%% file: paper.tex
\documentclass{article}

\usepackage[preprint]{neurips_2026}

\usepackage{microtype}
\usepackage{graphicx}
\usepackage{subcaption}
\usepackage{booktabs} 
\usepackage{algorithm}
\usepackage{algorithmic}

\newcommand{\RR}{\mathbb{R}}

\newcommand{\tr}{\operatorname{tr}}
\newcommand{\cV}{\mathcal{V}}
\newcommand{\cE}{\mathcal{E}}
\newcommand{\cK}{\mathcal{K}}
\newcommand{\cA}{\mathcal{A}}
\newcommand{\cR}{\mathcal{R}}
\newcommand{\cB}{\mathcal{B}}

\newcommand{\cS}{\mathcal{S}}
\newcommand{\dd}{\mathrm{d}}
\usepackage{tikz}

\usepackage{hyperref}
\usetikzlibrary{arrows.meta, positioning, calc, shapes.geometric, decorations.pathreplacing}

\usepackage{amsmath}
\usepackage{amssymb}
\usepackage{mathtools}
\usepackage{amsthm}
\usepackage{wrapfig}

\usepackage[capitalize,noabbrev]{cleveref}

\theoremstyle{plain}
\newtheorem{theorem}{Theorem}[section]
\newtheorem{proposition}[theorem]{Proposition}
\newtheorem{lemma}[theorem]{Lemma}
\newtheorem{corollary}[theorem]{Corollary}
\theoremstyle{definition}

\theoremstyle{remark}
\newtheorem{remark}[theorem]{Remark}

\usepackage[textsize=tiny]{todonotes}

\author{%
  Christian Pehle \\
  Cold Spring Harbor Laboratory\\ 
  Cold Spring Harbor, NY 11724 \\
  \texttt{pehle@cshl.edu} \\
  \And
  Jean-Jacques Slotine \\
  Massachusetts Institute of Technology\\
  Cambridge, MA 02139\\
  \texttt{jjs@mit.edu}
}

\title{Unlocked Backpropagation using Wave Scattering}

\begin{document}
\maketitle




\input{sections/00_abstract}
\input{sections/01_introduction}
\input{sections/02_variational_worldsheet}
\input{sections/03_discrete_scattering}
\input{sections/04_physical_interpretation}
\input{sections/05_related_work}
\input{sections/06_discussion}
\input{sections/07_limitations}

\bibliography{bibliography}
\bibliographystyle{plainnat}


\newpage

\appendix
\onecolumn

\input{appendix/A_worldsheet_relaxation}
\input{appendix/B_lqr_example}
\input{appendix/C_discrete_scattering_details}
\input{appendix/D_deep_linear_network}
\input{appendix/E_parameter_ports}
\input{appendix/F_continuous_delta_attention}
\input{appendix/G_physics_and_control}
\input{appendix/H_categorical_scattering}
\input{appendix/I_adaptive_nn_control}
\input{appendix/J_covariant_ccm_geometry}
\input{appendix/K_passive_wave_world_models}
\input{appendix/L_neural_ode_hnn_physics}
\input{appendix/M_classical_control_examples}
\input{appendix/N_text_flow_language_models}

\end{document}

%% file: sections/00_abstract.tex
\begin{abstract}
Both the backpropagation algorithm in machine learning and the maximum principle in optimal control theory are posed as a two-point boundary problem, resulting in a ``forward-backward'' lock. We derive a reformulation of the maximum principle in optimal control as a hyperbolic initial value problem by introducing an additional ``optimization time'' dimension. We introduce counter-propagating wave variables with finite propagation speed and recast the optimization problem in terms of scattering relationships between them. This relaxation of the original problem can be interpreted as a physical system that equilibrates and changes its physical properties in order to minimize reflections. We discretize this continuum theory to derive a family of fully unlocked algorithms suitable for training neural networks. Different parameter dynamics, including gradient descent, can be derived by demanding dissipation and minimization of reflections at parameter ports. 
Explicit physical sub-dynamics (i.e., analytical world-model elements) or constraints can be naturally incorporated, as they themselves correspond to known optimization problems. Our formulation automatically supports parallel, feedback, and hierarchical combinations of specialized subsystems, as well as communication delays between subsystems. It also suggests that backpropagation and more general distributed optimization problems can plausibly be implemented by biology.
\end{abstract}

%% file: sections/01_introduction.tex
\section{Introduction}

\begin{wrapfigure}{r}{0.5\textwidth}
\centering
\resizebox{0.49\textwidth}{!}{%
\begin{tikzpicture}[
    >={Stealth[length=2.5mm]},
    axis/.style={->, thick},
    char/.style={-{Stealth[length=2mm]}, thick, blue!60!black},
    charminus/.style={-{Stealth[length=2mm]}, thick, red!70!black},
    boundary/.style={very thick, black!80},
    sheet/.style={fill=blue!5, draw=black!30},
    port/.style={circle, draw=black!70, fill=white, inner sep=1.2pt},
    label/.style={font=\small},
    smalllabel/.style={font=\footnotesize},
]
\def\Tdepth{9}
\def\Taumax{6}
\fill[sheet] (0,0) rectangle (\Tdepth, \Taumax);
\draw[black!40] (0,0) rectangle (\Tdepth, \Taumax);
\foreach \k in {1,2,3,4,5,6,7,8} {
    \draw[gray!25, dashed, thin] (\k, 0) -- (\k, \Taumax);
}
\draw[axis] (-0.3, 0) -- (\Tdepth + 1.2, 0) node[right, label] {$t$ \footnotesize(depth)};
\draw[axis] (0, -0.3) -- (0, \Taumax + 0.8) node[above, label] {$\tau$};
\draw (0, 0.08) -- (0, -0.08);
\node[below left=1pt, smalllabel] at (0, 0) {$0$};
\draw (\Tdepth, 0.08) -- (\Tdepth, -0.08) node[below, smalllabel] {$T$};
\node[font=\large\itshape] at (\Tdepth/2, \Taumax + 1.4)
    {Worldsheet $\Sigma = [0,T]_t \times [0,\infty)_\tau$};
\draw[boundary, blue!50!black] (0, 0) -- (0, \Taumax);
\node[blue!50!black, label, rotate=90, anchor=south] at (-0.55, \Taumax/2)
    {$x(0,\tau) = x_0$};
\draw[boundary, red!60!black] (\Tdepth, 0) -- (\Tdepth, \Taumax);
\node[red!60!black, label, rotate=-90, anchor=south] at (\Tdepth + 0.50, 2.25)
    {$\lambda(T,\tau) = \nabla_x \Phi(x(T,\tau))$};
\draw[very thick, black!50, dashed] (0, 0) -- (\Tdepth, 0);
\node[smalllabel, anchor=south, black!60] at (\Tdepth/2, \Taumax + 0.1)
    {$\tau \to \infty$:\ \ $\mathcal{R}(z) \to 0$};
\draw[char] (0.8, 0.4) -- (3.0, 2.6);
\draw[char] (2.0, 0.4) -- (4.2, 2.6);
\draw[char] (3.2, 0.4) -- (5.4, 2.6);
\node[blue!60!black, label, anchor=south] at (4.0, 2.7) {$\omega_+$};
\draw[charminus] (5.9, 0.4) -- (3.7, 2.6);
\draw[charminus] (6.7, 0.4) -- (4.5, 2.6);
\draw[charminus] (8.4, 0.4) -- (6.2, 2.6);
\node[red!70!black, label, anchor=south] at (4.7, 2.7) {$\omega_-$};
\coordinate (Z) at (4.55, 1.75);
\node[port] at (Z) {};
\node[smalllabel, anchor=west, fill=blue!5, inner sep=1.2pt] at ($(Z) + (0.16, -0.05)$) {$z(t,\tau) = (x,\lambda,u)$};
\coordinate (Hit) at (\Tdepth, 4.5);
\draw[char, thick] (7.5, 3.0) -- (Hit);
\draw[charminus, thick] (Hit) -- (7.5, 6.0);
\node[port, fill=red!60!black] at (Hit) {};
\node[blue!60!black, smalllabel, anchor=east] at (7.4, 3.4) {incoming $\omega_+$};
\node[red!70!black, smalllabel, anchor=east] at (7.4, 5.6) {outgoing $\omega_-$};
\draw[->, thin, black!60] (\Tdepth + 0.95, 4.5) to[bend right=20] (\Tdepth + 0.05, 4.5);
\node[smalllabel, black!60, anchor=west, align=left] at (\Tdepth + 1.0, 4.5)
    {scattering\\at port};
\draw[->, thick, black!70] (1.2, 4.5) -- ++(0.5, 0.5);
\node[smalllabel, black!70, anchor=west] at (1.75, 4.75) {speed $c$};
\node[smalllabel, black!55, anchor=south west] at (0.2, 5.15)
    {$ds^2_\Sigma = -c^2\, d\tau^2 + dt^2$};
\node[smalllabel, anchor=north] at (\Tdepth/2, -0.55)
    {initial guess at $\tau = 0$};
\node[smalllabel, anchor=north, gray!70] at (\Tdepth/2, -1.05)
    {layers $k = 0, 1, \dots, N$};
\node[smalllabel, anchor=west] at (0.15, \Taumax + 0.65) {(optimization time)};
\end{tikzpicture}%
}
\caption{Wave scattering formulation of Backpropagation.}
\label{fig:worldsheet}
\end{wrapfigure}

The training of deep neural networks is traditionally dominated by the backpropagation algorithm. While highly successful, backpropagation imposes a strict sequential dependency known as the ``forward-backward lock'': the weights in earlier layers cannot be updated until the forward pass has reached the output and the error signal has propagated all the way back. This requirement for global locking prevents fully parallel updates. Furthermore, it lacks biological plausibility, as real synapses do not buffer past states while waiting for distant, global error signals to undergo plasticity. Indeed, enforcing this rigid orchestration and steep memory requirements is notoriously difficult to realize outside of a narrow set of conventional physical substrates (e.g. \cite{pai2023experimentally}).

Mathematically, backpropagation is the discrete realization of the Pontryagin Maximum Principle (PMP) \cite{bryson1969applied}\footnote{Sections 2.1 and 2.2}. It treats the training process as a Two-Point Boundary Value Problem (TPBVP): the network state is fixed at the input (time $t=0$), while the costate (adjoint) is fixed at the output (time $t=T$). Standard numerical methods often solve this iteratively, strictly enforcing the lock.

Physical systems do not solve boundary value problems instantaneously. When a mechanical force is applied to one end of a steel rod, the other end does not move immediately. Instead, a wave of stress propagates through the material at the speed of sound. The system reaches equilibrium not through a global solver, but through local interactions and the eventual dissipation of transient energy. In fact this is the principle by which any physical system supporting the propagation and dissipation of wave-like excitations adapts. Taking the finite speed of information propagation into account is also critical for feedback control and stable robotic teleoperation~\cite{niemeyer1991stable,wang2006contraction}, and our approach is inspired by the introduction of wave-variables in this context.

In this work, we propose a relaxation of the optimal control problem underpinning deep learning. By introducing an auxiliary ``optimization time'' dimension $\tau$, we lift the training dynamics from a trajectory on a line to a field theory on a $(1+1)$-dimensional worldsheet. We reinterpret the neural network not as a function composition, but as a physical transmission line. Signals propagate as counter-propagating waves with finite speed, and the enforcement of optimality conditions occurs via local wave scattering.

This reformulation yields the following contributions: \begin{enumerate} \item \textbf{Unlocked Parallelism:} We derive a hyperbolic relaxation of the Maximum Principle from its augmented action, in which layers are decoupled. Information propagates at a finite speed, allowing all layers to update their states and parameters simultaneously based solely on local variables. \item \textbf{Wave-PMP Algorithm:} We discretize this continuum theory into a concrete training algorithm (Alg.~\ref{alg:worldsheet_unlocked}) that handles variable layer widths via impedance-matching junctions, enabling fully asynchronous distributed training. \item \textbf{The Principle of Minimal Reflection:} We provide a physical derivation for parameter update laws. We show that Gradient Descent, Momentum, and Newton's method arise naturally from the boundary condition of minimizing wave reflections at the parameter ports. \end{enumerate}

This perspective suggests that the ``forward-backward lock'' is an artifact of taking the speed of information propagation to infinity. By restoring a finite speed of propagation, we recover a physical system that supports scattering, dissipation, and local relaxation. We believe these concepts may bridge the gap between backpropagation in artificial systems and physical learning systems.

%% file: sections/02_variational_worldsheet.tex
\section{Maximum Principle as Wave Scattering}
\label{sec:variational_worldsheet}

This section derives the wave-scattering formulation directly from
Pontryagin's maximum principle. We start from the continuous-time state and
adjoint equations, lift them onto a $(1+1)$-dimensional worldsheet
\begin{equation}
\label{eq:worldsheet}
\Sigma=[0,T]_t\times[0,\infty)_\tau,
\end{equation}
where $t$ is physical time (or network depth) and $\tau$ is auxiliary
relaxation time, and derive the governing wave equations. This continuous
picture makes three structural facts transparent: (i) where interior bulk
reflections come from, (ii) how damping enters through the ports, and
(iii) how penalizing feasibility defects in the action produces viscoelastic
bulk damping.

\subsection{From state and adjoint equations to the worldsheet}

Consider the continuous-time optimal-control problem
\begin{equation}
\label{eq:continuous_cost}
\min_{x,u}\; J[x,u]
= \Phi(x(T)) + \int_0^T L(x(t),u(t),t)\,\dd t,
\qquad
\partial_t x=f(x,u,t),
\qquad
x(0)=x_0,
\end{equation}
with state $x(t)\in\RR^n$, control $u(t)$, and Hamiltonian
$H(x,\lambda,u,t)=L(x,u,t)+\lambda^\top f(x,u,t)$. Stationarity of the
Pontryagin action
\begin{equation}
\label{eq:augmented_action}
\cS[x,\lambda,u]=\Phi(x(T))+\int_0^T\bigl(H(x,\lambda,u,t)-\lambda^\top\partial_t x\bigr)\,\dd t
\end{equation}
yields the two-point boundary-value problem of the maximum principle:
\begin{align}
\label{eq:continuous_dynamics}
\partial_t x &= \nabla_\lambda H = f(x,u,t), & x(0) &= x_0,\\
\label{eq:costate_dynamics}
\partial_t\lambda &= -\nabla_x H, & \lambda(T) &= \nabla_x\Phi(x(T)),\\
\label{eq:control_stationarity}
0 &= \nabla_u H .
\end{align}
Eliminating $u$ through \eqref{eq:control_stationarity} defines the reduced
Hamiltonian $H(x,\lambda,t)$; controls and learnable parameters return as
ports in Section~\ref{sec:minimal_reflection}. Collecting $y=(x,\lambda)$
and $X_H=(\nabla_\lambda H,-\nabla_x H)$, the violation of the state and
adjoint equations along a candidate trajectory is measured by the residuals
\begin{equation}
\label{eq:nonlinear_residual_operator_main}
r=\partial_t y-X_H(y),\qquad
r_x=\partial_t x-\nabla_\lambda H,\qquad
r_\lambda=\partial_t\lambda+\nabla_x H,
\end{equation}
which are directly the functional derivatives of the action
\eqref{eq:augmented_action}: $\delta\cS/\delta x=r_\lambda$ and
$\delta\cS/\delta\lambda=-r_x$.

To relax the two-point boundary-value problem locally in time, we lift
$y(t)$ to a field $y(t,\tau)$ on the worldsheet $\Sigma$. Given a symmetric
positive-definite metric $G\succ0$ on the state space and a finite
propagation speed $c>0$, we descend in $x$ and ascend in $\lambda$:
\begin{equation}
\label{eq:para_floer}
\boxed{\;\partial_\tau y+c\,\mathbb K\bigl(\partial_t y-X_H(y)\bigr)=0,\qquad
\mathbb K=\begin{pmatrix}0&G^{-1}\\ G&0\end{pmatrix}\;}
\qquad\Longleftrightarrow\qquad
\begin{aligned}
\partial_\tau x&=-c\,G^{-1}r_\lambda,\\
\partial_\tau\lambda&=-c\,G\,r_x .
\end{aligned}
\end{equation}
Its stationary states ($\partial_\tau x=\partial_\tau\lambda=0$) are
precisely the solutions of the maximum principle
\eqref{eq:continuous_dynamics}--\eqref{eq:costate_dynamics}. Updating $x$
against the adjoint residual and $\lambda$ against the state residual with
matching signs ($\mathbb K^2=+I$) gives a hyperbolic primal--dual wave
system with characteristic speeds $\pm c$, structurally identical to
voltage and current on a transmission line: it can be evolved forward from
any initial profile on $[0,T]$~\cite{nguyen2018contraction}. Flipping one
sign ($J^2=-I$) instead produces an elliptic (Laplace-like) system whose
spatial modes grow as $e^{|k|\tau}$ and cannot be solved as an
initial-value problem in $\tau$.\footnote{In symplectic geometry, the
elliptic sign choice ($J^2=-I$) is Floer's equation~\cite{floer1988morse}
for strips connecting the two boundary manifolds; flipping the sign to
$\mathbb K^2=+I$ replaces its Riemannian $(t,\tau)$ domain by a Lorentzian
wave equation (Appendix~\ref{app:paracomplex_proofs}).}

\begin{figure}[t]
\centering
\includegraphics[width=\textwidth]{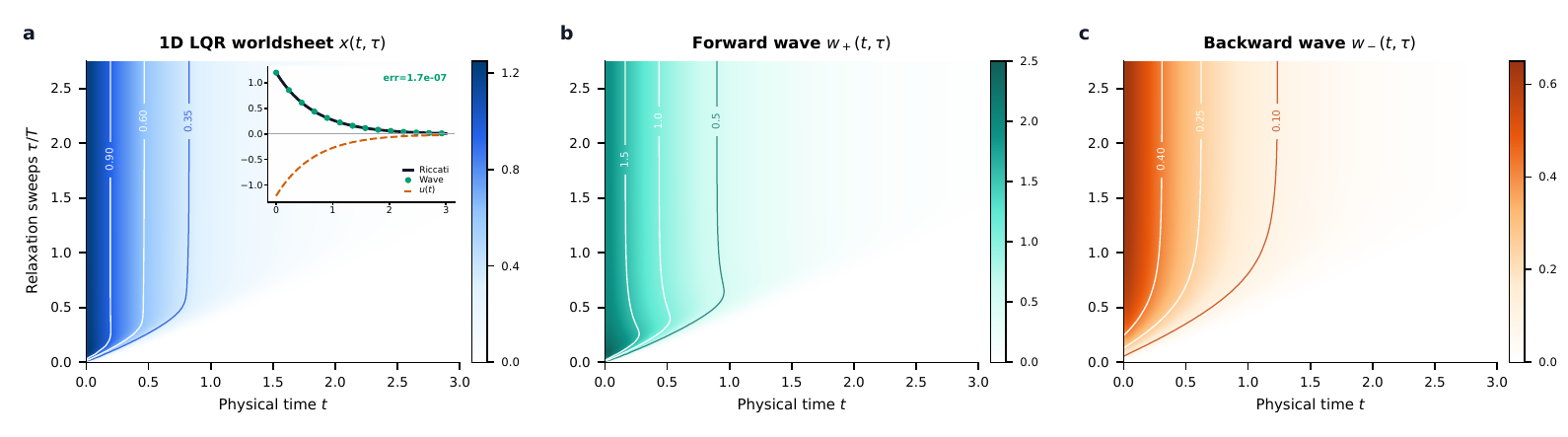}
\caption{\textbf{$(1+1)$-dimensional worldsheet relaxation and wave variables on a 1D finite-horizon LQR problem.}
Scalar LQR ($\dot x = -0.4x + u$, running costs $q=2$, $r=1$, terminal cost $p_T=3$, initial state $x(0)=1.2$ on $t\in[0,3]$) relaxed from a cold zero profile by the hyperbolic flow \eqref{eq:para_floer} with scalar metric $G=\theta^2$ ($\theta=1.5$). The wave variables \eqref{eq:residual_characteristic_waves_main} are $w_\pm(t,\tau)=\tfrac{1}{\sqrt{2}}\bigl(\theta x(t,\tau)\pm\theta^{-1}\lambda(t,\tau)\bigr)$, with exact inverse $x=\tfrac{1}{\sqrt{2}\theta}(w_+ + w_-)$ and costate $\lambda=\tfrac{\theta}{\sqrt{2}}(w_+ - w_-)$ (determining the control $u=-r^{-1}\lambda$).
\textbf{a},~State worldsheet $x(t,\tau)=\tfrac{1}{\sqrt{2}\theta}(w_+ + w_-)$ with iso-contours; inset: the settled wave trajectory $x(t)$ (teal dots) and control $u(t)$ (dashed vermilion) match the exact analytical Riccati solution (solid black) to $1.7\times 10^{-7}$.
\textbf{b},~Forward wave $w_+(t,\tau)=\tfrac{1}{\sqrt{2}}(\theta x+\theta^{-1}\lambda)$, injected at the clamped $t=0$ boundary via $w_+(0,\tau)=\sqrt{2}\theta\,x(0)-w_-(0,\tau)$ and propagating forward at speed $+c$ while attenuated by running cost and control ports.
\textbf{c},~Backward wave $w_-(t,\tau)=\tfrac{1}{\sqrt{2}}(\theta x-\theta^{-1}\lambda)$, generated by bulk scattering from $w_+$ (and terminal reflection $w_-(T,\tau)=\tfrac{G-p_T}{G+p_T}w_+(T,\tau)$) and propagating backward at speed $-c$ toward $t=0$, where it adjusts $w_+(0,\tau)$ and sets the costate $\lambda$ and feedback control $u(t)$.
See Appendix~\ref{app:lqr} for the full LQR derivation and Appendix~\ref{app:classical_control_examples} for nonlinear control, relay switching, differential-game, and filtering examples.}
\label{fig:lqr_worldsheet_sec2}
\end{figure}

Figure~\ref{fig:lqr_worldsheet_sec2} illustrates this relaxation directly on a scalar finite-horizon linear--quadratic regulator (LQR; Appendix~\ref{app:lqr}): information enters at the $t=0$ state boundary and propagates across the horizon $[0,T]$ via counter-propagating wave variables $w_\pm$ over $O(1)$ relaxation sweeps $\tau/T$ before settling onto the exact analytical Riccati solution (error $1.7\times 10^{-7}$). The same worldsheet dynamics extends without structural change to classical nonlinear, bang--bang, state-constrained, differential-game, and nonlinear-filtering benchmarks (Appendix~\ref{app:classical_control_examples}).

\subsection{Wave variables, bulk reflection, and port damping}

Factoring the metric as $G=\Theta^\top\Theta$ diagonalizes the hyperbolic
operator $\mathbb K$ into counter-propagating \emph{wave variables}. For
any state--costate perturbation (or vector-space value) $(\xi,\eta)$, define
\begin{equation}
\label{eq:residual_characteristic_waves_main}
w_\pm(\xi,\eta)=\tfrac{1}{\sqrt2}\left(\Theta\xi\pm\Theta^{-\top}\eta\right),
\qquad
\eta^\top\xi=\tfrac12\left(\|w_+\|^2-\|w_-\|^2\right).
\end{equation}
The bilinear pairing $\eta^\top\xi$ is the power flowing through a cut in
$t$; wave variables split it into incoming forward power $\tfrac12\|w_+\|^2$
and returning backward power $\tfrac12\|w_-\|^2$. For constant $G$,
projecting \eqref{eq:para_floer} onto the two polarizations gives the
first-order wave equations
\begin{equation}
\label{eq:characteristic_transport_main}
(\partial_\tau\pm c\,\partial_t)\,w_\pm=\pm c\,w_\pm(X_H),
\end{equation}
where $w_+$ propagates forward in $t$ at speed $+c$, $w_-$ propagates
backward at speed $-c$, and the Hamiltonian $X_H$ couples them locally
(Figure~\ref{fig:lqr_worldsheet_sec2}\textbf{b}--\textbf{c}).

\paragraph{Bulk reflection and port damping.}
Linearizing \eqref{eq:para_floer} around a trajectory---now allowing a
smoothly varying frame $\Theta(t)$ with metric $G(t)=\Theta(t)^\top\Theta(t)$---reveals
the exact physical role of each term in $H$. Let $A(t)=\nabla_x f$ be the
local state Jacobian, $Q(t)=H_{xx}\succeq0$ the running state-cost curvature
(the readout port), and $S(t)=-H_{\lambda\lambda}\succeq0$ the
control/input compliance (the write or control port). Transforming
$\delta p=\Theta\delta x$, $\delta q=\Theta^{-\top}\delta\lambda$, and
$\delta w_\pm=\tfrac{1}{\sqrt2}(\delta p\pm\delta q)$ yields the
linearized wave equations
\begin{align}
\label{eq:bulk_port_waves_plus}
\bigl(\partial_\tau+c(\partial_t-\widehat A_{\mathrm{skew}})\bigr)\delta w_+
&=-\frac{c}{2}(K+L)\,\delta w_+
  -\frac{c}{2}\bigl(K-L-2\widehat A_{\mathrm{sym}}\bigr)\delta w_-,\\
\label{eq:bulk_port_waves_minus}
\bigl(\partial_\tau-c(\partial_t-\widehat A_{\mathrm{skew}})\bigr)\delta w_-
&=-\frac{c}{2}\bigl(K-L+2\widehat A_{\mathrm{sym}}\bigr)\delta w_+
  -\frac{c}{2}(K+L)\,\delta w_-,
\end{align}
where $K=\Theta^{-\top}Q\Theta^{-1}$ and $L=\Theta S\Theta^\top$ are the
normalized port operators, and
$\widehat A=\Theta A\Theta^{-1}+\dot\Theta\,\Theta^{-1}
=\widehat A_{\mathrm{skew}}+\widehat A_{\mathrm{sym}}$
is the drift generator in the moving frame $\Theta(t)$, decomposed into its
skew-symmetric and symmetric parts. Equations
\eqref{eq:bulk_port_waves_plus}--\eqref{eq:bulk_port_waves_minus} separate
three distinct physical effects:
\begin{enumerate}
\item \emph{Bulk reflection is a geometric metric-compatibility condition.}
The symmetric drift
$\widehat A_{\mathrm{sym}}=\tfrac12\Theta^{-\top}(\mathcal L_A G)\,\Theta^{-1}$
enters off-diagonally with opposite signs ($\pm c\widehat A_{\mathrm{sym}}\delta w_\mp$):
it does no net damping, but instead acts as a local mass barrier that reflects
forward waves into backward waves. Bulk propagation is reflection-free
($\widehat A_{\mathrm{sym}}=0$) if and only if the metric $G(t)$ is
isometrically transported along the flow:
\begin{equation}
\label{eq:continuous_isometry_condition}
\mathcal L_A G \equiv \partial_t G + A^\top G + G A = 0 .
\end{equation}
In that frame, $\widehat A=\widehat A_{\mathrm{skew}}$ is purely a rotation
(a gauge connection) that transports $w_+$ and $w_-$ without scattering.
\item \emph{All dissipation occurs on the ports.}
The diagonal damping $-\tfrac{c}{2}(K+L)\delta w_\pm$ comes exclusively from
the distributed ports $Q$ and $S$, together with the boundary ports at
$t=0$ and $t=T$.
\item \emph{Port reflection vanishes at matched impedance.}
The port cross-term $-\tfrac{c}{2}(K-L)$ vanishes when the metric balances
the two ports ($K=L$, i.e.\ $GSG=Q$), and the terminal reflection vanishes
when $G(T)=\nabla^2\Phi$.
\end{enumerate}

\begin{theorem}[Wave-energy balance]
\label{thm:nonlinear_wave_energy_balance}
Let $G(t)=\Theta(t)^\top\Theta(t)\succ0$ be any smoothly varying spatial
metric ($\partial_\tau G=0$) and let $\delta y=(\delta x,\delta\lambda)$ be a
virtual displacement of a solution of \eqref{eq:para_floer}, with waves
$\delta w_\pm=w_\pm(\delta x,\delta\lambda)$. The wave energy
\begin{equation}
\label{eq:full_nonlinear_wave_energy_main}
\cE_{\rm wave}(\tau)=\tfrac12\int_0^T\bigl(\|\delta w_+\|^2+\|\delta w_-\|^2\bigr)\dd t
\end{equation}
satisfies
\begin{equation}
\label{eq:full_nonlinear_wave_energy_balance_main}
\frac{\dd\cE_{\rm wave}}{\dd\tau}
=\tfrac{c}{2}\left(\|\delta w_+(0)\|^2-\|\delta w_-(0)\|^2\right)
+\tfrac{c}{2}\left(\|\delta w_-(T)\|^2-\|\delta w_+(T)\|^2\right)
-D_\Sigma(\tau),
\end{equation}
\begin{equation}
\label{eq:interior_dissipation_main}
D_\Sigma(\tau)=c\int_0^T\bigl(\delta x^\top H_{xx}\,\delta x-\delta\lambda^\top H_{\lambda\lambda}\,\delta\lambda\bigr)\dd t,
\end{equation}
with the Hessians of the reduced Hamiltonian evaluated along the field.
Both the drift block $H_{x\lambda}=\nabla_x f$ and the metric Lie derivative
$\mathcal L_A G$ drop out identically: the bulk dynamics and spatial metric
variation neither create nor dissipate wave energy. The boundary ports
$x(0)=x_0$ and $\lambda(T)=\nabla\Phi$ are passive for convex $\Phi$, and
$\cE_{\rm wave}$ decays strictly wherever the ports provide positive
damping ($H_{xx}\succeq0$, $-H_{\lambda\lambda}\succeq0$).
\end{theorem}

For linear--quadratic control, $H_{xx}=Q$ and $-H_{\lambda\lambda}=S$ are
constant positive-semidefinite port conductances, so
Theorem~\ref{thm:nonlinear_wave_energy_balance} guarantees global
contraction to the unique optimum. In many neural-network and sequence
problems, however, interior port costs are small or absent ($Q=S=0$) and
the nonlinear curvature $\lambda^\top\nabla_x^2 f$ inside $H_{xx}$ is
indefinite. This naturally raises the question of how to add damping to the
worldsheet. The proof of Theorem~\ref{thm:nonlinear_wave_energy_balance} is
in Appendix~\ref{app:residual_relaxation_details}.

\subsection{Adding damping: ports versus bulk defect penalties}
\label{sec:augmented_viscoelastic}

There are two natural ways to add dissipation to \eqref{eq:para_floer},
which behave very differently with respect to locality and propagation speed.

\paragraph{Damping on the ports (strict finite speed).}
The first route keeps the bulk relaxation \eqref{eq:para_floer} strictly
first-order and hyperbolic, and supplies all dissipation through the ports:
passive or impedance-matched boundary terminations at $t=0$ and $t=T$,
distributed state/control ports ($Q,S\succeq0$), resistive parameter ports
(Section~\ref{sec:minimal_reflection}), and local wave averaging in the
discrete scheme (Section~\ref{sec:training_neural_network}). Because no
higher spatial derivatives are introduced, every interaction remains
strictly finite-speed and nearest-neighbor, with no instantaneous coupling
across the domain.

\paragraph{Adding bulk damping via an augmented action.}
A natural alternative when interior port damping is weak ($Q=S=0$) is to
penalize the feasibility residual $r_x=\partial_t x-f(x,u,t)$ directly in
the Pontryagin action \eqref{eq:augmented_action}:
\begin{equation}
\label{eq:augmented_defect_action}
\cS_\rho[x,\lambda,u]
=\Phi(x(T))+\int_0^T\Bigl(L(x,u,t)-\lambda^\top r_x
  +\frac{\rho}{2}\|r_x\|_P^2\Bigr)\dd t,
\qquad \rho\ge0,\; P\succ0 .
\end{equation}
Defining the \emph{effective costate} $\tilde\lambda\equiv\lambda-\rho P r_x$
(which coincides with $\lambda$ at every feasible trajectory $r_x=0$, with
terminal law $\tilde\lambda(T)=\nabla_x\Phi(x(T))$), the first variations
are $\delta\cS_\rho/\delta\lambda=-r_x$ and
$\delta\cS_\rho/\delta x=r_\lambda-\rho(\partial_t+D_x f^\top)(P r_x)$.
The resulting primal--dual worldsheet flow is
\begin{equation}
\label{eq:augmented_primal_dual_flow}
G\,\partial_\tau x=-c\,r_\lambda+c\,\rho\,(\partial_t+D_x f^\top)(P r_x),
\qquad
G^{-1}\partial_\tau\lambda=-c\,r_x .
\end{equation}
Linearizing around $f(x)=Ax$, $L=\tfrac12 x^\top Q x$ with $G=P=I$, and
writing $\mathcal D=\partial_t-A$ and $\mathcal D^*=-\partial_t-A^\top$,
eliminating $\lambda$ from \eqref{eq:augmented_primal_dual_flow} yields
the second-order worldsheet PDE
\begin{equation}
\label{eq:kelvin_voigt_worldsheet}
\partial_\tau^2 x
+ c\,(Q+\rho\,\mathcal D^*\mathcal D)\,\partial_\tau x
+ c^2\,\mathcal D^*\mathcal D\,x = 0 .
\end{equation}
For $A=0$ ($\mathcal D^*\mathcal D=-\partial_t^2$), this is the
\emph{Kelvin--Voigt strongly damped wave equation},
$\partial_\tau^2 x+cQ\,\partial_\tau x-c\rho\,\partial_t^2\partial_\tau x
-c^2\partial_t^2 x=0$, with energy balance
\begin{equation}
\label{eq:augmented_energy_balance}
\begin{aligned}
\frac{\dd}{\dd\tau}\frac12\int_0^T\bigl(\|\delta x\|^2+\|\delta\lambda\|^2\bigr)\dd t
&=-c\,\delta x(T)^\top\nabla^2\Phi\,\delta x(T)\\
&\quad -c\int_0^T\delta x^\top Q\,\delta x\,\dd t
 -c\,\rho\int_0^T\|\mathcal D\,\delta x\|^2\,\dd t \;\le\; 0 .
\end{aligned}
\end{equation}
Even when $Q=0$, the penalty $\rho>0$ damps every nonzero feasibility defect
$\|\mathcal D\,\delta x\|^2$ directly in the bulk; its dispersion relation
($s^2+c\rho k^2 s+c^2 k^2=0$ for $A=Q=0$) produces underdamped waves at
long wavelengths ($k<2/\rho$) and diffusive smoothing at short wavelengths
($k>2/\rho$). However, because $\mathcal D^*\mathcal D$ is second-order in
$\partial_t$, bulk defect damping introduces instantaneous spatial coupling
(parabolic smoothing that breaks strict finite propagation speed and couples
two neighbors simultaneously in discrete time). To preserve strictly
finite-speed, unlocked propagation without instantaneous coupling, we
therefore keep $\rho=0$ in the bulk and place damping on the ports.

\paragraph{Worldsheet action, metric Lie derivative, and directional splittings.}
At $\rho=0$, the first-order flow \eqref{eq:para_floer} also sits inside
the second-order worldsheet action (with $\mathbb M(t)=\operatorname{diag}(G(t),G(t)^{-1})$)
\begin{equation}
\label{eq:nonlinear_worldsheet_action}
\cS_\Sigma[y]
=\frac12\int_0^\infty\!\!\int_0^T
\Bigl(c^{-2}\|\partial_\tau y\|_{\mathbb M(t)}^2-\|\partial_t y-X_H(y)\|_{\mathbb M(t)}^2\Bigr)
\dd t\,\dd\tau ,
\end{equation}
whose solutions satisfy the damped wave equation
\begin{equation}
\label{eq:full_nonlinear_worldsheet_relaxation}
\frac{1}{c^2}\mathbb M(t)\,\partial_\tau^2y
=-\frac{\delta\cS_\Sigma^{\rm pot}}{\delta y}
-\frac{2}{c}\begin{pmatrix}H_{xx}&0\\0&-H_{\lambda\lambda}\end{pmatrix}\partial_\tau y
\end{equation}
(Appendix~\ref{app:residual_relaxation_details}). While the underlying
Pontryagin action \eqref{eq:augmented_action} and the wave-energy balance
(Theorem~\ref{thm:nonlinear_wave_energy_balance}) are independent of the
spatial metric choice $G(t)$, the Lie derivative $\mathcal L_A G$ directly
reshapes the potential energy of the relaxation action
\eqref{eq:nonlinear_worldsheet_action}. Writing the orthogonal covariant
derivative $D_t\equiv\partial_t-\widehat A_{\mathrm{skew}}$ and
$D_t\widehat A_{\mathrm{sym}}\equiv\dot{\widehat A}_{\mathrm{sym}}+[\widehat A_{\mathrm{sym}},\widehat A_{\mathrm{skew}}]$,
the linearized spatial residual energy in the moving frame $(\delta p,\delta q)$
with $Q=S=0$ expands identically as (Appendix~\ref{app:paracomplex_proofs}):
\begin{equation}
\label{eq:moving_frame_action_mass}
\|D_t\delta p-\widehat A_{\mathrm{sym}}\delta p\|^2
=\underbrace{\|D_t\delta p\|^2}_{\text{massless free transport}}
+\underbrace{\delta p^\top\bigl(\widehat A_{\mathrm{sym}}^2+D_t\widehat A_{\mathrm{sym}}\bigr)\delta p}_{\text{effective mass / reflection barrier}}
-\partial_t\bigl(\delta p^\top\widehat A_{\mathrm{sym}}\delta p\bigr),
\end{equation}
and analogously $\|D_t\delta q+\widehat A_{\mathrm{sym}}\delta q\|^2=\|D_t\delta q\|^2+\delta q^\top(\widehat A_{\mathrm{sym}}^2-D_t\widehat A_{\mathrm{sym}})\delta q+\partial_t(\delta q^\top\widehat A_{\mathrm{sym}}\delta q)$.
Whenever $\mathcal L_A G=0$, the relaxation action is \emph{massless}:
waves travel along the rotating frame $D_t$ with zero bulk
reflection. Whenever $\mathcal L_A G\neq0$, the metric mismatch injects an
effective Klein--Gordon mass barrier $\widehat A_{\mathrm{sym}}^2\pm D_t\widehat A_{\mathrm{sym}}$
that preserves total wave energy yet reflects low-frequency modes back into
internal transients. More generally, the wave
transform \eqref{eq:residual_characteristic_waves_main} uses a single metric
$G$ for both directions of travel; replacing it by two independent
symmetric operators $Y(t),Z(t)\succeq0$ defines \emph{directional waves}
$w_+=x+Z\lambda$ and $w_-=\lambda-Yx$. Requiring the continuous-time
reflection between $w_+$ and $w_-$ to vanish yields the coupled control and
filter Riccati equations along $t$, eliminating reflections even when the
isometry condition \eqref{eq:continuous_isometry_condition} fails, at the
cost of giving up the diagonal energy of
Theorem~\ref{thm:nonlinear_wave_energy_balance}.

\paragraph{State-dependent metrics and Control Contraction Metrics.}
Allowing the metric $M(x,t)=G(x,t)$ and its inverse (dual metric) $W=M^{-1}$ to vary with state connects \eqref{eq:bulk_port_waves_plus}--\eqref{eq:bulk_port_waves_minus} directly to Control Contraction Metrics (CCMs) \citep{manchester2017control} (Appendix~\ref{app:covariant_ccm_geometry}). In this setting, the two off-diagonal reflection blocks coincide with the Control and Filter Riccati operators; relaxing the former in the dual metric $W$ yields the convex dual-CCM inequality of \citet{manchester2017control}, while requiring the control fields to be Killing ($\mathcal{L}_{b_i}M=0$) makes bulk reflection independent of the input $u$. Lifting the worldsheet flow \eqref{eq:para_floer} covariantly with the Levi-Civita connection of $M(x)$ makes the wave-energy balance coordinate-invariant, eliminating indefinite Euclidean Hessian terms for contracting flows on curved manifolds. Finally, pairing bulk contraction with an active control port converts the low-frequency reflection barrier into a horizon-independent spectral gap as $T\to\infty$.

%% file: sections/03_discrete_scattering.tex
\section{The Discrete Setting: Discretization, Scattering, and Composability}
\label{sec:training_neural_network}

We now translate the continuum worldsheet into discrete steps
$k=0,\ldots,N-1$---representing either the layers of a deep network or
individual time steps along a discrete sequence. We proceed in four steps:
(i) discretizing one step into a local two-port scattering junction,
(ii) showing that single-metric reflection-freeness selects the
\emph{pullback metric} (orthogonal normalized Jacobians) and examining what
changes for non-orthogonal and nonlinear maps, (iii) composing many layers
via the \emph{Redheffer star product} and opening each layer into a
zero-inverse multi-stage Lie--SVD hybrid flow, and (iv) extending from
chains to forks, joins, residual connections, and general branching
architectures.

\subsection{Discretizing the worldsheet into local scattering junctions}

We first ask ourselves how to discretize the continuous worldsheet
\eqref{eq:para_floer} into a discrete algorithm. Two core approaches
stand out directly:
\begin{enumerate}
\item \emph{Explicit finite-difference (upwind) schemes.}
Discretizing \eqref{eq:characteristic_transport_main} on a fixed grid
at Courant number $\nu=c\,\Delta\tau/\Delta t\le 1$
\citep{leveque2002finite} is simple when state dimension is constant,
but adds numerical diffusion for $\nu<1$, leaves an $O(\Delta t)$
equilibrium bias on nonconstant profiles, imposes stiff explicit source
limits, and does not directly handle width changes ($n_k\neq n_{k+1}$).
\item \emph{Characteristic (Courant-one) scattering junctions.}
Stepping directly along the characteristics ($c\,\Delta\tau=\Delta t$)
turns free transport into a diffusion-free one-hop shift while enforcing
each step's exact forward and adjoint relations inside a local two-port
scatterer. This matches the bilinear (Cayley) port discretization of
wave digital filters \citep{fettweis1986wave} and, viewing each layer as
a monotone saddle operator, coincides with reflected-resolvent splitting
\citep{lions1979splitting,eckstein1992douglas} (Peaceman--Rachford at
$\omega=1$ and Douglas--Rachford at $\omega=\tfrac12$).
\end{enumerate}
More specialized discretizations also fit the same worldsheet---including
staggered Yee grids \citep{yee1966numerical}, higher-order discontinuous
Galerkin fluxes \citep{cockburn2001runge}, multisymplectic spacetime-cell
integrators \citep{marsden2001discrete,bridges2001multi}, and
red--black, asynchronous, or $O(\log N)$ hierarchical Redheffer
schedules \citep{redheffer1960supplementary,staff2005stability,FriedhoffEtAl2013}.
We focus below on characteristic scattering junctions because they avoid
transport diffusion, handle both discretized ODEs and arbitrary
width-changing neural layers, and preserve exact discrete backpropagation
fixed points.
Concretely, across step $k$, the discrete dynamics take the form
\begin{equation}
\label{eq:layer_map_main}
x_{k+1}=f_k(x_k;\theta_k),\qquad k=0,\ldots,N-1,
\end{equation}
minimizing the objective
\begin{equation}
\label{eq:discrete_objective_main}
J(\theta)=\ell(x_N,y)+\sum_{k=0}^{N-1}R_k(x_k,\theta_k),
\end{equation}
with discrete Hamiltonian
\begin{equation}
\label{eq:discrete_hamiltonian}
H_k(x_k,\lambda_{k+1},\theta_k)
=\lambda_{k+1}^\top f_k(x_k;\theta_k)+R_k(x_k,\theta_k).
\end{equation}
The discrete state, adjoint, and parameter stationarity residuals are
\begin{align}
\label{eq:node_residuals}
r_{x,0} &= x_0-x_{\mathrm{in}},\\
r_{x,k+1} &= x_{k+1}-f_k(x_k;\theta_k),\qquad k=0,\dots,N-1,\\
r_{\lambda,N} &= \lambda_N-\nabla_{x_N}\ell(x_N,y),\\
r_{\lambda,k} &= \lambda_k-J_{x,k}^\top\lambda_{k+1}-\nabla_{x_k}R_k(x_k,\theta_k),\qquad k=0,\dots,N-1,\\
\label{eq:param_residual}
r_{\theta,k} &= J_{\theta,k}^\top\lambda_{k+1}+\nabla_{\theta_k}R_k(x_k,\theta_k),\qquad k=0,\dots,N-1,
\end{align}
where $J_{x,k}=\nabla_{x_k}f_k$ and $J_{\theta,k}=\nabla_{\theta_k}f_k$.
Their simultaneous zero is the discrete Pontryagin maximum principle
(ordinary backpropagation \cite{lecun1988theoretical,bryson1969applied}).

At every node $k$, we choose a metric $G_k=\Theta_k^\top\Theta_k\succ0$ and
define algebraic state--costate wave variables
\begin{equation}
\label{eq:disc_wave_def}
w_{\pm,k}=\frac{1}{\sqrt2}\left(\Theta_k x_k\pm\Theta_k^{-\top}\lambda_k\right),
\end{equation}
with exact inverse
\begin{align}
\label{eq:disc_wave_inverse}
x_k &= \frac{1}{\sqrt2}\Theta_k^{-1}(w_{+,k}+w_{-,k}),\\
\lambda_k &= \frac{1}{\sqrt2}\Theta_k^\top(w_{+,k}-w_{-,k}),
\end{align}
and port-power identity
\begin{equation}
\label{eq:disc_wave_power}
\lambda_k^\top x_k=\frac12\left(\|w_{+,k}\|^2-\|w_{-,k}\|^2\right).
\end{equation}

Each step $k$ acts as a two-port scattering junction: given incoming waves
$u=w_{+,k}$ from the left and $v=w_{-,k+1}$ from the right, it solves its
own local forward and adjoint equations for the outgoing waves
$y=w_{+,k+1}$ and $z=w_{-,k}$:
\begin{align}
\label{eq:nonlinear_layer_scattering_main}
\frac{1}{\sqrt2}\Theta_{k+1}^{-1}(y+v)
&=f_k\left(\frac{1}{\sqrt2}\Theta_k^{-1}(u+z);\theta_k\right),\\
\frac{1}{\sqrt2}\Theta_k^\top(u-z)
&=J_{x,k}(\xi_k;\theta_k)^\top\frac{1}{\sqrt2}\Theta_{k+1}^\top(y-v)
+\nabla_{x_k}R_k(\xi_k,\theta_k),
\end{align}
where $\xi_k=\frac{1}{\sqrt2}\Theta_k^{-1}(u+z)$. Every step depends only
on its own local map and immediate neighbor ports, so all $N$ steps update
in parallel (Algorithm~\ref{alg:worldsheet_unlocked}). For any smooth map
(including width changes $n_k\neq n_{k+1}$), every fixed point of this
local iteration satisfies \eqref{eq:node_residuals}--\eqref{eq:param_residual}
exactly (Appendix~\ref{app:discrete_scattering_details}).

\paragraph{Hybrid flows and discrete jumps in one action.}
The continuous dynamics of Section~\ref{sec:variational_worldsheet} and the
discrete maps \eqref{eq:layer_map_main} are two limits of a single
\emph{hybrid} optimal-control action interleaving continuous ODE segments on
$[\underline{t}_k,\overline{t}_k]$ ($k=0,\dots,N$, duration
$\Delta t_k=\overline{t}_k-\underline{t}_k\ge 0$, arbitrary widths $n_k$)
with discrete jumps (layers, width projections, or physical contacts)
$x_{k+1}^+=\phi_k(x_k^-;\theta_k)$ ($k=0,\dots,N-1$):
\begin{equation}
\label{eq:hybrid_pontryagin_action}
\begin{aligned}
\cS_{\mathrm{hyb}}[x,\lambda,u,\theta]
&=\Phi(x_N^-)
+(\lambda_0^+)^\top(x_{\mathrm{in}}-x_0^+)
+\sum_{k=0}^{N}\int_{\underline{t}_k}^{\overline{t}_k}\bigl(H_k-\lambda_k^\top\partial_t x_k\bigr)\,\dd t \\
&\quad+\sum_{k=0}^{N-1}\Bigl((\lambda_{k+1}^+)^\top\bigl(\phi_k(x_k^-;\theta_k)-x_{k+1}^+\bigr)+R_k(x_k^-,\theta_k)\Bigr),
\end{aligned}
\end{equation}
where $x_k^+=x_k(\underline{t}_k)$, $x_k^-=x_k(\overline{t}_k)$, and likewise
for $\lambda_k^\pm$. Integrating $-\lambda_k^\top\partial_t\delta x_k$ by
parts on each stage cancels every post-jump variation $\delta x_{k+1}^+$
identically against the jump multiplier $\lambda_{k+1}^+$ and leaves the bulk
PMP equations on $(\underline{t}_k,\overline{t}_k)$ coupled to the discrete
jump equations $x_{k+1}^+=\phi_k(x_k^-;\theta_k)$,
$\lambda_k^-=J_{x,k}^\top\lambda_{k+1}^++\nabla_{x_k^-}R_k$ at every seam.
Setting $N=0$ recovers the pure continuous action \eqref{eq:augmented_action}.
Conversely, setting $H_k\equiv 0$ (free identity wires) leaves stationary
points constant across every stage ($x_k^+=x_k^-$, $\lambda_k^+=\lambda_k^-$),
recovering the exact discrete multi-layer PMP equations
\eqref{eq:discrete_objective_main}--\eqref{eq:param_residual} for any
wire transit time $\tau_k=\Delta t_k/c$: shrinking $\tau_k\to 0$ collapses
adjacent jumps instantaneously into locked backpropagation, whereas holding
$\tau_k=\Delta\tau$ finite (with half-step boundary wires $\tau_0=\tau_N=\tfrac12\Delta\tau$)
turns each free strip into an exact characteristic delay line
$w_{+,k}^-(\tau)=w_{+,k}^+(\tau-\Delta\tau)$,
$w_{-,k}^+(\tau)=w_{-,k}^-(\tau-\Delta\tau)$ and yields Algorithm~\ref{alg:worldsheet_unlocked}
at $\omega=1$ verbatim. Because affine jumps yield orthogonal two-port
scatterers (Section~\ref{sec:discrete_orthogonal_reflection}), inserting
discrete junctions into a continuous worldsheet leaves the global
wave-energy balance (Theorem~\ref{thm:nonlinear_wave_energy_balance}) intact:
continuous segments transport waves via $\partial_t-X_H$, while discrete
jumps scatter them conservatively across each interface.

\begin{figure}[t]
\centering
\includegraphics[width=\textwidth]{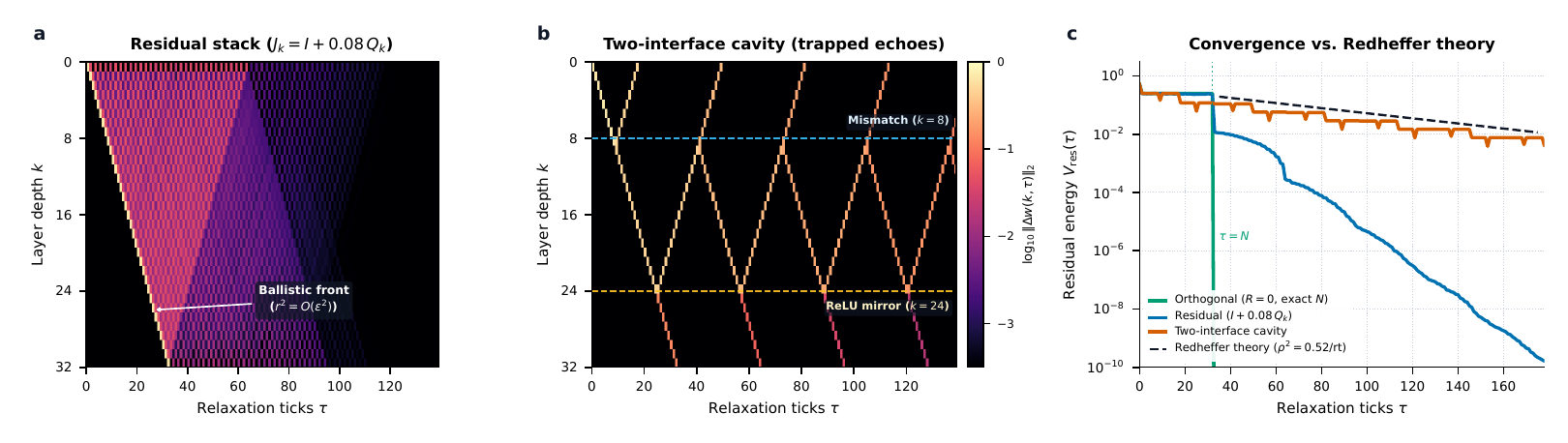}
\caption{\textbf{Ballistic residual transmission, ReLU mirror reflections, and Redheffer cavity echoes on a $32$-layer ($d=16$) strip.}
Exact Courant-1 unlocked wave scattering with reflection-free boundary terminations so every reflection originates strictly from interior layers.
\textbf{a},~Near-identity residual stack ($J_k = I + 0.08\,Q_k$): the wave propagates along a clean ballistic front with $O(\varepsilon^2)$ backscatter that fades rapidly behind it.
\textbf{b},~Two interior mismatches (a bottleneck $s=0.40$ at $k=8$ and a half-inactive ReLU mirror $s_i\in\{0,0.40\}$ at $k=24$): inactive ReLU coordinates reflect with $|r_i|=1$, trapping Fabry--P\'erot echoes between $k=8$ and $k=24$ that leak out on each bounce.
\textbf{c},~Residual energy $V_{\mathrm{res}}(\tau)$: an orthogonal chain ($R=0$) reaches machine precision ($7.1\times 10^{-31}$) in exact $N=32$ ticks, the residual stack decays to $1.5\times 10^{-10}$, and the cavity staircase decays at the exact Redheffer round-trip spectral rate ($\rho^2 = 0.52$ per $32$-tick round-trip, dashed line).}
\label{fig:deep_linear_strip}
\end{figure}

\begin{algorithm}[tbh!]
\caption{Unlocked training by local state--costate scattering}
\label{alg:worldsheet_unlocked}
\begin{algorithmic}[1]
\REQUIRE mini-batch $(x_{\mathrm{in}},y)$, maps $\{f_k(\cdot;\theta_k)\}$, penalties $\{R_k\}$,
metrics $\{\Theta_k\}$, parameter stepsize $\eta$, relaxation $\omega\in(0,1]$, inner steps $Q$.
\STATE Initialize waves $\{w_{+,k},w_{-,k}\}_{k=0}^N$ and parameters $\{\theta_k\}$.
\FOR{outer training steps $m=0,1,2,\dots$}
\FOR{$q=1$ to $Q$}
\STATE \textbf{Reconstruct.} For all nodes $k$ in parallel, compute $x_k,\lambda_k$ by \eqref{eq:disc_wave_inverse}.
\STATE \textbf{Boundary ports.} Set $w_{+,0}\gets\sqrt2\Theta_0x_{\mathrm{in}}-w_{-,0}$ and $w_{-,N}\gets w_{+,N}-\sqrt2\Theta_N^{-\top}\nabla_{x_N}\ell(x_N,y)$.
\STATE Reconstruct the boundary nodes again after the port assignments.
\STATE \textbf{Local scattering.} For all steps $k$ in parallel, solve (or locally linearize) \eqref{eq:nonlinear_layer_scattering_main} for outgoing waves $(\widehat w_{+,k+1},\widehat w_{-,k})$.
\STATE \textbf{Relaxed write.} For all interior directed waves, set
\begin{align*}
w_{+,k+1}&\gets(1-\omega)w_{+,k+1}+\omega\widehat w_{+,k+1},\\
w_{-,k}&\gets(1-\omega)w_{-,k}+\omega\widehat w_{-,k}.
\end{align*}
\ENDFOR
\STATE \textbf{Parameter ports.} Reconstruct $x_k,\lambda_k$ and compute mini-batch averaged $g_k=r_{\theta,k}$ for each step.
\STATE $\theta_k\gets\theta_k-\eta Z_{\theta,k}^{-1}g_k$ for all $k$ in parallel, with $Z_{\theta,k}\succ0$ fixed or adaptive.
\ENDFOR
\end{algorithmic}
\end{algorithm}

\subsection{Reflection-freeness and the pullback metric}
\label{sec:discrete_orthogonal_reflection}

We now determine when a single discrete junction propagates waves without
reflection, first for linear maps and then for nonlinear ones.

\paragraph{Linear case: orthogonal metrics and the pullback condition.}
Consider an affine step $x_{k+1}=J_k x_k+b_k$,
$\lambda_k=J_k^\top\lambda_{k+1}+g_k$ with $R_k=0$, and form the
metric-normalized Jacobian
\begin{equation}
\label{eq:normalized_jacobian_main}
A_k = \Theta_{k+1} J_k \Theta_k^{-1} \in \RR^{n_{k+1}\times n_k}.
\end{equation}
Substituting \eqref{eq:disc_wave_inverse} gives $y+v=A_k(u+z)$ and
$u-z=A_k^\top(y-v)$ (up to affine sources). Solving for $(z,y)$ yields the
exact block scattering matrix (Appendix~\ref{app:discrete_scattering_details}):
\begin{equation}
\label{eq:exact_linear_scattering_main}
\begin{aligned}
\begin{pmatrix} z \\ y \end{pmatrix}
&= S_{A_k} \begin{pmatrix} u \\ v \end{pmatrix} + \text{sources},\\
S_{A_k} &=
\begin{pmatrix}
(I+A_k^\top A_k)^{-1}(I-A_k^\top A_k) & 2(I+A_k^\top A_k)^{-1} A_k^\top \\[4pt]
2 A_k (I+A_k^\top A_k)^{-1} & (A_k A_k^\top-I)(I+A_k A_k^\top)^{-1}
\end{pmatrix}.
\end{aligned}
\end{equation}
By singular value decomposition $A_k=U_k\Pi_k\Sigma_k V_k^\top$
(with $U_k\in O(n_{k+1})$, $V_k\in O(n_k)$, rectangular selector $\Pi_k$,
and positive singular values $s_i=e^{d_i}>0$), the matrix $S_{A_k}$
is \emph{always orthogonal} ($S_{A_k}^\top S_{A_k}=I$, so every linear
junction is strictly lossless regardless of width or rank), and decomposes
an arbitrary discrete layer into the exact finite-step counterpart of
Section~\ref{sec:variational_worldsheet}'s orthogonal/metric-Lie-derivative
split:
\begin{enumerate}
\item \emph{Orthogonal transport ($U_k,V_k^\top$):} the discrete analogue
of $\widehat A_{\mathrm{skew}}$ ($\mathcal L_A G=0$), rotating waves across
the junction with zero reflection ($r=0$, $t=1$).
\item \emph{Integrated metric Lie derivative ($\Sigma_k=\operatorname{diag}(e^{d_i})$, $s_i>0$):}
integrating the continuous moving-frame wave PDE
\eqref{eq:bulk_port_waves_plus}--\eqref{eq:bulk_port_waves_minus}
(with $K=L=0$) over a finite metric stretch
$\int\widehat A_{\mathrm{sym}}\,\dd t=\operatorname{diag}(d_i)$ reproduces
the exact discrete reflection and transmission coefficients of $S_{A_k}$
with zero approximation error (Appendix~\ref{app:paracomplex_proofs}):
\begin{equation}
\label{eq:singular_reflection_main}
r_i = \frac{1-s_i^2}{1+s_i^2} = -\tanh(d_i), \qquad
t_i = \frac{2s_i}{1+s_i^2} = \operatorname{sech}(d_i), \qquad
r_i^2 + t_i^2 = 1 .
\end{equation}
The diagonal reflection blocks of $S_{A_k}$ vanish simultaneously ($d_i\equiv0$
with no null-space or cokernel modes) if and only if $G_k$ matches the
\emph{pullback metric}:
\begin{equation}
\label{eq:discrete_orthogonal_condition}
\begin{aligned}
S_{A_k}\text{ is reflection-free}
&\Longleftrightarrow
A_k^\top A_k = I_{n_k}\;\text{ and }\; A_k A_k^\top = I_{n_{k+1}}\\
&\Longleftrightarrow
G_k = J_k^\top G_{k+1} J_k \;\text{ with } J_k\text{ invertible}.
\end{aligned}
\end{equation}
Thus choosing the pullback metric $G_k=J_k^\top G_{k+1}J_k$ is the exact
discrete counterpart of solving $\mathcal L_A G=0$ in
\eqref{eq:continuous_isometry_condition}: it gauges away the metric stretch
($d_i\equiv0$), removes the mass/reflection barrier, and collapses the
layer purely to reflection-free orthogonal transport.
\item \emph{Irreducible obstructions (modes no SPD metric can absorb):}
whenever $J_k$ drops rank or changes width ($n_k\neq n_{k+1}$),
$\operatorname{rank}(A_k)\le\min(n_k,n_{k+1})$ forces $|n_{k+1}-n_k|$
zero singular values ($s_i=0\iff d_i\to-\infty$), leaving unavoidable
unit-amplitude mirror modes ($|r_i|=1$, e.g.\ an exact key overwrite in
causal delta-rule attention, Appendix~\ref{sec:continuous_delta_attention}).
\end{enumerate}

\paragraph{Nonlinear maps and multi-path obstructions.}
For a general nonlinear layer $x_{k+1}=f_k(x_k;\theta_k)$, two further
obstructions cannot be absorbed into a fixed single metric:
(i) \emph{state-dependent non-orthogonality}, because $J_{x,k}(x_k)$ varies
with the activation (an inactive ReLU coordinate has $s_i=0$, producing the
reflection mirror in Figure~\ref{fig:deep_linear_strip}b), and
(ii) \emph{second-order curvature}, because differentiating $r_{\lambda,k}=0$
introduces $\delta x_k^\top(\lambda_{k+1}^\top\nabla_x^2 f_k)\,\delta x_k$
(the discrete $H_{xx}$ of Theorem~\ref{thm:nonlinear_wave_energy_balance}),
which can be indefinite and break local losslessness. On branching graphs
with reconverging paths (Section~\ref{sec:residual_branching}), a third
obstruction arises whenever two parallel paths from $u$ to $v$ have
mismatched Jacobians ($J_{\mathrm{path1}}\neq J_{\mathrm{path2}}$ up to
isometry, i.e.\ nontrivial loop holonomy): no single metric $G_u$ can
simultaneously satisfy the pullback condition along both branches. Two
structural remedies follow: keep the bulk orthogonal or near-orthogonal
while damping at the ports (Section~\ref{sec:residual_branching}), or split
$G_k$ into directional conductances $(Y_k,Z_k)\succeq0$ so pullbacks
$Y_k=J_k^\top Y_{k+1}J_k+Q_k$ and pushforwards
$Z_{k+1}=J_k Z_k J_k^\top+S_k$ never require inverting $J_k$.

\subsection{Composing many layers: the Redheffer star product}
\label{sec:redheffer_composition}

The two-port scattering description is \emph{composable}: we can reason
about an entire deep network by wiring local junctions together. Write any two-port junction $k$ in block form
\begin{equation}
\label{eq:two_port_blocks}
\begin{pmatrix} w_{-,k} \\ w_{+,k+1} \end{pmatrix}
= S^{(k)} \begin{pmatrix} w_{+,k} \\ w_{-,k+1} \end{pmatrix} + \text{sources},
\qquad
S^{(k)} = \begin{pmatrix} R_{k,L} & T_{k,B} \\ T_{k,F} & R_{k,R} \end{pmatrix},
\end{equation}
where $R_{k,L},R_{k,R}$ are the left/right reflection operators and
$T_{k,F},T_{k,B}$ the forward/backward transmission operators. Connecting
junction $k$ in series with junction $k+1$ (eliminating the internal waves
at node $k+1$) yields an equivalent two-port scatterer given by the
\emph{Redheffer star product} $S^{(k)}\star S^{(k+1)}$
\citep{redheffer1960supplementary}:
\begin{equation}
\label{eq:redheffer_star_main}
\begin{aligned}
\bigl(S^{(k)}\star S^{(k+1)}\bigr)_{11}
&= R_{k,L} + T_{k,B}\bigl(I - R_{k+1,L} R_{k,R}\bigr)^{-1} R_{k+1,L} T_{k,F},\\
\bigl(S^{(k)}\star S^{(k+1)}\bigr)_{12}
&= T_{k,B}\bigl(I - R_{k+1,L} R_{k,R}\bigr)^{-1} T_{k+1,B},\\
\bigl(S^{(k)}\star S^{(k+1)}\bigr)_{21}
&= T_{k+1,F}\bigl(I - R_{k,R} R_{k+1,L}\bigr)^{-1} T_{k,F},\\
\bigl(S^{(k)}\star S^{(k+1)}\bigr)_{22}
&= R_{k+1,R} + T_{k+1,F}\bigl(I - R_{k,R} R_{k+1,L}\bigr)^{-1} R_{k,R} T_{k+1,B}.
\end{aligned}
\end{equation}

\paragraph{Summarizing everything below and above any layer.}
Because the star product is associative, standing at any single layer $k$
lets us collapse the entire subnetwork \emph{below} it (from the input port
through layers $0,\dots,k-1$) and the entire subnetwork \emph{above} it
(from layers $k+1,\dots,N-1$ through the loss port) into two effective
boundary relations at layer $k$'s ports:
\begin{equation}
\label{eq:above_below_summary}
\begin{aligned}
S_{\mathrm{below}}^{(k)}
&= S_{\mathrm{in}}\star S^{(0)}\star\cdots\star S^{(k-1)}
&\Longleftrightarrow\quad
x_k &= \bar x_k^\leftarrow - Z_k^\leftarrow\,\lambda_k,\\
S_{\mathrm{above}}^{(k+1)}
&= S^{(k+1)}\star\cdots\star S^{(N-1)}\star S_{\mathrm{loss}}
&\Longleftrightarrow\quad
\lambda_{k+1} &= \bar\lambda_{k+1}^\rightarrow + Y_{k+1}^\rightarrow\,x_{k+1}.
\end{aligned}
\end{equation}
Everything below layer $k$ is completely summarized by its forward state
prediction $\bar x_k^\leftarrow$ and upstream compliance (Th\'evenin
impedance) $Z_k^\leftarrow\succeq0$; everything above layer $k$ is
completely summarized by its backward costate drive
$\bar\lambda_{k+1}^\rightarrow$ and pulled-back downstream curvature
(Norton admittance) $Y_{k+1}^\rightarrow\succeq0$. From layer $k$'s
perspective, the rest of the deep network disappears into these two
effective terminations---and matching the directional wave frames at cut
$k$ to $(Z_k^\leftarrow,Y_{k+1}^\rightarrow)$ makes both half-networks
completely reflection-free.

This single associative operation help us understand how multi-layer networks behave:
\begin{enumerate}
\item \emph{Passivity closure at arbitrary depth.} Because every linear
junction satisfies $\|S_{A_k}\|_2=1$, the star product of lossless (or
port-damped passive) junctions remains lossless (or passive):
$\|S^{(0)}\star\cdots\star S^{(N-1)}\|_2\le1$ automatically across any
depth, width schedule, or rank drop.
\item \emph{Chaining the pullback metric and Riccati recursions.} When
every interior interface satisfies the pullback metric
$G_k=J_k^\top G_{k+1}J_k$, all reflection blocks vanish and the star
product collapses to pure transmission,
$S_{A_0}\star\cdots\star S_{A_{N-1}}=S_{A_{N-1}\cdots A_0}$, preserving
zero reflection across all $N$ layers. More generally, with directional
conductances $(Y_k,Z_k)$ and interior costs/controls $(Q_k,S_k)$,
associative star composition is algebraically equivalent to the discrete
control and filter Riccati recursions (and supports $O(\log N)$
balanced-tree composition, Appendix~\ref{app:discrete_scattering_details}).
\item \emph{Internal cavity reflections.} When consecutive layers are not
matched, the resolvent $(I-R_{k,R}R_{k+1,L})^{-1}=\sum_{m\ge0}(R_{k,R}R_{k+1,L})^m$
captures the exact sum of multiple internal reflections bouncing between
mismatched layers or nonlinearities (the trapped Fabry--P\'erot echoes in
Figure~\ref{fig:deep_linear_strip}b, whose staircase decay matches the exact
Redheffer round-trip rate in Figure~\ref{fig:deep_linear_strip}c), which decay as wave energy exits
through the ports.
\end{enumerate}

\paragraph{Transforming a layer into a multi-stage hybrid flow.}
Treating a dense rectangular layer $W\in\RR^{m\times n}$ as a single
algebraic jump \eqref{eq:exact_linear_scattering_main} requires inverting
$I+W^\top W$ at cost $O(n^3)$. Using the hybrid action
\eqref{eq:hybrid_pontryagin_action} together with the factorization
\begin{equation}
\label{eq:lie_svd_factorization_main}
W = e^{A_U}\,\Pi_{\mathrm{pad}}\,e^{\Sigma}\,\Pi_{\mathrm{drop}}\,e^{A_V},
\qquad
A_V\in\mathfrak{so}(n),\;\; A_U\in\mathfrak{so}(m),\;\;
\Sigma=\operatorname{diag}(\sigma_1,\dots,\sigma_r),
\end{equation}
where $r\le\min(m,n)$, $\Pi_{\mathrm{drop}}=[I_r\;\;0]:\RR^n\to\RR^r$, and
$\Pi_{\mathrm{pad}}=[I_r;\;0]:\RR^r\to\RR^m$ (with an optional fixed sign
reflection covering $\det=-1$), opens one layer into three continuous depth
stages $I_1=[0,1]$ ($\RR^n$), $I_2=[1,2]$ ($\RR^r$), and $I_3=[2,3]$ ($\RR^m$)
separated by two projection jumps (and an optional activation/bias jump
$\varphi(\cdot+b)$ at $s=3$):
\begin{equation}
\label{eq:lie_svd_hybrid_action_main}
\begin{aligned}
\cS_{\mathrm{layer}}
&=\int_0^1(\lambda^{(1)})^\top(A_V x^{(1)}-\partial_s x^{(1)})\,\dd s
+(\lambda^{(2)}(1^+))^\top\bigl(\Pi_{\mathrm{drop}}x^{(1)}(1^-)-x^{(2)}(1^+)\bigr)\\
&\quad+\int_1^2(\lambda^{(2)})^\top(\Sigma x^{(2)}-\partial_s x^{(2)})\,\dd s
+(\lambda^{(3)}(2^+))^\top\bigl(\Pi_{\mathrm{pad}}x^{(2)}(2^-)-x^{(3)}(2^+)\bigr)\\
&\quad+\int_2^3(\lambda^{(3)})^\top(A_U x^{(3)}-\partial_s x^{(3)})\,\dd s .
\end{aligned}
\end{equation}
In flat or scalar-impedance wave coordinates $w_\pm=\tfrac{1}{\sqrt{2b}}(bx\pm\lambda)$,
each stage simplifies completely without matrix inversion:
\begin{enumerate}
\item \emph{Skew-symmetric orthogonal waveguides ($\partial_s x=Ax$, $A^\top=-A\in\mathfrak{so}$).}
Because $A^\top=-A$, the state and costate obey the \emph{identical} ODE
$\partial_s x=Ax$, $\partial_s\lambda=A\lambda$, and $\mathcal L_A I=A^\top+A=0$.
The worldsheet PDE decouples into reflection-free orthogonal transport
$(\partial_\tau\pm c(\partial_s-A))w_\pm=0$ with massless worldsheet action,
block-diagonal scatterer $S_{\mathrm{orth}}(e^A):(u,v)\mapsto((e^A)^\top v,\,e^A u)$,
and impedance-independent $\mathfrak{so}$-gradient given by the wave angular
momentum (wedge product):
\begin{equation}
\label{eq:lie_svd_wave_gradients_main}
r_A = \operatorname{skew}\!\left(\int_0^1\lambda(s)x(s)^\top\,\dd s\right)
= \frac12\int_0^1\bigl(w_+(s)w_-(s)^\top-w_-(s)w_+(s)^\top\bigr)\,\dd s .
\end{equation}
\item \emph{Diagonal singular-value taper ($\partial_s x_i=\sigma_i x_i$, $i=1,\dots,r$).}
Because $\Sigma^\top=+\Sigma$, $\partial_s\lambda_i=-\sigma_i\lambda_i$ carries
the opposite sign, yielding $r$ independent 1D coupled-mode (Klein--Gordon)
channels $(\partial_\tau\pm c\partial_s)w_{\pm,i}=\pm c\sigma_i w_{\mp,i}$
with exact $O(r)$ diagonal orthogonal scatterer $S_{\mathrm{taper}}(\Sigma)$
from \eqref{eq:singular_reflection_main} ($s_i=e^{\sigma_i}$) and local
parameter gradient equal to the conserved channel net wave power flux:
\begin{equation}
\label{eq:lie_svd_sigma_gradient_main}
r_{\sigma_i}
= \int_1^2\lambda_i(s)x_i(s)\,\dd s
= \frac12\bigl(w_{+,i}^2-w_{-,i}^2\bigr).
\end{equation}
\item \emph{Projection jumps ($\Pi_{\mathrm{drop}},\Pi_{\mathrm{pad}}$).}
Active channels ($1\le i\le r$) pass transparently, dropped kernel channels
satisfy $\lambda_j=0\iff w_{-,j}=+w_{+,j}$ (lossless Neumann mirror, $R_L=+I_{n-r}$),
and zero-padded cokernel channels satisfy $x_j=0\iff w_{+,j}=-w_{-,j}$
(lossless Dirichlet mirror, $R_R=-I_{m-r}$).
\end{enumerate}
Composing these five lossless stages left-to-right under the Redheffer star
product reproduces the exact dense two-port scatterer
\eqref{eq:exact_linear_scattering_main} with zero matrix inversions
(Appendix~\ref{app:categorical_scattering}, Proposition~\ref{prop:monoidal_functor_svd_app}):
\begin{equation}
\label{eq:lie_svd_star_equivalence_main}
S_{\mathrm{orth}}(e^{A_V})\star S_{\Pi_{\mathrm{drop}}}\star S_{\mathrm{taper}}(\Sigma)\star S_{\Pi_{\mathrm{pad}}}\star S_{\mathrm{orth}}(e^{A_U}) = S_W .
\end{equation}
Chaining $L$ such blocks with pointwise nonlinearities $\varphi_k$ turns an
entire $L$-layer deep network into one piecewise-constant $(1+1)$-dimensional
hybrid wave PDE,
$(\partial_\tau\pm c\partial_s)w_\pm=\pm c\bigl(A(s)w_\pm+\Sigma(s)w_\mp\bigr)$,
with exact bulk conservation
$\partial_\tau\tfrac12(\|w_+\|^2+\|w_-\|^2)+\partial_s\tfrac{c}{2}(\|w_+\|^2-\|w_-\|^2)=0$
(since $A^\top=-A$ and $\Sigma^\top=\Sigma$). Contracting each layer's internal
star product \eqref{eq:lie_svd_star_equivalence_main} between inter-layer
delay wires recovers Algorithm~\ref{alg:worldsheet_unlocked} at $O(n^2)$
per layer (or $O(n\log n)$ with structured orthogonal factors) and purely
local wave-bilinear updates \eqref{eq:lie_svd_wave_gradients_main}--\eqref{eq:lie_svd_sigma_gradient_main}.

\subsection{Residual connections and branching architectures}
\label{sec:residual_branching}

Modern networks are rarely bare chains: they rely on residual (skip)
connections, parallel branches, and multi-way merges. Both fit natively
into local wave scattering without losing composability or passivity.

\paragraph{Forks and joins via Kirchhoff junctions (general $\Theta_j$ and scalar wave variables).}
Branching architectures combine two distinct kinds of forks and joins.
\emph{Coordinate slice-forks} and \emph{concat-joins} partition a state into
disjoint channel subsets via an orthogonal routing matrix $P=I_n$; they are
trivially reflection-free whenever channel metrics match (and the Lie--SVD
jumps $\Pi_{\mathrm{drop}},\Pi_{\mathrm{pad}}$ are simply slice/concat maps
capped by Neumann or Dirichlet mirrors). Non-trivial topological branching
arises from \emph{copy-forks} and \emph{sum-joins} (gluing worldsheet strips
into parallel and series pairs of pants):
\begin{itemize}
    \item at a 1-to-$d$ \emph{copy-fork} (parallel split) connecting trunk port $0$ to
branches $j=1,\dots,d$, all ports share the state ($x^{(0)}=x^{(1)}=\cdots=x^{(d)}$)
while returning costates sum ($\lambda^{(0)}=\sum_{j=1}^d\lambda^{(j)}$, adding
$\sum_{j=1}^d\lambda^{(j)\top}(x^{(0)}-x^{(j)})$ to the hybrid action
\eqref{eq:hybrid_pontryagin_action}). 
\item Dually, at a $d$-to-1 \emph{weighted sum-join}
(series merge, canonical weights $\alpha_j=1$), states sum
($x^{(0)}=\sum_{j=1}^d \alpha_j x^{(j)}$) while costates are shared
($\lambda^{(j)}=\alpha_j\lambda^{(0)}$, adding
$\lambda^{(0)\top}(\sum_{j=1}^d \alpha_j x^{(j)}-x^{(0)})$). Equipping each port
$j\in\{0,\dots,d\}$ with an arbitrary metric factor $\Theta_j$
($G_j=\Theta_j^\top\Theta_j\succ0$) and port-inward waves
$a_j=\tfrac{1}{\sqrt2}(\Theta_j x^{(j)}+\Theta_j^{-\top}\lambda_{\mathrm{in}}^{(j)})$,
$b_j=\tfrac{1}{\sqrt2}(\Theta_j x^{(j)}-\Theta_j^{-\top}\lambda_{\mathrm{in}}^{(j)})$
(using signed trunk waves $(\tilde a_0,\tilde b_0)=(-w_{-,0},-w_{+,0})$ on the
merge so $\sum_{j=1}^d\alpha_j x^{(j)}-x^{(0)}=0$), eliminating the shared junction
variable gives exact orthogonal block-Householder scatterers in closed form:
\begin{equation}
\label{eq:kirchhoff_householder}
\begin{gathered}
S_{\mathrm{split}} = 2\,C_{\mathrm{split}}C_{\mathrm{split}}^\top - I_{(d+1)n},
\qquad
S_{\mathrm{merge}} = I_{(d+1)n} - 2\,C_{\mathrm{merge}}C_{\mathrm{merge}}^\top, \\
S_{\mathrm{split}}^\top S_{\mathrm{split}} = S_{\mathrm{merge}}^\top S_{\mathrm{merge}} = I,
\end{gathered}
\end{equation}
where $G_{\mathrm{tot}}=\sum_{j=0}^d G_j=\Theta_{\mathrm{tot}}^\top\Theta_{\mathrm{tot}}$,
$H_{\mathrm{tot}}=G_0^{-1}+\sum_{j=1}^d \alpha_j^2 G_j^{-1}=\Gamma_{\mathrm{tot}}\Gamma_{\mathrm{tot}}^\top$,
and the tall orthonormal blocks ($C^\top C=I_n$) are
\begin{equation}
\label{eq:kirchhoff_C_blocks}
C_{\mathrm{split}}=\begin{bmatrix}\Theta_0\\ \Theta_1\\ \vdots\\ \Theta_d\end{bmatrix}\Theta_{\mathrm{tot}}^{-1},
\qquad
C_{\mathrm{merge}}=\begin{bmatrix}\Theta_0^{-\top}\\ \alpha_1\Theta_1^{-\top}\\ \vdots\\ \alpha_d\Theta_d^{-\top}\end{bmatrix}\Gamma_{\mathrm{tot}}^{-\top}.
\end{equation}
\end{itemize}

Trunk reflection vanishes ($S_{00}=0$) if and only if the trunk metric matches
the parallel or weighted series sum of the branches:
\begin{equation}
\label{eq:kirchhoff_metric_match}
G_0 = \sum_{j=1}^d G_j \quad (\text{parallel split}),
\qquad
G_0^{-1} = \sum_{j=1}^d \alpha_j^2 G_j^{-1} \quad (\text{series merge}).
\end{equation}
In the classical scalar wave-variable (Niemeyer--Slotine)
setting $G_j=b_j I_n$ ($\Theta_j=\sqrt{b_j}\,I_n$,
$w_{\pm,j}=\tfrac{1}{\sqrt{2b_j}}(b_j x^{(j)}\pm\lambda^{(j)})$)
\citep{niemeyer1991stable}, \eqref{eq:kirchhoff_householder}--\eqref{eq:kirchhoff_metric_match}
collapse coordinatewise to $O(dn)$ rank-one Householder reflections with dual weights
(matching the parallel and series adaptors of wave digital filters
\citep{fettweis1986wave}):
\begin{equation}
\label{eq:kirchhoff_slotine_scalar}
\begin{aligned}
S_{\mathrm{split}} &= \frac{2\,u u^\top}{u^\top u}-I
\qquad\Bigl(u_j=\sqrt{b_j},\;\; b_0=\textstyle\sum_{j=1}^d b_j\Bigr), \\
S_{\mathrm{merge}} &= I-\frac{2\,v v^\top}{v^\top v}
\qquad\Bigl(v_0=\tfrac{1}{\sqrt{b_0}},\;\; v_{j\ge 1}=\tfrac{\alpha_j}{\sqrt{b_j}},\;\; b_0^{-1}=\textstyle\sum_{j=1}^d \tfrac{\alpha_j^2}{b_j}\Bigr).
\end{aligned}
\end{equation}
For $d$ equal unweighted branches ($\alpha_j=1$, $b_1=\cdots=b_d=b$), a matched
split ($b_0=db$) divides an incoming trunk wave reflection-free into
$w_{+,j}=\tfrac{1}{\sqrt d}w_{+,0}$, while a matched merge ($b_0=b/d$) sums
incoming branch waves reflection-free into
$w_{+,0}=\tfrac{1}{\sqrt d}\sum_{j=1}^d w_{+,j}$ (note that
$S_{\mathrm{merge}}=-S_{\mathrm{split}}$ holds only when all port impedances
and weights are equal, whereas in general split weights scale as $\sqrt{b_j}$
and merge weights as $\alpha_j/\sqrt{b_j}$). Because both junctions satisfy
$\Delta P_{\mathrm{split}}=(\sum_j\delta\lambda^{(j)}-\delta\lambda^{(0)})^\top\delta x^{(0)}=0$
and $\Delta P_{\mathrm{merge}}=\delta\lambda^{(0)\top}(\delta x^{(0)}-\sum_j\alpha_j\delta x^{(j)})=0$
for \emph{arbitrary} metrics, and compose associatively under $\star$ (every
$d$-way fork or join factors into $d-1$ binary junctions), wiring arbitrary
acyclic or recurrent topologies---ResNets, U-Nets, multi-branch blocks,
attention/MLP sublayers, or cyclic equilibrium circuits---preserves
Theorem~\ref{thm:nonlinear_wave_energy_balance} automatically.

\paragraph{Fork--join diamonds with Lie--SVD branches: impedance scaling, transit alignment, and holonomy.}
Wiring a copy-fork to $d$ parallel Lie--SVD branches and reconverging at a
sum-join (e.g.\ a residual block $x_{\mathrm{out}}=\alpha_1 x_{\mathrm{in}}+\alpha_2 F(x_{\mathrm{in}})$)
introduces two loop effects absent in bare chains:
\begin{enumerate}
\item \emph{Depth-invariant impedance scaling via Cauchy--Schwarz.}
Matching both the fork ($b_{\mathrm{in}}=\sum_{j=1}^d b_j$) and the join
($b_{\mathrm{out}}^{-1}=\sum_{j=1}^d \alpha_j^2 b_j^{-1}$) across
constant-impedance branches obeys
\begin{equation}
\label{eq:fork_join_cauchy_schwarz}
\left(\sum_{j=1}^d |\alpha_j|\right)^2
\;\le\;
\left(\sum_{j=1}^d b_j\right)\!\left(\sum_{j=1}^d \frac{\alpha_j^2}{b_j}\right)
=\frac{b_{\mathrm{in}}}{b_{\mathrm{out}}},
\end{equation}
with equality if and only if $b_j\propto |\alpha_j|$. In particular, any
\emph{convex-combination merge} ($\sum_{j=1}^d\alpha_j=1$, $\alpha_j>0$,
such as $x_{\mathrm{out}}=(1-\varepsilon)x_{\mathrm{in}}+\varepsilon F(x_{\mathrm{in}})$)
with branch impedances $b_j=\alpha_j b_0$ achieves exact depth-invariant
matching $b_{\mathrm{out}}=b_{\mathrm{in}}=b_0$ across arbitrarily many
stacked blocks (sending $w_{+,0}$ reflection-free into $\sqrt{\alpha_j}w_{+,0}$
and recombining in-phase waves back into $w_{+,0}$ with zero trunk and zero
branch reflection). For unnormalized sums ($\alpha_j=1$, $b_j=b$), the ratio
$b_{\mathrm{out}}/b_{\mathrm{in}}=1/d^2$ reflects the pullback gain $d$ and
can either be absorbed inside each Lie--SVD branch's singular-value stage via
$\Delta\sigma_j=\tfrac12\ln(b_j^{\mathrm{end}}/b_j^{\mathrm{start}})$ or left
unmatched at constant $b$ (strictly lossless with bounded $|r|=1/3$ at $d=2$).
\item \emph{Branch transit alignment and loop holonomy.}
A fork--join diamond forms a closed loop: if parallel branches have unequal
transit times ($L_1/c_1\neq L_2/c_2$) or mismatched branch maps (nontrivial
metric holonomy $\mathcal H_\gamma=\Theta_1(0)\Theta_2(0)^{-1}\notin O(n)$),
pulses arriving asynchronously or out of phase at the join excite an
antisymmetric circulating mode inside the loop. Equalizing the skip wire's
lossless delay length to match the residual branch's hop count aligns pulse
arrivals, while terminal/regularizer ports and relaxed writes $\omega\in(0,1)$
damp any remaining holonomy-induced transients.
\end{enumerate}

\paragraph{Why residual connections are geometrically natural for waves.}
Consider a residual block $x_{k+1}=x_k+\varepsilon\,g_k(x_k;\theta_k)$ in
the constant flat metric $G_k=G_{k+1}=I$, so that
$A_k=J_k=I+\varepsilon\,D_x g_k$. Three consequences follow immediately
from \eqref{eq:singular_reflection_main}:
\begin{enumerate}
\item \emph{Weak reflection without per-layer metric tuning.} The singular
values of $A_k$ satisfy $s_i=1+O(\varepsilon)$, so the reflection amplitude
is $r_i=(1-s_i^2)/(1+s_i^2)=O(\varepsilon)$ and the \emph{reflected power}
is only $r_i^2=O(\varepsilon^2)$ (improving to $r_i^2=O(\varepsilon^4)$
when $D_x g_k$ is skew-symmetric). A deep residual network therefore acts
as a weakly perturbed, near-reflection-free transmission line in the flat
Euclidean metric $G_k=I$---explaining the clean ballistic propagation in
Figure~\ref{fig:deep_linear_strip}a.
\item \emph{Protection against rank-drop walls.} Even when the residual
branch $g_k$ is low-rank, bottlenecked, or ReLU-gated, the identity skip
keeps $J_k=I+\varepsilon D_x g_k$ invertible for small $\varepsilon$,
eliminating the $|r_i|=1$ total-reflection barrier of a plain layer.
\item \emph{Direct link to the continuous worldsheet.} Writing
$\varepsilon=\Delta t$, $x_{k+1}=x_k+\Delta t\,g_k(x_k;\theta_k)$ (or its
midpoint/Cayley counterpart) is directly the time discretization of
$\partial_t x=f(x,u)$ in Section~\ref{sec:variational_worldsheet}, closing
the loop between the continuous field theory and discrete deep
architectures.
\end{enumerate}

\paragraph{Normalization layers: LayerNorm, RMSNorm, and BatchNorm.}
Next, we ask ourselves what happens to waves when we insert a
normalization layer into a deep network. Because the reflection law \eqref{eq:singular_reflection_main}
depends only on the singular values of the local Jacobian, we can read the
answer directly from the derivative of each normalization map.
For a single activation vector $x\in\RR^n$, RMSNorm ($y=x/r$,
$r=\|x\|/\sqrt{n}$, $\hat x=x/\|x\|$) and LayerNorm
($y=(x-\bar x\mathbf{1})/\sigma$,
$\hat z=(x-\bar x\mathbf{1})/\|x-\bar x\mathbf{1}\|$) have Jacobians
that are scaled orthogonal projections onto the tangent space of a sphere:
\begin{equation}
\label{eq:norm_jacobians}
J_{\mathrm{RMS}}(x)=\frac{1}{r}\bigl(I_n-\hat x\hat x^\top\bigr),
\qquad
J_{\mathrm{LN}}(x)=\frac{1}{\sigma}\bigl(I_n-\tfrac{1}{n}\mathbf{1}\mathbf{1}^\top-\hat z\hat z^\top\bigr).
\end{equation}
Reading their singular values in the flat metric $G=I$ immediately reveals
two distinct behaviors:
\begin{itemize}
\item \emph{Invariant modes act as mirrors.} On the radial scale direction
$\hat x$ (and the constant mean direction $\mathbf{1}$ for LayerNorm),
the singular value is $s=0$, giving $r=+1$ and $t=0$: scale and shift
perturbations cannot pass through the layer and are reflected completely.
\item \emph{Tangential modes are impedance-matched at unit scale.} All
remaining $n-1$ (or $n-2$) directions share the single singular value
$s_\perp=1/\sigma$, with reflection amplitude
$r_\perp=(\sigma^2-1)/(\sigma^2+1)$. Whenever activations sit near unit
scale ($\sigma\approx 1$), $r_\perp\approx 0$ identically across every
informative direction---so LayerNorm and RMSNorm act as automatic
conformal impedance matchers, with any learned gain $\gamma_i$ tuning the
channel impedance ($s_i=\gamma_i/\sigma$, matched at $|\gamma_i|=\sigma$).
\end{itemize}
This immediately clarifies two familiar practical phenomena. First, placing
LayerNorm on the main trunk (Post-LN) inserts a two-mode total-reflection
mirror across the transmission line at every layer, whereas placing it
inside the residual branch (Pre-LN) lets the identity skip transmit all
$n$ modes cleanly with $O(\varepsilon^2)$ reflected power. Second, because
normalization divides by $\|x\|$, its second-order curvature
$H_{xx}=\nabla_x^2(\lambda^\top y)$ scales as $1/\|x\|^2$ (the second
fundamental form of the sphere) and is indefinite whenever the tangential
costate is nonzero---explaining why cold zero initializations near
$x\approx 0$ trigger large backward-wave transients unless activations
start away from the origin.

We can ask the same question when normalization acts across a mini-batch
instead of across features. In training mode, BatchNorm applies the same rank-2
projection across the $B$ examples in a batch independently for each
channel $i=1,\dots,n$. Unlike LayerNorm and RMSNorm---which keep every
example on its own independent $(1+1)$-dimensional worldsheet---training-mode
BatchNorm acts as a \emph{cross-sample coupler}: a wave perturbation on one
example scatters into the other $B-1$ examples through the shared batch
mean and variance, introducing $2n$ mirror modes across the batch (and
collapsing to a total mirror at $B=2$). By contrast, when the running
statistics $(\bar\mu_i,\bar\sigma_i)$ are held fixed on the fast wave
clock and updated only on a slow timescale (as in evaluation mode or
Section~\ref{sec:composability}), the examples decouple completely and
BatchNorm reduces to a diagonal scaling $J=\operatorname{diag}(\gamma_i/\bar\sigma_i)$
with zero mirror modes and zero curvature---a pure impedance transformer
that is reflection-free whenever $|\gamma_i|=\bar\sigma_i$.

%% file: sections/04_physical_interpretation.tex
\section{Physical interpretation}
\label{sec:physical_interpretation}
The lifted dynamics on $\Sigma=[0,T]_t\times[0,\infty)_\tau$ can be interpreted as a transmission line in depth or physical time, evolving in solver time. The relaxation \eqref{eq:para_floer} is a transmission line whose characteristics are the polarizations of phase-space tangent vectors; the discrete algorithm is its Courant-one discretization, composed of local two-port and Kirchhoff junctions (Section~\ref{sec:training_neural_network}).

At fixed $\tau$, the layerwise waves encode a candidate primal--adjoint pair. A solution of the maximum principle corresponds to vanishing residuals $r_x=r_\lambda=r_u=0$ in the continuous formulation, or $r_x=r_\lambda=r_\theta=0$ in the discrete formulation. Boundary conditions appear as terminations of the line: the input boundary enforces the prescribed state, and the output boundary enforces the terminal adjoint condition.

The parameter update is also a termination. If it reflects incoming stationarity energy, the state--costate waves can form oscillatory transients. If it absorbs the stationarity residual through a positive impedance $Z_\theta$, the update is preconditioned gradient flow on the state--adjoint consistency manifold. Curvature-aware metrics improve impedance matching, and inertial terminations give momentum-like dynamics, as detailed in Appendix~\ref{app:param_port}.

\subsection{The Principle of Minimal Reflection}
\label{sec:minimal_reflection}

In the Wave-PMP framework, the optimization of network parameters $\theta_k$ is treated as an impedance-matched boundary termination. Let $e_k = r_{\theta,k}$ represent the stationarity residual (the effort) and $v_k = \partial_\tau \theta_k$ the optimization velocity (the flow) at the parameter port. For a given positive-definite impedance metric $Z_k$, the incident wave $a_k$ and reflected wave $b_k$ at the parameter boundary are defined as:
$$a_k = \frac{1}{\sqrt{2}}\left(Z_k^{-1/2}e_k - Z_k^{1/2}v_k\right), \qquad b_k = \frac{1}{\sqrt{2}}\left(Z_k^{-1/2}e_k + Z_k^{1/2}v_k\right).$$

\textbf{Principle 4.1 (Minimal Reflection).} \textit{An optimal learning rule actively minimizes the reflected wave energy $\|b_k\|^2$ at the parameter ports. The perfectly reflectionless termination ($b_k = 0$) exactly prescribes the preconditioned update rule:}
$$\partial_\tau \theta_k = -Z_k^{-1} r_{\theta,k}.$$

By substituting this into local power dissipation $P_{\mathrm{diss},k} = -e_k^\top v_k$, we see that reflection-free termination guarantees strict passivity. When the local state and adjoint equations reach equilibrium, $r_{\theta,k}$ converges to the exact gradient $\nabla_{\theta_k} J$, and this zero-reflection boundary condition recovers preconditioned gradient flow.

\paragraph{Momentum as inductive matching and heavy-ball dynamics.}
For fixed symmetric $R\succeq 0$ and $L\succ 0$, equipping the parameter port with the dynamic impedance $Z(s)=R+sL$ yields the heavy-ball equation
\[
-\nabla J = R\dot\theta + L\ddot\theta,
\]
the Euler--Lagrange / Lagrange--d'Alembert dynamics of kinetic storage $\tfrac12\dot\theta^\top L\dot\theta$ and Rayleigh dissipation $\tfrac12\dot\theta^\top R\dot\theta$ (equivalently, an exponentially weighted action when $R=\gamma L$, mirroring the damped telegrapher structure in Appendix~\ref{app:mechanical_telegrapher}). For exact gradients, $\frac{\dd}{\dd\tau}(J+\tfrac12\dot\theta^\top L\dot\theta)=-\dot\theta^\top R\dot\theta$. Matching $Z(s)$ to the downstream parameter-port load $K_\theta$ (Section~\ref{sec:composability}) sets the damping ratio $\zeta=R/(2\sqrt{LK_\theta})$, with critical damping ($\zeta=1$) acting as inertial impedance matching. The same mechanism appears in the bulk relaxation: eliminating one field from \eqref{eq:full_nonlinear_worldsheet_relaxation} yields a damped second-order telegrapher equation \eqref{eq:telegraph_simple_app} whose wave impedance $Z_0=1/m$ trades off shunt and series damping, maximizing the contraction rate at the distortionless match $m_\star=\sqrt{a/b}$ and generalizing to distortionless two-conductance transport across arbitrary drift.

For a complete proof of passivity, matched parameter termination, and how dynamic impedance naturally yields momentum-based optimization, we refer the reader to Appendix~\ref{app:param_port}.

\paragraph{Discrete minimal-reflection RLC circuits: unifying GD, Heavy Ball, Nesterov, and Newton.}
In discrete time, \citet{boyd2024circuits} formulate convex optimization as a
nonlinear monotone gradient resistor $i=\nabla f(\theta)$ coupled to a
passive lumped RLC interconnect $[C]+[R_d\parallel L]+[-R_0]$, and restore
discrete passivity after explicit Runge--Kutta discretization via
computer-assisted PEP/QCQP tuning. Translating the same circuit into
discrete wave-variable scattering \citep{SlotineLi1987,niemeyer1991stable}
(which coincides with the bilinear port discretization of wave digital
filters \citep{fettweis1986wave}) using the general-metric Kirchhoff
junctions of Section~\ref{sec:residual_branching} preserves discrete-time
passivity \emph{by construction} while eliminating all implicit proximal
solves. Specifically, equipping the gradient port with
an SPD port-resistance tensor $R_0=G_\theta^{-1}=\Theta_\theta^{-1}\Theta_\theta^{-\top}$
and placing Boyd et al.'s matched negative resistor $-R_0$ in series with
$\nabla f(\theta)$ gives terminal voltage $v_+=\theta-R_0\nabla f(\theta)$,
so the frame-normalized incident and reflected waves become explicit:
\begin{equation}
\label{eq:wdf_matched_grad_port}
a_g = \Theta_\theta(v_+ + R_0\nabla f(\theta)) = \Theta_\theta\theta,
\qquad
b_g = \Theta_\theta(v_+ - R_0\nabla f(\theta)) = \Theta_\theta\theta - 2\Theta_\theta^{-\top}\nabla f(\theta),
\end{equation}
with exact local scattering Jacobian $S_f(\theta)=\partial b_g/\partial a_g=I-2\Theta_\theta^{-\top}\nabla^2 f(\theta)\Theta_\theta^{-1}$.
Three classical optimization regimes drop out immediately from
\eqref{eq:wdf_matched_grad_port} (Appendix~\ref{app:param_port}):
(i)~\emph{passivity} ($\|S_f\|_2\le 1$) holds if and only if
$0\preceq\nabla^2 f\preceq G_\theta$ (recovering the stepsize bound
$\eta=2R_0\le 2/L_f$ via $G_\theta$-metric Baillon--Haddad co-coercivity);
(ii)~\emph{minimal worst-case reflection} over $\mu G\preceq\nabla^2 f\preceq L_f G$
selects $R_0^\star=(L_f+\mu)^{-1}G^{-1}$ with $\|S_f^\star\|_2=(\kappa_G-1)/(\kappa_G+1)$,
and \emph{exact zero reflection} ($S_f\equiv 0$) holds if and only if
$G_\theta=2\nabla^2 f(\theta)$ (Newton, Gauss--Newton, or Natural Gradient); and
(iii)~\emph{wiring the gradient port to bilinear $C$ ($+z^{-1}$), $L$ ($-z^{-1}$),
and $R_d$} through reflection-free parallel ($G_P=G_d+G_L$) and series
($R_0=R_C+R_P$) Kirchhoff junctions collapses the entire RLC tree into a
4-line passive update ($\Gamma_L=R_P G_L$, $B=I-2\Gamma_L$, initialized at
$c_0=\theta_0,\ell_0=0$):
\begin{equation}
\label{eq:wdf_rlc_four_line}
\theta_k = c_k + \Gamma_L\ell_k, \qquad
g_k = \nabla f(\theta_k), \qquad
c_{k+1} = c_k - 2R_C\,g_k, \qquad
\ell_{k+1} = B\,\ell_k + 2R_P\,g_k .
\end{equation}
For \emph{arbitrary} (even non-commuting) SPD tensors $G_C,G_L,G_d\succ0$,
\eqref{eq:wdf_rlc_four_line} obeys the exact discrete Tellegen--Lyapunov
energy balance
$\cE_{k+1}-\cE_k=-\|v_P^k\|_{G_d}^2-\langle\theta_k-\theta^\star-G_0^{-1}g_k,\,g_k\rangle\le0$
whenever $0\preceq\nabla^2 f\preceq G_0=(R_C+R_P)^{-1}$. For proportional
tensors ($R_j=r_j G^{-1}$, $\Gamma_L=\gamma_L I$, $B=\beta I$,
$r\equiv\gamma_L r_P/r_C$), eliminating $(c_k,\ell_k)$ recovers Riemannian
Gradient Descent ($r=0$), Polyak Heavy Ball ($r=\beta$), and Nesterov
Accelerated Gradient ($r=\beta^2$) as exact special cases of one passive
scattering tree (Theorem~\ref{thm:wdf_rlc_unification_app}).

\subsection{Port equivalents, stability, and multiple timescales}
\label{sec:composability}

Building on the Redheffer star composition (Section~\ref{sec:redheffer_composition}) and Kirchhoff branching (Section~\ref{sec:residual_branching}), Wave-PMP mirrors classical electronic port reduction and preserves stability across three tiers, echoing passivity and contraction combination properties \citep{lohmiller1998contraction,wang2006contraction,van2014port}:

\paragraph{Parameter-port loads and buffering regimes.}
The effective boundary summaries $(Z_k^\leftarrow, Y_{k+1}^\to)$ of \eqref{eq:above_below_summary} (whose continuum limit $\Delta t\to 0$ yields the forward and backward Riccati ODEs tapering along the horizon) directly determine the local parameter-port load seen by layer $k$:
\[
K_{\theta,k} = B_k^\top Y_{k+1}^\to \bigl(I + J_k Z_k^\leftarrow J_k^\top Y_{k+1}^\to\bigr)^{-1} B_k,
\qquad B_k = D_{\theta_k} f_k .
\]
Setting $Z_k^\leftarrow = 0$ corresponds to \emph{unilateral high-gain buffering} (zero downstream loading of forward states, as in CMOS/op-amp stages and causal sequence models such as continuous-time delta-rule attention, Appendix~\ref{sec:continuous_delta_attention}, while retaining a matched backward channel), whereas $Z_k^\leftarrow \succ 0$ corresponds to \emph{bidirectional transmission-line composition}.

\paragraph{Three-tier compositional stability.}
\begin{enumerate}
\item \emph{Passivity closure (one metric, arbitrary topology and delay).} Every affine junction ($\|S_{A_k}\|_2 = 1$), Kirchhoff junction ($\|S_{\mathrm{split}}\|_2 = 1$), convex running/terminal cost, and finite transmission delay is locally incrementally nonexpansive. Wiring ports with shared metrics cancels internal fluxes identically; with averaging $0 < \omega < 1$, the lifted energy (including packets in flight) decreases monotonically, preserving global stability under arbitrary graph topologies and finite communication delays (Theorem~\ref{thm:nonlinear_wave_energy_balance}).
\item \emph{Cascade and small-gain stability (two conductances).} Unmatched directional pairs $(Y_k^\to, Z_k^\leftarrow)$ are not automatically nonexpansive. When one Riccati recursion holds, the iteration becomes a nilpotent triangular cascade terminating in finite ticks; under moderate mismatch (e.g., frozen conductances on a drifting plant, Appendix~\ref{sec:exp_physics_control}), stability persists whenever the round-trip reflection satisfies the small-gain bound $\|R_\rightarrow R_\leftarrow\| < 1$.
\item \emph{Multiple timescales (concurrent adaptation).} Because mismatched $(Y^\to, Z^\leftarrow)$ affect convergence rate without biasing stationary PMP equilibria, practical implementations separate up to four clocks: fast wave propagation ($\tau$), lazy conductance updates $(Y^\to, Z^\leftarrow)$, slow parameter updates ($\zeta$, rescaling $Z_\theta$ in Appendix~\ref{app:param_port}), and physical environment drift ($t_{\mathrm{env}}$). Coupling fast waves (contraction rate $\alpha_w$) with moving parameters (mobility $\eta = \|Z_\theta^{-1}\|$ and cross sensitivities $L_{w\theta}, L_{\theta w}$) remains stable when the impedance hierarchy satisfies $\|Z_\theta\| > L_{\theta w} L_{w\theta}/\alpha_w$.
\end{enumerate}

\paragraph{Mirror-descent ports and stable deep adaptive control.}
Setting the parameter-port impedance to the Hessian of a convex mirror potential, $Z_\theta(\theta)=\nabla^2\psi(\theta)$, recovers continuous-time mirror-descent adaptation with exact Bregman-energy passivity \citep{boffi2021implicit,boffi2022nonparametric}. This port view clarifies why deep adaptive networks can lose stability under scalar speed-gradient rules \citep{fradkov1999nonlinear,boffi2022nonparametric}: deep Jacobians amplify control transients, and updating hidden weights shifts the ideal readout weights, adding an uncancelled drift term to the Lyapunov derivative. Using Givens-factored orthogonal hidden layers ($Q_\ell\in SO(m)$) keeps feature and Jacobian norms independent of depth. Pairing them with either load-matched hidden ports (giving two-timescale stability) or a parallel fixed random-feature branch summed at a passive junction (providing a constant Lyapunov reference independent of hidden-layer motion) restores closed-loop stability at arbitrary depth $L$ while outperforming random features alone on the 5D nonlinear benchmark of \citet{boffi2022nonparametric} (Appendix~\ref{app:adaptive_nn_control}, Table~\ref{tab:boffi_adaptive_nn}).

\paragraph{Neural ODEs, Hamiltonian circuits, hybrid impacts, and noisy symbolic discovery.}
Identifying the depth index with physical time directly resolves five structural failure modes of continuous-time scientific machine learning using the same two wave bases (Appendix~\ref{app:neural_ode_hnn_physics}, Table~\ref{tab:physics_wave_summary}): (i)~two-conductance split Riccati shadowing $(Z_k^\leftarrow,Y_k^\to)$ eliminates chaotic adjoint blowup in Neural ODEs \citep{chen2018neural} and Hamiltonian Neural Networks \citep{greydanus2019hamiltonian} (keeping Lorenz-63 gradients at $0.97\pm 0.02$ versus $6.76\times 10^7$ at $23.5$ Lyapunov times and reducing chaotic H\'enon--Heiles rollout error over both single-shooting and finite-difference HNNs); (ii)~transporting the symmetric metric $G_k=J_k^\top G_{k+1}J_k$ along a known conservative background makes unmodeled forces appear directly as localized wave reflections (enabling unconstrained 2D neural residual recovery on precessing Kepler orbits with $17.3\times$ orbital error reduction), while invariant-modulated $SO(d+1)$ Cayley orthogonal ODEs achieve $O(1)$-memory reversibility to $4.72\times 10^{-14}$ ($515\times$ measured peak memory reduction) and $100\%$ concentric-ring separation; (iii)~wrapping saltation matrices \citep{wunderlich2021eventprop,chen2021learning} as local scattering junctions yields $88.6\text{--}585\times$ lower five-bounce impact error than soft-contact, saltation-free, and smooth Neural ODE baselines and keeps chaotic $31$-bounce Sinai billiard gradients $O(1)$ versus $5.53\times 10^6$; (iv)~pairing orthogonal-factorized potentials with discrete-gradient ports \citep{gonzalez1996time} yields Circuit-HNNs conserving energy to $6.57\times 10^{-14}$ over $5{,}000$ steps (versus $4.84\times 10^{-2}$ for Leapfrog HNN) with exact open-system Tellegen power balance ($3.68\times 10^{-15}$); and (v)~coupling split Riccati shadowing to a proximal-$\ell_1$ parameter port discovers sparse governing equations from noisy position-only observations without numerical differentiation ($100\%$ exact support and $24.2\times$ lower coefficient error than Savitzky--Golay \texttt{SINDy} \citep{brunton2016discovering} at $\sigma=0.030$).

\paragraph{Text flow language models: separating vocabulary compression from wave-variable gains.}
Building on continuous flow matching \citep{lipman2022flow}, flow map matching \citep{boffi2024flow}, and the Flow/Flow Map Language Model (FLM/FMLM) framework of \citet{lee2026flow} at $80\text{M}\text{--}170\text{M}$ scale (Appendix~\ref{app:text_flow_language_models}, Table~\ref{tab:text_flow_ablation}), we isolate two independent gains. First, independent of wave variables, projecting noisy one-hot inputs $X_t\in\RR^{L\times V}$ into hidden width $d\ll V$ via the exact Gaussian Cholesky factor $L L^\top = W^\top W$ removes all $O(BLV)$ input tensors ($72\text{--}81\times$ faster input step), while restoring the exact $O(BLd)$ analytical noise variance $v_{\mathrm{noise},j}=\sum_{b,\ell}(1-t_{b,\ell})^2(\partial\mathcal{L}/\partial h_{t,b,\ell,j})^2$ into AdamW's second moment eliminates the $9.3\times$ rare-token shrinkage trap of naive compressed training. On a completely standard $80\text{M}$ Transformer, this standalone compression fix cuts peak GPU memory by $33\%$ and step time by $15\%$ while lowering GPT-2 Large perplexity from $153.4\to 106.6$ at $64$ steps and $118.9\to 76.8$ at $1{,}024$ steps. Second, adding wave-variable layers (orthogonal-factorized projections trained without autograd inside Transformer blocks) independently lowers full-vocabulary perplexity ($153.4\to 132.5$ at $64$ steps, $118.9\to 91.3$ at $1{,}024$ steps), combines cleanly with variance-matched compression ($110.1$ PPL at $64$ steps and $82.1$ at $1{,}024$ steps, $-29\%$ memory), supplies analytic forward-mode time derivatives that prevent collapse during few-step distillation ($-28.8$ PPL at $4$ steps), enables $O(1)$-memory forward--backward trajectory steering across the denoising horizon (matching full BPTT infilling at $100\%$ constraint accuracy and half the memory, without OOM at $K=32$), and scales stably to $10^{18}$ FLOPs with orthogonality error below $7.5\times 10^{-6}$.

\paragraph{Scope.}
Exact Redheffer decoupling holds locally around a fixed operating point (analogous to small-signal circuit models). Strong nonlinear curvature (indefinite second-order pullbacks), overly rapid metric changes (which add metric-power terms), or non-reciprocal backward paths ($B \neq J^\top$) can inject energy and require regularization, slow impedance updates, or local damping.

\subsection{Convergence, dwell time, and normative optimal stopping}
\label{sec:optimal_stopping}

Restoring a finite propagation speed raises an immediate operational
question for Algorithm~\ref{alg:worldsheet_unlocked}: \emph{how long should
the system wait (dwell) on each mini-batch before switching to the next?}
Waiting too briefly corrupts updates with unsettled waves; waiting too long
wastes wall-clock cycles harvesting diminishing returns from a single batch.

\paragraph{Causal floor and the bias--variance dwell U-curve.}
Across $N$ layers with inter-layer delay $d\ge0$, switching mini-batches
faster than the causal round-trip time $T_{\mathrm{rt}}=2N(1+d)$ causes
batch $m$'s outgoing forward wave to collide mid-network with batch
$(m+1)$'s returning error wave, corrupting local outer-product efforts.
Above this stability floor, holding each mini-batch for a fixed dwell
$\tau$ over a wall-clock budget $T$ trades decaying wave-transient bias
against stochastic sampling variance across the $T/\tau$ visited batches:
\begin{equation}
\label{eq:dwell_mse_ucurve}
J_{\mathrm{MSE}}(\tau)
=\underbrace{\left(\frac{\Delta_0}{\tau(1-\rho)}\right)^2}_{\text{wave-transient bias}^2}
+\underbrace{\frac{\tau\,\sigma_{\mathrm{batch}}^2}{T}}_{\text{mini-batch variance}}
\quad\Longrightarrow\quad
\tau^\star=\left(\frac{2\,T\,\Delta_0^2}{(1-\rho)^2\,\sigma_{\mathrm{batch}}^2}\right)^{1/3},
\end{equation}
where $\Delta_0=\|w_m^\star-w_{m-1}^\star\|$ is the inter-batch shock
amplitude and $\rho\in(0,1)$ is the per-tick contraction factor.
Equation~\eqref{eq:dwell_mse_ucurve} directly explains three distinct
operating regimes: (i)~increasing interconnect delay $d$ increases both
$T_{\mathrm{rt}}$ and $\rho$, shifting the optimal dwell $\tau^\star$
rightward; (ii)~higher mini-batch noise $\sigma_{\mathrm{batch}}^2$ favors
faster rotation; and (iii)~on correlated continuous streams (such as the
physical control and identification streams of
Appendix~\ref{sec:exp_physics_control}), $\Delta_0\to0$ collapses
$\tau^\star\to1$, allowing continuous tracking with zero holding.

\paragraph{Normative optimal stopping via Charnov's Marginal Value Theorem.}
A fixed dwell $\tau^\star$ still wastes cycles on easy, already-classified
batches while cutting off hard batches early. Remarkably, an adaptive,
\emph{zero-free-threshold} stopping rule follows from monitoring the two
boundary ports alone. Measuring boundary wave stationarity
$r_{\mathrm{rel}}(t)$ from successive boundary wave differences defines a
local fidelity gate $\gamma(r_{\mathrm{rel}})=[1+(r_{\mathrm{rel}}/r_0)^2]^{-1}\in(0,1]$,
while the loss-port reflection energy
$E_{\mathrm{refl,loss}}(t)\propto\|\lambda_N(t)\|_{G_N^{-1}}^2$ tracks the
active batch's remaining error gradient. Their product defines batch $m$'s
purely local \emph{settled descent rate}
$R_m(t)\equiv\gamma(r_{\mathrm{rel}}(t))\,E_{\mathrm{refl,loss}}(t)$, with
cumulative descent $G_m(\tau_m)=\int_0^{\tau_m}R_m(s)\,\dd s$ over dwell
$\tau_m$. Modeling mini-batch switching as a renewal-reward process across
batch types with probabilities $p_m$ gives the long-run average settled
descent rate per wall-clock tick:
\begin{equation}
\label{eq:renewal_reward_rate}
\bar R\bigl(\{\tau_m\}\bigr)
=\frac{\mathbb E_m[G_m(\tau_m)]}{\mathbb E_m[\tau_m]}
=\frac{\sum_m p_m\int_0^{\tau_m}R_m(s)\,\dd s}{\sum_m p_m\tau_m}.
\end{equation}

\begin{proposition}[Normative MVT stopping rule \citep{charnov1976optimal}]
\label{prop:charnov_mvt}
Maximizing the long-run average settled descent rate \eqref{eq:renewal_reward_rate}
over per-batch dwell times $\{\tau_m\}$ yields the parameter-free optimal
stopping condition
\begin{equation}
\label{eq:charnov_mvt_rule}
R_m(\tau_m^\star)=\bar R^\star
\qquad\text{with}\qquad
\dot R_m(\tau_m^\star)<0.
\end{equation}
\end{proposition}

\begin{proof}
Differentiating \eqref{eq:renewal_reward_rate} with respect to $\tau_m$ by
the quotient rule gives
$\partial\bar R/\partial\tau_m=(p_m/\mathbb E[\tau])\bigl(R_m(\tau_m)-\bar R\bigr)$,
which vanishes if and only if $R_m(\tau_m^\star)=\bar R^\star$. At any such
stationary point, $\partial^2\bar R/\partial\tau_m^2=(p_m/\mathbb E[\tau])\dot R_m(\tau_m^\star)<0$,
confirming a strict maximum on the descending (diminishing-returns) branch.
\end{proof}

Implemented online by tracking a slow running average $\bar R(t)$ across
batches, \eqref{eq:charnov_mvt_rule} switches on the first tick past the
causal floor ($\tau_m\ge T_{\mathrm{rt}}$) at which the active batch's
settled descent rate drops to $\bar R(t)$. With no tuned stopping threshold,
easy batches whose error wave is already small depart immediately at
$T_{\mathrm{rt}}$, whereas hard batches dwell automatically until
parameter updates deplete their boundary reflection energy down to the
global average $\bar R$.

%% file: sections/05_related_work.tex
\section{Related Work}
\label{sec:related_work}

\paragraph{Decoupling and unlocking backpropagation.}
System-level approaches mitigate sequential dependencies via pipelining: GPipe \cite{huang2019gpipe} and PipeDream \cite{narayanan2019pipedream} overlap forward and backward computation across microbatches without changing the learning rule. Algorithmic approaches instead decouple layer updates through local or approximate error signals; synthetic gradients \citep{jaderberg2017decoupled} learn predictors of gradients so that earlier layers update before the full backward pass, and equilibrium propagation \citep{scellier2017equilibrium} avoids explicit backpropagation by exploiting fixed-point conditions. Closest in spirit are approaches that compute gradients directly using the physics of the system \cite{cin2025training, pourcel2025lagrangian}. We do not approximate gradients or replace the adjoint dynamics; we introduce a relaxation whose fixed points satisfy the exact maximum-principle conditions while admitting concurrent propagation along depth.

\paragraph{Parallel-in-time and relaxation methods.}
The forward--backward structure of PMP has motivated shooting, multiple shooting, and forward--backward sweep methods in optimal control, and more generally parallel-in-time schemes such as Parareal \cite{staff2005stability} and multigrid-reduction-in-time \cite{FriedhoffEtAl2013} that iterate coarse-to-fine corrections to enable concurrency across time. Our worldsheet formulation is aligned in spirit but explicitly designs \emph{hyperbolic} propagation with finite speed and uses scattering variables that yield a passivity-based energy balance.

\paragraph{Backpropagation and the brain.}
A major motivation for alternatives to backpropagation is its limited biological plausibility \cite{bengio2015towards, lillicrap2020backpropagation}, with the weight transport and temporal credit assignment problems as central obstacles. Our framework suggests that brain matter, viewed as an excitable medium supporting wave propagation, scattering, and dissipation, may implicitly realize the optimization objective through its physical properties.

\paragraph{Energy-based and equilibrium models.}
Relaxation to a fixed point also underlies deep equilibrium models \citep{scellier2017equilibrium}, energy-based models, and the joint embedding predictive architecture \citep{lecun2022path,assran2023self,assran2025v,chen2025vl}, which treat inference as settling into a low-energy state. Implementations of JEPA still rely on global backpropagation for parameter updates, and EBMs typically model relaxation as a diffusive process with infinite signal speed. The wave-scattering construction differs in two ways: the relaxation is hyperbolic rather than diffusive, with explicit finite propagation speed $c<\infty$ and strictly local nearest-neighbor scattering; and parameter updates do not require the inner relaxation to settle, so the system can track non-stationary data without alternating between settling and updating.

\paragraph{Wave variables, scattering, and passivity.}
Wave and scattering variables are classical tools in circuit theory and passivity-based control; an early connection to optimal control appears in \cite{redheffer1960supplementary}. In networked and delayed control, scattering transformations preserve passivity and stability under communication delays, notably in bilateral teleoperation \citep{niemeyer1991stable,niemeyer2004telemanipulation} and time-delayed systems \cite{SlotineLohmiller2001,wang2006contraction}. Passivity analysis and port-Hamiltonian modeling provide a general framework for energy-conserving interconnections \cite{SlotineLi1987,anderson1989} and dissipative ports \citep{van2014port}, while contraction analysis \citep{lohmiller1998contraction} and Control Contraction Metrics (CCMs) \citep{manchester2017control} provide intrinsic Riemannian criteria, convex dual-metric synthesis, and submanifold contraction guarantees for nonlinear feedback design. Our energy identity and boundary scattering interpretation align with these traditions, but the ``transmission line'' here is the computational graph, dissipation corresponds to optimization progress, and state-dependent wave metrics directly inherit the dual-CCM and covariant Riemannian geometry of \citet{manchester2017control} (Appendix~\ref{app:covariant_ccm_geometry}).


\paragraph{Embedded physical dynamics and a common language.}
A longstanding divide separates physical mechanics, adaptive control, and deep learning: equations of motion are written in Hamiltonian or Euler--Lagrange form, control is formulated through passivity and Riccati equations, and neural networks are trained by non-local backpropagation. Writing the maximum principle in wave variables bridges these descriptions without changing the underlying variational structure. When the worldsheet coordinate is physical time, conservative mechanics appears as a lossless transmission medium, unknown couplings and contact forces appear as local scatterers driving parameter ports, and receding-horizon control appears as a matched future termination. Physical identification, passive control under delay, visual world-model learning, and continuous-time physics discovery (chaotic Neural ODEs \citep{chen2018neural}, Hamiltonian networks \citep{greydanus2019hamiltonian}, hybrid EventProp/Neural Event ODEs \citep{wunderlich2021eventprop,chen2021learning}, and noisy sparse equation discovery \citep{brunton2016discovering}) thus share a common local vocabulary of energy, impedance, and reflection (Appendices~\ref{sec:exp_physics_control}, \ref{app:wave_world_models}, and~\ref{app:neural_ode_hnn_physics}). In particular, continuous-time mirror-descent and nonparametric RKHS adaptive control \citep{fradkov1999nonlinear,boffi2021implicit,boffi2022nonparametric} emerge as the myopic ($H=0$), infinite-depth-speed ($c\to\infty$) limit of Wave-PMP with a Bregman nonlinear-capacitor parameter port, and our orthogonal-layer and matched-impedance constructions directly address the depth-stability open problem for nonlinearly parameterized neural adaptive controllers raised in \citet{boffi2022nonparametric} (Appendix~\ref{app:adaptive_nn_control}). 

This implies, in particular, that underlying sub-system physics can be incorporated explicitly as part of a neural network and its training, rather than being used merely as a source of training data through simulation. Of course, such direct embedding of physical equations (an analytical world-model of the corresponding dynamics) is a powerful form of data compression.

While the first version of this paper (\texttt{arXiv:2602.10461}) was posted in February 2026, recently \citet{seely2026pcalm} developed a related unlocked version of predictive coding (PC-ALM) with forward initialization. Actually, predictive coding \citep{rao1999predictive} is a particular instance of feedback hierarchies, originally studied in the presence of communication delays in Section~II of \citet{wang2006contraction}, which is one of the wave-variable papers that motivated the current work.

%% file: sections/06_discussion.tex
\section{Discussion}
\label{sec:discussion}

We have reformulated backpropagation as a local scattering process on a $(1+1)$-dimensional worldsheet. Restoring a finite propagation speed $c < \infty$ decouples the layers: state and adjoint information equilibrates through nearest-neighbor wave exchange while preserving the maximum-principle conditions at fixed points. Unlike equilibrium-based methods, parameter updates need not wait for the inner relaxation to converge; optimization can proceed concurrently with new samples streaming in. The discrete algorithm is reached either by depth-discretizing the continuum residual potential or by applying the variational principle directly to the finite-dimensional Pontryagin system---two faces of the same fixed-point structure.

Backpropagation has an $O(N)$ critical path: every parameter update waits for a full forward and backward sweep. Algorithm~\ref{alg:worldsheet_unlocked} allows all layers to compute simultaneously and supports streaming/asynchronous updates, but exact global credit still propagates at finite speed. Total optimization time remains bounded by the light cone $\tau \gtrsim N/c$, but no layer is idle: every layer processes wavefronts at every step, eliminating the pipeline bubbles of microbatch schemes \citep{huang2019gpipe,narayanan2019pipedream}. The locking cost is most acute in recurrent and continuous-time models and we see promise of this method there.

The choice of parameter update rule is equivalent to the choice of impedance at the parameter port. Gradient descent is a uniform resistive load, producing oscillatory reflections when curvature mismatches the learning rate. Newton-type methods are transparent terminations near a local minimum. Convergence rate is set by the mismatch between impedance and local curvature, and the activation-space pullback metrics of Section~\ref{sec:training_neural_network} together with the covariant state-dependent metric and convex dual-CCM synthesis of Appendix~\ref{app:covariant_ccm_geometry} \citep{manchester2017control} suggest that matching can be approached both layer by layer and intrinsically across nonlinear state spaces. One interesting future direction is extending these convex state-dependent metric constructions to scalable online adaptive impedances.

\textbf{High-Gain Electronics and Energy Efficiency.}
In analog and neuromorphic electronics, enforcing algebraic constraints via high-gain feedback (analogous to large wave speed $c$ driving $r_x, r_\lambda \to 0$) collides with interconnect propagation delay: phase lag turns high-gain feedback loops unstable unless bandwidth is severely rolled off. Wave-variable scattering resolves this gain--delay bottleneck: passivity is preserved independently of interconnect delay \citep{niemeyer2004telemanipulation,wang2006contraction}, and every local affine junction remains norm-bounded ($\|S_A\|_2 = 1$) even at arbitrarily high layer gain, provided the hardware supplies a reciprocal or explicit dual path for $J^\top$.

While simulating wave relaxation on synchronous digital GPUs generally incurs overhead relative to standard backpropagation, the formulation directly targets the dominant physical energy bottlenecks of analog and wave-supporting substrates: (i) removing global activation buffering and long-range memory movement through strictly local nearest-neighbor exchange; (ii) minimizing reflected power ($\|b\|^2$ in $e^\top v = \tfrac12(\|a\|^2 - \|b\|^2)$) and settling time via link and parameter impedance matching; (iii) amortizing relaxation energy on streaming and receding-horizon tasks via persistent, incremental fields (Appendix~\ref{sec:exp_physics_control}); and (iv) replacing overdamped resistive creep with critically damped reactive/inertial dynamics (telegrapher and heavy-ball matching). Realizing net Joule savings remains substrate-dependent---balancing these architectural gains against static amplifier bias power, analog noise floors, and the implementation of local sources and terminations.

%% file: sections/07_limitations.tex
\section{Limitations}

Three open choices govern stability and convergence: the Courant number, the per-layer activation metrics $G_k$ or impedances $\Theta_k$, and the parameter port impedance $Z_{\theta,k}$. Adaptive or state-dependent metrics offer a principled remedy: while the continuum theory guarantees energy dissipation under passive boundary conditions and extends covariantly to curved state-dependent metrics and convex dual-CCM synthesis \citep{manchester2017control} (Appendix~\ref{app:covariant_ccm_geometry}), stability of the discrete scattering iteration under online adaptive metrics and scalable synthesis of nonlinear metrics in high dimensions remain open. Memory is comparable to backpropagation, with wave variables replacing stored activations. The most pressing directions are theoretically grounded stability analysis at width-changing junctions and saturating activations, and empirical evaluation on recurrent settings where the BPTT locking cost is most pronounced.

%% file: appendix/A_worldsheet_relaxation.tex
\section{Worldsheet relaxation details}
\label{app:residual_relaxation_details}

Section~\ref{app:paracomplex_proofs} proves the statements of
Section~\ref{sec:variational_worldsheet} about the first-order relaxation
\eqref{eq:para_floer}. Section~\ref{app:directional_riccati_waves} develops
the continuous-time Riccati-aligned directional paracomplex involution
$\mathbb K_{P,0}$ with local Maxwell-in-$\tau$ memory for strongly chaotic,
indefinite-Hessian optimal control problems. Section~\ref{app:zeros_and_relaxation} gives the
second-order relaxation with general residual weights, and
Section~\ref{app:explicit_drstar} its explicit adjoint residual operator.
Section~\ref{app:wave_variables} treats the velocity--residual waves of the
second-order form. Section~\ref{app:mechanical_telegrapher} collects the
mechanical-energy and telegrapher results.

\subsection{The first-order hyperbolic relaxation}
\label{app:paracomplex_proofs}

\paragraph{Algebra of $\mathbb K$.}
Let $\mathbb M=\operatorname{diag}(G,G^{-1})$ and
$\mathbb P=\bigl(\begin{smallmatrix}0&I\\I&0\end{smallmatrix}\bigr)$. Direct
multiplication gives
\begin{equation}
\mathbb K^2=I,\qquad \mathbb M\mathbb K=\mathbb P=(\mathbb M\mathbb K)^\top,\qquad
\mathbb K^\top\mathbb M\mathbb K=\mathbb M .
\label{eq:K_algebra_app}
\end{equation}
So $\mathbb K$ is an $\mathbb M$-self-adjoint $\mathbb M$-isometry, and
$\tfrac12\langle\cdot,\cdot\rangle_{\mathbb P}$ is the power pairing
$\eta^\top\xi$. The map
$U:(\xi,\eta)\mapsto(w_+,w_-)$ of \eqref{eq:residual_characteristic_waves_main}
satisfies $U\mathbb KU^{-1}=\operatorname{diag}(I,-I)$ and
$U^{-\top}\mathbb PU^{-1}=\operatorname{diag}(I,-I)$, which gives the power
identity in \eqref{eq:residual_characteristic_waves_main} and the projected
equations \eqref{eq:characteristic_transport_main}.

\paragraph{Principal symbols.}
For $y\propto e^{ikt}$ the principal parts of Floer's equation and of
\eqref{eq:para_floer} are $\partial_s y=-ikJy$ and
$\partial_\tau y=-ick\,\mathbb Ky$. Since $J^2=-1$, the eigenvalues of $J$
are $\pm i$ and those of $-ikJ$ are $\pm k$: half of the modes grow as
$e^{|k|s}$. Since $\mathbb K^2=1$, the eigenvalues of $-ick\,\mathbb K$ are
$\pm ick$: all modes oscillate with phase speed $c$.

\paragraph{Worldsheet signature.}
The sign of $\mathbb K^2$ is set by the signature of the worldsheet metric.
For $h_\sigma=\dd t^2-\sigma c^2\dd\tau^2$ and
$\mathbb J_\sigma=\bigl(\begin{smallmatrix}0&G^{-1}\\ \sigma G&0\end{smallmatrix}\bigr)$,
so that $\mathbb J_\sigma^2=\sigma I$, the equation
$\partial_\tau y+c\,\mathbb J_\sigma\bigl(\partial_t y-X_H(y)\bigr)=0$
is Floer's equation for $\sigma=-1$ and \eqref{eq:para_floer} for $\sigma=+1$.
Its principal symbol has eigenvalues $\pm ikc\sqrt\sigma$. For Lorentzian
$h_\sigma$ ($\sigma>0$) the equation is hyperbolic with speed $c\sqrt\sigma$;
for Riemannian $h_\sigma$ ($\sigma<0$) it is elliptic; at the degenerate
metric $\sigma=0$ the symbol is nilpotent. A change of type therefore passes
through a degenerate worldsheet metric, and only the Lorentzian side is an
initial-value problem in $\tau$. The second-order action
$\tfrac12\int\sqrt{|h|}\,h^{ab}\langle D_ay,D_by\rangle_{\mathbb M}\,\dd t\,\dd\tau$,
with $D_ty=\partial_ty-X_H$ and $D_\tau y=\partial_\tau y$, exists for both
signatures. For Riemannian $h$ it is Floer's energy, which splits into
$\|\partial_\tau y+Jr\|_{\mathbb M}^2$ and a boundary term proportional to the
change of the action functional \eqref{eq:augmented_action}, so the action
decreases monotonically along Floer's equation. For Lorentzian $h$ it is
\eqref{eq:nonlinear_worldsheet_action}, which factorizes along the light
cone (see below), and the decrease of energy comes from the convexity
condition of Theorem~\ref{thm:nonlinear_wave_energy_balance} instead. On the
Lorentzian side, a metric that varies over the worldsheet amounts, through
its light cone, to a local propagation speed $c(t,\tau)$, together with a
conformal factor that weights the Hamiltonian term.

\begin{proof}[Proof of Theorem~\ref{thm:nonlinear_wave_energy_balance}]
Linearizing \eqref{eq:para_floer} along a solution with smoothly varying
spatial metric $G(t)=\Theta(t)^\top\Theta(t)\succ0$ ($\partial_\tau G=0$)
gives the variational equations
\begin{align*}
\partial_\tau\delta x&=-c\,G(t)^{-1}\bigl(\partial_t\delta\lambda+H_{xx}\delta x+H_{x\lambda}\delta\lambda\bigr),\\
\partial_\tau\delta\lambda&=-c\,G(t)\bigl(\partial_t\delta x-H_{\lambda x}\delta x-H_{\lambda\lambda}\delta\lambda\bigr),
\end{align*}
with $H_{\lambda x}=H_{x\lambda}^\top$ and Hessians evaluated along the
solution. By \eqref{eq:residual_characteristic_waves_main},
$\tfrac12(\|\delta w_+\|^2+\|\delta w_-\|^2)=\tfrac12(\delta x^\top G(t)\delta x+\delta\lambda^\top G(t)^{-1}\delta\lambda)$.
Differentiating in $\tau$ (with $\partial_\tau G=0$) and inserting the
variational equations cancels $G(t)$ and $G(t)^{-1}$ pointwise at every $t$:
\begin{align*}
\frac{\dd\cE_{\rm wave}}{\dd\tau}
&=-c\int_0^T\bigl(\delta x^\top\partial_t\delta\lambda+\delta\lambda^\top\partial_t\delta x\bigr)\dd t\\
&\quad-c\int_0^T\bigl(\delta x^\top H_{xx}\delta x+\delta x^\top H_{x\lambda}\delta\lambda
-\delta\lambda^\top H_{\lambda x}\delta x-\delta\lambda^\top H_{\lambda\lambda}\delta\lambda\bigr)\dd t .
\end{align*}
The first integral is $-c\,[\delta\lambda^\top\delta x]_0^T$, which by the power
identity equals the boundary terms in
\eqref{eq:full_nonlinear_wave_energy_balance_main}. In the second integral the
mixed terms cancel because they are transposes of each other, which leaves
$-D_\Sigma$. (Equivalently, in the moving-frame equations
\eqref{eq:bulk_port_waves_plus}--\eqref{eq:bulk_port_waves_minus},
$\widehat A_{\mathrm{skew}}$ drops out by skew-symmetry and the cross-terms
$\pm c\,\delta w_\pm^\top\widehat A_{\mathrm{sym}}\delta w_\mp$ cancel
identically.) At $t=0$, the condition $x(0)=x_0$ gives $\delta x(0)=0$ and a
vanishing boundary power. At $t=T$, the condition $\lambda(T)=\nabla\Phi(x(T))$
gives $\delta\lambda(T)=\nabla^2\Phi\,\delta x(T)$ and boundary power
$-c\,\delta x(T)^\top\nabla^2\Phi\,\delta x(T)\le0$ for convex $\Phi$.
\end{proof}

\begin{proposition}[Lagrange--d'Alembert form]
\label{prop:lda_form_app}
Let $G(t)$ be smoothly varying in $t$ ($\partial_\tau G=0$), let the
Hessians of $H$ be constant in $\tau$, and write $F=DX_H$. Every
solution of \eqref{eq:para_floer} satisfies
\eqref{eq:full_nonlinear_worldsheet_relaxation}, where
$-\delta\cS_\Sigma^{\rm pot}/\delta y=\partial_t(\mathbb M(t)r)+F^\top\mathbb M(t)r$
is the force of the potential term of \eqref{eq:nonlinear_worldsheet_action}.
\end{proposition}

\begin{proof}
Since $\mathbb M(t)\mathbb K(t)=\mathbb P$ and $\mathbb P\mathbb K(t)=\mathbb M(t)$
with constant $\mathbb P$, differentiating $\partial_\tau y=-c\,\mathbb K(t)r$
in $\tau$ gives
\[
c^{-2}\mathbb M(t)\partial_\tau^2y
=-c^{-1}\mathbb P\bigl(\partial_t\partial_\tau y-F\partial_\tau y\bigr)
=\partial_t\bigl(\mathbb M(t)r\bigr)-\mathbb M(t)\mathbb K(t)F\mathbb K(t)\,r .
\]
Subtracting the Euler--Lagrange force leaves $Dr$ with
$D=-\mathbb M(t)\mathbb K(t)F\mathbb K(t)-F^\top\mathbb M(t)$. With
$F=\bigl(\begin{smallmatrix}H_{\lambda x}&H_{\lambda\lambda}\\-H_{xx}&-H_{x\lambda}\end{smallmatrix}\bigr)$,
a block computation gives
$D=2\operatorname{diag}(H_{xx},-H_{\lambda\lambda})\,\mathbb K(t)$: the mixed
blocks cancel. Inserting $r=-c^{-1}\mathbb K(t)\partial_\tau y$ and $\mathbb K(t)^2=I$
gives $Dr=-\tfrac2c\operatorname{diag}(H_{xx},-H_{\lambda\lambda})\partial_\tau y$.
\end{proof}

\paragraph{Moving-frame action expansion and exact finite-stretch scattering.}
Write $M(t)\equiv\widehat A_{\mathrm{sym}}(t)=M(t)^\top$ and
$D_t\equiv\partial_t-\widehat A_{\mathrm{skew}}$. Expanding
$\|D_t\delta p-M\delta p\|^2=\|D_t\delta p\|^2+\|M\delta p\|^2-2\langle D_t\delta p,M\delta p\rangle$
and using $\partial_t(\delta p^\top M\delta p)=2\langle D_t\delta p,M\delta p\rangle
+\delta p^\top(\dot M+[M,\widehat A_{\mathrm{skew}}])\delta p$
(since $\widehat A_{\mathrm{skew}}^\top=-\widehat A_{\mathrm{skew}}$)
yields \eqref{eq:moving_frame_action_mass} verbatim, with the opposite
sign on $M$ for $\delta q$. Finally, for a pure stationary stretch
($\partial_\tau=0$, $K=L=0$, $\widehat A_{\mathrm{skew}}=0$,
$\int_0^{\Delta t}\widehat A_{\mathrm{sym}}\,\dd t=\operatorname{diag}(d_i)$),
\eqref{eq:bulk_port_waves_plus}--\eqref{eq:bulk_port_waves_minus} integrate
channelwise to the hyperbolic transfer matrix
$\bigl(\begin{smallmatrix}w_+(\Delta t)\\ w_-(\Delta t)\end{smallmatrix}\bigr)
=\bigl(\begin{smallmatrix}\cosh d_i&\sinh d_i\\ \sinh d_i&\cosh d_i\end{smallmatrix}\bigr)
\bigl(\begin{smallmatrix}w_+(0)\\ w_-(0)\end{smallmatrix}\bigr)$.
Solving for outgoing $(z,y)=(w_-(0),w_+(\Delta t))$ in terms of incoming
$(u,v)=(w_+(0),w_-(\Delta t))$ gives
$z=-\tanh(d_i)\,u+\operatorname{sech}(d_i)\,v$ and
$y=\operatorname{sech}(d_i)\,u+\tanh(d_i)\,v$, which matches
\eqref{eq:exact_linear_scattering_main}--\eqref{eq:singular_reflection_main}
identically upon setting $s_i=e^{d_i}$ and rotating input/output frames by
$V_k^\top,U_k$.

\paragraph{Factorization and the two forms.}
With $a=\partial_\tau y$ and $b=r$, \eqref{eq:K_algebra_app} gives
$c^{-2}\|a\|_{\mathbb M}^2-\|b\|_{\mathbb M}^2
=\langle c^{-1}a+\mathbb Kb,\,c^{-1}a-\mathbb Kb\rangle_{\mathbb M}$.
The first-order flow sets the first factor to zero. The second-order
equation has twice as many Cauchy data; its solutions with
$c^{-1}\partial_\tau y+\mathbb Kr\neq0$ at $\tau=0$ are not solutions of
\eqref{eq:para_floer}. On the first-order submanifold the
velocity--residual waves of Section~\ref{app:wave_variables} reduce to
$\omega_\pm=\mp\sqrt2\,\Theta_y\Pi_\pm r$ with the spectral projectors
$\Pi_\pm=\tfrac12(I\pm\mathbb K)$: each null direction carries one
polarization of the residual.

\paragraph{Splittings.}
For $L_\rightarrow=\mathrm{graph}(Y)$ and $L_\leftarrow=\{(-Z\mu,\mu)\}$,
the decomposition $(\xi,\eta)=(\xi_\rightarrow,Y\xi_\rightarrow)+(-Z\mu,\mu)$
has $\xi_\rightarrow=(I+ZY)^{-1}(\xi+Z\eta)$ and $\mu=\eta-Y\xi_\rightarrow$.
Expanding $\eta^\top\xi$ gives
$\xi_\rightarrow^\top Y\xi_\rightarrow-\mu^\top Z\mu+\mu^\top(I-ZY)\xi_\rightarrow$.
Both subspaces are Lagrangian because $Y$ and $Z$ are symmetric. For a
lossless layer $x_{k+1}=J_kx_k$, $\lambda_k=J_k^\top\lambda_{k+1}$, an
incident forward wave produces no backward wave if and only if
$Y_k=J_k^\top Y_{k+1}J_k$, and an incident backward wave produces no forward
wave if and only if $Z_{k+1}=J_kZ_kJ_k^\top$. Running costs and process
noise add $Q_k$ and $W_k$, which turns these into the control and filter
Riccati recursions.

\subsection{Riccati-aligned directional wave relaxation with Maxwell-in-$\tau$ memory}
\label{app:directional_riccati_waves}

When the underlying open-loop dynamics are strongly unstable (large positive
Lyapunov exponents) and the Pontryagin Hessian $H_{xx}=L_{xx}+\lambda^\top f_{xx}$
is strongly indefinite along the trajectory---as in underactuated swing-up
problems such as the four-state Acrobot---the single-metric relaxation
\eqref{eq:para_floer} suffers both a nonzero bulk reflection barrier
($\mathcal L_A G\neq0$) and a loss of bulk passivity in
Theorem~\ref{thm:nonlinear_wave_energy_balance} ($D_\Sigma\not\ge0$).
The bi-Lagrangian splitting above resolves both obstructions within the same
first-order hyperbolic worldsheet equation by replacing the symmetric
paracomplex involution $\mathbb K_{G,G^{-1}}$ with a Riccati-aligned
involution.

\paragraph{General bi-Lagrangian involution.}
Given symmetric operators $Y,Z\succeq0$ defining the directional wave map
$w_+=x+Z\lambda$ and $w_-=\lambda-Yx$, transforming the characteristic
eigenvalue matrix $\operatorname{diag}(I,-I)$ back to $(x,\lambda)$ yields the
two-operator paracomplex involution
\begin{equation}
\label{eq:general_K_YZ_app}
\mathbb K_{Y,Z}
=\begin{pmatrix}I&Z\\-Y&I\end{pmatrix}^{-1}
\begin{pmatrix}I&0\\0&-I\end{pmatrix}
\begin{pmatrix}I&Z\\-Y&I\end{pmatrix}
=\begin{pmatrix}
(I+ZY)^{-1}(I-ZY) & 2(I+ZY)^{-1}Z\\
2(I+YZ)^{-1}Y & -(I+YZ)^{-1}(I-YZ)
\end{pmatrix}.
\end{equation}
Direct block multiplication confirms $\mathbb K_{Y,Z}^2=+I$ and
$\mathbb K_{Y,Z}^\top\Omega+\Omega\mathbb K_{Y,Z}=0$ for
$\Omega=\bigl(\begin{smallmatrix}0&I\\-I&0\end{smallmatrix}\bigr)$, so
$\mathbb K_{Y,Z}$ is an anti-symplectic paracomplex structure for every
symmetric pair $(Y,Z)$. Two canonical limits sit inside
\eqref{eq:general_K_YZ_app}:
\begin{itemize}
\item \emph{Single-metric symmetric waves \eqref{eq:para_floer}:}
$(Y,Z)=(G,G^{-1})$ gives
$\mathbb K_{G,G^{-1}}=\bigl(\begin{smallmatrix}0&G^{-1}\\G&0\end{smallmatrix}\bigr)$.
\item \emph{Riccati-aligned directional waves:}
$(Y,Z)=(P,0)$ gives the triangular involution
\begin{equation}
\label{eq:K_P0_app}
\mathbb K_{P,0}=\begin{pmatrix}I&0\\2P&-I\end{pmatrix},
\qquad \mathbb K_{P,0}^2=+I,
\qquad \mathbb K_{P,0}^\top\Omega+\Omega\mathbb K_{P,0}=0,
\end{equation}
whose $+1$ Lagrangian eigenspace is $L_\rightarrow=\{\lambda=Px\}$ and whose
$-1$ Lagrangian eigenspace is $L_\leftarrow=\{x=0\}$, with directional wave
variables $w_+=x$ and $w_-=\lambda-Px$.
\end{itemize}

\paragraph{Continuous-time $(t,\tau)$ hyperbolic PDE with Maxwell-in-$\tau$ memory.}
To keep local cell updates well-conditioned far from optimality without
biasing the converged Pontryagin solution, introduce two stationary
(zero-wave-speed) cell-local viscoelastic memory fields
$\bar x(t,\tau),\bar u(t,\tau)$ with stiffnesses $\mu_x,\mu_u\ge0$ and
relaxation rate $\eta>0$, and define the Maxwell-regularized Pontryagin
Hamiltonian
\begin{equation}
\label{eq:maxwell_hamiltonian_app}
H_{\mathrm M}(x,\lambda,u;\bar x,\bar u)
=L(x,u)+\frac{\mu_x}{2}\|x-\bar x\|^2+\frac{\mu_u}{2}\|u-\bar u\|^2+\lambda^\top f(x,u),
\end{equation}
with local algebraic minimizer $u^\star(x,\lambda;\bar u)=\arg\min_u H_{\mathrm M}$
solving $\nabla_u L+\mu_u(u^\star-\bar u)+f_u^\top\lambda=0$. Let
$r_x=\partial_t x-f(x,u^\star)$ and
$r_\lambda=\partial_t\lambda+\nabla_x L(x,u^\star)+\mu_x(x-\bar x)+f_x(x,u^\star)^\top\lambda$.
Replacing $\mathbb K$ by $\mathbb K_{P,0}$ in \eqref{eq:para_floer} and letting
the reflection-canceling admittance $P(t,\tau)\succeq0$ ride the left-moving
characteristic gives the coupled continuous-time worldsheet system
\begin{equation}
\label{eq:directional_compact_app}
\begin{gathered}
\partial_\tau\begin{pmatrix}x\\\lambda\end{pmatrix}
+c\begin{pmatrix}I&0\\2P&-I\end{pmatrix}\begin{pmatrix}r_x\\r_\lambda\end{pmatrix}=0,
\qquad
(\partial_\tau-c\,\partial_t)P=c\,\mathcal R(P,x,u^\star), \\[4pt]
\partial_\tau\bar x=\eta(x-\bar x),
\qquad
\partial_\tau\bar u=\eta(u^\star-\bar u),
\end{gathered}
\end{equation}
where $\mathcal R(P,x,u)=Q_{\mathrm{eff}}(x,u)+f_x^\top P+Pf_x-P S_{\mathrm{eff}}(x,u)P$
is the Gauss--Newton Riccati operator with
$Q_{\mathrm{eff}}=L_{xx}+\mu_x I\succeq0$ and
$S_{\mathrm{eff}}=f_u(L_{uu}+\mu_u I)^{-1}f_u^\top\succeq0$.

Eliminating $\lambda=Px+w_-$ via $(\partial_\tau-c\partial_t)P=c\,\mathcal R$
cancels $(\partial_t P)x$ identically and casts \eqref{eq:directional_compact_app}
into strict diagonal first-order hyperbolic normal form with characteristic
speeds $(+c,-c,-c,0,0)$:
\begin{equation}
\label{eq:directional_normal_form_app}
\begin{aligned}
(\partial_\tau+c\,\partial_t)\,x
&=c\,f(x,u^\star), &&\quad(\text{speed }+c)\\
(\partial_\tau-c\,\partial_t)\,w_-
&=c\Bigl(\nabla_x L(x,u^\star)+\mu_x(x-\bar x)+f_x^\top(Px+w_-) \\
&\qquad +P f(x,u^\star)-\mathcal R(P,x,u^\star)\,x\Bigr), &&\quad(\text{speed }-c)\\
(\partial_\tau-c\,\partial_t)\,P
&=c\,\mathcal R(P,x,u^\star), &&\quad(\text{speed }-c)\\
\partial_\tau\bar x&=\eta(x-\bar x),\quad \partial_\tau\bar u=\eta(u^\star-\bar u), &&\quad(\text{speed }0)
\end{aligned}
\end{equation}
with boundary ports $x(0,\tau)=x_0$,
$P(T,\tau)=\nabla^2\Phi(x(T,\tau))$, and
$w_-(T,\tau)=\nabla\Phi(x(T,\tau))-P(T,\tau)x(T,\tau)$
(for quadratic terminal cost $\Phi(x)=\tfrac12(x-r_f)^\top P_T(x-r_f)$, simply
$P(T,\tau)=P_T$ and $w_-(T,\tau)=-P_T r_f$).

\begin{proposition}[Unbiased nonlinear PMP fixed points and LQ decoupling]
\label{prop:directional_pmp_app}
Let $(x,w_-,P,\bar x,\bar u)$ be any stationary solution ($\partial_\tau=0$)
of \eqref{eq:directional_normal_form_app} with $\lambda:=Px+w_-$.
\begin{enumerate}
\item \emph{Exact nonlinear maximum principle:} Regardless of the choice of
operator $\mathcal R$, the memory springs vanish identically
($\bar x=x$, $\bar u=u^\star$) and $(x,\lambda,u^\star)$ satisfies the exact
unbiased nonlinear Pontryagin boundary-value problem
$\partial_t x=f(x,u^\star)$,
$\partial_t\lambda=-\nabla_x H(x,\lambda,u^\star)$,
$\nabla_u H(x,\lambda,u^\star)=0$, with $x(0)=x_0$ and
$\lambda(T)=\nabla\Phi(x(T))$.
\item \emph{Exact reflectionless LQ decoupling:} In the linear-quadratic case
$f(x,u)=Ax+Bu$, $L(x,u)=\tfrac12(x-r)^\top Q(x-r)+\tfrac12 u^\top R u$ with
$\mu_x=\mu_u=0$ and $S=BR^{-1}B^\top$, all $x$-terms cancel out of the $w_-$
equation, reducing \eqref{eq:directional_normal_form_app} to
\begin{equation}
\label{eq:directional_lq_decoupled_app}
(\partial_\tau+c\,\partial_t)\,x=c\bigl((A-SP)x-Sw_-\bigr),\qquad
(\partial_\tau-c\,\partial_t)\,w_-=c\bigl((A-SP)^\top w_--Qr\bigr),
\end{equation}
so both forward and backward waves propagate under the closed-loop stabilized
operator $A-SP$ rather than the open-loop dynamics $A$, with zero bulk
reflection from $x$ into $w_-$ and exact convergence in one round trip
$\tau=2T/c$.
\end{enumerate}
\end{proposition}

\begin{proof}
(1)~At $\partial_\tau=0$, the memory equations give $\bar x=x$ and
$\bar u=u^\star$, so $\nabla_u H(x,\lambda,u^\star)=0$ and
$\partial_t x=f(x,u^\star)$. The remaining two equations read
$-\partial_t P=\mathcal R$ and
$-\partial_t w_-=\nabla_x L+f_x^\top\lambda+Pf-\mathcal R x$.
Differentiating $\lambda=Px+w_-$ gives
$\partial_t\lambda=(\partial_t P)x+P\partial_t x+\partial_t w_-
=-\mathcal R x+Pf-(\nabla_x L+f_x^\top\lambda+Pf-\mathcal R x)
=-\nabla_x L-f_x^\top\lambda$, where $\mathcal R x$ cancels identically.
Boundary ports give $x(0)=x_0$ and
$\lambda(T)=P(T)x(T)+w_-(T)=\nabla\Phi(x(T))$.
(2)~For $f=Ax+Bu$, $u^\star=-R^{-1}B^\top(Px+w_-)$ gives
$f(x,u^\star)=(A-SP)x-Sw_-$. Substituting
$\nabla_x L=Q(x-r)$, $f_x=A$, and $\mathcal R=Q+A^\top P+PA-PSP$ into the
right-hand side of the $w_-$ equation cancels $(Q+A^\top P+PA-PSP)x$ term by
term, leaving $(A-SP)^\top w_--Qr$.
\end{proof}

\paragraph{Unlocked discrete cell scatterer.}
Discretizing \eqref{eq:directional_normal_form_app} on cells $[t_k,t_{k+1}]$
with RK2 step $\Phi_k(x_k,u_k)$ (Jacobians $A_k=\Phi_{x,k}$, $B_k=\Phi_{u,k}$)
and smooth box-constrained control cost
$L_u(u)=Ru_{\max}^2\bigl[\tfrac{u}{u_{\max}}\operatorname{artanh}(\tfrac{u}{u_{\max}})+\tfrac12\ln(1-u^2/u_{\max}^2)\bigr]$
(the exact Legendre dual of $u^\star=u_{\max}\tanh(-B^\top\lambda/(Ru_{\max}))$)
produces a strictly nearest-neighbor local cell map: given incoming $x_k$ from
cell $k-1$ and $(w_{-,k+1},P_{k+1})$ from cell $k+1$, each cell performs a
single scalar Gauss--Newton update on $u_k$ using
$\lambda_{k+1}=P_{k+1}\Phi_k(x_k,u_k)+w_{-,k+1}$, emits the forward wave
$x_{k+1}^{\mathrm{new}}=\Phi_k(x_k,u_k)$, updates the backward admittance via
the local discrete Riccati step
$P_k^{\mathrm{new}}=Q_{\mathrm{eff},k}+A_k^\top P_{k+1}A_k
-(A_k^\top P_{k+1}B_k)H_{uu,k}^{-1}(B_k^\top P_{k+1}A_k)$,
and emits $w_{-,k}^{\mathrm{new}}=\lambda_k^{\mathrm{new}}-P_k^{\mathrm{new}}x_k$
with $\lambda_k^{\mathrm{new}}=\Delta t(\nabla_x L_k+\mu_x(x_k-\bar x_k))+A_k^\top\lambda_{k+1}$,
followed by CFL relaxation with factor $\omega\in(0,1]$ and local memory
update $\bar x_k\leftarrow(1-\eta)\bar x_k+\eta x_k$,
$\bar u_k\leftarrow(1-\eta)\bar u_k+\eta u_k$. Every cell updates in parallel
(Red--Black or synchronous Jacobi) with zero outer SQP loops, zero sequential
rollouts, and zero global line search. An appliccation to the Acrobot task are provided in Appendix~\ref{app:ex_acrobot_directional}.

\subsection{Second-order relaxation with general residual weights}
\label{app:zeros_and_relaxation}

For positive-definite weights $\mathsf W\succ0$, $\mathsf W_b\succ0$ and
the full residual including the control and the boundary,
\begin{equation}
\label{eq:full_residual_app}
\cR(z)=
\begin{bmatrix}
\partial_t y-F_y(y,u,t)\\
\nabla_u H(x,\lambda,u,t)
\end{bmatrix},
\qquad
\cB(z)=
\begin{bmatrix}
x(0,\tau)-x_0\\
\lambda(T,\tau)-\nabla_x\Phi(x(T,\tau))
\end{bmatrix},
\end{equation}
with $z=(x,\lambda,u)$ and $F_y=(\nabla_\lambda H,-\nabla_xH)$, define the
residual potential
\begin{equation}
\label{eq:residual_potential}
\cV[z]=\tfrac12\|\cR(z)\|_{\mathsf W}^2+\tfrac12\|\cB(z)\|_{\mathsf W_b}^2 .
\end{equation}

\begin{proposition}[Zeros of the residual potential]
\label{prop:zeros_residual_potential}
$\cV[z]=0$ if and only if the Pontryagin residuals vanish and the boundary
conditions $x(0,\tau)=x_0$ and $\lambda(T,\tau)=\nabla_x\Phi(x(T,\tau))$ hold.
\end{proposition}

\begin{proof}
Every term in \eqref{eq:residual_potential} is nonnegative. Since the
weights are positive definite, the sum is zero exactly when each residual
and boundary residual is zero.
\end{proof}

With inertia metric $\mathsf M_z\succ0$, the worldsheet action
\begin{equation}
\label{eq:general_worldsheet_action_app}
\cS_\Sigma^{\cV}[z]
=\int_0^\infty\left\{
\frac{1}{2c^2}\int_0^T(\partial_\tau z)^\top\mathsf M_z(\partial_\tau z)\,\dd t
-\cV[z(\cdot,\tau)]
\right\}\dd\tau
\end{equation}
reduces to \eqref{eq:nonlinear_worldsheet_action} for
$\mathsf W=\mathbb M$, $\mathsf M_z=\mathbb M$, $u$ eliminated and the boundary
residuals imposed exactly.

\begin{proposition}[Second-order worldsheet relaxation]
\label{prop:full_nonlinear_worldsheet_relaxation}
The adjoint $D\cR(z)^*$ is the formal $L^2$ adjoint and includes
integration by parts in $t$. The Lagrange--d'Alembert equation of
\eqref{eq:general_worldsheet_action_app} with Rayleigh damping
$\frac12\int z_\tau^\top\Gamma\mathsf M_z z_\tau\,\dd t$ is
\begin{equation}
\label{eq:full_nonlinear_worldsheet_relaxation_app}
\frac{1}{c^2}\mathsf M_z z_{\tau\tau}+\Gamma\mathsf M_z z_\tau
=-D\cR(z)^*\mathsf W\cR(z)-D\cB(z)^*\mathsf W_b\cB(z).
\end{equation}
\end{proposition}

\begin{proof}
For a variation $\delta z$,
$\delta\cV=\langle D\cR(z)^*\mathsf W\cR(z)+D\cB(z)^*\mathsf W_b\cB(z),\delta z\rangle$.
Substituting into the damped Euler--Lagrange equation for the kinetic
energy $\frac{1}{2c^2}\int z_\tau^\top\mathsf M_z z_\tau\,\dd t$ gives the
displayed equation.
\end{proof}

Here the damping $\Gamma$ is chosen freely. In the first-order relaxation
of Section~\ref{sec:variational_worldsheet}, the damping is fixed by the
cost Hessian (Proposition~\ref{prop:lda_form_app}). The overdamped limit of
\eqref{eq:full_nonlinear_worldsheet_relaxation_app} is the residual-gradient
flow $\Gamma\mathsf M_z\partial_\tau z=-\delta\cV/\delta z$.

\subsection{Explicit form of the nonlinear adjoint residual}
\label{app:explicit_drstar}

This subsection expands the interior part of $D\cR(z)^*$. We adopt the convention that $H_{ab}$ denotes the Jacobian of $\nabla_a H$ with respect to $b$, evaluated along the current worldsheet field. Write
\begin{equation}
\label{eq:rho_def_app}
\rho_y=\begin{bmatrix}\rho_x\\ \rho_\lambda\end{bmatrix}
=\mathsf G_y\begin{bmatrix}r_x\\ r_\lambda\end{bmatrix},
\qquad
\rho_u=\mathsf G_u r_u,
\end{equation}
where $\mathsf G_y$ and $\mathsf G_u$ are the blocks of $\mathsf W$ and $r_u=\nabla_uH$.
For a variation $(\xi,\eta,\nu)$ of $(x,\lambda,u)$,
\begin{equation}
\label{eq:DR_explicit_app}
D\cR(z)\begin{bmatrix}\xi\\ \eta\\ \nu\end{bmatrix}
=\begin{bmatrix}
\xi_t-H_{\lambda x}\xi-H_{\lambda\lambda}\eta-H_{\lambda u}\nu\\
\eta_t+H_{xx}\xi+H_{x\lambda}\eta+H_{xu}\nu\\
H_{ux}\xi+H_{u\lambda}\eta+H_{uu}\nu
\end{bmatrix}.
\end{equation}
Integrating the first two rows by parts in $t$ gives the interior formal adjoint
\begin{align}
\label{eq:DRstar_x_app}
(D\cR(z)^*\rho)_x
&=-\partial_t\rho_x-H_{\lambda x}^\top\rho_x+H_{xx}^\top\rho_\lambda+H_{ux}^\top\rho_u,\\
\label{eq:DRstar_lambda_app}
(D\cR(z)^*\rho)_\lambda
&=-\partial_t\rho_\lambda-H_{\lambda\lambda}^\top\rho_x+H_{x\lambda}^\top\rho_\lambda+H_{u\lambda}^\top\rho_u,\\
\label{eq:DRstar_u_app}
(D\cR(z)^*\rho)_u
&=-H_{\lambda u}^\top\rho_x+H_{xu}^\top\rho_\lambda+H_{uu}^\top\rho_u.
\end{align}
Substituting into \eqref{eq:full_nonlinear_worldsheet_relaxation_app}, the interior nonlinear worldsheet equations become
\begin{align}
\label{eq:explicit_nonlinear_x_relax_app}
\frac{1}{c^2}\mathsf M_x x_{\tau\tau}+\Gamma_x\mathsf M_x x_\tau
&=\partial_t\rho_x+H_{\lambda x}^\top\rho_x-H_{xx}^\top\rho_\lambda-H_{ux}^\top\rho_u,\\
\label{eq:explicit_nonlinear_lambda_relax_app}
\frac{1}{c^2}\mathsf M_\lambda\lambda_{\tau\tau}+\Gamma_\lambda\mathsf M_\lambda\lambda_\tau
&=\partial_t\rho_\lambda+H_{\lambda\lambda}^\top\rho_x-H_{x\lambda}^\top\rho_\lambda-H_{u\lambda}^\top\rho_u,\\
\label{eq:explicit_nonlinear_u_relax_app}
\frac{1}{c^2}\mathsf M_u u_{\tau\tau}+\Gamma_u\mathsf M_u u_\tau
&=H_{\lambda u}^\top\rho_x-H_{xu}^\top\rho_\lambda-H_{uu}^\top\rho_u,
\end{align}
up to endpoint forces from $D\cB(z)^*\mathsf W_b\cB(z)$. Stacking the state and costate as $y=(x,\lambda)$ with kinetic metric $\mathsf G_y$ gives the compact form
\begin{align}
\label{eq:nonlinear_y_relaxation_app}
\frac{1}{c^2}\mathsf G_y y_{\tau\tau}+\Gamma_y\mathsf G_y y_\tau
&=\partial_t\rho_y+D_yF_y(y,u,t)^\top\rho_y-D_y(\nabla_u H)^\top\rho_u,\\
\label{eq:nonlinear_u_relaxation_app}
\frac{1}{c^2}\mathsf M_u u_{\tau\tau}+\Gamma_u\mathsf M_u u_\tau
&=D_uF_y(y,u,t)^\top\rho_y-D_u(\nabla_u H)^\top\rho_u.
\end{align}
For the standard Hamiltonian $H=L(x,u,t)+\lambda^\top f(x,u,t)$,
\begin{equation}
H_\lambda=f,\qquad
H_x=L_x+f_x^\top\lambda,\qquad
H_u=L_u+f_u^\top\lambda,
\end{equation}
so $H_{\lambda x}=f_x$, $H_{\lambda\lambda}=0$, $H_{\lambda u}=f_u$, and $H_{u\lambda}=f_u^\top$.

\begin{proof}[Derivation of \eqref{eq:DRstar_x_app}--\eqref{eq:DRstar_u_app}]
Pair \eqref{eq:DR_explicit_app} with $(\rho_x,\rho_\lambda,\rho_u)$ and integrate the terms $\rho_x^\top\xi_t$ and $\rho_\lambda^\top\eta_t$ by parts. The boundary contribution is $[\rho_x^\top\xi+\rho_\lambda^\top\eta]_0^T$, which is handled by the boundary residual operator $\cB$. The remaining interior coefficients of $(\xi,\eta,\nu)$ are exactly \eqref{eq:DRstar_x_app}--\eqref{eq:DRstar_u_app}.
\end{proof}

\subsection{Velocity--residual waves of the second-order form}
\label{app:wave_variables}

For the second-order relaxation, let $\Theta_y^\top\Theta_y=\mathsf G_y$ and define the velocity--residual waves
\begin{equation}
\label{eq:wave_definition_app}
\omega_\pm=\tfrac{1}{\sqrt2}\left(c^{-1}\Theta_y\partial_\tau y\mp\Theta_y r_y\right),
\end{equation}
with inverse
\begin{align}
\label{eq:wave_inverse_velocity_app}
\partial_\tau y &= \tfrac{c}{\sqrt2}\Theta_y^{-1}(\omega_++\omega_-),\\
\label{eq:wave_inverse_residual_app}
r_y &= \tfrac{1}{\sqrt2}\Theta_y^{-1}(\omega_- -\omega_+),
\end{align}
and local energy
\begin{equation}
\label{eq:wave_energy_density_app}
\tfrac12\left(\|\omega_+\|^2+\|\omega_-\|^2\right)
=\tfrac{1}{2c^2}(\partial_\tau y)^\top\mathsf G_y(\partial_\tau y)
+\tfrac12 r_y^\top\mathsf G_y r_y.
\end{equation}
For constant $\Theta_y$, \eqref{eq:nonlinear_y_relaxation_app} gives
\begin{align}
\label{eq:exact_nonlinear_wave_plus_app}
(\partial_\tau+c\partial_t)\omega_+
&=\tfrac{1}{\sqrt2}\Big[
c\Theta_y^{-\top}\big(D_yF_y^\top\rho_y-D_y(\nabla_u H)^\top\rho_u-\Gamma_y\mathsf G_y\partial_\tau y\big)
\nonumber\\
&\qquad\quad +\Theta_y\big(D_yF_y\,\partial_\tau y+D_uF_y\,\partial_\tau u\big)
\Big],\\
\label{eq:exact_nonlinear_wave_minus_app}
(\partial_\tau-c\partial_t)\omega_-
&=\tfrac{1}{\sqrt2}\Big[
c\Theta_y^{-\top}\big(D_yF_y^\top\rho_y-D_y(\nabla_u H)^\top\rho_u-\Gamma_y\mathsf G_y\partial_\tau y\big)
\nonumber\\
&\qquad\quad -\Theta_y\big(D_yF_y\,\partial_\tau y+D_uF_y\,\partial_\tau u\big)
\Big].
\end{align}

\begin{theorem}[Second-order wave-energy balance]
\label{thm:second_order_wave_energy_balance}
With constant positive-definite metrics and symmetric positive-semidefinite damping, the energy
\begin{equation}
\label{eq:second_order_wave_energy_app}
\cE_{\omega}(\tau)
=\tfrac12\int_0^T\!(\|\omega_+\|^2+\|\omega_-\|^2)\,\dd t
+\tfrac{1}{2c^2}\int_0^T\!(\partial_\tau u)^\top\mathsf M_u(\partial_\tau u)\,\dd t
+\tfrac12\int_0^T\! r_u^\top\mathsf G_u r_u\,\dd t
\end{equation}
satisfies
\begin{equation}
\begin{aligned}
\frac{\dd\cE_{\omega}}{\dd\tau}
&=\tfrac{c}{2}\left(\|\omega_+(0,\tau)\|^2-\|\omega_-(0,\tau)\|^2\right)
+\tfrac{c}{2}\left(\|\omega_-(T,\tau)\|^2-\|\omega_+(T,\tau)\|^2\right)
\\
&\quad -\int_0^T\!\left[(\partial_\tau y)^\top\Gamma_y\mathsf G_y(\partial_\tau y)
+(\partial_\tau u)^\top\Gamma_u\mathsf M_u(\partial_\tau u)\right]\dd t .
\end{aligned}
\end{equation}
\end{theorem}

\begin{proof}
By \eqref{eq:wave_energy_density_app}, the first term in $\cE_\omega$ is the sum of the state--costate kinetic energy and residual energy. Differentiating, using
$r_{y,\tau}=y_{\tau t}-D_yF_y\,y_\tau-D_uF_y\,u_\tau$ and
$r_{u,\tau}=D_y(\nabla_u H)\,y_\tau+D_u(\nabla_u H)\,u_\tau$, and integrating by parts in $t$,
\begin{align*}
\frac{\dd}{\dd\tau}\cE_{\omega}
&=\int_0^T y_\tau^\top\!\left(\tfrac{1}{c^2}\mathsf G_y y_{\tau\tau}-\partial_t\rho_y-D_yF_y^\top\rho_y+D_y(\nabla_u H)^\top\rho_u\right)\dd t\\
&\quad+\int_0^T u_\tau^\top\!\left(\tfrac{1}{c^2}\mathsf M_u u_{\tau\tau}-D_uF_y^\top\rho_y+D_u(\nabla_u H)^\top\rho_u\right)\dd t
+\left[\rho_y^\top y_\tau\right]_0^T.
\end{align*}
Substituting \eqref{eq:nonlinear_y_relaxation_app}--\eqref{eq:nonlinear_u_relaxation_app} leaves the damping terms and the boundary flux. By \eqref{eq:wave_inverse_velocity_app}--\eqref{eq:wave_inverse_residual_app},
$\rho_y^\top y_\tau=r_y^\top\mathsf G_y y_\tau=\tfrac{c}{2}(\|\omega_-\|^2-\|\omega_+\|^2)$.
\end{proof}

\subsection{Mechanical energy balance and telegrapher equation}
\label{app:mechanical_telegrapher}

\begin{theorem}[Nonlinear mechanical energy balance]
\label{thm:nonlinear_mechanical_energy_balance}
Let $z$ solve \eqref{eq:full_nonlinear_worldsheet_relaxation_app} and define
\begin{equation}
\label{eq:nonlinear_mechanical_energy_app}
\cE(z,z_\tau)=\frac{1}{2c^2}\int_0^T z_\tau^\top\mathsf M_z z_\tau\,\dd t+\cV[z].
\end{equation}
Then, up to the boundary power terms contained in $D\cB(z)^*$,
\begin{equation}
\frac{\dd\cE}{\dd\tau}=-\int_0^T z_\tau^\top\Gamma\mathsf M_z z_\tau\,\dd t.
\end{equation}
With fixed or passive boundary ports, the omitted boundary power is zero or nonpositive, so the nonlinear mechanical energy is non-increasing.
\end{theorem}

\begin{proof}
Differentiate \eqref{eq:nonlinear_mechanical_energy_app} in $\tau$ and use $\frac{\dd}{\dd\tau}\cV[z]=\langle\delta\cV/\delta z,z_\tau\rangle$ together with \eqref{eq:full_nonlinear_worldsheet_relaxation_app}. Conservative terms cancel, leaving only Rayleigh damping and boundary power.
\end{proof}

Linearizing the full nonlinear equation around a reference field exposes the telegrapher characteristic structure of the wave dynamics. Write the principal linearized residual as $e_y=\partial_t\delta y-\mathsf A\delta y-\mathsf B\nu$.

\begin{proposition}[Telegrapher equation from linearized residual dynamics]
\label{prop:telegrapher_from_residual_deviations}
For the quadratic potential $\cV_y=\frac12\int e_y^\top\mathsf G e_y\,\dd t$, the interior functional derivative is
\begin{equation}
\label{eq:residual_functional_derivative_app}
\frac{\delta\cV_y}{\delta y}=-\partial_t(\mathsf G e_y)-\mathsf A^\top\mathsf G e_y,
\end{equation}
up to boundary terms. The Lagrange--d'Alembert equation with kinetic energy $\frac{1}{2c^2}\int y_\tau^\top\mathsf G y_\tau\,\dd t$ and damping force $-\Gamma\mathsf G y_\tau$ is
\begin{equation}
\label{eq:telegraph_general_app}
\frac{1}{c^2}\mathsf G\delta y_{\tau\tau}+\Gamma\mathsf G\delta y_\tau
=\partial_t(\mathsf G e_y)+\mathsf A^\top\mathsf G e_y,
\end{equation}
which for constant coefficients with $\mathsf A=\mathsf B=0$ reduces to the telegrapher equation
\begin{equation}
\label{eq:telegraph_simple_app}
\delta y_{\tau\tau}+c^2\Gamma_y\delta y_\tau=c^2\delta y_{tt}.
\end{equation}
\end{proposition}

\begin{proof}
Vary $y$ by $\delta y$. Then $\delta e_y=\partial_t\delta y-\mathsf A\delta y$, hence
\begin{align*}
\delta\cV_y
&=\int_0^T e_y^\top\mathsf G(\partial_t\delta y-\mathsf A\delta y)\,\dd t\\
&=\left[e_y^\top\mathsf G\delta y\right]_0^T
-\int_0^T\!\left(\partial_t(\mathsf G e_y)+\mathsf A^\top\mathsf G e_y\right)^\top\delta y\,\dd t.
\end{align*}
The conservative variational equation is $(1/c^2)\mathsf G y_{\tau\tau}=-\delta\cV_y/\delta y$. Adding the d'Alembert damping force gives \eqref{eq:telegraph_general_app}.
\end{proof}

\begin{theorem}[Linear mechanical energy balance]
\label{thm:mechanical_energy_balance}
Let $\delta y$ solve \eqref{eq:telegraph_general_app}. Define
\begin{equation}
\label{eq:mechanical_energy_app}
\cE_y(\tau)=\frac{1}{2c^2}\int_0^T y_\tau^\top\mathsf G y_\tau\,\dd t+\cV_y(\tau).
\end{equation}
Then
\begin{equation}
\label{eq:lda_energy_balance_app}
\frac{\dd\cE_y}{\dd\tau}
=-\int_0^T y_\tau^\top\Gamma\mathsf G y_\tau\,\dd t
+\left[e_y^\top\mathsf G y_\tau\right]_0^T.
\end{equation}
Fixed, periodic, or passive boundary conditions make the boundary contribution vanish or become nonpositive, so the total mechanical energy is non-increasing.
\end{theorem}

\begin{proof}
Differentiate \eqref{eq:mechanical_energy_app}. The kinetic derivative contributes $\int c^{-2}y_\tau^\top\mathsf G y_{\tau\tau}\,\dd t$. The derivative of $\cV_y$ contributes $\langle\delta\cV_y/\delta y,y_\tau\rangle$ plus the boundary term in \eqref{eq:residual_functional_derivative_app}. Substituting \eqref{eq:telegraph_general_app} cancels the conservative terms and leaves \eqref{eq:lda_energy_balance_app}.
\end{proof}

%% file: appendix/B_lqr_example.tex
\section{Example: Linear--Quadratic Regulator}
\label{app:lqr}
\begin{figure*}
  \centering
  \includegraphics[width=\textwidth]{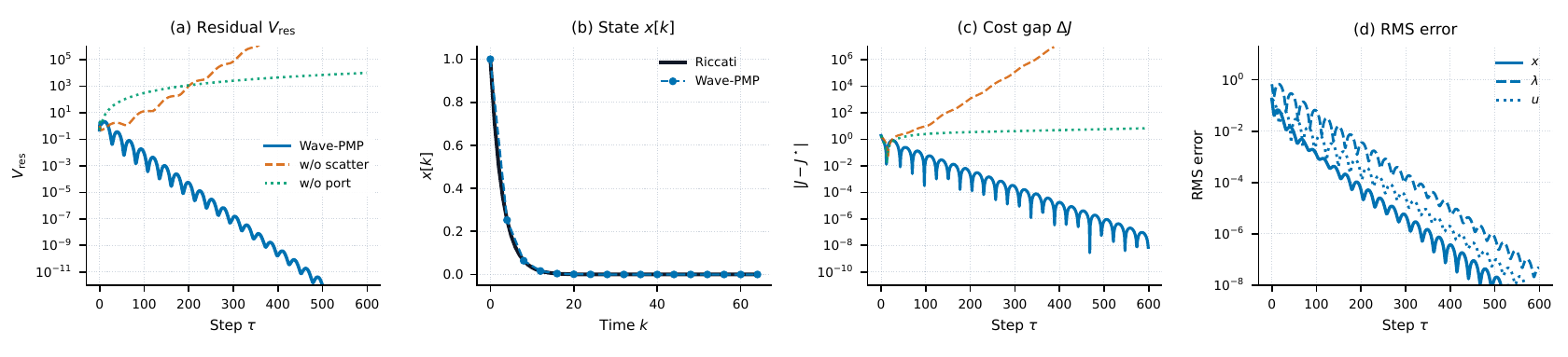}
  \caption{\textbf{Wave-PMP convergence for scalar finite-horizon LQR control with unstable dynamics ($a=1.04$, 64-step horizon).} 
  (a) PMP residual energy $V_{\mathrm{res}}$ decays exponentially to $10^{-10}$ with wave scattering and damped control port (blue), 
  while ablations removing either component (dashed/dotted) fail to converge. 
  (b) Recovered state trajectory matches the analytical Riccati solution. 
  (c) Optimality gap reaches machine precision ($\sim 10^{-9}$). 
  (d) RMS errors for state $x$, costate $\lambda$, and control $u$ all converge to $10^{-6}$.}
  \label{fig:scalar_lqr}
\end{figure*}

Consider the finite-horizon LQR problem
\begin{equation}
\label{eq:lqr_cost_app}
\begin{aligned}
\min_{x,u}\; J[x,u]
&=\frac12 x(T)^\top S_Tx(T)
+\frac12\int_0^T\left(x^\top Qx+u^\top \mathsf R u\right)\,\dd t,\\
\text{s.t.}\qquad \partial_t x&=Ax+Bu,
\qquad x(0)=x_0 .
\end{aligned}
\end{equation}
where $Q=Q^\top\succeq0$, $S_T=S_T^\top\succeq0$, and $\mathsf R=\mathsf R^\top\succ0$.  The Hamiltonian is
\begin{equation}
\label{eq:lqr_hamiltonian_app}
H(x,\lambda,u)=\frac12x^\top Qx+\frac12u^\top\mathsf R u+\lambda^\top(Ax+Bu).
\end{equation}
The augmented action is
\begin{equation}
\label{eq:lqr_augmented_action_app}
\mathcal{S}[x,\lambda,u]
=\frac12x(T)^\top S_Tx(T)
+\int_0^T\left[\frac12x^\top Qx+\frac12u^\top\mathsf R u
+\lambda^\top(Ax+Bu-\partial_t x)\right]\,\dd t .
\end{equation}
Varying $x,\lambda,u$ gives
\begin{align}
\label{eq:lqr_pmp_app}
\partial_t x&=Ax+Bu, & x(0)&=x_0,\\
\partial_t\lambda&=-Qx-A^\top\lambda, & \lambda(T)&=S_Tx(T),\\
0&=\mathsf R u+B^\top\lambda .
\end{align}
Indeed, the variation of the action is
\begin{align}
\delta\cA
&=\left(S_Tx(T)-\lambda(T)\right)^\top\delta x(T)
+\int_0^T\left(Qx+A^\top\lambda+\partial_t\lambda\right)^\top\delta x\,\dd t\nonumber\\
&\quad +\int_0^T(Ax+Bu-\partial_t x)^\top\delta\lambda\,\dd t
+\int_0^T(\mathsf R u+B^\top\lambda)^\top\delta u\,\dd t,
\end{align}
where $\delta x(0)=0$.  This is the standard LQR maximum-principle system.

The residual formulation keeps the same equations but views them as a zero set.  Define the full residuals
\begin{align}
\label{eq:lqr_residuals_full_app}
r_x&=\partial_t x-Ax-Bu,\\
r_\lambda&=\partial_t\lambda+Qx+A^\top\lambda,\\
r_u&=\mathsf R u+B^\top\lambda,\\
b_0&=x(0,\tau)-x_0,
\qquad
b_T=\lambda(T,\tau)-S_Tx(T,\tau).
\end{align}
With identity residual weights, the full residual potential is
\begin{equation}
\label{eq:lqr_residual_potential_app}
\cV_{\rm LQR}[x,\lambda,u]
=\frac12\int_0^T\left(\|r_x\|^2+\|r_\lambda\|^2+\|r_u\|^2\right)\,\dd t
+\frac12\|b_0\|^2+\frac12\|b_T\|^2 .
\end{equation}
Its zero set is exactly \eqref{eq:lqr_pmp_app}.  The interior first variation of \eqref{eq:lqr_residual_potential_app} is
\begin{align}
\label{eq:lqr_functional_gradient_app}
\frac{\delta\cV_{\rm LQR}}{\delta x}
&=-\partial_t r_x-A^\top r_x+Qr_\lambda,\\
\frac{\delta\cV_{\rm LQR}}{\delta\lambda}
&=-\partial_t r_\lambda+Ar_\lambda+Br_u,\\
\frac{\delta\cV_{\rm LQR}}{\delta u}
&=-B^\top r_x+\mathsf R r_u,
\end{align}
with endpoint terms supplied by $b_0,b_T$ and by integration by parts.  Therefore the worldsheet relaxation is the explicit linear damped residual system
\begin{align}
\label{eq:lqr_worldsheet_x_app}
\frac1{c^2}M_x\partial_{\tau\tau}x+\Gamma_xM_x\partial_\tau x
&=\partial_t r_x+A^\top r_x-Qr_\lambda,\\
\label{eq:lqr_worldsheet_lambda_app}
\frac1{c^2}M_\lambda\partial_{\tau\tau}\lambda+\Gamma_\lambda M_\lambda\partial_\tau\lambda
&=\partial_t r_\lambda-Ar_\lambda-Br_u,\\
\label{eq:lqr_worldsheet_u_app}
\frac1{c^2}M_u\partial_{\tau\tau}u+\Gamma_uM_u\partial_\tau u
&=B^\top r_x-\mathsf R r_u .
\end{align}
The right-hand sides are the adjoint-linearized residual force $-D\cR^*\cR$, not the pointwise residual vector itself.

The stationarity residual can also be imposed exactly.  From $r_u=0$,
\begin{equation}
\label{eq:lqr_feedback_u_app}
u=-\mathsf R^{-1}B^\top\lambda .
\end{equation}
Substitution gives the reduced Hamiltonian system
\begin{equation}
\label{eq:lqr_hamiltonian_system_app}
\partial_t
\begin{bmatrix}x\\ \lambda\end{bmatrix}
=
\underbrace{\begin{bmatrix}
A&-B\mathsf R^{-1}B^\top\\
-Q&-A^\top
\end{bmatrix}}_{F_{\rm LQR}}
\begin{bmatrix}x\\ \lambda\end{bmatrix},
\qquad
x(0)=x_0,
\qquad
\lambda(T)=S_Tx(T).
\end{equation}
Writing $y=(x,\lambda)$ and $r=y_t-F_{\rm LQR}y$, a reduced residual potential with metric $\mathsf G\succ0$ is
\begin{equation}
\label{eq:lqr_reduced_residual_potential_app}
\widetilde\cV_{\rm LQR}[y]
=\frac12\int_0^T r^\top\mathsf G r\,\dd t
+\frac{\beta_0}{2}\|x(0)-x_0\|_{B_0}^2
+\frac{\beta_T}{2}\|\lambda(T)-S_Tx(T)\|_{B_T}^2 .
\end{equation}
Its interior derivative is
\begin{equation}
\label{eq:lqr_reduced_functional_gradient_app}
\frac{\delta\widetilde\cV_{\rm LQR}}{\delta y}
=-\partial_t(\mathsf G r)-F_{\rm LQR}^\top\mathsf G r.
\end{equation}
Thus
\begin{equation}
\label{eq:lqr_worldsheet_equation_app}
\frac1{c^2}\mathsf G y_{\tau\tau}+\Gamma\mathsf G y_\tau
=\partial_t(\mathsf G r)+F_{\rm LQR}^\top\mathsf G r .
\end{equation}
For constant $F_{\rm LQR}$ and $\mathsf G=I$, this expands to
\begin{equation}
\label{eq:lqr_worldsheet_expanded_app}
\frac1{c^2}y_{\tau\tau}+\Gamma y_\tau
=y_{tt}+(F_{\rm LQR}^\top-F_{\rm LQR})y_t-F_{\rm LQR}^\top F_{\rm LQR}y .
\end{equation}
The principal part is the wave operator in $(t,\tau)$; the LQR dynamics enter as lower-order gyroscopic and stiffness terms.  The velocity--residual waves are
\begin{equation}
\label{eq:lqr_residual_waves_app}
\omega_\pm
=\frac1{\sqrt2}\left(c^{-1}\mathsf G^{1/2}y_\tau
\mp\mathsf G^{1/2}(y_t-F_{\rm LQR}y)\right),
\end{equation}
and their energy density is
\begin{equation}
\frac12\left(\|\omega_+\|^2+\|\omega_-\|^2\right)
=\frac1{2c^2}y_\tau^\top\mathsf G y_\tau+\frac12r^\top\mathsf G r .
\end{equation}

At a zero-residual equilibrium, the Riccati equation is recovered.  If $\lambda(t)=P(t)x(t)$, then \eqref{eq:lqr_feedback_u_app} gives $u=-\mathsf R^{-1}B^\top Px$ and
\begin{equation}
\dot x=(A-B\mathsf R^{-1}B^\top P)x.
\end{equation}
Combining $\dot\lambda=\dot P x+P\dot x$ with $\dot\lambda=-Qx-A^\top Px$ gives, for all $x$,
\begin{equation}
\label{eq:lqr_riccati_app}
-\dot P=A^\top P+PA-PB\mathsf R^{-1}B^\top P+Q,
\qquad
P(T)=S_T.
\end{equation}

Finally, the LQR boundary conditions can be written in algebraic scattering form.  With
\begin{equation}
w_\pm=\frac1{\sqrt2}\left(\Theta x\pm\Theta^{-\top}\lambda\right),
\end{equation}
the input condition gives
\begin{equation}
\label{eq:lqr_input_scattering_app}
w_+(0)=\sqrt2\Theta x_0-w_-(0),
\end{equation}
and the terminal condition $\lambda(T)=S_Tx(T)$ gives
\begin{equation}
\label{eq:lqr_terminal_scattering_app}
w_-(T)=(I+K_T)^{-1}(I-K_T)w_+(T),
\qquad
K_T=\Theta^{-\top}S_T\Theta^{-1}.
\end{equation}

%% file: appendix/C_discrete_scattering_details.tex
\section{Discrete residuals and layer scattering}
\label{app:discrete_scattering_details}

\subsection{Discrete residual potential}
Let $M_k=\Theta_k^\top\Theta_k\succ0$ and $M_{\theta,k}=\Theta_{\theta,k}^\top\Theta_{\theta,k}\succ0$. A discrete residual potential is
\begin{align}
\label{eq:discrete_residual_potential}
\cV_d
&=\frac12\|\Theta_0 r_{x,0}\|^2
+\frac12\sum_{k=0}^{N-1}\|\Theta_{k+1}r_{x,k+1}\|^2
+\frac12\|\Theta_N^{-\top}r_{\lambda,N}\|^2\nonumber\\
&\quad
+\frac12\sum_{k=0}^{N-1}\|\Theta_k^{-\top}r_{\lambda,k}\|^2
+\frac12\sum_{k=0}^{N-1}\|\Theta_{\theta,k}^{-\top}r_{\theta,k}\|^2 .
\end{align}
With positive definite metrics, $\cV_d=0$ if and only if all residuals in \eqref{eq:node_residuals} vanish. Hence its zeros are exactly the discrete Pontryagin equations and boundary conditions.

\subsection{Exact nonlinear layer scattering}
For incident waves $u=w_{+,k}$ and $v=w_{-,k+1}$, the exact local nonlinear scattering step solves \eqref{eq:nonlinear_layer_scattering_main} for outgoing waves $y=w_{+,k+1}$ and $z=w_{-,k}$. This is simply the nonlinear state equation and nonlinear adjoint equation written in port coordinates. It can be solved by any local nonlinear method, for example a damped Newton step initialized at the previous outgoing waves. No global forward or backward sweep is involved.

\subsection{Affine small-signal scattering map}
For the explicit formula used as a cheap local approximation, freeze
\begin{equation}
\label{eq:affine_layer_main}
f_k(\zeta;\theta_k)\approx J_k\zeta+n_k,
\qquad
J_k=J_{x,k}(x_k;\theta_k),
\qquad
n_k=f_k(x_k;\theta_k)-J_kx_k,
\end{equation}
and set
\begin{equation}
\label{eq:scattering_sources_main}
s_k=\nabla_{x_k}R_k(x_k,\theta_k),\quad
A_k=\Theta_{k+1}J_k\Theta_k^{-1},\quad
\bar n_k=\sqrt2\Theta_{k+1}n_k,\quad
\bar s_k=\sqrt2\Theta_k^{-\top}s_k.
\end{equation}
The local affine state--adjoint equations
\begin{equation}
\label{eq:local_affine_pmp_main}
x_{k+1}=J_kx_k+n_k,
\qquad
\lambda_k=J_k^\top\lambda_{k+1}+s_k
\end{equation}
are equivalent to
\begin{align}
\label{eq:scat_eq_1}
w_{+,k+1}+w_{-,k+1}
&=A_k(w_{+,k}+w_{-,k})+\bar n_k,\\
\label{eq:scat_eq_2}
w_{+,k}-w_{-,k}
&=A_k^\top(w_{+,k+1}-w_{-,k+1})+\bar s_k.
\end{align}
Given the incident pair $(w_{+,k},w_{-,k+1})$, the outgoing pair is
\begin{align}
\label{eq:scatter_plus}
w_{+,k+1}
&=T_+(A_k)w_{+,k}+R_-(A_k)w_{-,k+1}
-(I+A_kA_k^\top)^{-1}(A_k\bar s_k-\bar n_k),\\
\label{eq:scatter_minus}
w_{-,k}
&=R_+(A_k)w_{+,k}+T_-(A_k)w_{-,k+1}
-(I+A_k^\top A_k)^{-1}(A_k^\top\bar n_k+\bar s_k),
\end{align}
where
\begin{align}
\label{eq:scattering_coefficients}
T_+(A)&=2A(I+A^\top A)^{-1},
& R_+(A)&=(I-A^\top A)(I+A^\top A)^{-1},\\
T_-(A)&=2A^\top(I+AA^\top)^{-1},
& R_-(A)&=-(I-AA^\top)(I+AA^\top)^{-1}.
\end{align}
The matrices $I+A^\top A$ and $I+AA^\top$ are positive definite even when $A$ is rectangular or rank deficient.

\begin{proof}[Derivation]
Let $u=w_{+,k}$ and $v=w_{-,k+1}$ be incident waves, and let $y=w_{+,k+1}$ and $z=w_{-,k}$ be unknown outgoing waves. Equations \eqref{eq:scat_eq_1}--\eqref{eq:scat_eq_2} give
\begin{equation*}
y=A_k(u+z)+\bar n_k-v,
\qquad
z=u-A_k^\top(y-v)-\bar s_k.
\end{equation*}
Substituting the second equation into the first yields
\begin{equation*}
(I+A_kA_k^\top)y=2A_ku+(A_kA_k^\top-I)v-A_k\bar s_k+\bar n_k,
\end{equation*}
which is \eqref{eq:scatter_plus}. Substituting \eqref{eq:scatter_plus} back into the expression for $z$, or eliminating $y$ first, gives \eqref{eq:scatter_minus}. The identities
$A(I+A^\top A)^{-1}=(I+AA^\top)^{-1}A$ and
$(I+A^\top A)^{-1}A^\top=A^\top(I+AA^\top)^{-1}$ give the displayed coefficient forms.

\end{proof}

\subsection{Pullback metrics and scattering coefficients}
\label{app:pullback_scattering}
This subsection isolates the homogeneous small-signal two-port, since it explains the impedance interpretation of the neural-network metrics.  Let $A=\Theta_{k+1}J_k\Theta_k^{-1}$ and suppress affine offsets.  With incident waves $u=w_{+,k}$ and $v=w_{-,k+1}$ and outgoing waves $z=w_{-,k}$ and $y=w_{+,k+1}$,
\begin{equation}
\label{eq:homogeneous_scattering_matrix_app}
\begin{bmatrix}
z\\ y
\end{bmatrix}
=
\begin{bmatrix}
R_L(A) & T_R(A)\\
T_L(A) & R_R(A)
\end{bmatrix}
\begin{bmatrix}
u\\ v
\end{bmatrix},
\end{equation}
where
\begin{align}
\label{eq:homogeneous_scattering_blocks_app}
R_L(A)&=(I-A^\top A)(I+A^\top A)^{-1},
& T_R(A)&=2A^\top(I+AA^\top)^{-1},\\
T_L(A)&=2A(I+A^\top A)^{-1},
& R_R(A)&=-(I-AA^\top)(I+AA^\top)^{-1}.
\end{align}
If $A=U\Sigma V^\top$ is a singular value decomposition, then a left-incident singular mode $V_i$ has reflection coefficient
\begin{equation}
\label{eq:singular_reflection_coefficient_app}
\rho_i=\frac{1-\sigma_i^2}{1+\sigma_i^2},
\qquad
\tau_i=\frac{2\sigma_i}{1+\sigma_i^2}
\end{equation}
for reflection and transmission.  Hence the mode is reflectionless exactly when $\sigma_i=1$.  A zero singular value gives $|\rho_i|=1$, representing unavoidable reflection of a direction annihilated by the layer linearization.

The singular values are controlled by the activation metrics.  Since
\begin{equation}
\label{eq:pullback_metric_app}
A^\top A=\Theta_k^{-\top}J_k^\top G_{k+1}J_k\Theta_k^{-1},
\qquad
G_j=\Theta_j^\top\Theta_j,
\end{equation}
the pullback metric condition
\begin{equation}
\label{eq:pullback_metric_condition_app}
G_k=J_k^\top G_{k+1}J_k
\end{equation}
gives $A^\top A=I$ whenever the pullback is positive definite.  In that case $R_L(A)=0$ and all rank-preserved left-incident small-signal modes pass through the layer without reflection.  If $J_k$ is rank deficient, the pullback is only semidefinite and the nullspace modes cannot be matched by a positive-definite metric without regularization.  Damped choices such as
\begin{equation}
\label{eq:damped_pullback_metric_app}
G_k=J_k^\top G_{k+1}J_k+\epsilon I,
\qquad \epsilon>0,
\end{equation}
are therefore approximate impedance matches.  The same calculation may be applied recursively or approximately, giving a layerwise pullback-metric interpretation of adaptive activation impedance.

\begin{proposition}[Fixed points and discrete PMP]
\label{prop:fixed_points_discrete_pmp}
Fix $\theta$. A wave configuration satisfying the exact nonlinear layer scattering equations \eqref{eq:nonlinear_layer_scattering_main} for all layers and the boundary equations in Algorithm~\ref{alg:worldsheet_unlocked} is equivalent to a solution of the state and adjoint equations in \eqref{eq:node_residuals}. If, in addition, the averaged parameter residuals $r_{\theta,k}$ vanish, the full discrete PMP system holds. The affine map \eqref{eq:scatter_plus}--\eqref{eq:scatter_minus} has the same fixed points when its frozen quantities are recomputed from the fixed-point state.
\end{proposition}

\begin{proof}
By the inverse wave transform, \eqref{eq:nonlinear_layer_scattering_main} is exactly the pair of equations $x_{k+1}=f_k(x_k;\theta_k)$ and $\lambda_k=J_{x,k}^\top\lambda_{k+1}+\nabla_{x_k}R_k(x_k,\theta_k)$. The boundary scattering equations impose $x_0=x_{\rm in}$ and $\lambda_N=\nabla\ell(x_N,y)$. Adding $r_{\theta,k}=0$ gives parameter stationarity. For the affine map, a fixed point with $J_k,n_k,s_k$ recomputed from the fixed-point state reduces to the same equations because $n_k=f_k(x_k;\theta_k)-J_kx_k$ and $s_k=\nabla_{x_k}R_k(x_k,\theta_k)$.
\end{proof}

%% file: appendix/D_deep_linear_network.tex
\section{Example: Deep Linear Network}
\label{app:deep_linear}
This appendix writes the discrete construction for a deep linear network.  Use a mini-batch formulation
\begin{equation}
\label{eq:deep_linear_forward_app}
X_k\in\RR^{n_k\times m},
\qquad
W_k\in\RR^{n_{k+1}\times n_k},
\qquad
X_{k+1}=W_kX_k,
\qquad
k=0,\ldots,N-1,
\end{equation}
where the columns of $X_k$ are examples and the single-example case is $m=1$.  For simplicity use the summed square loss with optional weight decay,
\begin{equation}
\label{eq:deep_linear_loss_app}
J(W)=\frac12\|X_N-Y\|_F^2+\sum_{k=0}^{N-1}\frac{\rho_k}{2}\|W_k\|_F^2,
\qquad
\rho_k\ge0.
\end{equation}
If the averaged loss is preferred, every terminal adjoint and weight gradient below is divided by $m$ in the loss term.

Using the Frobenius pairing $\langle A,B\rangle=\tr(A^\top B)$, define the discrete augmented action
\begin{equation}
\label{eq:deep_linear_augmented_action_app}
\mathcal{S}_d
=\frac12\|X_N-Y\|_F^2
+\sum_{k=0}^{N-1}\frac{\rho_k}{2}\|W_k\|_F^2
+\sum_{k=0}^{N-1}\left\langle \Lambda_{k+1},W_kX_k-X_{k+1}\right\rangle .
\end{equation}
The variations are explicit.  The terminal variation gives
\begin{equation}
\delta_{X_N}\cA_d=\left\langle X_N-Y-\Lambda_N,\delta X_N\right\rangle,
\end{equation}
so $\Lambda_N=X_N-Y$.  For $1\le k\le N-1$,
\begin{equation}
\delta_{X_k}\cA_d
=\left\langle W_k^\top\Lambda_{k+1}-\Lambda_k,\delta X_k\right\rangle,
\end{equation}
so $\Lambda_k=W_k^\top\Lambda_{k+1}$.  Finally,
\begin{equation}
\delta_{W_k}\cA_d
=\left\langle \Lambda_{k+1}X_k^\top+\rho_kW_k,\delta W_k\right\rangle .
\end{equation}
Thus the discrete PMP equations are
\begin{align}
\label{eq:deep_linear_pmp_app}
X_{k+1}&=W_kX_k,\\
\Lambda_N&=X_N-Y,\\
\Lambda_k&=W_k^\top\Lambda_{k+1},
\qquad k=N-1,\ldots,0,\\
0&=\Lambda_{k+1}X_k^\top+\rho_kW_k,
\qquad k=0,\ldots,N-1.
\end{align}
Consequently, once the state and adjoint equations hold, the stationarity residual
\begin{equation}
\label{eq:deep_linear_gradient_app}
E_{W,k}=\Lambda_{k+1}X_k^\top+\rho_kW_k
\end{equation}
is exactly $\nabla_{W_k}J$.

The full residual formulation introduces
\begin{align}
\label{eq:deep_linear_residuals_app}
E_{X,0}&=X_0-X_{\rm in},
& E_{X,k+1}&=X_{k+1}-W_kX_k,\\
E_{\Lambda,N}&=\Lambda_N-(X_N-Y),
& E_{\Lambda,k}&=\Lambda_k-W_k^\top\Lambda_{k+1},\\
E_{W,k}&=\Lambda_{k+1}X_k^\top+\rho_kW_k,
&& k=0,\ldots,N-1.
\end{align}
With identity residual weights,
\begin{equation}
\label{eq:deep_linear_residual_potential_app}
\cV_{\rm lin}
=\frac12\|E_{X,0}\|_F^2
+\frac12\sum_{k=0}^{N-1}\|E_{X,k+1}\|_F^2
+\frac12\|E_{\Lambda,N}\|_F^2
+\frac12\sum_{k=0}^{N-1}\|E_{\Lambda,k}\|_F^2
+\frac12\sum_{k=0}^{N-1}\|E_{W,k}\|_F^2 .
\end{equation}
The potential \eqref{eq:deep_linear_residual_potential_app} vanishes if and only if all equations in \eqref{eq:deep_linear_pmp_app} hold.  It is nonlinear in the full stacked variable $(X,\Lambda,W)$, because the residuals contain products $W_kX_k$, $W_k^\top\Lambda_{k+1}$, and $\Lambda_{k+1}X_k^\top$.

The true residual-potential force is not the residual vector itself.  Direct differentiation gives, for $1\le j\le N-1$,
\begin{align}
\label{eq:deep_linear_grad_X_app}
\nabla_{X_j}\cV_{\rm lin}
&=E_{X,j}-W_j^\top E_{X,j+1}+E_{W,j}^\top\Lambda_{j+1},\\
\label{eq:deep_linear_grad_Lambda_app}
\nabla_{\Lambda_j}\cV_{\rm lin}
&=E_{\Lambda,j}-W_{j-1}E_{\Lambda,j-1}+E_{W,j-1}X_{j-1},\\
\label{eq:deep_linear_grad_W_app}
\nabla_{W_j}\cV_{\rm lin}
&=-E_{X,j+1}X_j^\top-\Lambda_{j+1}E_{\Lambda,j}^\top+\rho_jE_{W,j}.
\end{align}
The endpoint formulas are obtained by dropping out-of-range terms and adding the boundary or terminal residuals.  For example,
\begin{align}
\label{eq:deep_linear_endpoint_grads_app}
\nabla_{X_N}\cV_{\rm lin}
&=E_{X,N}-E_{\Lambda,N},\\
\nabla_{\Lambda_N}\cV_{\rm lin}
&=E_{\Lambda,N}-W_{N-1}E_{\Lambda,N-1}+E_{W,N-1}X_{N-1}.
\end{align}
A full damped variational relaxation of the deep linear network is therefore
\begin{equation}
\label{eq:deep_linear_full_relaxation_app}
\frac1{c^2}\mathsf M q_{\tau\tau}+\Gamma\mathsf M q_\tau=-\nabla_q\cV_{\rm lin}(q),
\qquad
q=(X_0,\ldots,X_N,\Lambda_0,\ldots,\Lambda_N,W_0,\ldots,W_{N-1}).
\end{equation}
This is the finite-dimensional counterpart of the nonlinear worldsheet relaxation.

The parameter-port update used in the main algorithm is the matched termination for the stationarity residual rather than the full residual-potential force \eqref{eq:deep_linear_grad_W_app}:
\begin{equation}
\label{eq:deep_linear_port_update_app}
\partial_\tau W_k=-M_{W,k}^{-1}E_{W,k}.
\end{equation}
On the state--adjoint consistency manifold, $E_{W,k}=\nabla_{W_k}J$, so \eqref{eq:deep_linear_port_update_app} is preconditioned gradient flow on the deep linear objective.

For fixed $W_k$, the local state--adjoint scattering map is closed form.  Define matrix-valued waves
\begin{equation}
\label{eq:deep_linear_matrix_waves_app}
w_{\pm,k}=\frac1{\sqrt2}\left(\Theta_kX_k\pm\Theta_k^{-\top}\Lambda_k\right),
\end{equation}
and let $U=w_{+,k}$ and $V=w_{-,k+1}$ be incident waves, with outgoing waves $Y_+=w_{+,k+1}$ and $Z=w_{-,k}$.  Set
\begin{equation}
\label{eq:deep_linear_A_app}
A_k=\Theta_{k+1}W_k\Theta_k^{-1}.
\end{equation}
The state and adjoint equations become
\begin{align}
\label{eq:deep_linear_scattering_eqs_app}
Y_+ + V &= A_k(U+Z),\\
U-Z &= A_k^\top(Y_+-V).
\end{align}
Solving this two-port system gives
\begin{align}
\label{eq:deep_linear_scatter_plus_app}
Y_+&=2A_k(I+A_k^\top A_k)^{-1}U
-(I-A_kA_k^\top)(I+A_kA_k^\top)^{-1}V,\\
\label{eq:deep_linear_scatter_minus_app}
Z&=(I-A_k^\top A_k)(I+A_k^\top A_k)^{-1}U
+2A_k^\top(I+A_kA_k^\top)^{-1}V.
\end{align}
These are the affine scattering formulas \eqref{eq:scatter_plus}--\eqref{eq:scatter_minus} with zero affine offsets, applied columnwise to the mini-batch.  They require inverses only of $I+A_k^\top A_k$ and $I+A_kA_k^\top$, which are positive definite even when $W_k$ is rectangular.

The boundary terminations are also explicit.  The input boundary $X_0=X_{\rm in}$ gives
\begin{equation}
\label{eq:deep_linear_input_scattering_app}
w_{+,0}=\sqrt2\Theta_0X_{\rm in}-w_{-,0}.
\end{equation}
For the square terminal loss, $\Lambda_N=X_N-Y$.  Writing $a=w_{+,N}$ and $b=w_{-,N}$, the inverse wave transform gives
\begin{equation}
\frac1{\sqrt2}\Theta_N^\top(a-b)=\frac1{\sqrt2}\Theta_N^{-1}(a+b)-Y.
\end{equation}
Thus, with $K_N=\Theta_N^{-\top}\Theta_N^{-1}$,
\begin{equation}
\label{eq:deep_linear_terminal_scattering_app}
b=(I+K_N)^{-1}\left((I-K_N)a+\sqrt2\Theta_N^{-\top}Y\right).
\end{equation}
For $\Theta_N=I$, this reduces to $w_{-,N}=Y/\sqrt2$.

%% file: appendix/E_parameter_ports.tex
\section{Parameter ports, adaptive impedance, and minimal reflection}
\label{app:param_port}

This appendix gives the port calculation used in Section~\ref{sec:minimal_reflection}.  At layer $k$, let
\begin{equation}
\label{eq:param_port_effort_flow_app}
e_k=r_{\theta,k},
\qquad
v_k=\partial_\tau\theta_k,
\qquad
Z_k=Z_k^\top\succ0.
\end{equation}
The power supplied to the parameter degrees of freedom is $e_k^\top v_k$; the dissipated optimization power is therefore
\begin{equation}
\label{eq:param_diss_power_app}
P_{{\rm diss},k}=-e_k^\top v_k.
\end{equation}
When the state and adjoint equations are satisfied, $e_k=r_{\theta,k}=\nabla_{\theta_k}J$, so $P_{{\rm diss},k}=-\dot J_k$ is the local contribution to descent.  Away from state--adjoint consistency, $e_k$ is a Hamiltonian stationarity residual and the port law is a local residual-damping rule.

\begin{theorem}[Matched parameter termination]
\label{thm:matched_parameter_termination_app}
For $Z_k\succ0$, define incident and reflected parameter waves
\begin{align}
\label{eq:param_wave_incident_app}
a_k&=\frac{1}{\sqrt2}\left(Z_k^{-1/2}e_k-Z_k^{1/2}v_k\right),\\
\label{eq:param_wave_reflected_app}
b_k&=\frac{1}{\sqrt2}\left(Z_k^{-1/2}e_k+Z_k^{1/2}v_k\right).
\end{align}
Then
\begin{equation}
\label{eq:param_power_identity_app}
-e_k^\top v_k
=\frac12\left(\|a_k\|^2-\|b_k\|^2\right).
\end{equation}
Consequently, any termination satisfying $\|b_k\|\le\|a_k\|$ is passive.  The reflectionless termination $b_k=0$ is exactly
\begin{equation}
\label{eq:min_reflection_update}
\partial_\tau\theta_k=-Z_k^{-1}r_{\theta,k}.
\end{equation}
If the state and adjoint residuals vanish, then this is preconditioned gradient flow for the reduced training objective.
\end{theorem}

\begin{proof}
Expanding the squared norms gives
\begin{align*}
\|a_k\|^2
&=\frac12\left(e_k^\top Z_k^{-1}e_k-2e_k^\top v_k+v_k^\top Z_kv_k\right),\\
\|b_k\|^2
&=\frac12\left(e_k^\top Z_k^{-1}e_k+2e_k^\top v_k+v_k^\top Z_kv_k\right).
\end{align*}
Subtracting yields \eqref{eq:param_power_identity_app}.  If $b_k=0$, then $Z_k^{-1/2}e_k+Z_k^{1/2}v_k=0$, and hence $v_k=-Z_k^{-1}e_k=-Z_k^{-1}r_{\theta,k}$.
\end{proof}

\paragraph{Adaptive resistive impedance.}
The impedance may depend on layer, parameter block, minibatch, or optimization time:
\begin{equation}
\label{eq:adaptive_resistive_impedance_app}
\partial_\tau\theta_k=-Z_k(\tau)^{-1}r_{\theta,k},
\qquad
Z_k(\tau)\succ0.
\end{equation}
The port remains passive at each instant because
\begin{equation}
\label{eq:adaptive_resistive_dissipation_app}
-r_{\theta,k}^\top\partial_\tau\theta_k
=r_{\theta,k}^\top Z_k(\tau)^{-1}r_{\theta,k}\ge0.
\end{equation}
Thus scalar gradient flow corresponds to $Z_k=\eta^{-1}I$, while diagonal or block-diagonal choices give adaptive preconditioned stationarity flow. Choosing the state-dependent Hessian metric $Z_k(\theta)=\nabla^2\psi(\theta)\succ 0$ of a strongly convex mirror map $\psi$ realizes continuous-time mirror descent $\frac{\dd}{\dd\tau}\nabla\psi(\theta)=-r_\theta$ \citep{boffi2021implicit,boffi2022nonparametric} as a nonlinear capacitor storing Bregman energy $d_\psi(\theta^\star\|\theta)$ (developed in detail in Proposition~\ref{prop:bregman_port_app}, Appendix~\ref{app:adaptive_nn_control}). A local positive-definite quadratic model
\begin{equation}
\label{eq:local_quadratic_impedance_app}
J(\theta+\delta)\approx J(\theta)+g^\top\delta+\frac12\delta^\top H\delta,
\qquad H\succ0,
\end{equation}
has curvature impedance $H$; choosing $Z=H$ gives the continuous-time Newton direction $\theta_\tau=-H^{-1}g$.  This exact matching statement is local to the positive-definite quadratic approximation.  In nonconvex settings one should use a positive-definite approximation such as $H+\mu I$, a Gauss--Newton or Fisher-type metric, or another stable structured preconditioner.

\paragraph{Dynamic impedance and momentum.}
Momentum is better interpreted as a dynamic, energy-storing termination rather than a purely resistive zero-reflection load.  A passive linear dynamic impedance has the constitutive law
\begin{equation}
\label{eq:dynamic_impedance_app}
e_k+R_kv_k+L_k\partial_\tau v_k=0,
\qquad
R_k\succeq0,
\quad
L_k\succ0.
\end{equation}
Since $v_k=\partial_\tau\theta_k$, this gives
\begin{equation}
\label{eq:heavy_ball_port_app}
L_k\partial_{\tau\tau}\theta_k+R_k\partial_\tau\theta_k+r_{\theta,k}=0,
\end{equation}
which is the continuous heavy-ball form on the consistency manifold.  The stored kinetic energy of the port is
\begin{equation}
\label{eq:dynamic_impedance_energy_app}
K_{\theta,k}=\frac12v_k^\top L_kv_k.
\end{equation}
Multiplying \eqref{eq:dynamic_impedance_app} by $v_k^\top$ gives
\begin{equation}
\label{eq:dynamic_impedance_passivity_app}
e_k^\top v_k+\frac{d}{d\tau}K_{\theta,k}+v_k^\top R_kv_k=0.
\end{equation}
Thus, when $e_k$ is the reduced gradient, the combined objective plus kinetic energy decreases at rate $v_k^\top R_kv_k$.  In this sense momentum is an impedance choice: it stores and dissipates energy. Equivalently, \eqref{eq:heavy_ball_port_app} is the Euler--Lagrange / Lagrange--d'Alembert dynamics of kinetic storage $K_{\theta,k}$ and Rayleigh dissipation $\tfrac12 v_k^\top R_k v_k$ (or an exponentially weighted action when $R_k=\gamma L_k$). Matching $Z_k(s)=R_k+sL_k$ to a local downstream parameter-port load $H_{\theta,k}$ sets the damping ratio $\zeta=R_k/(2\sqrt{L_k H_{\theta,k}})$, with critical damping ($\zeta=1$) corresponding to inertial impedance matching, paralleling the distortionless telegrapher match $m_\star=\sqrt{a/b}$ in the bulk relaxation (Appendix~\ref{app:mechanical_telegrapher}).

\subsection{Discrete-time minimal-reflection wave-variable RLC circuits}
\label{app:wdf_rlc_boyd}

We establish the discrete-time minimal-reflection, Tellegen--Lyapunov, and
algorithm-unification statements of Section~\ref{sec:minimal_reflection}
connecting \citet{boyd2024circuits} to wave-variable scattering
\citep{SlotineLi1987,niemeyer1991stable} (the bilinear port discretization
of wave digital filters \citep{fettweis1986wave}) with general SPD
port-resistance tensors
$R_j=G_j^{-1}=\Theta_j^{-1}\Theta_j^{-\top}\succ0$.

\begin{proposition}[Discrete minimal reflection and Newton equivalence]
\label{prop:riemannian_min_refl_app}
Let $f\in C^2(\RR^p)$ and $G_\theta=\Theta_\theta^\top\Theta_\theta=R_0^{-1}\succ0$.
For the matched negative-resistor gradient port \eqref{eq:wdf_matched_grad_port}
with local scattering Jacobian
$S_f(\theta)=I-2\Theta_\theta^{-\top}\nabla^2 f(\theta)\Theta_\theta^{-1}$:
\begin{enumerate}
\item $\|S_f(\theta)\|_2\le 1 \iff 0\preceq\nabla^2 f(\theta)\preceq G_\theta$.
If this holds globally, $f$ satisfies $G_\theta$-metric Baillon--Haddad
co-coercivity $\langle\theta-\theta',\,\nabla f(\theta)-\nabla f(\theta')\rangle\ge\|\nabla f(\theta)-\nabla f(\theta')\|_{G_\theta^{-1}}^2$,
and $a_g\mapsto b_g$ is globally nonexpansive.
\item $S_f(\theta)\equiv 0 \iff G_\theta=2\nabla^2 f(\theta)$, in which case
terminating the port with a matched capacitor ($a_g[k+1]=b_g[k]$) gives the
exact Newton step $\theta_{k+1}=\theta_k-[\nabla^2 f(\theta_k)]^{-1}\nabla f(\theta_k)$.
\item Over $\mu G\preceq\nabla^2 f(\theta)\preceq L_f G$ ($\kappa_G=L_f/\mu$),
minimizing $\sup_\theta\|S_f(\theta)\|_2$ over scalings $G_\theta=\alpha G$
uniquely yields $G_\theta^\star=(L_f+\mu)G$ ($R_0^\star=(L_f+\mu)^{-1}G^{-1}$)
and minimal reflection norm $(\kappa_G-1)/(\kappa_G+1)$.
\end{enumerate}
\end{proposition}

\begin{proof}
Set $\widehat H(\theta)=\Theta_\theta^{-\top}\nabla^2 f(\theta)\Theta_\theta^{-1}$.
(1)~Since $S_f=I-2\widehat H$ is symmetric, $\|S_f\|_2\le 1\iff 0\preceq\widehat H\preceq I\iff 0\preceq\nabla^2 f\preceq G_\theta$.
In frame coordinates $z=\Theta_\theta\theta$, $\hat f(z)=f(\Theta_\theta^{-1}z)$
is convex and $1$-smooth, so Euclidean Baillon--Haddad gives
$\langle\Theta_\theta(\theta-\theta'),\,\Theta_\theta^{-\top}(\nabla f(\theta)-\nabla f(\theta'))\rangle\ge\|\Theta_\theta^{-\top}(\nabla f(\theta)-\nabla f(\theta'))\|_2^2$,
which implies $\tfrac14(\|a_g(\theta)-a_g(\theta')\|_2^2-\|b_g(\theta)-b_g(\theta')\|_2^2)
=\langle\theta-\theta',\,\nabla f(\theta)-\nabla f(\theta')\rangle-\|\nabla f(\theta)-\nabla f(\theta')\|_{G_\theta^{-1}}^2\ge0$.
(2)~Immediate from $I-2\widehat H=0\iff G_\theta=2\nabla^2 f$ and $\theta_{k+1}=\Theta_\theta^{-1}b_g$.
(3)~For $G_\theta=\alpha G$, the spectrum of $S_f$ lies in $[1-2L_f/\alpha,\,1-2\mu/\alpha]$;
equating $1-2\mu/\alpha^\star=-(1-2L_f/\alpha^\star)$ gives $\alpha^\star=L_f+\mu$ and reflection norm $(L_f-\mu)/(L_f+\mu)$.
\end{proof}

\begin{theorem}[Discrete Tellegen--Lyapunov balance and three-way unification]
\label{thm:wdf_rlc_unification_app}
For arbitrary (possibly non-commuting) SPD tensors $G_C,G_L,G_d\succ0$
($R_j=G_j^{-1}$), let $G_P=G_d+G_L$, $R_P=G_P^{-1}$, $G_0^{-1}=R_0=R_C+R_P$,
$\Gamma_L=R_P G_L$, and $B=I-2\Gamma_L$. Along the 4-line update
\eqref{eq:wdf_rlc_four_line}, define the stored wave energy relative to an
equilibrium $\theta^\star$ ($\nabla f(\theta^\star)=0$) by
$\cE_k\equiv\tfrac14\|c_k-\theta^\star\|_{G_C}^2+\tfrac14\|\ell_k\|_{G_L}^2$.
Then:
\begin{enumerate}
\item For any vector sequence $g_k\in\RR^p$, with parallel tank voltage $v_P^k=\Gamma_L\ell_k-R_P g_k$,
\begin{equation}
\label{eq:tellegen_lyapunov_app}
\cE_{k+1}-\cE_k
= -\|v_P^k\|_{G_d}^2 - \bigl\langle\theta_k-\theta^\star-G_0^{-1}g_k,\,g_k\bigr\rangle .
\end{equation}
Whenever $0\preceq\nabla^2 f\preceq G_0$, Proposition~\ref{prop:riemannian_min_refl_app}(1)
guarantees strict monotone Lyapunov descent $\cE_{k+1}-\cE_k\le-\|v_P^k\|_{G_d}^2\le0$.
\item For proportional tensors $R_j=r_j G^{-1}$ ($\Gamma_L=\gamma_L I$,
$B=\beta I$, $r\equiv\gamma_L r_P/r_C$), eliminating $(c_k,\ell_k)$ yields
for all $k\ge 1$:
\begin{equation}
\label{eq:wdf_rlc_2step_app}
(\theta_{k+1}-\theta_k)-\beta(\theta_k-\theta_{k-1})
= -2r_C(1-r)\,G^{-1}g_k + 2r_C(\beta-r)\,G^{-1}g_{k-1},
\end{equation}
recovering identically: Riemannian Gradient Descent ($r=0$, $\eta=2r_C$),
Polyak Heavy Ball ($r=\beta$, $\alpha=2r_C(1-\beta)$), and Nesterov
Accelerated Gradient ($r=\beta^2$, $\eta=2r_C(1-\beta)$).
\end{enumerate}
\end{theorem}

\begin{proof}
(1)~From $c_{k+1}-\theta^\star=(c_k-\theta^\star)-2R_C g_k$ and $G_C R_C=I$,
$\Delta\cE_{C,k}=-\langle c_k-\theta^\star,g_k\rangle+\|g_k\|_{R_C}^2$.
For $\ell_{k+1}=B\ell_k+2R_P g_k$, expand
$\Delta\cE_{L,k}=\tfrac14\ell_k^\top(B^\top G_L B-G_L)\ell_k+\ell_k^\top B^\top G_L R_P g_k+g_k^\top R_P G_L R_P g_k$.
Using only $G_P=G_d+G_L$, $\Gamma_L=R_P G_L$, $I-\Gamma_L^\top=G_d R_P$, and
$B=I-2\Gamma_L$ (no commutativity required):
$R_P G_L R_P=R_P-R_P G_d R_P$,
$B^\top G_L R_P=-\Gamma_L^\top+2\Gamma_L^\top G_d R_P$, and
$\tfrac14(B^\top G_L B-G_L)=-G_L R_P(I-G_L R_P)G_L=-\Gamma_L^\top G_d\Gamma_L$.
Collecting terms completes the square:
$\Delta\cE_{L,k}=-\|\Gamma_L\ell_k-R_P g_k\|_{G_d}^2-\langle\Gamma_L\ell_k,g_k\rangle+\|g_k\|_{R_P}^2$.
Adding $\Delta\cE_{C,k}$ with $\theta_k=c_k+\Gamma_L\ell_k$ and
$R_C+R_P=G_0^{-1}$ gives \eqref{eq:tellegen_lyapunov_app}.
(2)~For $R_j=r_j G^{-1}$, $(c_{k+1}-c_k)-\beta(c_k-c_{k-1})=-2r_C G^{-1}g_k+2\beta r_C G^{-1}g_{k-1}$
and $(\ell_{k+1}-\beta\ell_k)-(\ell_k-\beta\ell_{k-1})=2r_P G^{-1}(g_k-g_{k-1})$.
Adding the first equation to $\gamma_L$ times the second (with $\gamma_L r_P=r\,r_C$)
gives \eqref{eq:wdf_rlc_2step_app}, whose special cases $r\in\{0,\beta,\beta^2\}$
match GD ($\ell_k\equiv0$ when $r_P=0$), Heavy Ball, and Nesterov verbatim.
\end{proof}

\begin{figure}[t]
\centering
\includegraphics[width=\textwidth]{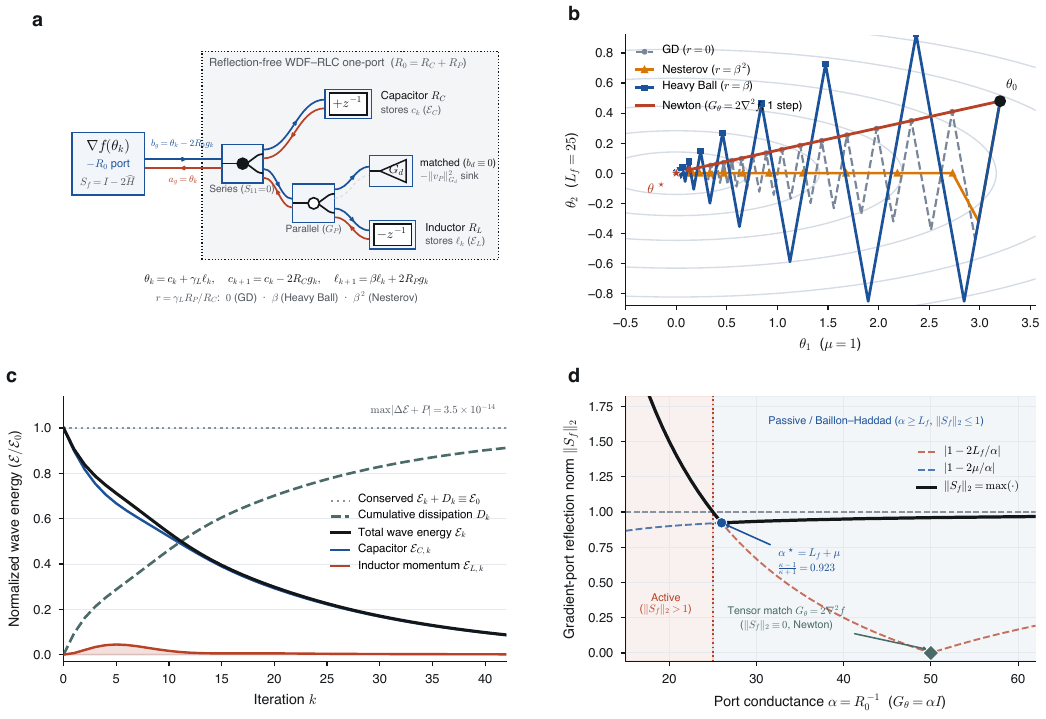}
\caption{\textbf{Minimal-reflection WDF--RLC electrical optimizer one-port (Proposition~\ref{prop:riemannian_min_refl_app} and Theorem~\ref{thm:wdf_rlc_unification_app}).}
(a)~Double-line wave-digital one-port circuit in $\mathbf{Scat}$ terminating the matched negative-resistor gradient port $\nabla f(\theta_k)$ ($S_f = I - 2\widehat H$): a reflection-free Kirchhoff series spider ($R_0 = R_C + R_P$, $S_{11}=0$) couples a capacitive delay U-turn ($+z^{-1}$, storing $c_k$) to a parallel tank ($G_P = G_d + G_L$) combining a matched dissipative resistor ($b_d \equiv 0$, absorbing $-\|v_P\|^2_{G_d}$) and an inductive sign-inverting delay U-turn ($-z^{-1}$, storing momentum $\ell_k$).
(b)~Optimization trajectories on an ill-conditioned quadratic ($\mu = 1, L_f = 25$) comparing GD ($r=0$), Polyak Heavy Ball ($r=\beta$), Nesterov ($r=\beta^2$), and exact tensor-matched Newton ($G_\theta = 2\nabla^2 f$, single zero-reflection step).
(c)~Exact discrete Tellegen--Lyapunov wave-energy balance along strictly passive WDF--RLC Heavy Ball ($R_0 = L_f^{-1}$): capacitor energy $\mathcal{E}_{C,k}$, inductive momentum energy $\mathcal{E}_{L,k}$, monotone total energy $\mathcal{E}_k$, cumulative dissipation $D_k$, and conserved invariant $\mathcal{E}_k + D_k \equiv \mathcal{E}_0$ ($\max|\Delta\mathcal{E}+P| = 3.5\times 10^{-14}$).
(d)~Gradient-port reflection norm $\|S_f\|_2$ versus scalar port conductance $\alpha = R_0^{-1}$, showing the Baillon--Haddad passivity boundary $\alpha \ge L_f$, the minimal-reflection cusp $\alpha^\star = L_f + \mu$ ($\|S_f^\star\|_2 = (\kappa-1)/(\kappa+1)$), and zero-reflection Newton matching ($G_\theta = 2\nabla^2 f$, $\|S_f\|_2 \equiv 0$).}
\label{fig:wdf_rlc_optimizer}
\end{figure}

%% file: appendix/F_continuous_delta_attention.tex
\section{Continuous-Time Causal Linear and Delta-Rule Attention}
\label{sec:continuous_delta_attention}

Standard causal softmax attention couples every position $t$ directly to all earlier positions $u \le t$, creating an all-to-past computation graph with a growing key--value cache and no finite propagation speed along the sequence coordinate.
Replacing softmax with a kernel feature map $\phi(\cdot)$ and a fixed-size associative memory state $S(t) \in \RR^{m \times d_v}$ replaces that nonlocal graph by a strictly local evolution along sequence time $t \in [0,T]$, placing sequence modeling directly on the worldsheet $\Sigma = [0,T]_t \times [0,\infty)_\tau$ of Section~\ref{sec:variational_worldsheet}.

\paragraph{Continuous-time memory dynamics and key-selective contraction.}
Let $\phi_q(t), \phi_k(t) \in \RR^m$ and $v(t) \in \RR^{d_v}$ denote the query, key, and value streams at sequence time $t$, with passive decay rate $\alpha(t) \ge 0$ and write rate $\beta(t) \ge 0$.
In continuous time, decaying linear attention evolves as $\partial_t S = -\alpha(t) S + \beta(t)\,\phi_k(t) v(t)^\top$ with readout $o(t) = S(t)^\top \phi_q(t)$.
Evaluating the predicted value $\hat v(t) = S(t)^\top \phi_k(t)$ and writing the residual $v(t) - \hat v(t)$ yields the continuous-time Delta rule, which we equip with an optional skew-symmetric generator $\Omega(t) = -\Omega(t)^\top \in \RR^{m \times m}$ (e.g., rotary phase precession):
\begin{equation}
\label{eq:cont_delta_forward}
\partial_t S(t) = -\bigl(D(t) + \Omega(t)\bigr) S(t) + \beta(t)\,\phi_k(t) v(t)^\top,
\qquad S(0) = 0,
\qquad o(t) = S(t)^\top \phi_q(t),
\end{equation}
where the symmetric positive-semidefinite erasure operator is
\begin{equation}
\label{eq:delta_dissipation}
D(t) := \alpha(t) I + \beta(t)\,\phi_k(t)\phi_k(t)^\top = D(t)^\top \succeq \alpha(t) I \succeq 0.
\end{equation}
Two properties distinguish~\eqref{eq:cont_delta_forward} from discrete linear recurrences.
First, when $\Omega = 0$, the dynamics are exact online gradient flow $\partial_t S = -\nabla_S E_t(S)$ on the instantaneous associative-memory energy
\[
E_t(S) = \frac{\alpha(t)}{2}\|S\|_F^2 + \frac{\beta(t)}{2}\bigl\|S^\top \phi_k(t) - v(t)\bigr\|_2^2.
\]
Second, because $\langle S, \Omega S \rangle = 0$ and $D(t) \succeq 0$, the homogeneous flow is unconditionally contractive for any unnormalized key stream $\phi_k(t)$ and any $\alpha(t), \beta(t) \ge 0$: memory decays slowly (at rate $\alpha$) in directions orthogonal to recent keys, while perturbations and in-flight drift along rewritten key directions are selectively flushed at rate $\alpha + \beta\|\phi_k\|^2$.
Including $\Omega(t)$ preserves passivity while adding norm-preserving rotational modes that a purely symmetric real generator cannot realize in continuous time.

\paragraph{Self-adjoint Pontryagin costate and local efforts.}
For a running readout loss $\int_0^T \ell(o(t), t)\,\dd t$ with output gradient $g_o(t) = \nabla_o \ell(o(t), t)$, the Hamiltonian $H(S,\Lambda,t) = \ell(S^\top\phi_q,t) + \langle \Lambda, -(D+\Omega)S + \beta\phi_k v^\top \rangle$ gives the Pontryagin residuals
\begin{equation}
\label{eq:delta_pmp_residuals}
r_S = (D_t + D) S - \beta\phi_k v^\top,
\qquad
r_\Lambda = (D_t - D)\Lambda + \phi_q g_o^\top,
\end{equation}
where $D_t := \partial_t + \Omega(t)$ is the $SO(m)$ covariant derivative (skew-adjoint in $L^2$: $D_t^* = -D_t$), with boundary conditions $S(0)=0$ and $\Lambda(T)=0$.
Setting $r_\Lambda = 0$ yields
\begin{equation}
\label{eq:cont_delta_adjoint}
-\partial_t \Lambda(t) = -\bigl(D(t) - \Omega(t)\bigr)\Lambda(t) + \phi_q(t) g_o(t)^\top,
\qquad \Lambda(T) = 0.
\end{equation}
The backward costate $\Lambda(t) \in \RR^{m \times d_v}$ therefore obeys the \emph{same} contractive Delta-rule ODE in reverse sequence time (with reversed rotation $-\Omega$): queries write credit $\phi_q g_o^\top$ into $\Lambda$, while the same key projector $\beta\phi_k\phi_k^\top$ erases credit past subsequent key overwrites.
All sensitivities entering the local token projections and layer parameters are strictly local in $t$: writing $\hat v = S^\top \phi_k$ and $\hat g_v = \Lambda^\top \phi_k$,
\begin{equation}
\label{eq:delta_local_efforts}
\begin{gathered}
\nabla_{\phi_q} H = S g_o,
\quad
\nabla_v H = \beta \hat g_v,
\quad
\nabla_{\phi_k} H = \beta\bigl(\Lambda(v - \hat v) - S \hat g_v\bigr),
\\
\nabla_\alpha H = -\langle \Lambda, S \rangle,
\quad
\nabla_\beta H = \hat g_v^\top (v - \hat v).
\end{gathered}
\end{equation}

\paragraph{Continuum residual relaxation as a massive covariant telegrapher equation.}
Because~\eqref{eq:delta_pmp_residuals} is first-order local in $t$, the relaxation~\eqref{eq:para_floer} and its waves apply directly along sequence time with finite speed $c$.
Moreover, for fixed trajectories $(\phi_q, \phi_k, v, \alpha, \beta, \Omega)$ and exogenous $g_o$, the adjoints $(D_t \pm D)^* = -D_t \pm D$ yield an exact Klein--Gordon factorization of the normal operators entering the damped second-order relaxation (Appendix~\ref{app:mechanical_telegrapher}):
\begin{equation}
\label{eq:delta_telegrapher_factor}
(D_t \pm D)^*(D_t \pm D)
= -(\partial_t + \Omega)^2 + D(t)^2 \mp \bigl(\dot D(t) + [\Omega(t), D(t)]\bigr).
\end{equation}
Thus passive decay and Delta-rule erasure act as a spatially varying squared mass $D(t)^2 \succeq \alpha(t)^2 I$, while $\Omega(t)$ enters purely through the covariant derivative $D_t = \partial_t + \Omega(t)$.

\paragraph{Cayley discretization, key-overwrite reflection, and directional conductances.}
Discretizing~\eqref{eq:cont_delta_forward} over sequence step $n = 0, \dots, N-1$ of size $\Delta t$ by the midpoint (Crank--Nicolson) rule gives the affine step $S_{n+1} = F_n S_n + \widetilde U_n$, where
\begin{equation}
\label{eq:delta_cayley}
F_n = \Bigl(I + \tfrac{\Delta t}{2}(D_n + \Omega_n)\Bigr)^{-1}\Bigl(I - \tfrac{\Delta t}{2}(D_n + \Omega_n)\Bigr),
\qquad
\widetilde U_n = \Delta t\Bigl(I + \tfrac{\Delta t}{2}(D_n + \Omega_n)\Bigr)^{-1} \beta_n \phi_{k,n} v_n^\top.
\end{equation}
Because $D_n \succeq 0$ and $\Omega_n = -\Omega_n^\top$, the Cayley transform is unconditionally contractive ($\|F_n\|_2 \le 1$).
In the pure Delta-rule case ($\alpha_n = 0$, $\Omega_n = 0$), the Sherman--Morrison formula reduces~\eqref{eq:delta_cayley} to an $O(m)$ update with the exact algebraic form of discrete DeltaNet,
\begin{equation}
\label{eq:delta_sherman_morrison}
S_{n+1} = S_n + \tilde\beta_n\,\phi_{k,n}\bigl(v_n - S_n^\top \phi_{k,n}\bigr)^\top,
\qquad
\tilde\beta_n := \frac{\Delta t\,\beta_n}{1 + \tfrac{\Delta t}{2}\beta_n\|\phi_{k,n}\|_2^2} \in \Bigl[0,\frac{2}{\|\phi_{k,n}\|_2^2}\Bigr),
\end{equation}
automatically confining $\tilde\beta_n$ to the discrete stability interval.
Along $\phi_{k,n}$, the eigenvalue of $F_n$ crosses zero (exact key erasure) at $\tfrac{\Delta t}{2}\beta_n\|\phi_{k,n}\|_2^2 = 1$ and smoothly enters $(-1,0)$ (Householder-like reflection, enabling sign flips and state tracking) above this threshold.

This structure exposes a sharp distinction between one-metric and directional wave scattering (Section~\ref{sec:training_neural_network}, Appendix~\ref{app:pullback_scattering}, and Appendix~\ref{sec:exp_physics_control}).
Under one-metric scattering with $\Theta_n = I$, the reflection matrix $R_+(F_n) = (I - F_n^\top F_n)(I + F_n^\top F_n)^{-1}$ reflects any singular value $s$ with amplitude $(1 - s^2)/(1 + s^2)$; consequently, every exact key overwrite ($s \to 0$) acts as a unit-magnitude mirror ($R_+ \to +1$) that backscatters wave energy along $\phi_{k,n}$.
Under the directional conductance pair $(Y_n^\to, Z_n^\leftarrow)$, by contrast, prescribed $S_0 = 0$ yields $Z_n^\leftarrow \equiv 0$, while for isotropic readout curvature $q_{o,n} I$ the forward conductance satisfies the backward Lyapunov recursion
\begin{equation}
\label{eq:delta_lyapunov}
Y_n^\to = F_n^\top Y_{n+1}^\to F_n + q_{o,n}\,\phi_{q,n}\phi_{q,n}^\top,
\qquad Y_N^\to = 0,
\end{equation}
with continuum limit $-\dot Y^\to = -(D - \Omega)Y^\to - Y^\to(D + \Omega) + q_o \phi_q \phi_q^\top$.
This pair eliminates reflection even at exact key erasures ($s = 0$), decoupling the forward and backward sweeps along the sequence axis.

\paragraph{Chunk-boundary ports and scope.}
Storing $(S_n, \Lambda_n)$ at every token costs $O(T m d_v)$, which can be prohibitive for long sequences.
Because~\eqref{eq:delta_cayley} is affine in $S_n$, composing $B$ steps across each chunk $[t_c, t_{c+1}]$ collapses interior tokens into a single affine chunk map $S_{c+1} = \mathcal{F}_c S_c + \mathcal{U}_c$ (analogous to chunkwise DeltaNet), so the wave solver need only maintain and exchange ports $(S_c, \Lambda_c)$ at the $C = T/B$ chunk boundaries---or a single state in sliding-window streaming.
Finally, the operator factorization~\eqref{eq:delta_telegrapher_factor} and Lyapunov matching~\eqref{eq:delta_lyapunov} hold for fixed trajectories $(\phi_q, \phi_k, v, \alpha, \beta, \Omega)$; in stacked multi-layer architectures, upstream depth waves modulate these streams parametrically, coupling sequence ports to depth ports and leaving nonlinear convergence analysis open.

%% file: appendix/G_physics_and_control.tex
\section{Recovering Physical Equations and Receding-Horizon Control}
\label{sec:exp_physics_control}

When the depth index $k$ is identified with physical time rather than an abstract layer number, inference about the past, discovery of equations of motion, and optimal control of the future cease to be separate problems. They become different boundary sectors of a single local scattering medium on the $(1+1)$-dimensional worldsheet. Observations act as shunts along the past horizon, unknown couplings enter as parameter ports $r_{\theta,k}$, and control objectives enter as future boundary loads. What distinguishes the physical setting from shuffled-minibatch learning is causal continuity: the state and costate waves need not be erased at every step, because the dynamics itself transports the field from one physical time slice to the next.

\paragraph{Recovering equations of motion through the parameter port.}
To isolate how unknown physical laws are extracted by local scattering, consider a nonlinear driven oscillator with known conservative core $(m,k)=(1.4,1.1)$ and five unknown anharmonic, dissipative, and actuator couplings $(\theta,d,a,b_0,b_2)$,
\begin{equation}
m\dot v = -kq - \theta q^3 - dv - av^3 + (b_0 + b_2 q^2)u,
\label{eq:oscillator_id}
\end{equation}
discretized into $160$ implicit-midpoint cells ($\Delta t=0.025$) and observed only through noisy positions ($\sigma=0.01$) across eight trajectories. Choosing the link metric as the Hessian of the known mechanical energy, $G=\Theta^\top\Theta=\operatorname{diag}(k,m)$, renders the conservative evolution an exact isometry: the known mechanics transmits waves without reflection, so only the unknown dissipative and nonlinear terms scatter the field and supply effort to the parameter port. From a single cold start, evolving the fields and parameters jointly recovers all five coefficients across three dataset seeds to within $0.014$ least-squares standard errors (held-out RMSE within $0.7\%$ of offline least squares; effort cosine $>0.999995$). By contrast, resetting the wave fields to zero every $K$ flights injects a repeated cold-start transient that shifts the stationary couplings by a predictable linear-response bias $-H^{-1}b/K$ (measured scaling $3.96\text{--}4.02$ between $K=65{,}536$ and $K=262{,}144$, cosine $>0.9998$). Persistence of the field is therefore not merely an efficiency device; on deep physical chains it is required for unbiased parameter recovery.

On a live stream with an abrupt mid-stream parameter jump (sliding horizon $H=32$ on a driven Duffing oscillator; frozen-window gradient agreement $6.7\times10^{-8}$), replacing the single symmetric metric with directional forward and backward conductances $(Y_k^\to,Z_k^\leftarrow)$ computed once from nominal $(m,k)$ allows persistent shifted waves to track exact moving-horizon estimation at $K=4$ local ticks per sample (pre/post-jump parameter error $0.031/0.067$ vs.\ $0.034/0.071$, median gradient cosine $+0.938$; $K=2$ yields $0.029/0.056$). Cold restarts fail across the same horizon (error $2.10/2.27$, negative cosine), and $K=1$ violates the causal transport budget after the jump ($2.15$).

\paragraph{Receding-horizon control and unified past--future coupling.}
In receding-horizon control, eliminating the actuation inside each junction yields the control and filter Riccati recursions as the exact reflectionless pair $(Y_k^\to,Z_k^\leftarrow)$: $Y_k^\to$ pulls future cost-to-go backward, while $Z_k^\leftarrow$ pushes past state uncertainty forward. On a drifting linear--quadratic plant (ten seeds, horizons $H\in\{8,12,24,36\}$), persistent fields with fixed nominal conductances remain within $0.3\%$ of exact adaptive MPC at $K=8$ ticks per action and within $9\%$ at $K=4$, independent of horizon length $H$, whereas cold restarts blow up by factors of $60$ to $1.9\times10^4$ below eight ticks.

The same local structure extends to nonlinear Euler--Lagrange mechanics with hard torque limits enforced inside every cell junction, under abrupt unknown payload pickups (Table~\ref{tab:physics_control}). On a planar two-link manipulator ($+1\,\mathrm{kg}$ payload, $H=20$) and a full $6$-DoF Universal Robots UR5e simulated in MuJoCo ($+4\,\mathrm{kg}$ payload on the wrist, $H=16$, collaborative torque caps actively engaged on $25\%$ of steps), persistent directional Wave-MPC with fixed nominal-pose conductances tracks exact Box-NMPC to within $3.4\text{--}4.8\%$ at $K=4$ and $1.0\text{--}1.1\%$ at $K=8$, outperforming both one-metric persistent waves and a $400\,\mathrm{Hz}$ Slotine--Li adaptive controller while enforcing actuator bounds anticipatively rather than by post-hoc clipping. Closing the loop without any external estimator---coupling a past sliding-window Wave-PMP parameter port ($H_{\mathrm{est}}=16,K=4$, noisy encoders) directly to a future Wave-MPC horizon ($H_{\mathrm{ctrl}}=20,K=4$) on the two-link arm---identifies the $+1.0\,\mathrm{kg}$ jump online ($\hat m_p=1.003\,\mathrm{kg}$) and achieves $1.006\times$ the cost of an oracle controller ($36.10$ vs.\ $35.90$; unadapted $44.92$). Moreover, in a four-pendulum cooperative chain with one-way link delays $T_d\in[0,1.0]\,\mathrm{s}$, direct delayed coupling diverges at $T_d=0.30\text{--}0.40\,\mathrm{s}$, whereas wave-variable scattering remains passive and synchronized ($0.35\text{--}0.47\,\mathrm{rad}$ RMSE) at every tested delay. Appendix~\ref{app:wave_world_models} extends this unified past-identification and future-control architecture to systems observed purely through raw camera images, embedding passive contact port-Hamiltonian mechanics inside a learned latent phase space across six visual control benchmarks, while Appendix~\ref{app:neural_ode_hnn_physics} extends the same metric-wave and split-Riccati formulations to chaotic Neural ODEs and Hamiltonian networks, $O(1)$-memory $SO(d+1)$ Cayley ODEs, hybrid impact systems, discrete-gradient Circuit-HNNs, and noisy sparse equation discovery.

\begin{table}[t]
\centering
\small
\caption{\textbf{Persistent directional Wave-MPC on nonlinear manipulators and unified past--future adaptation.} Cost ratios relative to exact Box-NMPC (or oracle payload on the unified row); fixed nominal-pose conductances on both manipulators.}
\label{tab:physics_control}
\resizebox{\textwidth}{!}{%
\begin{tabular}{lccccc}
\toprule
Task & Wave $K=2$ & Wave $K=4$ & Wave $K=8$ & One-metric ($K=2$) & Cold ($K=4$) / Slotine--Li \\
\midrule
2-link arm ($H=20$, $+1\,\mathrm{kg}$) & $1.117\times$ & $1.034\times$ & $1.010\times$ & $1.349\times$ & $3.75\times$ / $1.034\times$ \\
6-DoF UR5e MuJoCo ($H=16$, $+4\,\mathrm{kg}$) & $1.148\times$ & $1.048\times$ & $1.011\times$ & $1.316\times$ & $6.10\times$ / $1.242\times$ \\
Unified past Wave-ID + future MPC ($K=4$) & --- & $1.006\times$ & --- & --- & unadapted: $1.251\times$ \\
\bottomrule
\end{tabular}%
}
\end{table}

\paragraph{Composing physical subsystems, equation recovery, and control.}
Because physical interconnections, past estimation, and future planning share the same wave-port algebra, they compose directly in a single closed loop (Figure~\ref{fig:composed_wave_system}). We couple an actuated nonlinear rotary drive $A$ (known $(m_A,g_A,d_A)$, torque bound $|u|\le 8.5\,\mathrm{N\cdot m}$) to an unactuated nonlinear load $B$ with three unknown couplings $\theta_B=(k_B,\beta_B,d_B)$ across a discrete-passive physical wave link (power-balance error $1.1\times 10^{-14}$), while an abrupt load jump shifts $\theta_B$ mid-episode from $(0.90,0.40,0.25)$ to $(2.60,1.50,0.80)$. Along the past window ($H_{\mathrm{est}}=20,K=4$, noisy sensors, no external estimator; frozen-window gradient agreement $9.2\times 10^{-10}$), persistent directional waves paired with the matched Gauss--Newton parameter-port impedance $Z_\theta^{-1}r_\theta$ recover all three load parameters online (pre/post-jump error $0.029/0.132$, matching exact sequential smoothing at $0.026/0.143$). Along the future horizon ($H_{\mathrm{ctrl}}=24,K=4$; frozen agreement $1.6\times 10^{-9}$ with exact constrained control), persistent waves steer the unactuated load $B$ through the physical link using the live estimate $\hat\theta_B(t)$, achieving total cost $15.79$ (matching exact sequential identification at $15.80$, versus $11.68$ with exact control and known $\theta_B$) and reducing post-jump load error by $5.6\times$ relative to fixed parameters ($0.029$ vs.\ $0.166\,\mathrm{rad}$; resetting fields every step collapses to cost $315.5$ and $0.592\,\mathrm{rad}$). As the one-way physical transmission delay between $A$ and $B$ increases, direct delayed spring--damper coupling loses passivity and diverges for all $T_d\ge 0.16\,\mathrm{s}$, whereas the wave-port link remains strictly passive and bounded across the entire range ($0.15\text{--}0.89\,\mathrm{rad}$ for $T_d\in[0.04,0.40]\,\mathrm{s}$).

\begin{figure}[t]
\centering
\includegraphics[width=\textwidth]{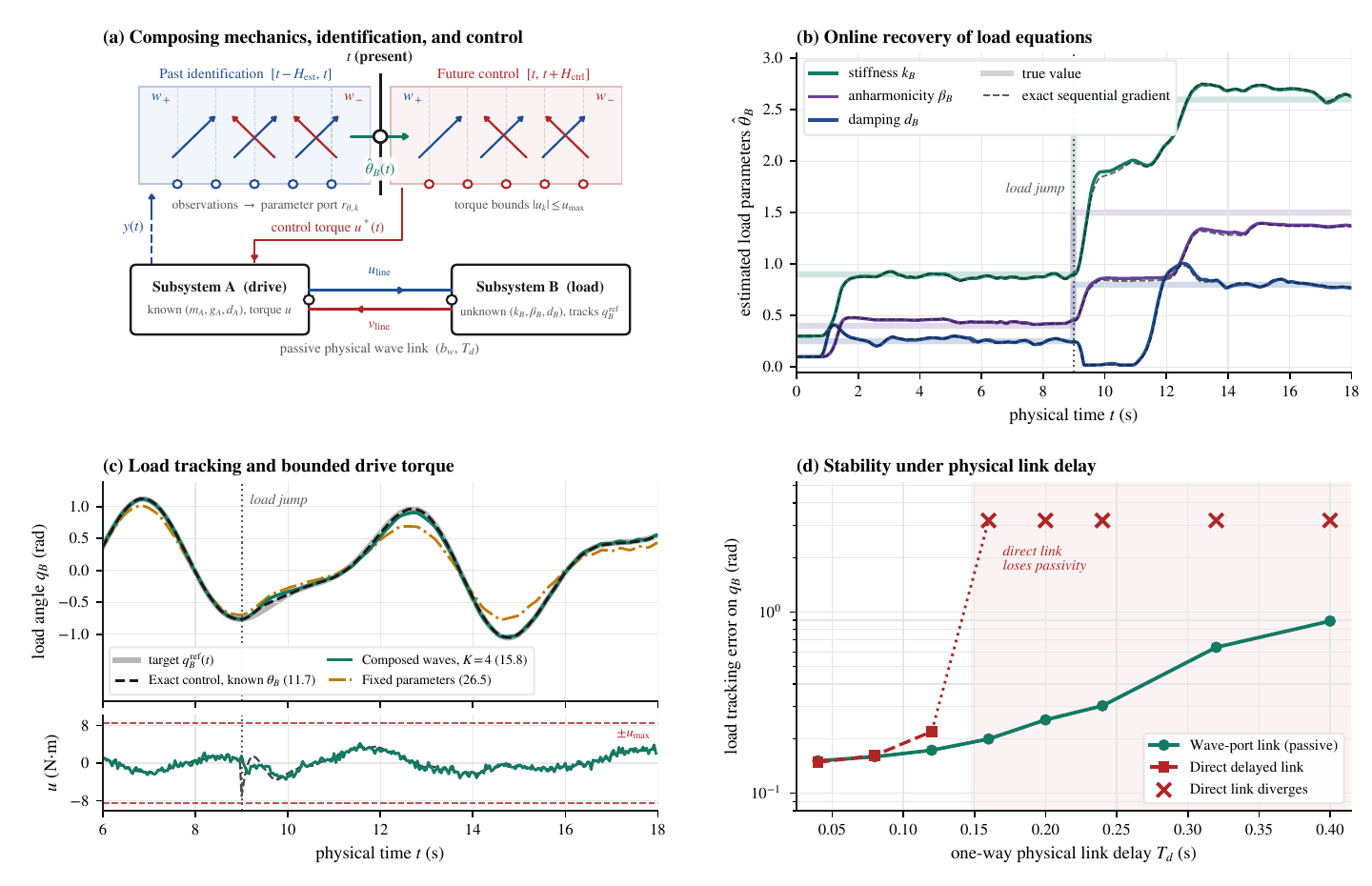}
\caption{\textbf{Composing physical mechanics, online equation recovery, and bounded control through wave ports.} (a)~Space--time architecture joining an actuated drive $A$ and an unactuated unknown load $B$ across a passive physical wave link ($T_d$), coupled at the present slice $t$ to past-window identification ($H_{\mathrm{est}}=20,K=4$) and future-horizon bounded control ($H_{\mathrm{ctrl}}=24,K=4$). (b)~Online recovery of Subsystem~$B$'s stiffness $k_B$, cubic anharmonicity $\beta_B$, and damping $d_B$ across an abrupt load jump, tracking the exact sequential gradient (dashed). (c)~Closed-loop tracking of the unactuated load $q_B(t)$ and bounded drive torque $u(t)$ compared against exact control with known $\theta_B$ and control with fixed parameters. (d)~Tracking error versus one-way physical link delay $T_d$: the passive wave-port link remains bounded while direct delayed coupling diverges for all $T_d\ge 0.16\,\mathrm{s}$.}
\label{fig:composed_wave_system}
\end{figure}

\paragraph{Regime of validity.}
These experiments also delimit where fixed local metrics cease to suffice. On strongly non-convex, torque-limited pendulum swing-up ($|u|\le 0.55\,mgl$, $H=56$), open-loop trajectories amplify perturbations exponentially and fixed nominal conductances fail; locally matched Riccati conductances along the multiple-shooting trajectory restore swing-up and recover from a $+1.7\,\mathrm{rad/s}$ impulse kick to $0.3^\circ$ error at $K=6\text{--}8$, while cold restarts fail ($164.9^\circ$). Convex dual-metric Control Contraction Metric synthesis \citep{manchester2017control}---pinning $W(x_{\mathrm{eq}})=P_{\mathrm{LQR}}^{-1}$ at the equilibrium while enforcing global nonlinear contraction across state space, together with persistent geodesic wave tracking (Appendix~\ref{app:covariant_ccm_geometry})---offers a principled state-dependent generalization of fixed nominal conductances for such non-convex regimes. Similarly, the fixed-batch oscillator recovery requires $\sim 10^4\times$ more cell solves than offline trust-region least squares, our local tick budgets measure parallel depth rather than serial flop count, and short-horizon ($H=12$) pure-latent planning without a value function does not stabilize an open-loop unstable visual cart-pole (Appendix~\ref{app:wave_world_models}). Within these boundaries, however, one local scattering law accommodates estimation, identification, and constrained control on equal footing.

%% file: appendix/H_categorical_scattering.tex
\section{Functorial and categorical structure of wave scattering}
\label{app:categorical_scattering}

The composability of two-port scattering junctions (Section~\ref{sec:redheffer_composition}), the parameter-port formulation of optimization (Appendix~\ref{app:param_port}), the elimination of matrix inversions via singular value decomposition, and the pullback metric characterization of reflection-free propagation (Section~\ref{sec:discrete_orthogonal_reflection}) admit a unified categorical formulation. Placing learnable parameters directly into the construction via the parametric category $\mathbf{Para}(\cdot)$ shows that compiling a parameterized hybrid neural network into a multi-port wave-scattering circuit is a \emph{strict symmetric monoidal functor}, while both the joint state--parameter wave transform and state/parameter impedance re-gauging are \emph{monoidal natural isomorphisms}.

\subsection{The parametric category $\mathbf{Para}(\mathbf{Hyb})$ and the functors $\widetilde{\mathcal{P}}$ and $\widetilde{\mathcal{W}}_{G,Z}$}
\label{app:subsec_functors}

Let $\mathbf{Hyb}$ denote the symmetric monoidal category whose objects are finite-dimensional Euclidean spaces equipped with symmetric positive-definite metrics $(X, G_X)$ ($X = \RR^n$, $G_X = \Theta_X^\top \Theta_X \succ 0$, with $(X, G_X) \otimes (Y, G_Y) = (X \oplus Y, G_X \oplus G_Y)$), and whose morphisms are generated under sequential composition $\circ$ and parallel stacking $\oplus$ by smooth hybrid maps (orthogonal flows $e^A$, diagonal stretches $e^\Sigma$, nonlinear maps $\phi$) together with the linear Kirchhoff copy/add structure maps:
\begin{equation}
\label{eq:kirchhoff_generators_app}
\Delta_X : x \mapsto (x, x), \qquad
!_X : x \mapsto (), \qquad
+_X : (x_1, x_2) \mapsto x_1 + x_2, \qquad
0_X : () \mapsto 0 .
\end{equation}

To treat learnable weights, weight sharing, and reparameterizations on equal footing with activations, we pass to the parametric category $\mathbf{Para}(\mathbf{Hyb})$:
\begin{itemize}
\item \emph{Objects} are metric state spaces $(X, G_X)$ as in $\mathbf{Hyb}$.
\item \emph{Morphisms} $(X, G_X) \to (Y, G_Y)$ are parameterized hybrid layers $\bigl((P, Z_P),\, f\bigr)$, consisting of a parameter space equipped with a symmetric positive-definite parameter-port impedance $(P, Z_P)$ ($P = \RR^p$, $Z_P = Z_P^\top \succ 0$, matching \eqref{eq:param_port_effort_flow_app}) and a morphism $f : P \times X \to Y$ in $\mathbf{Hyb}$ (written $f(x; \theta)$). Parameter-free structural maps (such as $\Delta_X, !_X, +_X, 0_X$ or fixed activations) have trivial parameter space $P = \{0\}$.
\item \emph{Composition and monoidal product.} Composing $\bigl((P, Z_P), f\bigr) : X \to Y$ with $\bigl((Q, Z_Q), g\bigr) : Y \to Z$ stacks their parameter spaces in parallel:
\begin{equation}
\label{eq:para_composition_app}
\bigl((Q, Z_Q), g\bigr) \circ_{\mathbf{Para}} \bigl((P, Z_P), f\bigr)
= \Bigl((Q \oplus P,\, Z_Q \oplus Z_P),\; (\phi, \theta, x) \mapsto g\bigl(f(x; \theta); \phi\bigr)\Bigr),
\end{equation}
and parallel stacking $\bigl((P, Z_P), f\bigr) \oplus \bigl((Q, Z_Q), g\bigr)$ likewise pairs $(P \oplus Q, Z_P \oplus Z_Q)$ with $(\theta, \phi, x_1, x_2) \mapsto (f(x_1;\theta),\, g(x_2;\phi))$.
\item \emph{Reparameterizations ($2$-morphisms).} A smooth map $\alpha : (P', Z_{P'}) \to (P, Z_P)$ acts on $\bigl((P, Z_P), f\bigr)$ by pullback to $\bigl((P', Z_{P'}),\, (\theta', x) \mapsto f(x; \alpha(\theta'))\bigr)$. Two standard architectural mechanisms are reparameterizations in this sense:
\begin{enumerate}
\item \emph{Weight sharing} across $K$ spatial sites or tied layers is precomposition of the stacked parameter wire $P^{\oplus K}$ with the $K$-way copy spider $\Delta_P^{(K)} : P \to P^{\oplus K}$.
\item \emph{SVD reparameterization} of a dense linear layer $W \in \RR^{m\times n}$ of rank $r$ is the map
\begin{equation}
\label{eq:svd_reparam_app}
\alpha_{\mathrm{SVD}} : \mathfrak{so}(n) \times \RR^r \times \mathfrak{so}(m) \longrightarrow \RR^{m\times n},
\qquad
(A_V, \Sigma, A_U) \longmapsto e^{A_U}\,\Pi_{\mathrm{pad}}\,e^{\Sigma}\,\Pi_{\mathrm{drop}}\,e^{A_V}
\end{equation}
(absorbing any $\det=-1$ sign flip into a fixed diagonal sign jump), which simultaneously factors the state map into elementary stages and splits the dense parameter port into three local parameter ports.
\end{enumerate}
\end{itemize}

Let $\mathbf{Rel}$ be the symmetric monoidal category of relations on vector spaces, and let $\mathbf{Para}(\mathbf{Rel})$ be its parametric lifting where parameter ports stack in parallel under composition. Two symmetric monoidal functors map $\mathbf{Para}(\mathbf{Hyb})$ into $\mathbf{Para}(\mathbf{Rel})$:

\begin{enumerate}
\item \textbf{The parametric Pontryagin (reverse-derivative) functor $\widetilde{\mathcal{P}} : \mathbf{Para}(\mathbf{Hyb}) \to \mathbf{Para}(\mathbf{Rel})$.}
On state objects, $\widetilde{\mathcal{P}}(X, G_X) = T^*X = X \oplus X^*$ with pairing $\langle x, \lambda \rangle = \lambda^\top x$; on parameter spaces, $\widetilde{\mathcal{P}}_{\mathrm{param}}(P, Z_P) = T^*P = P^* \oplus P$ carries parameter effort--velocity pairs $(r_\theta, v_\theta)$ (where $v_\theta = \partial_\tau \theta$) with optimization power pairing $-r_\theta^\top v_\theta$ \eqref{eq:param_diss_power_app}. A parameterized layer $\bigl((P, Z_P), f\bigr)$ (optionally with running cost $R(x, \theta)$) maps to the three-port Pontryagin relation on $T^*P \times T^*X \times T^*Y$:
\begin{equation}
\label{eq:pontryagin_functor_app}
y = f(x; \theta), \qquad
\lambda_x = \bigl(D_x f(x; \theta)\bigr)^\top \lambda_y + \nabla_x R, \qquad
r_\theta = \bigl(D_\theta f(x; \theta)\bigr)^\top \lambda_y + \nabla_\theta R .
\end{equation}
On the Kirchhoff generators \eqref{eq:kirchhoff_generators_app}, $\widetilde{\mathcal{P}}(\Delta_X)$ enforces shared state and summed costates ($x_0 = x_1 = x_2$, $\lambda_0 = \lambda_1 + \lambda_2$), $\widetilde{\mathcal{P}}(!_X)$ enforces $\lambda = 0$ with $x$ free, $\widetilde{\mathcal{P}}(+_X)$ enforces summed states and shared costate ($y = x_1 + x_2$, $\lambda_1 = \lambda_2 = \lambda_y$), and $\widetilde{\mathcal{P}}(0_X)$ enforces $y = 0$ with $\lambda_y$ free. Under a reparameterization $\alpha : P' \to P$, the parameter effort pulls back by $r_{\theta'} = (D\alpha(\theta'))^\top r_\theta$; in particular, weight sharing $\Delta_P^{(K)}$ sums the $K$ local efforts ($r_\theta = \sum_{k=1}^K r_{\theta, k}$). Functoriality $\widetilde{\mathcal{P}}(g \circ_{\mathbf{Para}} f) = \widetilde{\mathcal{P}}(g) \circ_{\mathbf{Para}} \widetilde{\mathcal{P}}(f)$ is reverse-mode backpropagation with parameter gradients \eqref{eq:node_residuals}--\eqref{eq:param_residual}.

\item \textbf{The parametric wave-scattering functor $\widetilde{\mathcal{W}}_{G,Z} : \mathbf{Para}(\mathbf{Hyb}) \to \mathbf{Para}(\mathbf{Scat}) \hookrightarrow \mathbf{Para}(\mathbf{Rel})$.}
On state objects, $\widetilde{\mathcal{W}}_{G,Z}(X, G_X) = \RR^n \oplus \RR^n$ carries counter-propagating state--costate waves $(w_+, w_-)$ with pairing $\frac{1}{2}(\|w_+\|^2 - \|w_-\|^2)$. On parameter spaces, $\widetilde{\mathcal{W}}_{\mathrm{param}}(P, Z_P) = \RR^p \oplus \RR^p$ carries incident and reflected parameter waves $(a_\theta, b_\theta)$ \eqref{eq:param_wave_incident_app}--\eqref{eq:param_wave_reflected_app} with pairing $\frac{1}{2}(\|a_\theta\|^2 - \|b_\theta\|^2)$. Each parameterized layer $\bigl((P, Z_P), f\bigr)$ maps to a three-port scattering junction coupling the left/right state wave ports via \eqref{eq:nonlinear_layer_scattering_main} and exporting the local effort $r_\theta$ through the parameter port $(a_\theta, b_\theta)$.
\end{enumerate}

\begin{proposition}[Parametric monoidal functoriality and zero-inverse SVD compilation]
\label{prop:monoidal_functor_svd_app}
For any composable parameterized layers $\bigl((P, Z_P), f\bigr) : (X, G_X) \to (Y, G_Y)$ and $\bigl((Q, Z_Q), g\bigr) : (Y, G_Y) \to (Z, G_Z)$ in $\mathbf{Para}(\mathbf{Hyb})$,
\begin{equation}
\label{eq:functorial_redheffer_app}
\widetilde{\mathcal{W}}_{G,Z}\bigl(g \circ_{\mathbf{Para}} f\bigr)
= \widetilde{\mathcal{W}}_{G,Z}(f) \star_{\mathrm{state}} \widetilde{\mathcal{W}}_{G,Z}(g),
\qquad
\widetilde{\mathcal{W}}_{G,Z}(f_1 \oplus f_2)
= \widetilde{\mathcal{W}}_{G,Z}(f_1) \oplus \widetilde{\mathcal{W}}_{G,Z}(f_2),
\end{equation}
where $\star_{\mathrm{state}}$ takes the Redheffer star product \eqref{eq:redheffer_star_main} on the shared state--costate port $Y$ (in left-to-right diagram order) while stacking the parameter ports $(P, Z_P)$ and $(Q, Z_Q)$ in parallel. In particular, applying the SVD reparameterization \eqref{eq:svd_reparam_app} to any linear layer $W \in \RR^{m\times n}$ compiles $\widetilde{\mathcal{W}}_{G,Z}(W)$ without matrix inversion into
\begin{equation}
\label{eq:svd_functorial_split_app}
\begin{aligned}
\widetilde{\mathcal{W}}_{G,Z}\bigl(W \circ \alpha_{\mathrm{SVD}}\bigr)
&= \widetilde{\mathcal{W}}_{G,Z}(e^{A_V}) \star_{\mathrm{state}} \widetilde{\mathcal{W}}_{G,Z}(\Pi_{\mathrm{drop}}) \\
&\quad \star_{\mathrm{state}} \widetilde{\mathcal{W}}_{G,Z}(e^{\Sigma}) \star_{\mathrm{state}} \widetilde{\mathcal{W}}_{G,Z}(\Pi_{\mathrm{pad}}) \star_{\mathrm{state}} \widetilde{\mathcal{W}}_{G,Z}(e^{A_U}),
\end{aligned}
\end{equation}
where:
\begin{enumerate}
\item $\widetilde{\mathcal{W}}_{G,Z}(e^{A_V})$ and $\widetilde{\mathcal{W}}_{G,Z}(e^{A_U})$ are reflection-free orthogonal transports ($S_{e^A} = \operatorname{diag}(e^A, e^{-A})$) whose local parameter-port efforts (for scalar state metric $G = b I$) are the impedance-independent wave angular-momentum integrals
\begin{equation}
\label{eq:svd_lie_effort_app}
r_A = \operatorname{skew}\!\left(\int_0^1 \lambda(s)\,x(s)^\top\,\dd s\right)
= \frac{1}{2}\int_0^1 \bigl(w_+(s)\,w_-(s)^\top - w_-(s)\,w_+(s)^\top\bigr)\,\dd s ;
\end{equation}
\item $\widetilde{\mathcal{W}}_{G,Z}(e^\Sigma)$ decouples into $r$ scalar $2\times 2$ hyperbolic scatterers \eqref{eq:singular_reflection_main} whose local parameter-port efforts equal the conserved channel wave-power fluxes
\begin{equation}
\label{eq:svd_sigma_effort_app}
r_{\sigma_i} = \int_0^1 \lambda_i(s)\,x_i(s)\,\dd s
= \frac{1}{2}\bigl(w_{+,i}^2 - w_{-,i}^2\bigr);
\end{equation}
\item $\widetilde{\mathcal{W}}_{G,Z}(\Pi_{\mathrm{drop}})$ and $\widetilde{\mathcal{W}}_{G,Z}(\Pi_{\mathrm{pad}})$ are parameter-free strict Neumann ($\widetilde{\mathcal{W}}(!): w_- = +w_+$) and Dirichlet ($\widetilde{\mathcal{W}}(0): w_+ = -w_-$) mirror terminations on $\ker W$ and $\ker W^\top$, respectively.
\end{enumerate}
\end{proposition}

\begin{proof}
Relational composition in $\mathbf{Para}(\mathbf{Rel})$ eliminates the intermediate wave pair $(w_{+,Y}, w_{-,Y})$ while leaving the parameter ports $(P, Z_P)$ and $(Q, Z_Q)$ open and disjoint; by \eqref{eq:two_port_blocks}--\eqref{eq:redheffer_star_main}, eliminating $(w_{+,Y}, w_{-,Y})$ produces the Redheffer star product on the state ports, giving \eqref{eq:functorial_redheffer_app}. Applying \eqref{eq:functorial_redheffer_app} to \eqref{eq:svd_reparam_app} yields \eqref{eq:svd_functorial_split_app}. For $e^A$ ($A \in \mathfrak{so}(n)$), Duhamel's formula gives $r_A = \operatorname{skew}(\int_0^1 \lambda(s) x(s)^\top \dd s)$, and substituting $x = (w_+ + w_-)/\sqrt{2b}$, $\lambda = \sqrt{b/2}(w_+ - w_-)$ cancels $b$ and removes the symmetric dyads $w_+ w_+^\top$ and $-w_- w_-^\top$, leaving \eqref{eq:svd_lie_effort_app}. Along $\partial_s x_i = \sigma_i x_i$, $\partial_s \lambda_i = -\sigma_i \lambda_i$, the product $\lambda_i(s) x_i(s)$ is conserved and equals $\frac{1}{2}(w_{+,i}^2 - w_{-,i}^2)$ by \eqref{eq:disc_wave_power}, giving \eqref{eq:svd_sigma_effort_app}. Finally, $\lambda = 0$ in $\widetilde{\mathcal{P}}(!)$ gives $w_- = +w_+$, and $y = 0$ in $\widetilde{\mathcal{P}}(0)$ gives $w_+ = -w_-$.
\end{proof}

\paragraph{Optimization algorithms as electrical one-port terminations in $\mathbf{Scat}$.}
In $\mathbf{Para}(\mathbf{Scat})$, an optimizer is not an external imperative loop outside the model; it is a passive one-port circuit in the \emph{same} scattering category $\mathbf{Scat}$---built from the exact same Kirchhoff series/parallel spiders ($\bullet$, $\circ$), unit delays, and resistive loads---plugged directly into each open parameter port. Both the neural network and the optimization dynamics therefore close under a single global Tellegen energy balance:
\begin{itemize}
\item \emph{Continuous-time terminations (Theorem~\ref{thm:matched_parameter_termination_app}):} closing $(a_\theta, b_\theta)$ with a matched resistor ($b_\theta = 0$) enforces preconditioned gradient flow $\partial_\tau \theta = -Z_P^{-1} r_\theta$; closing it with a series $R_P$--$L_P$ dynamic impedance ($r_\theta + R_P v_\theta + L_P \partial_\tau v_\theta = 0$) yields continuous heavy-ball momentum \eqref{eq:heavy_ball_port_app}; capping it with a white Neumann open circuit ($v_\theta = 0 \iff b_\theta = +a_\theta$) freezes weights for inference with zero power extraction.
\item \emph{Discrete-time electrical RLC optimizer circuits (\citealp{boyd2024circuits}; Theorem~\ref{thm:wdf_rlc_unification_app}):} in discrete wave-digital port coordinates $(a_g, b_g)$ with port resistance $R_0 = G_0^{-1}$, terminating the parameter port with a reflection-free Kirchhoff \emph{series} spider ($R_0 = R_C + R_P$) joining a capacitive memory wire ($+z^{-1}$, storing $c_k$) and a reflection-free Kirchhoff \emph{parallel} spider ($G_P = G_d + G_L$) joining a dissipative damping resistor ($G_d$, absorbing $-\|v_P^k\|_{G_d}^2$) and an inductive momentum wire ($-z^{-1}$, storing $\ell_k$) recovers Riemannian gradient descent ($r = 0$, parallel tank deleted), exact Newton (matched metric gauge $G_\theta = 2\nabla^2 f(\theta)$, Proposition~\ref{prop:riemannian_min_refl_app}), Polyak heavy ball ($r = \beta$), and Nesterov accelerated gradient ($r = \beta^2$) as impedance-ratio specializations of a single one-port circuit in $\mathbf{Scat}$.
\end{itemize}
Because every state layer and every optimizer termination lives in $\mathbf{Scat}$ and obeys the same port-power identity, joint training is simply the full closed-circuit composition of $\widetilde{\mathcal{W}}_{G,Z}(\mathrm{Net})$ with its data source, terminal loss sink, and per-layer RLC parameter terminations (Figure~\ref{fig:string_functor_svd_rlc}).

\begin{figure}[t]
\centering
\includegraphics[width=\textwidth]{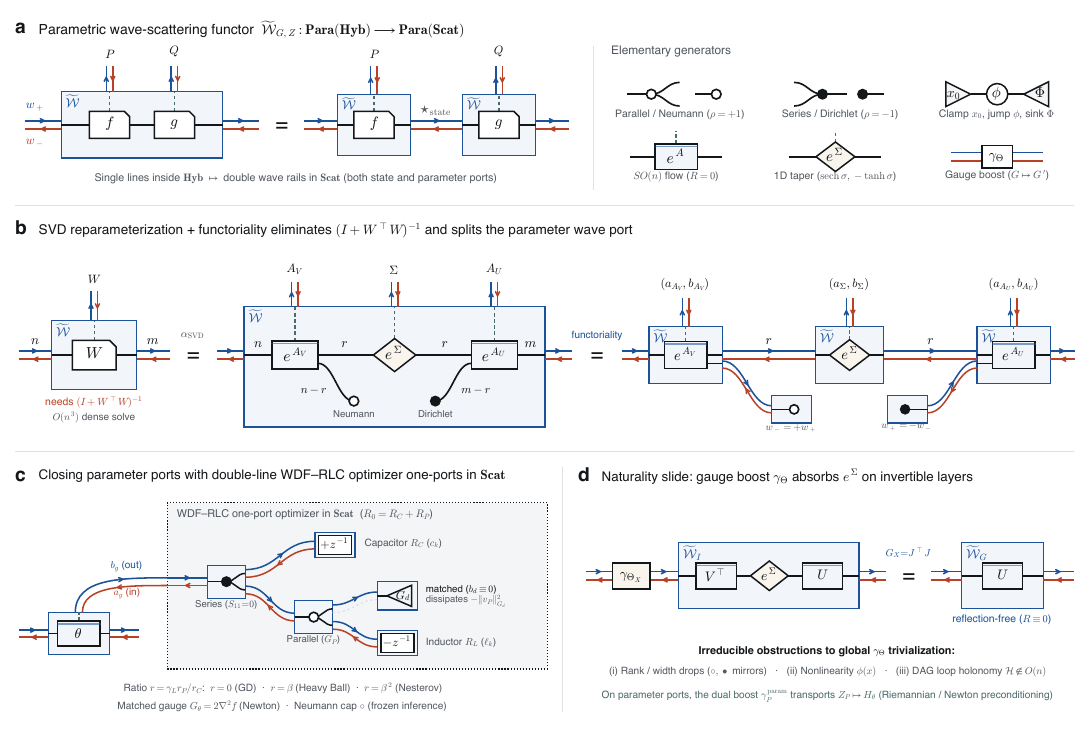}
\caption{\textbf{String-diagram grammar of parametric wave scattering in $\mathbf{Para}(\mathbf{Scat})$.} (a)~Strict symmetric monoidal functoriality $\widetilde{\mathcal{W}}_{G,Z}(g \circ_{\mathbf{Para}} f) = \widetilde{\mathcal{W}}_{G,Z}(f) \star_{\mathrm{state}} \widetilde{\mathcal{W}}_{G,Z}(g)$ drawn with Melli\`es functorial boxes, together with the six elementary generators: Kirchhoff parallel/Neumann white spiders ($\circ$), Kirchhoff series/Dirichlet black spiders ($\bullet$), boundary clamps/jumps/sinks ($x_0, \phi, \Phi$), orthogonal flows ($e^A$), 1D singular-value tapers ($e^\Sigma$), and metric gauge boosts ($\gamma_\Theta$). (b)~Combining the SVD reparameterization $\alpha_{\mathrm{SVD}}$ with functoriality eliminates the $O(n^3)$ matrix solve $(I + W^\top W)^{-1}$ and splits the dense parameter port into three local wave-torque and wave-power ports. (c)~Closing any open parameter port with a passive WDF--RLC one-port circuit in $\mathbf{Scat}$ (Theorem~\ref{thm:wdf_rlc_unification_app}) unifies gradient descent ($r=0$), Polyak heavy ball ($r=\beta$), Nesterov acceleration ($r=\beta^2$), Newton ($G_\theta = 2\nabla^2 f$), and frozen inference (Neumann cap $\circ$) inside a single Tellegen-conservative circuit. (d)~Naturality slide of the metric gauge isomorphism $\gamma_{\Theta_X}$ absorbing $e^\Sigma$ on invertible layers (Corollary~\ref{cor:pullback_trivialization_app}) and the three obstructions to global trivialization (Remark~\ref{rem:three_obstructions_app}).}
\label{fig:string_functor_svd_rlc}
\end{figure}

\subsection{The joint state--parameter natural isomorphism}
\label{app:subsec_cayley_nat_iso}

At every state object $(X, G_X)$ ($G_X = \Theta_X^\top \Theta_X$) and parameter object $(P, Z_P)$ ($Z_P = Z_P^\top \succ 0$), define the state and parameter wave-variable isomorphisms \eqref{eq:disc_wave_def} and \eqref{eq:param_wave_incident_app}--\eqref{eq:param_wave_reflected_app}:
\begin{equation}
\label{eq:cayley_component_app}
\mathcal{C}_{X, G_X}
\begin{pmatrix} x \\ \lambda \end{pmatrix}
= \frac{1}{\sqrt{2}}
\begin{pmatrix} \Theta_X & \Theta_X^{-\top} \\ \Theta_X & -\Theta_X^{-\top} \end{pmatrix}
\begin{pmatrix} x \\ \lambda \end{pmatrix}
= \begin{pmatrix} w_+ \\ w_- \end{pmatrix},
\qquad
\mathcal{C}^{\mathrm{param}}_{P, Z_P}
\begin{pmatrix} r_\theta \\ v_\theta \end{pmatrix}
= \frac{1}{\sqrt{2}}
\begin{pmatrix} Z_P^{-1/2} & -Z_P^{1/2} \\ Z_P^{-1/2} & Z_P^{1/2} \end{pmatrix}
\begin{pmatrix} r_\theta \\ v_\theta \end{pmatrix}
= \begin{pmatrix} a_\theta \\ b_\theta \end{pmatrix}.
\end{equation}

\begin{proposition}[Naturality of the joint state--parameter wave transform]
\label{prop:cayley_natural_iso_app}
The pair $\widetilde{\mathcal{C}}_{G,Z} = \bigl(\{\mathcal{C}_{X, G_X}\},\, \{\mathcal{C}^{\mathrm{param}}_{P, Z_P}\}\bigr)$ defines a symmetric monoidal natural isomorphism $\widetilde{\mathcal{C}}_{G,Z} : \widetilde{\mathcal{P}} \xRightarrow{\;\sim\;} \widetilde{\mathcal{W}}_{G,Z}$. For every parameterized layer $\bigl((P, Z_P), f\bigr) : (X, G_X) \to (Y, G_Y)$ in $\mathbf{Para}(\mathbf{Hyb})$, the three-port naturality square commutes in $\mathbf{Rel}$:
\begin{equation}
\label{eq:cayley_naturality_square_app}
\begin{array}{ccc}
T^*P \oplus \widetilde{\mathcal{P}}(X, G_X) & \xrightarrow{\;\;\widetilde{\mathcal{P}}(f)\;\;} & \widetilde{\mathcal{P}}(Y, G_Y) \\[4pt]
\Big\downarrow\scriptstyle{\mathcal{C}^{\mathrm{param}}_{P, Z_P} \oplus \mathcal{C}_{X, G_X}} & & \Big\downarrow\scriptstyle{\mathcal{C}_{Y, G_Y}} \\[6pt]
\widetilde{\mathcal{W}}_{\mathrm{param}}(P, Z_P) \oplus \widetilde{\mathcal{W}}_{G,Z}(X, G_X) & \xrightarrow{\;\;\widetilde{\mathcal{W}}_{G,Z}(f)\;\;} & \widetilde{\mathcal{W}}_{G,Z}(Y, G_Y)
\end{array}
\end{equation}
and is natural with respect to reparameterizations $\alpha : (P', Z_{P'}) \to (P, Z_P)$. Moreover, $\widetilde{\mathcal{C}}_{G,Z}$ is an isometry of both the state power form $\lambda^\top x = \frac{1}{2}(\|w_+\|^2 - \|w_-\|^2)$ and the parameter dissipation form $-r_\theta^\top v_\theta = \frac{1}{2}(\|a_\theta\|^2 - \|b_\theta\|^2)$.
\end{proposition}

\begin{proof}
By definition of $\widetilde{\mathcal{W}}_{G,Z}(f)$, a triple of wave pairs $\bigl((a_\theta, b_\theta), (w_{+,X}, w_{-,X}), (w_{+,Y}, w_{-,Y})\bigr)$ lies in $\widetilde{\mathcal{W}}_{G,Z}(f)$ if and only if its inverse image under $\mathcal{C}^{\mathrm{param}}_{P, Z_P} \oplus \mathcal{C}_{X, G_X} \oplus \mathcal{C}_{Y, G_Y}$ satisfies \eqref{eq:pontryagin_functor_app}, which is equivalent to \eqref{eq:cayley_naturality_square_app}. Monoidality follows from the block-diagonal structure of $\Theta_{X\oplus Y}$ and $Z_{P\oplus Q}$; naturality under $\alpha : P' \to P$ holds identically because $\widetilde{\mathcal{W}}_{G,Z}$ conjugates the pullback relation $(r_{\theta'}, v_{\theta'}) \sim_\alpha (r_\theta, v_\theta)$ by $\mathcal{C}^{\mathrm{param}}_{P', Z_{P'}}$ and $\mathcal{C}^{\mathrm{param}}_{P, Z_P}$. The two power identities are \eqref{eq:disc_wave_power} and \eqref{eq:param_power_identity_app}.
\end{proof}

Three consequences follow immediately from Proposition~\ref{prop:cayley_natural_iso_app}:
\begin{enumerate}
\item \emph{Exact fixed-point and gradient equivalence.} Because $\widetilde{\mathcal{C}}_{G,Z}$ is a natural isomorphism on both state and parameter ports, every stationary wave configuration of a composed network $\widetilde{\mathcal{W}}_{G,Z}(\mathrm{Net})$ is in bijection with an exact Pontryagin / backpropagation equilibrium of $\widetilde{\mathcal{P}}(\mathrm{Net})$ with identical parameter gradients $r_\theta$.
\item \emph{Passivity across state and parameter ports.} Because $\widetilde{\mathcal{C}}_{G,Z}$ preserves both power forms, every Tellegen-conservative layer maps to an orthogonal state scatterer while any passive parameter termination ($\|b_\theta\| \le \|a_\theta\|$) strictly dissipates optimization energy.
\item \emph{Extension to directional $(Y^\rightarrow, Z^\leftarrow)$ splittings.} Replacing the symmetric state metric $G$ by a transverse pair $(Y^\rightarrow \succeq 0, Z^\leftarrow \succeq 0)$ replaces $\mathcal{C}_{X, G_X}$ with the asymmetric natural isomorphism $\mathcal{C}_{Y,Z}(x, \lambda) = (x + Z^\leftarrow \lambda,\, \lambda - Y^\rightarrow x)$ (invertible whenever $Y^\rightarrow, Z^\leftarrow \succeq 0$), propagating through rank-deficient Jacobians via Riccati recursions while leaving the parameter port $\mathcal{C}^{\mathrm{param}}_{P, Z_P}$ unchanged.
\end{enumerate}

\subsection{State and parameter impedance re-gauging as natural isomorphisms}
\label{app:subsec_metric_gauge_nat_iso}

Fixing the network architecture and changing the state metric bundle $G \mapsto G'$ ($G'_v = \Theta_v'^\top \Theta'_v$) and parameter impedances $Z \mapsto Z'$ ($Z'_P \succ 0$) induces a composite monoidal natural isomorphism
\begin{equation}
\label{eq:gauge_nat_iso_def_app}
\widetilde{\gamma}_{(\Theta', Z') \leftarrow (\Theta, Z)}
:= \widetilde{\mathcal{C}}_{G', Z'} \circ \widetilde{\mathcal{C}}_{G, Z}^{-1}
\;:\; \widetilde{\mathcal{W}}_{G, Z} \xRightarrow{\;\sim\;} \widetilde{\mathcal{W}}_{G', Z'}.
\end{equation}

\begin{proposition}[Lorentz-boost form of state and parameter impedance re-gauging]
\label{prop:metric_gauge_natural_iso_app}
Let $M_v := \Theta'_v \Theta_v^{-1} \in GL(n_v)$ on state wire $v$ and $N_P := Z_P'^{1/2} Z_P^{-1/2} \in GL(p)$ on parameter port $P$. The components of $\widetilde{\gamma}$ on state and parameter waves are
\begin{equation}
\label{eq:metric_gauge_matrix_app}
\gamma_v
= \frac{1}{2}
\begin{pmatrix}
M_v + M_v^{-\top} & M_v - M_v^{-\top} \\[2pt]
M_v - M_v^{-\top} & M_v + M_v^{-\top}
\end{pmatrix},
\qquad
\gamma^{\mathrm{param}}_P
= \frac{1}{2}
\begin{pmatrix}
N_P + N_P^{-\top} & N_P^{-\top} - N_P \\[2pt]
N_P^{-\top} - N_P & N_P + N_P^{-\top}
\end{pmatrix}.
\end{equation}
In particular:
\begin{enumerate}
\item An orthogonal change of state vielbein ($M_v = Q \in O(n_v)$, leaving $G'_v = G_v$ unchanged) gives $\gamma_v = \operatorname{diag}(Q, Q)$ (a decoupled, reflection-free rotation).
\item Pure symmetric stretches $M_v = e^{D_v} \succ 0$ and $N_P = e^{D_P} \succ 0$ yield exact Lorentz-boost (telegraph) transfer matrices on both state and parameter ports:
\begin{equation}
\label{eq:metric_lorentz_boost_app}
\gamma_v = \begin{pmatrix} \cosh D_v & \sinh D_v \\ \sinh D_v & \cosh D_v \end{pmatrix},
\qquad
\gamma^{\mathrm{param}}_P = \begin{pmatrix} \cosh D_P & -\sinh D_P \\ -\sinh D_P & \cosh D_P \end{pmatrix}.
\end{equation}
\end{enumerate}
\end{proposition}

\begin{proof}
Multiplying the explicit matrices in \eqref{eq:cayley_component_app} gives
\begin{align*}
\mathcal{C}_{v,G'} \mathcal{C}_{v,G}^{-1}
&= \frac{1}{2}
\begin{pmatrix} \Theta'_v & \Theta_v'^{-\top} \\ \Theta'_v & -\Theta_v'^{-\top} \end{pmatrix}
\begin{pmatrix} \Theta_v^{-1} & \Theta_v^{-1} \\ \Theta_v^\top & -\Theta_v^\top \end{pmatrix}
= \frac{1}{2}
\begin{pmatrix}
M_v + M_v^{-\top} & M_v - M_v^{-\top} \\
M_v - M_v^{-\top} & M_v + M_v^{-\top}
\end{pmatrix},\\
\mathcal{C}^{\mathrm{param}}_{P,Z'_P} \bigl(\mathcal{C}^{\mathrm{param}}_{P,Z_P}\bigr)^{-1}
&= \frac{1}{2}
\begin{pmatrix} Z_P'^{-1/2} & -Z_P'^{1/2} \\ Z_P'^{-1/2} & Z_P'^{1/2} \end{pmatrix}
\begin{pmatrix} Z_P^{1/2} & Z_P^{1/2} \\ -Z_P^{-1/2} & Z_P^{-1/2} \end{pmatrix}
= \frac{1}{2}
\begin{pmatrix}
N_P^{-\top} + N_P & N_P^{-\top} - N_P \\
N_P^{-\top} - N_P & N_P^{-\top} + N_P
\end{pmatrix},
\end{align*}
since $(Z_P'^{1/2} Z_P^{-1/2})^{-\top} = Z_P'^{-1/2} Z_P^{1/2} = N_P^{-\top}$ by symmetry of $Z_P, Z'_P$. Substituting $M_v = Q \in O(n_v)$, $M_v = e^{D_v}$, and $N_P = e^{D_P}$ immediately gives \eqref{eq:metric_lorentz_boost_app}.
\end{proof}

\begin{corollary}[Pullback trivialization on state ports and curvature matching on parameter ports]
\label{cor:pullback_trivialization_app}
The natural isomorphism \eqref{eq:gauge_nat_iso_def_app} unifies reflection cancellation on state wires and adaptive preconditioning on parameter wires:
\begin{enumerate}
\item \emph{State ports (pullback metric).} For any invertible linear layer $J = U e^\Sigma V^\top : X \to Y$, applying $\gamma_{\Theta_X}$ with $\Theta_Y = I$ and $\Theta_X = e^\Sigma V^\top$ ($G_X = J^\top G_Y J$, \eqref{eq:discrete_orthogonal_condition}) yields the commutative naturality square
\begin{equation}
\label{eq:pullback_trivialization_square_app}
\begin{array}{ccc}
\widetilde{\mathcal{W}}_{I, Z}(X) & \xrightarrow{\;\;\widetilde{\mathcal{W}}_{I,Z}(V^\top)\,\star_{\mathrm{state}}\,\widetilde{\mathcal{W}}_{I,Z}(e^\Sigma)\,\star_{\mathrm{state}}\,\widetilde{\mathcal{W}}_{I,Z}(U)\;\;} & \widetilde{\mathcal{W}}_{I, Z}(Y) \\[4pt]
\Big\downarrow\scriptstyle{\gamma_{\Theta_X}} & & \Big\downarrow\scriptstyle{\mathrm{id}} \\[6pt]
\widetilde{\mathcal{W}}_{J^\top J, Z}(X) & \xrightarrow{\qquad\qquad \widetilde{\mathcal{W}}_{J^\top J, I, Z}(U)\;\;(\text{reflection-free}) \qquad\qquad} & \widetilde{\mathcal{W}}_{I, Z}(Y)
\end{array}
\end{equation}
in which $\Theta_Y J \Theta_X^{-1} = U \in O(n)$: sliding $\gamma_{\Theta_X}$ across $J$ absorbs $e^\Sigma$ identically, proving that bulk reflection on an invertible linear step is a pure gauge coboundary.
\item \emph{Parameter ports (Hessian / Riemannian preconditioning).} Under a local quadratic objective model with curvature $H_\theta \succ 0$ \eqref{eq:local_quadratic_impedance_app}, applying $\gamma^{\mathrm{param}}_P$ to transport the parameter impedance from $Z_P = I$ to $Z'_P = H_\theta$ renders the matched resistive termination ($b'_\theta = 0$) identically the continuous-time Newton step $\partial_\tau \theta = -H_\theta^{-1} r_\theta$ (Proposition~\ref{prop:riemannian_min_refl_app}).
\end{enumerate}
\end{corollary}

\begin{remark}[Three obstructions to global reflection-free trivialization]
\label{rem:three_obstructions_app}
A global natural isomorphism $\gamma : \widetilde{\mathcal{W}}_{I, Z} \xRightarrow{\sim} \widetilde{\mathcal{W}}_{\mathrm{orth}, Z}$ trivializing all interior state reflections across a network DAG $\mathcal{G}$ exists if and only if the pullback relation $G_u = J_e^\top G_v J_e$ (together with the Kirchhoff vertex matching laws \eqref{eq:kirchhoff_metric_match}) admits a consistent SPD solution across every edge and cycle. Consistent with Section~\ref{sec:discrete_orthogonal_reflection} and Section~\ref{sec:residual_branching}, global trivialization fails for exactly three reasons:
\begin{enumerate}
\item \emph{Local rank or dimension drops ($\ker J_e \neq \{0\}$ or $n_u \neq n_v$).} The pullback $J_e^\top G_v J_e$ is singular, so no invertible gauge $\Theta_u \in GL(n_u)$ exists; the unavoidable obstructions are the Neumann ($\rho=+1$) and Dirichlet ($\rho=-1$) mirrors (cured only by enlarging to the directional splitting $\mathcal{C}_{Y,Z}$).
\item \emph{State-dependent nonlinearity and non-integrability of differential frames.} For a nonlinear layer $\phi(x)$, $J_\phi(x)$ depends on the activation, so no fixed metric $\Theta$ makes $\Theta J_\phi(x) \Theta^{-1}$ orthogonal across all states, and in flat Euclidean coordinates the indefinite curvature term $\lambda^\top \nabla_x^2 \phi$ breaks local losslessness. Passing to a state-dependent Riemannian metric $M(x) = \Theta(x)^\top\Theta(x)$ sharpens this distinction via \citet[Section~IV.C]{manchester2017control}: global algebraic wave variables $w_\pm=\tfrac{1}{\sqrt{2}}(\theta(x)\pm(\partial\theta/\partial x)^{-\top}\lambda)$ require an \emph{integrable} flat frame $\Theta(x)=\partial\theta/\partial x$ ($R\equiv 0$, corresponding to feedback linearizability), whereas differential wave variables $\delta w_\pm=\tfrac{1}{\sqrt{2}}(\Theta(x)\delta x\pm\Theta(x)^{-\top}\delta^{\mathrm{cov}}\lambda)$ require no Schwarz integrability of $\Theta(x)$, admit convex synthesis of $W(x)=M(x)^{-1}$, and reduce the nonlinear curvature obstruction to $\mathcal{L}_f\nabla$ and residual-coupled Riemann curvature $R$ (Theorem~\ref{thm:covariant_wave_energy_balance_app}).
\item \emph{Global DAG loop holonomy.} Around a reconverging diamond with parallel paths $f_1, f_2 : u \rightrightarrows v$, path~$1$ pulls $G_v$ back to $J_1^\top G_v J_1$ while path~$2$ pulls $G_v$ back to $J_2^\top G_v J_2$. Whenever the loop holonomy $\mathcal{H} = \Theta_1 \Theta_2^{-1} \notin O(n_u)$, no single metric $G_u$ at the split node can absorb $e^\Sigma$ along both branches simultaneously, leaving an irreducible topological reflection determined by $\frac{1}{2}\log(\mathcal{H}^\top \mathcal{H})$---the exact discrete counterpart of continuous Riemann curvature $R\neq 0$ in item~(2). Similarly, when symmetries (such as SVD gauge orbits or normalization spheres) yield a nonlinear equilibrium submanifold rather than an isolated fixed point, transverse submanifold contraction \citep[Section~V]{manchester2017control} replaces pointwise contraction (Appendix~\ref{app:geodesic_waves_submanifold}).
\end{enumerate}
\end{remark}

\subsection{Worked architectural compilations: LeNet-5 and nanoGPT}
\label{app:subsec_worked_examples}

To illustrate how standard neural architectures compile mechanically in $\mathbf{Para}(\mathbf{Hyb}) \xrightarrow{\widetilde{\mathcal{W}}_{G,Z}} \mathbf{Para}(\mathbf{Scat})$ with zero matrix inversions, Figures~\ref{fig:string_lenet5} and~\ref{fig:string_nanogpt} give complete string diagrams for a compressing convolutional chain (LeNet-5) and a constant-width Pre-LN residual transformer (nanoGPT).

\paragraph{LeNet-5 ($32\times 32 \to 10$, Figure~\ref{fig:string_lenet5}).}
Each valid $5\times 5$ convolution splits into a state-level $\mathrm{im2col}$ parallel spider ($\circ$) fanning input pixels across spatial patches ($784$ patches of width $25$ for $\mathrm{Conv}_1$; $100$ patches of width $150$ for $\mathrm{Conv}_2$), a shared kernel SVD-factored block ($e^{A_U}\Pi e^\Sigma e^{A_V}$), and a parameter-level weight-sharing copy spider ($\Delta_P^{(784)}$) that sums all local wave torques into a single WDF--RLC parameter port. Each $2\times 2$ average-pooling layer admits an exact closed-form SVD via the orthogonal 2D Haar transform $H_4 \in O(4)$: the $\mathrm{LL}$ average mode passes through the scalar taper $e^{-\ln 2}$ ($r = \tfrac{3}{5}, t = \tfrac{4}{5}$, or $s_c = \alpha_c/2$ with learned gain $\alpha_c$), while the three detail subbands $(\mathrm{LH}, \mathrm{HL}, \mathrm{HH})$ terminate in strict Neumann caps ($\circ$, $\rho = +1$, $\lambda = 0$). Because LeNet-5 is a strongly compressing chain ($4704 \to 1176 \to 1600 \to 400 \to 120 \to 84 \to 10$), enforcing $\rho = +1$ at every Neumann cap is essential to prevent phantom $L_2$ adjoints from overwhelming the terminal cross-entropy drive.

\begin{figure}[t]
\centering
\includegraphics[width=\textwidth]{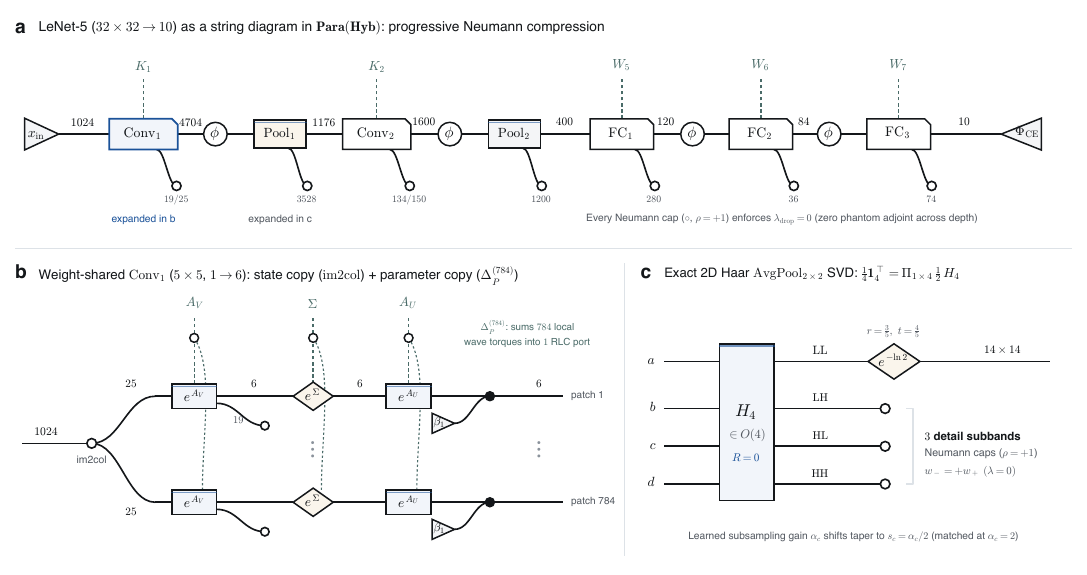}
\caption{\textbf{String-diagram compilation of LeNet-5 ($32\times 32 \to 10$) in $\mathbf{Para}(\mathbf{Hyb})$.} (a)~Full 1D hybrid chain from the Dirichlet input clamp $x_{\mathrm{in}}$ to the passive cross-entropy sink $\Phi_{\mathrm{CE}}$, showing progressive width compression and strict Neumann caps ($\circ$, $\rho=+1$) at every rank drop. (b)~Expanded weight-shared $\mathrm{Conv}_1$ cell ($5\times 5$, $1\to 6$): an $\mathrm{im2col}$ white spider fans pixels to $784$ patch fibers while a parameter copy spider $\Delta_P^{(784)}$ ties all $784$ local $(A_V, \Sigma, A_U)$ ports into a single shared WDF--RLC optimizer port. (c)~Exact 2D Haar SVD of $2\times 2$ average pooling: $H_4 \in O(4)$ rotates each $2\times 2$ tile reflection-free, routing $\mathrm{LL}$ through $e^{-\ln 2}$ and terminating $(\mathrm{LH}, \mathrm{HL}, \mathrm{HH})$ in three Neumann caps.}
\label{fig:string_lenet5}
\end{figure}

\paragraph{nanoGPT (GPT-2 124M: $12$L, $d=768$, $H=12$, $T=1024$, $V=50{,}257$, Figure~\ref{fig:string_nanogpt}).}
Token and positional embeddings ($E_{i_t}, p_t$) act as Dirichlet ket ports merged at a black series spider ($\bullet$) into $T=1024$ parallel $768$-wire token fibers. Inside each Pre-LN block, white parallel spiders ($\circ$) and black series spiders ($\bullet$) bridge a straight ballistic $768$-wire skip trunk (reflected power $O(\varepsilon^2)$, Section~\ref{sec:residual_branching}) around two branches:
\begin{itemize}
\item \emph{Causal self-attention:} $\mathrm{LN}_1$ terminates its two invariant directions ($\mathbf{1}, \hat z$) in Neumann/Dirichlet mirrors while impedance-matching the remaining $766$ tangential modes; a three-way parallel spider feeds square $768\times 768$ SVD-factored projections ($W_Q, W_K, W_V$, zero dimension mirrors) into a zero-inverse causal WDF attention crossbar whose bilinear QK scorer and $O(i)$ $2\times 2$ Woodbury softmax monotone resistor drive an attention-matched bipartite $\circ$-to-$\bullet$ value crossbar ($r_0 = r_{\mathrm{out}} = 0$, $O(1)$ edge tapers $\tilde g_{ij} = A_{ij}\sqrt{b^R_{ij}/b^L_{ij}} \approx \sqrt{c_j}$), followed by $W_O$.
\item \emph{$4\times$ MLP and thin-SVD LM head:} $W_{\mathrm{up}}$ ($768 \to 3072$) injects $2304$ Dirichlet zero-pad modes ($\bullet$) before the pointwise $\mathrm{GELU}$ jump, and $W_{\mathrm{down}}$ ($3072 \to 768$) terminates $2304$ modes in Neumann caps ($\circ$). At readout, the $50{,}257\times 768$ language-modeling head uses a thin Stiefel factor $\mathrm{St}(V, 768)$ (implicitly enforcing $49{,}489$ Dirichlet mirrors in $O(Vd)$ flops without forming $SO(V)$), tied back to $E_{i_t}$ by a weight-tying parameter copy spider $\Delta_{\mathrm{wte}}$ before the $O(V)$ Sherman--Morrison cross-entropy sinks.
\end{itemize}

\begin{figure}[t]
\centering
\includegraphics[width=\textwidth]{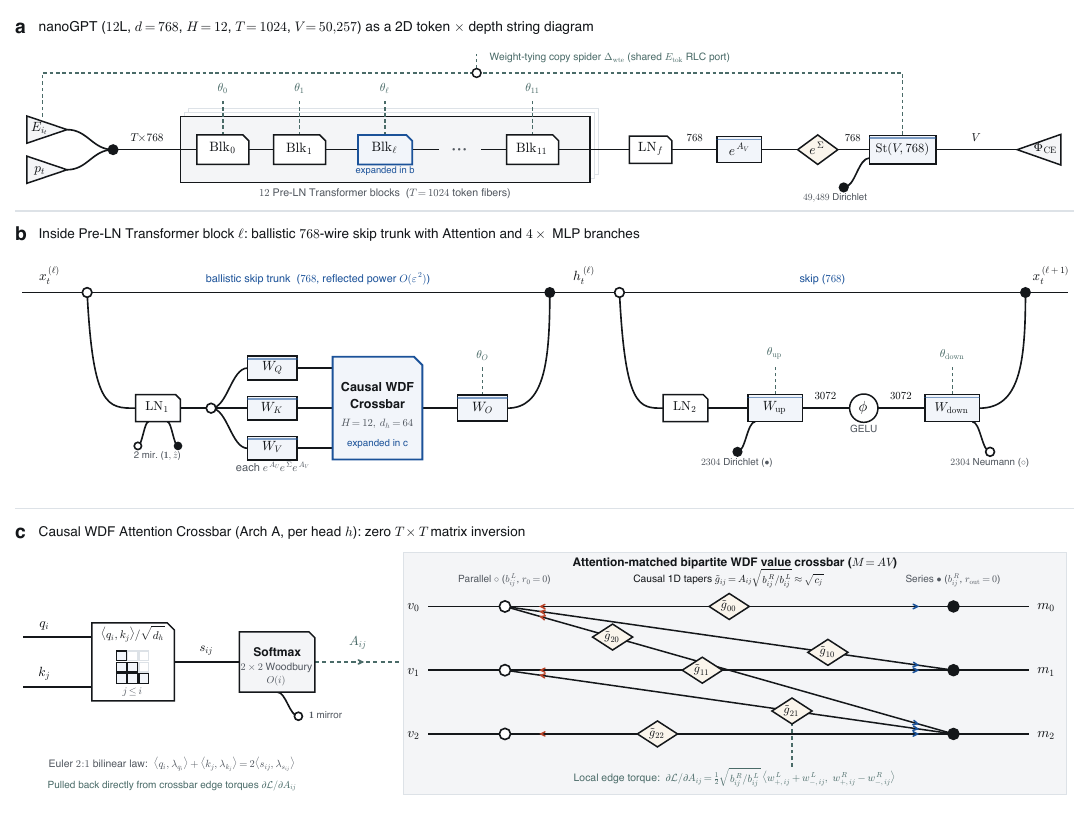}
\caption{\textbf{String-diagram compilation of nanoGPT ($12$L, $d=768$, $H=12$, $T=1024$, $V=50{,}257$) in $\mathbf{Para}(\mathbf{Hyb})$.} (a)~2D token $\times$ depth architecture from embedding ket ports ($E_{i_t}, p_t$) through $12$ Pre-LN blocks, $\mathrm{LN}_f$, and the thin-SVD LM head $\mathrm{St}(V, 768)$ to $T$ cross-entropy sinks $\Phi_{\mathrm{CE}}$, with a weight-tying copy spider $\Delta_{\mathrm{wte}}$ sharing one token-embedding RLC parameter port. (b)~Inside Pre-LN block~$\ell$: white ($\circ$) and black ($\bullet$) Kirchhoff spiders bridge the ballistic $768$-wire skip trunk around causal self-attention and the $4\times$ MLP ($2304$ Dirichlet $\bullet$ and $2304$ Neumann $\circ$ modes). (c)~Zero-inverse causal WDF attention crossbar (per head $h$): bilinear QK scoring and a $2\times 2$ Woodbury softmax resistor drive an attention-matched bipartite $\circ$-to-$\bullet$ value crossbar through $O(1)$ causal taper diamonds $\tilde g_{ij}$ with local colliding-wave edge torques $\partial\mathcal{L}/\partial A_{ij}$.}
\label{fig:string_nanogpt}
\end{figure}

%% file: appendix/I_adaptive_nn_control.tex
\section{Mirror-Descent Parameter Ports and Stable Deep Adaptive Neural Control}
\label{app:adaptive_nn_control}

This appendix provides the full derivations, proofs, and numerical verification
for the adaptive neural control results summarized in
Section~\ref{sec:composability}, connecting nonparametric and neural adaptive
control \citep{boffi2021implicit,boffi2022nonparametric} to our minimal-reflection
parameter ports (Appendix~\ref{app:param_port}) and orthogonal wave-scattering
layers (Section~\ref{sec:training_neural_network}).

\subsection{Setup and speed-gradient adaptation as a Wave-PMP limit}

Consider the continuous-time nonlinear plant with matched uncertainty
$h:\RR^n\to\RR^d$ and tracking-error dynamics of \citet{boffi2022nonparametric}:
\begin{equation}
\label{eq:boffi_plant_app}
\dot x = f(x,t) + g(x,t)\bigl(u(x,t) - h(x)\bigr),
\qquad
\dot e = f_e(e,t) + g_e(x,t)\bigl(u(x,t) - h(x)\bigr),
\end{equation}
equipped with a $C^1$ control Lyapunov function (or Riemannian contraction
energy) $Q:\RR^s\times\RR_{\ge 0}\to\RR_{\ge 0}$ satisfying
\begin{equation}
\label{eq:boffi_clf_app}
\langle \nabla Q(e,t), f_e(e,t)\rangle + \frac{\partial Q}{\partial t}(e,t)
\le -\rho(\|e\|_2),
\qquad
\mu_1(\|e\|_2) \le Q(e,t) \le \mu_2(\|e\|_2),
\end{equation}
for $\rho,\mu_1,\mu_2\in\cK_\infty$. Define the \emph{output boundary costate}
(error-port effort injected into the adaptive controller):
\begin{equation}
\label{eq:boffi_boundary_costate_app}
\lambda_{\mathrm{out}}(t) := g_e(x(t),t)^\top \nabla Q(e(t),t) \in \RR^d .
\end{equation}
More generally, when the plant admits a Control Contraction Metric $M(x,t)$ \citep{manchester2017control} rather than a quadratic CLF, taking $Q(x,x^\star,t)=\tfrac{1}{2}E(x,x^\star,t)$ (the Riemannian geodesic energy along a minimal geodesic $\gamma(t,\cdot)$ from $x^\star(t)$ to $x(t)$) identifies the exact output boundary costate via the first-variation formula \citep[Eq.~(12)]{manchester2017control} as $\lambda_{\mathrm{out}}(t)=g(x,t)^\top M(x,t)\gamma_s(t,1)$, extending Theorems~\ref{thm:two_timescale_waveguide_app}--\ref{thm:hybrid_zero_witness_app} to arbitrary universally stabilizable nonlinear plants while allowing persistent geodesic waves along $s\in[0,1]$ to track $\gamma(t,\cdot)$ concurrently (Appendix~\ref{app:geodesic_waves_submanifold}).
For a nonlinearly parameterized neural network $u(x,t)=\phi(x;\hat\theta(t))$,
the speed-gradient adaptation law \citep{fradkov1999nonlinear,boffi2022nonparametric}
\begin{equation}
\label{eq:boffi_speed_grad_app}
\dot{\hat\theta}
= -\gamma\,\nabla_{\hat\theta}\dot Q(e,\hat\theta,t)
= -\gamma\,\bigl(\nabla_{\hat\theta}\phi(x,\hat\theta)\bigr)^\top \lambda_{\mathrm{out}}(t)
\end{equation}
is the simultaneous collapse of Wave-PMP along three axes: a myopic physical-time
horizon ($H_{\mathrm{est}}=H_{\mathrm{ctrl}}=0$, with instantaneous boundary
costate \eqref{eq:boffi_boundary_costate_app}), infinite propagation speed across
network depth ($c\to\infty$, locked backpropagation through
$\nabla_{\hat\theta}\phi$), and a flat scalar resistive parameter port
($Z_\theta=\gamma^{-1}I$).

\subsection{Mirror descent as a Bregman nonlinear-capacitor parameter port}

Let $\psi:\mathcal{O}\to\RR$ be a twice-differentiable, strongly convex mirror
map with associated Bregman divergence
$d_\psi(\alpha^\star\|\hat\alpha):=\psi(\alpha^\star)-\psi(\hat\alpha)
-\langle\nabla\psi(\hat\alpha),\alpha^\star-\hat\alpha\rangle\ge 0$.

\begin{proposition}[Bregman nonlinear-capacitor parameter port]
\label{prop:bregman_port_app}
Equip a parameter port with effort $e_\alpha=r_\alpha$, flow
$v_\alpha=\dot{\hat\alpha}$, and state-dependent minimal-reflection impedance
$Z_\alpha(\hat\alpha)=\nabla^2\psi(\hat\alpha)\succ 0$ (a nonlinear capacitor
with dual charge $q=\nabla\psi(\hat\alpha)$ and stored incremental energy
$d_\psi(\alpha^\star\|\hat\alpha)$). Then:
\begin{enumerate}
\item The reflectionless port law $b_\alpha=0\iff\dot{\hat\alpha}=-Z_\alpha(\hat\alpha)^{-1}r_\alpha$
coincides identically with continuous-time mirror descent
\citep{boffi2021implicit,boffi2022nonparametric}:
\begin{equation}
\label{eq:bregman_mirror_port_app}
\frac{\dd}{\dd t}\nabla\psi(\hat\alpha) = \nabla^2\psi(\hat\alpha)\,\dot{\hat\alpha} = -r_\alpha,
\qquad
P_{\mathrm{diss}} = -r_\alpha^\top\dot{\hat\alpha}
= r_\alpha^\top\bigl[\nabla^2\psi(\hat\alpha)\bigr]^{-1}r_\alpha \ge 0.
\end{equation}
\item For any smooth (possibly time-varying) reference parameter
$\alpha^\star(t)$, the stored Bregman energy obeys the exact power balance
\begin{equation}
\label{eq:bregman_power_balance_app}
\frac{\dd}{\dd t}\,d_\psi\bigl(\alpha^\star(t)\,\big\|\,\hat\alpha(t)\bigr)
= -\langle \hat\alpha - \alpha^\star,\, r_\alpha\rangle
+ \bigl\langle \nabla\psi(\alpha^\star) - \nabla\psi(\hat\alpha),\, \dot\alpha^\star(t)\bigr\rangle .
\end{equation}
\end{enumerate}
\end{proposition}

\begin{proof}
Part~(1) follows immediately from Theorem~\ref{thm:matched_parameter_termination_app}
with $Z_\alpha(\hat\alpha)=\nabla^2\psi(\hat\alpha)\succ 0$. For Part~(2),
differentiating $d_\psi(\alpha^\star\|\hat\alpha)$ directly gives
\begin{align*}
\frac{\dd}{\dd t} d_\psi(\alpha^\star\|\hat\alpha)
&= \langle\nabla\psi(\alpha^\star),\dot\alpha^\star\rangle
 - \langle\nabla\psi(\hat\alpha),\dot{\hat\alpha}\rangle
 - \Bigl\langle\frac{\dd}{\dd t}\nabla\psi(\hat\alpha),\,\alpha^\star-\hat\alpha\Bigr\rangle
 - \langle\nabla\psi(\hat\alpha),\,\dot\alpha^\star-\dot{\hat\alpha}\rangle \\
&= \Bigl\langle \hat\alpha - \alpha^\star,\, \frac{\dd}{\dd t}\nabla\psi(\hat\alpha)\Bigr\rangle
 + \langle\nabla\psi(\alpha^\star)-\nabla\psi(\hat\alpha),\,\dot\alpha^\star\rangle .
\end{align*}
Substituting \eqref{eq:bregman_mirror_port_app} yields
\eqref{eq:bregman_power_balance_app}.
\end{proof}

Choosing $\psi(\alpha)=\tfrac{1}{2\gamma}\|\alpha\|_F^2$ recovers a linear
Euclidean capacitor ($Z_\alpha=\gamma^{-1}I$), whereas choosing the hypentropy
potential $\psi_\beta(\alpha)=\tfrac{1}{\gamma}\sum_j(\alpha_j\operatorname{arsinh}(\alpha_j/\beta)-\sqrt{\alpha_j^2+\beta^2})$
\citep{ghai2020exponentiated,boffi2022nonparametric} yields the diagonal
impedance $Z_{\alpha,jj}=[\gamma\sqrt{\hat\alpha_j^2+\beta^2}]^{-1}$,
smoothly interpolating between $\ell_2$ and multiplicative geometry.

\subsection{Power-balance anatomy of deep NN instability}

Factor an $L$-layer controller into its linear readout
$\hat\alpha\in\RR^{d\times m}$ and hidden feature network
$\Phi(\cdot;\hat\phi):\RR^n\to\RR^m$:
\begin{equation}
\label{eq:factored_nn_ctrl_app}
u(x,t) = \hat\alpha(t)\,\Phi\bigl(x;\hat\phi(t)\bigr).
\end{equation}
Let $\sigma_\Delta:\RR_{\ge 0}\to\RR_{\ge 0}$ be the $C^1$ smooth deadzone of
\citet{boffi2022nonparametric} (Example~6.3), with $\sigma_\Delta(q)=0$ for
$q\le\Delta$, $\sigma_\Delta(q)=(q-\Delta)^2/(4\gamma_{\mathrm{dz}})$ for
$q\in(\Delta,\Delta+2\gamma_{\mathrm{dz}})$, and
$\sigma_\Delta(q)=q-(\Delta+\gamma_{\mathrm{dz}})$ for
$q\ge\Delta+2\gamma_{\mathrm{dz}}$, so that
$\sigma_\Delta'(q)=\operatorname{clip}((q-\Delta)/(2\gamma_{\mathrm{dz}}),0,1)$
is $(1/(2\gamma_{\mathrm{dz}}))$-Lipschitz independently of $\Delta$.
For any frozen hidden weights $\hat\phi$, let $\alpha^\star(\hat\phi)$ be a
smooth reference readout on the operating ball $\cB_x=\{\|x\|_2\le B_x\}$ with
uniform representation residual
$\varepsilon(\hat\phi):=\sup_{x\in\cB_x}\|\alpha^\star(\hat\phi)\Phi(x;\hat\phi)-h(x)\|_2$.
Define the composite storage energy
\begin{equation}
\label{eq:composite_storage_app}
V(e,\hat\alpha,\hat\phi,t)
:= \sigma_\Delta\bigl(Q(e,t)\bigr) + d_\psi\bigl(\alpha^\star(\hat\phi)\,\big\|\,\hat\alpha\bigr).
\end{equation}
Driving the readout port \eqref{eq:bregman_mirror_port_app} with
$r_\alpha=\sigma_\Delta'(Q)\,\lambda_{\mathrm{out}}\,\Phi(x;\hat\phi)^\top$
cancels the bilinear cross-term
$\sigma_\Delta'(Q)\langle\lambda_{\mathrm{out}},(\hat\alpha-\alpha^\star)\Phi\rangle$
identically via Proposition~\ref{prop:bregman_port_app}, yielding the exact
power balance
\begin{equation}
\label{eq:exact_vdot_anatomy_app}
\dot V
\le -\sigma_\Delta'(Q)\Bigl[\rho(\|e\|_2) - \|\lambda_{\mathrm{out}}\|_2\,\varepsilon(\hat\phi)\Bigr]
+ \bigl\langle \nabla\psi(\alpha^\star(\hat\phi)) - \nabla\psi(\hat\alpha),\,
  D_{\hat\phi}\alpha^\star(\hat\phi)\,\dot{\hat\phi}\bigr\rangle .
\end{equation}
Equation~\eqref{eq:exact_vdot_anatomy_app} exposes why the flat-$\gamma$ Swish
MLP of \citet{boffi2022nonparametric} destabilizes with depth: (i)~the second
term injects uncancelled \emph{hidden-drift power} scaling with
$\|\dot{\hat\phi}\|$; (ii)~in an unconstrained $L$-layer MLP,
$\|\nabla_{\hat\phi}\Phi\|$ grows with both depth $L$ and state norm $\|x\|_2$,
triggering large control spikes; and (iii)~in discrete Euler integration, large
$\|\nabla_\theta\phi\|_F$ violates the gradient-port passivity bound
$\Delta t\,\gamma\,\|\nabla_\theta\phi\|_F^2\le 2$
(Proposition~\ref{prop:riemannian_min_refl_app}).

\subsection{Orthogonal hidden layers, matched hidden impedance, and closed-loop stability}

\paragraph{Fix 1: Lossless orthogonal hidden layers and conformal RMS matching.}
We construct $\Phi_{\mathrm{wave}}(x;\hat\phi)\in\RR^{m+1}$ (power-of-two width
$m$) from a bounded input lift $h_0=\tanh(W_{\mathrm{in}}(x/s_x)+b_{\mathrm{in}})$,
$L$ Givens-factored orthogonal layers $Q_\ell(\theta_\ell)\in SO(m)$
(products of $O(m\log m)$ $2\times 2$ Givens rotations on the compact torus
$\mathbb{T}^{(m/2)\log_2 m}$, zero matrix inversions) with residual blend
$\beta_{\mathrm{res}}\in(0,1]$:
\begin{equation}
\label{eq:ortho_waveguide_arch_app}
h_\ell = (1-\beta_{\mathrm{res}})\,h_{\ell-1}
+ \beta_{\mathrm{res}}\,\tanh\bigl(Q_\ell(\theta_\ell)\,h_{\ell-1} + b_\ell\bigr),
\qquad \ell=1,\dots,L,
\end{equation}
followed by the conformal RMS output impedance matcher of
Section~\ref{sec:residual_branching},
$\tilde h_L = h_L / \sqrt{\epsilon + \|h_L\|_2^2/m}$, augmented with a constant
DC channel $\Phi_{\mathrm{wave}}(x;\hat\phi)=[\tilde h_L;\,1]\in\RR^{m+1}$.

\begin{lemma}[Depth-invariant gain and zero bulk reflection]
\label{lem:ortho_waveguide_app}
For every depth $L\ge 1$, input $x\in\RR^n$, and Givens angles $\{\theta_\ell\}_{\ell=1}^L$:
(1)~$\|\Phi_{\mathrm{wave}}(x;\hat\phi)\|_2\le\sqrt{m+1}$ uniformly in $x$ and $L$;
(2)~every layer Jacobian satisfies $\|J_\ell\|_2\le 1$, so internal adjoints
are non-expanding ($\|\lambda_{\ell-1}\|_2\le\|\lambda_\ell\|_2$); and
(3)~in Slotine wave variables ($G_\ell=bI_m$), each orthogonal layer
$Q_\ell\in SO(m)$ satisfies the pullback condition $Q_\ell^\top Q_\ell=I_m$ and
scatters reflection-free ($r=0,t=1$), with exact local $2\times 2$ Givens wave
torque $g_{\theta_{\ell,s,j}}=w_p u_q - u_p w_q$ on every wire pair $(p,q)$.
\end{lemma}

\begin{proof}
Since $|\tanh(\cdot)|\le 1$, convexity preserves $h_\ell\in[-1,1]^m$, and
$\|\tilde h_L\|_2^2=\|h_L\|_2^2/(\epsilon+\|h_L\|_2^2/m)<m$, giving
$\|\Phi_{\mathrm{wave}}\|_2<\sqrt{m+1}$. Because $\|Q_\ell\|_2=1$ and
$0<\operatorname{sech}^2(\cdot)\le 1$,
$\|J_\ell\|_2\le (1-\beta_{\mathrm{res}})+\beta_{\mathrm{res}}\cdot 1=1$.
Zero reflection follows from \eqref{eq:discrete_orthogonal_condition}, and the
wave-torque identity follows from
$w_p u_q - u_p w_q = x_p\lambda_q - \lambda_p x_q = \partial_{\theta}\langle\lambda,R(\theta)x\rangle$.
\end{proof}

\paragraph{Fix 2: Effort- and load-matched hidden parameter ports.}
For each hidden parameter block $\hat\phi_\ell$ with raw speed-gradient
sensitivity $g_{\phi,\ell}=(\nabla_{\hat\phi_\ell}\Phi_{\mathrm{wave}})^\top\hat\alpha^\top\lambda_{\mathrm{out}}$
and deadzone-gated effort $r_{\phi,\ell}=\sigma_\Delta'(Q)\,g_{\phi,\ell}$,
we terminate the port with the adaptive resistive impedance
\begin{equation}
\label{eq:matched_hidden_impedance_app}
\begin{gathered}
Z_{\phi,\ell}(t)
:= \frac{1}{\gamma_{\mathrm{hid}}}\Bigl(c_0 + c_L\|\hat\alpha(t)\|_F^2\|h_{\ell-1}(t)\|_2^2
+ \kappa_g\|g_{\phi,\ell}(t)\|_F\Bigr) I \;\succ\; 0, \\
\dot{\hat\phi}_\ell = -Z_{\phi,\ell}(t)^{-1}\,r_{\phi,\ell}(t),
\end{gathered}
\end{equation}
with $c_0,c_L,\kappa_g>0$ (and radial ball projection on non-torus biases and
$W_{\mathrm{in}}$ so $\hat\phi\in\cK_{\mathrm{hid}}$ remains compact).
Equation~\eqref{eq:matched_hidden_impedance_app} is strictly passive
($-r_{\phi,\ell}^\top\dot{\hat\phi}_\ell\ge 0$), matches the diagonal
parameter-port load $K_{\phi,\ell}\propto\|\hat\alpha\|_F^2\|h_{\ell-1}\|_2^2$
(Section~\ref{sec:composability}), and enforces the deadzone-factored uniform
speed bound
\begin{equation}
\label{eq:hidden_speed_bound_app}
\|\dot{\hat\phi}_\ell(t)\|_F \le \sigma_\Delta'\bigl(Q(e,t)\bigr)\,\frac{\gamma_{\mathrm{hid}}}{\kappa_g}
\qquad \forall\, t\ge 0,\; \forall\, x,\; \forall\, L .
\end{equation}

\begin{theorem}[Two-timescale stability of pure deep orthogonal NNs]
\label{thm:two_timescale_waveguide_app}
Consider \eqref{eq:boffi_plant_app}--\eqref{eq:boffi_clf_app} coupled to
$u=\hat\alpha\,\Phi_{\mathrm{wave}}(x;\hat\phi)$ with readout mirror-descent
port \eqref{eq:bregman_mirror_port_app} and matched hidden ports
\eqref{eq:matched_hidden_impedance_app} over compact $\cK_{\mathrm{hid}}$.
Set $V_0:=V(e(0),\hat\alpha(0),\hat\phi(0),0)$ and choose
$R>\mu_1^{-1}(V_0+\Delta+\gamma_{\mathrm{dz}})$ defining $\cB_x$,
$B_\lambda:=\sup_{\|e\|\le R,\|x\|\le B_x}\|\lambda_{\mathrm{out}}\|_2$, and
$B_{\Delta\nabla\psi}(V_0):=\sup_{d_\psi(\alpha^\star\|\hat\alpha)\le V_0}\|\nabla\psi(\alpha^\star)-\nabla\psi(\hat\alpha)\|_F$.
Suppose $\alpha^\star:\cK_{\mathrm{hid}}\to\mathcal{O}$ is
$L_{\alpha^\star}$-Lipschitz with $\varepsilon_{\max}:=\sup_{\hat\phi\in\cK_{\mathrm{hid}}}\varepsilon(\hat\phi)$
and $\rho(\mu_2^{-1}(\Delta))>B_\lambda\varepsilon_{\max}$. If the aggregate
hidden speed $\bar v_{\mathrm{hid}}:=\sqrt{L_{\mathrm{blk}}}\,\gamma_{\mathrm{hid}}/\kappa_g$
satisfies the impedance-hierarchy bound
\begin{equation}
\label{eq:timescale_hierarchy_app}
\bar v_{\mathrm{hid}}
\le \frac{\rho\bigl(\mu_2^{-1}(\Delta)\bigr) - B_\lambda\,\varepsilon_{\max}}
         {2\,L_{\alpha^\star}\,B_{\Delta\nabla\psi}(V_0)},
\end{equation}
then for any depth $L\ge 1$, the closed-loop trajectory exists globally, stays
bounded in $\{V\le V_0\}$, satisfies
$\dot V\le -\tfrac{1}{2}\sigma_\Delta'(Q)[\rho(\|e\|_2)-B_\lambda\varepsilon_{\max}]\le 0$,
and obeys $\limsup_{t\to\infty}\|e(t)\|_2\le\mu_1^{-1}(\Delta)$.
\end{theorem}

\begin{proof}
On any maximal interval $[0,T)$ with $\|e(t)\|_2\le R$ and $V(t)\le V_0$,
substituting $\|\dot{\hat\phi}\|_F\le\sigma_\Delta'(Q)\,\bar v_{\mathrm{hid}}$
from \eqref{eq:hidden_speed_bound_app} into \eqref{eq:exact_vdot_anatomy_app}
yields
$\dot V\le -\sigma_\Delta'(Q)[\rho(\|e\|_2)-B_\lambda\varepsilon_{\max}-B_{\Delta\nabla\psi}(V_0)L_{\alpha^\star}\bar v_{\mathrm{hid}}]$.
For $Q\le\Delta$, $\sigma_\Delta'(Q)=0$ so $\dot V\le 0$. For $Q>\Delta$,
$\|e\|_2\ge\mu_2^{-1}(\Delta)$, so \eqref{eq:timescale_hierarchy_app} implies
$\dot V\le -\tfrac{1}{2}\sigma_\Delta'(Q)[\rho(\|e\|_2)-B_\lambda\varepsilon_{\max}]\le 0$.
Hence $V(t)\le V_0$, which forces
$Q(e(t),t)\le V_0+\Delta+\gamma_{\mathrm{dz}}<\mu_1(R)\implies\|e(t)\|_2<R$
strictly inside the ball, giving $T=\infty$ and global boundedness.
Integrating $\dot V$ shows
$\int_0^\infty\sigma_\Delta'(Q)[\rho(\|e\|_2)-B_\lambda\varepsilon_{\max}]\,\dd t\le 2V_0<\infty$;
uniform boundedness of $\Phi_{\mathrm{wave}}$ (Lemma~\ref{lem:ortho_waveguide_app})
and $\dot{\hat\phi}$ \eqref{eq:hidden_speed_bound_app} bounds $\dot e$ and
$\dot Q$, so Barbalat's lemma gives $\limsup_{t\to\infty}Q(e(t),t)\le\Delta$
and hence $\limsup_{t\to\infty}\|e(t)\|_2\le\mu_1^{-1}(\Delta)$.
\end{proof}

\paragraph{Fix 3: Kirchhoff parallel hybrid (RFF $\parallel$ deep orthogonal branch).}
To eliminate the hidden-drift term altogether---directly addressing the hybrid
kernel--neural question posed in \citet{boffi2022nonparametric}---wire a fixed
random Fourier feature branch $\Psi_{\mathrm{RFF}}(x)\in\RR^K$
\citep{rahimi2007random} in parallel with $\Phi_{\mathrm{wave}}(x;\hat\phi)$
through a lossless Kirchhoff merge junction (Section~\ref{sec:residual_branching}):
\begin{equation}
\label{eq:kirchhoff_hybrid_ctrl_app}
u(x,t) = \hat\alpha_{\mathrm{RFF}}(t)\,\Psi_{\mathrm{RFF}}(x)
+ \hat\alpha_{\mathrm{NN}}(t)\,\Phi_{\mathrm{wave}}\bigl(x;\hat\phi(t)\bigr),
\end{equation}
with mirror-descent readout ports
$\frac{\dd}{\dd t}\nabla\psi_1(\hat\alpha_{\mathrm{RFF}})=-\sigma_\Delta'(Q)\lambda_{\mathrm{out}}\Psi_{\mathrm{RFF}}^\top$
and
$\frac{\dd}{\dd t}\nabla\psi_2(\hat\alpha_{\mathrm{NN}})=-\sigma_\Delta'(Q)\lambda_{\mathrm{out}}\Phi_{\mathrm{wave}}^\top$,
and matched hidden ports \eqref{eq:matched_hidden_impedance_app}.

\begin{theorem}[Unconditional zero-witness stability of the Kirchhoff parallel hybrid]
\label{thm:hybrid_zero_witness_app}
Suppose there exists a constant readout $\alpha_{\mathrm{RFF}}^\star$ achieving
uniform error $\varepsilon_{\mathrm{RFF}}(K)=\sup_{\|x\|\le B_x}\|\alpha_{\mathrm{RFF}}^\star\Psi_{\mathrm{RFF}}(x)-h(x)\|_2$
(e.g.\ $O(1/\sqrt{K})$ with high probability for $h\in\mathcal{F}_2$
\citep{boffi2022nonparametric}), and choose
$\Delta>\mu_2(\rho^{-1}(B_\lambda\varepsilon_{\mathrm{RFF}}(K)))$.
Then for \emph{any} depth $L\ge 1$, \emph{any} mirror potentials
$\psi_1,\psi_2$, and \emph{any} hidden learning rate $\gamma_{\mathrm{hid}}>0$,
the composite Lyapunov function with constant zero witness
$(\alpha_{\mathrm{RFF}}^\star,0)$,
\begin{equation}
\label{eq:hybrid_zero_witness_V_app}
V_{\mathrm{hyb}}(e,\hat\alpha_{\mathrm{RFF}},\hat\alpha_{\mathrm{NN}},t)
:= \sigma_\Delta\bigl(Q(e,t)\bigr)
+ d_{\psi_1}\bigl(\alpha_{\mathrm{RFF}}^\star\,\big\|\,\hat\alpha_{\mathrm{RFF}}\bigr)
+ d_{\psi_2}\bigl(0\,\big\|\,\hat\alpha_{\mathrm{NN}}\bigr),
\end{equation}
obeys the exact drift-free dissipation inequality
\begin{equation}
\label{eq:hybrid_zero_witness_vdot_app}
\dot V_{\mathrm{hyb}}(t)
\le -\sigma_\Delta'\bigl(Q(e,t)\bigr)\Bigl[\rho(\|e\|_2) - \|\lambda_{\mathrm{out}}\|_2\,\varepsilon_{\mathrm{RFF}}(K)\Bigr]
\le 0
\qquad \forall\, t\ge 0 .
\end{equation}
Consequently, all trajectories exist globally, remain uniformly bounded, and
satisfy $\limsup_{t\to\infty}\|e(t)\|_2\le\mu_1^{-1}(\Delta)$.
\end{theorem}

\begin{proof}
Writing
$u-h=(\hat\alpha_{\mathrm{RFF}}-\alpha_{\mathrm{RFF}}^\star)\Psi_{\mathrm{RFF}}
+(\hat\alpha_{\mathrm{NN}}-0)\Phi_{\mathrm{wave}}
+(\alpha_{\mathrm{RFF}}^\star\Psi_{\mathrm{RFF}}-h)$
and applying Proposition~\ref{prop:bregman_port_app} to both Bregman terms
shows that both target-drift derivatives vanish identically because
$\dot\alpha_{\mathrm{RFF}}^\star\equiv 0$ and $\frac{\dd}{\dd t}0\equiv 0$,
while both bilinear cross-terms cancel $\langle\lambda_{\mathrm{out}},(\hat\alpha_{\mathrm{RFF}}-\alpha_{\mathrm{RFF}}^\star)\Psi_{\mathrm{RFF}}+\hat\alpha_{\mathrm{NN}}\Phi_{\mathrm{wave}}\rangle$
identically for every $\hat\phi(t)$. This yields
\eqref{eq:hybrid_zero_witness_vdot_app} without any hidden-drift remainder.
Since $\|\Phi_{\mathrm{wave}}\|_2<\sqrt{m+1}$ (Lemma~\ref{lem:ortho_waveguide_app})
and $\|\dot{\hat\phi}\|_F\le\bar v_{\mathrm{hid}}<\infty$
\eqref{eq:hidden_speed_bound_app}, global existence and convergence follow by
continuation and Barbalat's lemma exactly as in
Theorem~\ref{thm:two_timescale_waveguide_app}.
\end{proof}

Crucially, the witness $\alpha_{\mathrm{NN}}^\star=0$ is purely an analysis
device: the readout law for $\hat\alpha_{\mathrm{NN}}$ carries no decay toward
zero, so the deep orthogonal network still learns nonlinear features outside
the kernel regime while the parallel RFF witness certifies stability at every
depth.

\subsection{Unlocked Wave-PMP realization and 5D quartic benchmark}

To remove the $c\to\infty$ depth lock as well, we equip the $L$ chained
orthogonal layers ($\beta_{\mathrm{res}}=1$) with persistent Slotine waves
$(u_\ell,w_\ell)$ at characteristic impedance $b>0$. At each plant tick, all
$L$ layers scatter concurrently with zero matrix inversions and $0\,\mathrm{B}$
activation replay tape: each orthogonal butterfly plus bias scatters
reflection-free ($u_{\mathrm{mid}}=Q_\ell u_L+\sqrt{b/2}\,b_\ell$,
$w_L=Q_\ell^\top(w_{\mathrm{mid}}-\sqrt{b/2}\,b_\ell)$) while synthesizing exact
bias-compensated $2\times 2$ Givens wave torques, and each pointwise $\tanh$
solves its scalar 1D two-port equation
$b z_i+\operatorname{sech}^2(z_i)(b\tanh z_i-\sqrt{2b}\,w_{\mathrm{out},i})-\sqrt{2b}\,u_{\mathrm{in},i}=0$
locally (choosing $b$ large enough to dominate the indefinite second-order
curvature $H_{zz}=\lambda_{\mathrm{out}}\tanh''(z)$ of
Section~\ref{sec:discrete_orthogonal_reflection}). Because the physical plant
streams continuously ($\Delta_0\to 0\implies\tau^\star\to 1$ in
Section~\ref{sec:optimal_stopping}), persistent waves track depth adjoints with
just $2$ concurrent scattering ticks per plant step.

\begin{table}[t]
\centering
\small
\caption{\textbf{Depth-stability comparison on the unstable 5D quartic plant of \citet{boffi2022nonparametric} (\S7.3, Eq.~7.5).}
System $\dot x = A(x-x_d)+\dot x_d+u-\tfrac{1}{4}x^4$,
$x_{d,i}(t)=\sin(2\pi t)+\cos(\sqrt{2}\pi t)$, $Q=\tfrac{1}{2}e^\top P e$
($A^\top P+PA=-I$), $T=10\,\mathrm{s}$, $\Delta t=10^{-3}$, width $m=16$.
Unconstrained Swish MLPs exhibit large control spikes and blow up for
$L\in\{4,8\}$, whereas all three wave-scattering constructions remain strictly
stable across $L\in\{1,2,4,8\}$ while actively adapting hidden weights
($\|\Delta\hat\phi\|_2>0$).}
\label{tab:boffi_adaptive_nn}
\resizebox{\textwidth}{!}{%
\begin{tabular}{lcccccc}
\toprule
Controller & Depth $L$ & Status & Tail Mean $\|e\|_2$ & Tail Max $\|e\|_2$ & Peak $\|u\|_2$ & $\|\Delta\hat\phi\|_2$ \\
\midrule
No adaptation ($u\equiv 0$) & $0$ & Blowup ($0.678\,\mathrm{s}$) & $\infty$ & $\infty$ & $0.00$ & $0.000$ \\
RFF baseline ($K=64$) & $1$ & Stable & $0.2277$ & $0.5393$ & $46.38$ & $0.000$ \\
RFF baseline ($K=200$) & $1$ & Stable & $0.1937$ & $0.5337$ & $34.22$ & $0.000$ \\
\midrule
Unconstrained Swish MLP \citep{boffi2022nonparametric} & $1$ & Stable & $0.2176$ & $0.4397$ & $71.93$ & $1.522$ \\
Unconstrained Swish MLP \citep{boffi2022nonparametric} & $2$ & Stable & $0.1289$ & $0.5087$ & $89.46$ & $1.233$ \\
Unconstrained Swish MLP \citep{boffi2022nonparametric} & $4$ & Blowup ($0.806\,\mathrm{s}$) & $\infty$ & $\infty$ & $26{,}808.22$ & $4.6\times 10^7$ \\
Unconstrained Swish MLP \citep{boffi2022nonparametric} & $8$ & Blowup ($0.565\,\mathrm{s}$) & $\infty$ & $\infty$ & $23{,}081.07$ & $2.1\times 10^{12}$ \\
\midrule
Fix 1+2: Pure Ortho Layers + $Z_\phi$ (Thm.~\ref{thm:two_timescale_waveguide_app}) & $1,2,4,8$ & All stable & $0.148\text{--}0.172$ & $0.564\text{--}0.635$ & $26.9\text{--}27.8$ & $0.90\text{--}0.98$ \\
Fix 1+2+3: Hybrid ($K=64$ RFF $\parallel$ Ortho NN, Thm.~\ref{thm:hybrid_zero_witness_app}) & $1,2,4,8$ & All stable & $0.130\text{--}0.142$ & $0.311\text{--}0.350$ & $39.8\text{--}41.3$ & $0.62\text{--}0.67$ \\
Unlocked Wave-PMP ($0\,\mathrm{B}$ tape, $0$ inv, $2$ ticks) & $1,2,4,8$ & All stable & $0.104\text{--}0.138$ & $0.256\text{--}0.343$ & $31.8\text{--}35.3$ & $0.65\text{--}0.71$ \\
\bottomrule
\end{tabular}%
}
\end{table}

Table~\ref{tab:boffi_adaptive_nn} summarizes the depth sweep on the unstable 5D
quartic benchmark of \citet{boffi2022nonparametric}. Without adaptation, the
plant escapes at $t=0.678\,\mathrm{s}$. Under flat scalar speed-gradient
adaptation, the unconstrained Swish MLP develops $2\text{--}3\times$ larger
control spikes than RFF at $L\in\{1,2\}$ and blows up at $L\in\{4,8\}$. By
contrast, the pure orthogonal network with matched hidden impedance remains
stable across all depths with depth-invariant peak control effort
($\|u\|_2\approx 27$) and beats $K=200$ RFF while moving its hidden parameters
significantly ($\|\Delta\hat\phi\|_2\approx 0.95$). Both the hypentropy
Kirchhoff hybrid and the unlocked zero-tape Wave-PMP realization remain stable
at every depth $L\in\{1,2,4,8\}$, with unlocked Wave-PMP achieving the lowest
tracking error ($0.104$ mean / $0.256$ max at $L=8$).

%% file: appendix/J_covariant_ccm_geometry.tex
\section{Covariant Riemannian Wave-PMP, Control Contraction Metrics, and Geometric Obstructions}
\label{app:covariant_ccm_geometry}

This appendix develops the intrinsic differential-geometric and contraction-theoretic structure of Wave-PMP by connecting the moving-frame scattering equations of Section~\ref{sec:variational_worldsheet} to Control Contraction Metrics (CCMs) \citep{manchester2017control}. Four results follow directly from this synthesis: (i)~the two off-diagonal bulk-scattering blocks of \eqref{eq:bulk_port_waves_plus}--\eqref{eq:bulk_port_waves_minus} are algebraically identical to the Control and Filter Riccati operators, whose convex relaxations in dual and primal metrics are the CCM and observer-contraction LMIs; (ii)~lifting Wave-PMP covariantly to a curved state-dependent Riemannian metric $M(x)$ replaces the coordinate-dependent Hessian $\lambda^\top\nabla_x^2 f$ in Theorem~\ref{thm:nonlinear_wave_energy_balance} by the Lie derivative of the Levi-Civita connection $\mathcal{L}_f\nabla$ (determined solely by $\nabla(\mathcal{L}_f M)$) plus residual-coupled Riemann curvature; (iii)~pairing bulk contraction with a distributed control port converts the Klein--Gordon mass barrier into a horizon-independent spectral gap; and (iv)~persistent Wave-PMP solves the CCM geodesic boundary-value problem concurrently with plant evolution, while submanifold contraction \citep{manchester2017control} characterizes convergence to nonlinear equilibrium and gauge manifolds.

\subsection{Bulk reflection, Riccati operators, and convex CCM synthesis}
\label{app:ccm_riccati_scattering}

Consider a smooth control-affine system on $\mathcal{X}=\RR^n$,
\begin{equation}
\label{eq:ccm_control_affine_app}
\partial_t x = f(x,t) + B(x,t)\,u = f(x,t) + \sum_{i=1}^m b_i(x,t)\,u_i,
\end{equation}
with running cost $L(x,u,t)=\ell(x,t)+\tfrac{1}{2}u^\top R(x,t)u$ ($R\succ 0$). Eliminating $u=-R^{-1}B^\top\lambda$ yields the reduced Hamiltonian
\begin{equation}
\label{eq:ccm_reduced_ham_app}
H(x,\lambda,t) = \ell(x,t) + \lambda^\top f(x,t) - \tfrac{1}{2}\lambda^\top S(x,t)\lambda,
\qquad
S(x,t) := -H_{\lambda\lambda} = B(x,t) R(x,t)^{-1} B(x,t)^\top \succeq 0 .
\end{equation}
Equip $\mathcal{X}$ with a smoothly varying Riemannian metric $M(x,t)=\Theta(x,t)^\top\Theta(x,t)\succ 0$ and its inverse (dual metric) $W(x,t)=M(x,t)^{-1}=\Theta(x,t)^{-1}\Theta(x,t)^{-\top}$, playing the role of $G$ and $G^{-1}$ in Section~\ref{sec:variational_worldsheet}. Along solutions of \eqref{eq:ccm_control_affine_app}, write $A(x,u,t)=\frac{\partial f}{\partial x}+\sum_{i=1}^m u_i\frac{\partial b_i}{\partial x}$ and denote the Lie derivatives along a vector field $v$ by
\begin{equation}
\label{eq:lie_metric_dual_app}
\mathcal{L}_v M := \partial_v M + \left(\frac{\partial v}{\partial x}\right)^\top M + M\frac{\partial v}{\partial x},
\qquad
\mathcal{L}_v W := \partial_v W - \frac{\partial v}{\partial x} W - W\left(\frac{\partial v}{\partial x}\right)^\top
= -W(\mathcal{L}_v M)W,
\end{equation}
with $\dot M=\partial_t M+\partial_{f+Bu}M$ and $\dot W=\partial_t W+\partial_{f+Bu}W$.

\begin{proposition}[Scattering--Riccati identity and convex CCM relaxation]
\label{prop:ccm_riccati_scattering_app}
In the moving-frame linearized wave equations \eqref{eq:bulk_port_waves_plus}--\eqref{eq:bulk_port_waves_minus} with $G=M=\Theta^\top\Theta$, $K=\Theta^{-\top}Q\Theta^{-1}$, $L=\Theta S\Theta^\top$, and $2\widehat A_{\mathrm{sym}}=\Theta^{-\top}(\dot M+A^\top M+MA)\Theta^{-1}=\Theta(-\dot W+AW+WA^\top)\Theta^\top$:
\begin{enumerate}
\item \textbf{Control-independent bulk reflection $\Longleftrightarrow$ Killing control fields (CCM Condition~C2).}
The total bulk-reflection tensor splits as
\[
\dot M + A^\top M + M A = \bigl(\partial_t M + \mathcal{L}_f M\bigr) + \sum_{i=1}^m u_i\,\mathcal{L}_{b_i} M .
\]
Hence the bulk scattering barrier $\widehat A_{\mathrm{sym}}$ is independent of the control input $u$ if and only if each column $b_i$ is a Killing vector field for $M$ ($\mathcal{L}_{b_i}M=0$, equivalently $\mathcal{L}_{b_i}W=0$, Condition~C2 of \citet{manchester2017control}).
\item \textbf{Directional reflection-freeness $\Longleftrightarrow$ Control and Filter Riccati equations.}
The forward-to-backward reflection block in \eqref{eq:bulk_port_waves_minus} and the backward-to-forward reflection block in \eqref{eq:bulk_port_waves_plus} satisfy the exact congruence identities
\begin{align}
\label{eq:ctrl_riccati_refl_app}
\Theta^\top\bigl(K - L + 2\widehat A_{\mathrm{sym}}\bigr)\Theta
&= \dot M + A^\top M + M A - M B R^{-1} B^\top M + Q,\\
\label{eq:filt_riccati_refl_app}
\Theta^{-1}\bigl(K - L - 2\widehat A_{\mathrm{sym}}\bigr)\Theta^{-\top}
&= \dot W - A W - W A^\top + W Q W - S .
\end{align}
Setting \eqref{eq:ctrl_riccati_refl_app} to zero is the Control Differential Riccati Equation; setting \eqref{eq:filt_riccati_refl_app} to zero is the Filter Differential Riccati Equation (equivalently $\dot I_F+A^\top I_F+I_F A - Q + I_F S I_F=0$ for the information metric $I_F=W^{-1}$). Simultaneous zero reflection in both directions holds if and only if $\widehat A_{\mathrm{sym}}=0$ and $K=L$.
\item \textbf{Convex dual-metric relaxation (CCM) and primal-metric observer dual.}
Transforming \eqref{eq:ctrl_riccati_refl_app} by $\Theta^{-1}(\cdot)\Theta^{-\top}$, replacing the state-cost term $W Q W$ by the intrinsic contraction rate $2\lambda W$ and $B R^{-1} B^\top$ by $\rho(x,t) B B^\top$, and relaxing equality to negative definiteness yields the jointly convex dual-CCM inequality \citep[Eq.~(11)]{manchester2017control}:
\begin{equation}
\label{eq:dual_ccm_lmi_app}
-\frac{\partial W}{\partial t} - \partial_f W + \frac{\partial f}{\partial x}W + W\left(\frac{\partial f}{\partial x}\right)^\top - \rho(x,t)\,B B^\top + 2\lambda W \;\prec\; 0,
\qquad
\mathcal{L}_{b_i} W = 0 .
\end{equation}
Equivalently, $2\widehat A_{\mathrm{sym}} - L_\rho \prec -2\lambda I$ ($L_\rho=\rho\,\Theta B B^\top\Theta^\top$): every bulk-reflection direction not naturally contracting at rate $\lambda$ is strictly dominated by the dissipative control port. Dually, for an observation port $Q=\sigma(x,t)C^\top C$, relaxing the information form of \eqref{eq:filt_riccati_refl_app} yields the observer-contraction LMI $\partial_t M+\partial_f M+f_x^\top M+M f_x-\sigma C^\top C+2\lambda M\prec 0$, which is jointly convex in the \emph{primal} metric $(M,\sigma)$.
\end{enumerate}
\end{proposition}

\begin{proof}
Part~(1) follows immediately from expanding $\dot M=\partial_t M+\partial_f M+\sum_i u_i\partial_{b_i}M$ and $A=f_x+\sum_i u_i\partial_x b_i$. For Part~(2), recall $M=\Theta^\top\Theta$, $W=\Theta^{-1}\Theta^{-\top}$, $K=\Theta^{-\top}Q\Theta^{-1}$, and $L=\Theta S\Theta^\top=\Theta^{-\top}(M S M)\Theta^{-1}$. Conjugating $K-L+2\widehat A_{\mathrm{sym}}$ by $\Theta^\top(\cdot)\Theta$ gives $Q-MSM+(\dot M+A^\top M+MA)$, which is \eqref{eq:ctrl_riccati_refl_app}. Conjugating $K-L-2\widehat A_{\mathrm{sym}}$ by $\Theta^{-1}(\cdot)\Theta^{-\top}$ and using $\Theta^{-1}(2\widehat A_{\mathrm{sym}})\Theta^{-\top}=W(\dot M+A^\top M+MA)W=-\dot W+AW+WA^\top$ gives $WQW-S-(-\dot W+AW+WA^\top)$, which is \eqref{eq:filt_riccati_refl_app}. Multiplying \eqref{eq:filt_riccati_refl_app} on both sides by $I_F=W^{-1}$ and using $I_F\dot W I_F=-\dot I_F$ yields $-\dot I_F-I_F A-A^\top I_F+Q-I_F S I_F=0$. Part~(3) follows by substituting $W Q W\mapsto 2\lambda W$ and $B R^{-1} B^\top\mapsto\rho B B^\top$ into the dual form of \eqref{eq:ctrl_riccati_refl_app}, matching Finsler's lemma in \citet[Proposition~2 \& Eq.~(11)]{manchester2017control}.
\end{proof}

\paragraph{Uniting local Riccati impedance matching with global nonlinear contraction.}
Because \eqref{eq:dual_ccm_lmi_app} is convex in $(W,\rho)$, one can impose the linear point constraints $W(x_{\mathrm{eq}})=P_{\mathrm{LQR}}^{-1}$ and $\rho(x_{\mathrm{eq}})=2r^{-1}$ at an operating equilibrium $x_{\mathrm{eq}}$ (where $P_{\mathrm{LQR}}$ solves the algebraic Riccati equation with $R=rI$, so the LQR feedback $K_{\mathrm{LQR}}=-r^{-1}B^\top P_{\mathrm{LQR}}=-\frac{1}{2}\rho B^\top W^{-1}$ matches the CCM differential gain; \citealp[Section~VI]{manchester2017control}). The resulting metric $M(x)=W(x)^{-1}$ is locally reflection-free in the sense of \eqref{eq:ctrl_riccati_refl_app} near $x_{\mathrm{eq}}$ while retaining universal exponential stabilizability across strongly nonlinear regions where fixed nominal Riccati conductances fail (Appendix~\ref{sec:exp_physics_control}).

\subsection{Intrinsic covariant Wave-PMP on a curved Riemannian manifold}
\label{app:covariant_wave_pmp}

Theorem~\ref{thm:nonlinear_wave_energy_balance} was established for a state-independent metric $G(t)$ ($\partial_\tau G=0$), leaving the non-tensorial Euclidean second derivative $\lambda^\top\nabla_x^2 f$ inside $H_{xx}$. Allowing a curved state-dependent Riemannian metric $M(x)=\Theta(x)^\top\Theta(x)\succ 0$ on $\mathcal{X}$ with inverse $W(x)=M(x)^{-1}$ and Levi-Civita connection $\nabla$ (Christoffel symbols $\Gamma^k_{ij}$, Riemann curvature tensor $R(U,V)Z=D_U D_V Z-D_V D_U Z-D_{[U,V]}Z$) reveals that the operator $\mathbb K(x)=\bigl(\begin{smallmatrix}0&W(x)\\ M(x)&0\end{smallmatrix}\bigr)$ is the musical isomorphism $(\sharp,\flat)$ between $T\mathcal{X}$ and $T^*\mathcal{X}$ \citep{manchester2017control}, while $\mathbb M(x)=\operatorname{diag}(M(x),W(x))$ is the Sasaki metric on $T^*\mathcal{X}$.

Along any worldsheet field $(x(t,\tau),\lambda(t,\tau))\in T^*\mathcal{X}$, define the covariant derivatives of the covector $\lambda$,
\begin{equation}
\label{eq:cov_deriv_lambda_app}
(D_t\lambda)_i := \partial_t\lambda_i - \Gamma^k_{ij}(\partial_t x^j)\lambda_k,
\qquad
(D_\tau\lambda)_i := \partial_\tau\lambda_i - \Gamma^k_{ij}(\partial_\tau x^j)\lambda_k,
\end{equation}
and split the differential of $H:T^*\mathcal{X}\to\RR$ via the Levi-Civita horizontal/vertical decomposition into its fiber gradient $(d_V H)^j:=\partial H/\partial\lambda_j\in T_x\mathcal{X}$ and its horizontal covariant gradient $(d_H H)_i:=\partial_i H+\Gamma^k_{ij}\lambda_k(d_V H)^j\in T_x^*\mathcal{X}$. For the control-affine reduced Hamiltonian $H(x,\lambda)=\ell(x)+\langle\lambda,f(x)\rangle-\tfrac{1}{2}\langle\lambda,S(x)\lambda\rangle$, metric compatibility ($\nabla M=0$) gives the purely tensorial expressions
\begin{equation}
\label{eq:cov_ham_grads_app}
d_V H = f(x) - S(x)\lambda,
\qquad
(d_H H)_i = (\nabla\ell)_i + \langle\lambda,\nabla_i f\rangle - \tfrac{1}{2}\langle\lambda,(\nabla_i S)\lambda\rangle .
\end{equation}
Define the tensorial state and covariant adjoint residuals
\begin{equation}
\label{eq:cov_residuals_app}
r_x := \partial_t x - d_V H \in T_x\mathcal{X},
\qquad
r_\lambda^{\mathrm{cov}} := D_t\lambda + d_H H = r_\lambda - \Gamma^k_{ij}r_x^j\lambda_k \in T_x^*\mathcal{X} .
\end{equation}
Because $r_\lambda^{\mathrm{cov}}$ differs from the coordinate residual $r_\lambda=\partial_t\lambda+\partial_x H$ by a term linear in $r_x$, the simultaneous zero set $\{r_x=0,\,r_\lambda^{\mathrm{cov}}=0\}$ coincides identically with Pontryagin's maximum principle \eqref{eq:continuous_dynamics}--\eqref{eq:costate_dynamics}. We therefore lift \eqref{eq:para_floer} to the \emph{intrinsic covariant Wave-PMP flow} on $T^*\mathcal{X}$:
\begin{equation}
\label{eq:covariant_para_floer_app}
\boxed{\;\partial_\tau x = -c\,W(x)\,r_\lambda^{\mathrm{cov}}, \qquad D_\tau\lambda = -c\,M(x)\,r_x .\;}
\end{equation}
By construction, every term in \eqref{eq:covariant_para_floer_app} is a genuine tensor, so the worldsheet PDE is invariant under arbitrary smooth coordinate diffeomorphisms $\xi=\phi(x)$ (with $M_\xi=(\partial\phi^{-1}/\partial\xi)^\top M(\partial\phi^{-1}/\partial\xi)$ and $W_\xi=(\partial\phi/\partial x)W(\partial\phi/\partial x)^\top$, mirroring Theorem~2 of \citet{manchester2017control}).

\begin{theorem}[Intrinsic covariant wave-energy balance]
\label{thm:covariant_wave_energy_balance_app}
Let $(\mathcal{X},M(x))$ be a smooth Riemannian manifold with Levi-Civita connection $\nabla$ and curvature tensor $R$, and let $(x(t,\tau;s),\lambda(t,\tau;s))$ be a smooth one-parameter family of solutions of \eqref{eq:covariant_para_floer_app} for $H(x,\lambda)=\ell(x)+\langle\lambda,f(x)\rangle-\tfrac{1}{2}\langle\lambda,S(x)\lambda\rangle$. Let $\delta x:=\partial_s x\in T_x\mathcal{X}$ and $\delta^{\mathrm{cov}}\lambda:=D_s\lambda\in T_x^*\mathcal{X}$ be the covariant virtual displacements, with Sasaki wave energy
\begin{equation}
\label{eq:sasaki_wave_energy_app}
\cE_{\mathrm{wave}}(\tau)
= \frac{1}{2}\int_0^T\Bigl(\|\delta x\|_{M(x)}^2 + \|\delta^{\mathrm{cov}}\lambda\|_{W(x)}^2\Bigr)\,\dd t
= \frac{1}{2}\int_0^T\bigl(\|\delta w_+\|^2 + \|\delta w_-\|^2\bigr)\,\dd t .
\end{equation}
Then $\cE_{\mathrm{wave}}(\tau)$ obeys the exact coordinate-invariant balance
\begin{equation}
\label{eq:covariant_wave_energy_balance_app}
\begin{aligned}
\frac{\dd\cE_{\mathrm{wave}}}{\dd\tau}
&= \frac{c}{2}\Bigl(\|\delta w_+(0)\|^2-\|\delta w_-(0)\|^2 + \|\delta w_-(T)\|^2-\|\delta w_+(T)\|^2\Bigr)\\
&\quad - c\int_0^T\Bigl(\operatorname{Hess}_M\ell(\delta x,\delta x) + \langle\delta^{\mathrm{cov}}\lambda,\,S(x)\,\delta^{\mathrm{cov}}\lambda\rangle\Bigr)\dd t\\
&\quad - c\int_0^T\Bigl(\bigl\langle\lambda,\,(\mathcal{L}_f\nabla)(\delta x,\delta x)\bigr\rangle
  - \tfrac{1}{2}\bigl\langle\lambda,\,(\nabla^2_{\delta x,\delta x}S)\lambda\bigr\rangle
  - \bigl\langle\lambda,\,R(S\lambda,\delta x)\delta x\bigr\rangle\Bigr)\dd t\\
&\quad - c\int_0^T\bigl\langle\lambda,\,R(r_x,\delta x)\delta x\bigr\rangle\,\dd t
  - \int_0^T\bigl\langle\lambda,\,R(\partial_\tau x,\delta x)(W\delta^{\mathrm{cov}}\lambda)\bigr\rangle\,\dd t,
\end{aligned}
\end{equation}
where $(\mathcal{L}_f\nabla)(X,Y):=\nabla_X\nabla_Y f-\nabla_{\nabla_X Y}f+R(f,X)Y$ is the Lie derivative of the Levi-Civita connection along $f$, determined solely by the covariant gradient of the bulk-reflection/contraction tensor $h:=\mathcal{L}_f M$ via the Koszul--Lie identity
\begin{equation}
\label{eq:koszul_lie_identity_app}
\bigl\langle\lambda,\,(\mathcal{L}_f\nabla)(\delta x,\delta x)\bigr\rangle
= (\nabla_{\delta x} h)(\delta x,\,W\lambda) - \tfrac{1}{2}(\nabla_{W\lambda} h)(\delta x,\delta x).
\end{equation}
\end{theorem}

\begin{proof}
Since $\nabla M=0$ and $\nabla W=0$, differentiating \eqref{eq:sasaki_wave_energy_app} in $\tau$ gives
\[
\frac{\dd\cE_{\mathrm{wave}}}{\dd\tau}
= \int_0^T\Bigl(\langle M\delta x,\,D_\tau\delta x\rangle + \langle W\delta^{\mathrm{cov}}\lambda,\,D_\tau\delta^{\mathrm{cov}}\lambda\rangle\Bigr)\,\dd t .
\]
Because $\nabla$ is torsion-free, $D_\tau\delta x = D_\tau(\partial_s x) = D_s(\partial_\tau x) = -c\,W\,D_s r_\lambda^{\mathrm{cov}}$, so $\langle M\delta x,D_\tau\delta x\rangle=-c\langle\delta x,D_s r_\lambda^{\mathrm{cov}}\rangle$. By the Ricci identity for a covector, $D_\tau D_s\lambda - D_s D_\tau\lambda = -R(\partial_\tau x,\delta x)^*\lambda$, where $\langle R(U,V)^*\lambda,Z\rangle=\langle\lambda,R(U,V)Z\rangle$; substituting $D_\tau\lambda=-c\,M r_x$ gives
\[
\langle W\delta^{\mathrm{cov}}\lambda,\,D_\tau\delta^{\mathrm{cov}}\lambda\rangle
= -c\langle\delta^{\mathrm{cov}}\lambda,\,D_s r_x\rangle
  - \langle\lambda,\,R(\partial_\tau x,\delta x)(W\delta^{\mathrm{cov}}\lambda)\rangle .
\]
Next we expand $D_s r_x$ and $D_s r_\lambda^{\mathrm{cov}}$. From \eqref{eq:cov_ham_grads_app}--\eqref{eq:cov_residuals_app} and torsion-freeness ($D_s\partial_t x=D_t\delta x$),
\[
D_s r_x = D_t\delta x - (\nabla_{\delta x} f - (\nabla_{\delta x}S)\lambda) + S\,\delta^{\mathrm{cov}}\lambda .
\]
For $r_\lambda^{\mathrm{cov}}=D_t\lambda+d_H H$, the Ricci identity gives $D_s D_t\lambda = D_t\delta^{\mathrm{cov}}\lambda - R(\delta x,\partial_t x)^*\lambda$, while differentiating \eqref{eq:cov_ham_grads_app} yields
\[
D_s(d_H H) = (\operatorname{Hess}_M\ell)\,\delta x
+ \langle\lambda,\nabla_{\delta x}\nabla f\rangle
- \tfrac{1}{2}\langle\lambda,(\nabla^2_{\delta x}S)\lambda\rangle
+ \bigl(\nabla f - (\nabla S)\lambda\bigr)^*\delta^{\mathrm{cov}}\lambda .
\]
Pairing $-c\langle\delta x,D_s r_\lambda^{\mathrm{cov}}\rangle - c\langle\delta^{\mathrm{cov}}\lambda,D_s r_x\rangle$, the transport terms combine into $-c\,\partial_t\langle\delta^{\mathrm{cov}}\lambda,\delta x\rangle$ (giving the exact boundary wave-power flux upon integration), and the mixed horizontal--vertical drift blocks $\pm c\langle\delta^{\mathrm{cov}}\lambda,(\nabla_{\delta x}f-(\nabla_{\delta x}S)\lambda)\rangle$ cancel identically. Using skew-symmetry $-R(\delta x,\partial_t x)=R(\partial_t x,\delta x)$ and substituting $\partial_t x = f - S\lambda + r_x$ yields
\[
-\langle\lambda,\nabla^2_{\delta x,\delta x}f\rangle - \langle\lambda,R(\partial_t x,\delta x)\delta x\rangle
= -\langle\lambda,(\mathcal{L}_f\nabla)(\delta x,\delta x)\rangle
  + \langle\lambda,R(S\lambda,\delta x)\delta x\rangle
  - \langle\lambda,R(r_x,\delta x)\delta x\rangle,
\]
establishing \eqref{eq:covariant_wave_energy_balance_app}. Finally, differentiating the Koszul formula $2M(\nabla_X Y,Z)=X(M(Y,Z))+Y(M(X,Z))-Z(M(X,Y))+\cdots$ along the flow of $f$ yields $2M((\mathcal{L}_f\nabla)(X,Y),Z)=(\nabla_X h)(Y,Z)+(\nabla_Y h)(X,Z)-(\nabla_Z h)(X,Y)$ with $h=\mathcal{L}_f M$; setting $X=Y=\delta x$ and $Z=W\lambda$ gives \eqref{eq:koszul_lie_identity_app}.
\end{proof}

\begin{corollary}[Nonlinear losslessness of Killing and homothetic flows, and curvature coupling]
\label{cor:killing_homothetic_lossless_app}
From \eqref{eq:covariant_wave_energy_balance_app}--\eqref{eq:koszul_lie_identity_app}:
\begin{enumerate}
\item \textbf{Killing and homothetically contracting flows have zero $\mathcal{L}_f\nabla$ even on curved manifolds.}
Whenever the contraction tensor is parallel ($\nabla(\mathcal{L}_f M)=0$)---in particular for any nonlinear Killing flow ($\mathcal{L}_f M=0$) or any constant-rate homothetic contraction ($\mathcal{L}_f M=-2\lambda M$)---we have $\mathcal{L}_f\nabla\equiv 0$ identically, even when $R\neq 0$ and the Euclidean Hessian $\lambda^\top\nabla_x^2 f$ is nonzero and indefinite.
\item \textbf{CCM Condition~C2 kills both first- and second-order control scattering.}
For \eqref{eq:ccm_control_affine_app}, $\mathcal{L}_{b_i}M=0$ implies $\mathcal{L}_{b_i}\nabla=0$, simultaneously eliminating both the first-order control-induced reflection $\sum_i u_i\mathcal{L}_{b_i}M$ and the second-order control-induced Hessian $\sum_i u_i\langle\lambda,(\mathcal{L}_{b_i}\nabla)(\delta x,\delta x)\rangle$.
\item \textbf{Flat pullback metrics and residual-coupled curvature.}
For any flat pullback metric ($R\equiv 0$, including all feedback-linearizable systems of \citealp[Corollary~1]{manchester2017control}), all Riemann curvature terms in \eqref{eq:covariant_wave_energy_balance_app} vanish identically. For curved metrics ($R\neq 0$, non-integrable differential frames $\Theta(x)$), the last two terms of \eqref{eq:covariant_wave_energy_balance_app} vanish on every PMP equilibrium ($r_x=0$, $\partial_\tau x=0$).
\end{enumerate}
\end{corollary}

\subsection{Horizon-independent spectral gap via contraction and port damping}
\label{app:contraction_spectral_gap}

Section~\ref{sec:variational_worldsheet} showed that in an unactuated bulk ($Q=S=0$), metric mismatch $\widehat A_{\mathrm{sym}}\neq 0$ creates a Klein--Gordon mass barrier that traps low-frequency modes and slows boundary dissipation exponentially in horizon length $T$. When a distributed control port $S\succ 0$ is present, however, bulk contraction reverses this behavior and prevents the $O(1/T^2)$ horizon slowdown of massless ($Q=0,\widehat A_{\mathrm{sym}}=0$) transport.

\begin{proposition}[Horizon-independent contraction rate under bulk contraction + port damping]
\label{prop:contraction_spectral_gap_app}
Consider the constant-coefficient linearized worldsheet \eqref{eq:bulk_port_waves_plus}--\eqref{eq:bulk_port_waves_minus} in a moving frame with $K=0$, isotropic control port $L=s_0 I$ ($s_0>0$), and uniform symmetric drift $\widehat A_{\mathrm{sym}}=\alpha I$. Every bulk Fourier mode $e^{s\tau+ikt}$ obeys the dispersion relation
\begin{equation}
\label{eq:kg_port_dispersion_app}
s^2 + c\,s_0\,s + c^2(k^2 + \alpha^2) = 0
\quad\Longrightarrow\quad
s_\pm(k) = -\frac{c\,s_0}{2} \pm \sqrt{\left(\frac{c\,s_0}{2}\right)^2 - c^2(k^2+\alpha^2)} .
\end{equation}
Consequently:
\begin{enumerate}
\item In an isometric bulk ($\alpha=0$), the slowest finite-horizon mode ($k_{\min}\sim\pi/(2T)$) decays as $\operatorname{Re}(s_+)\sim -c\pi^2/(4 s_0 T^2)\to 0$ as $T\to\infty$.
\item Whenever $|\alpha|>0$---in particular for a contracting plant ($\alpha=-\lambda<0$) or a plant pre-stabilized by an affine CCM feedback $u=k_{\mathrm{CCM}}(x,t)+v$ (preserving feasible trajectories by feedback invariance, \citealp[Theorem~2]{manchester2017control})---the bulk decay rate remains bounded away from zero independently of horizon length $T$:
\begin{equation}
\label{eq:horizon_independent_gap_app}
\sup_{k\in\RR}\operatorname{Re}(s_\pm(k))
= -\frac{c}{2}\Bigl(s_0 - \operatorname{Re}\sqrt{s_0^2 - 4\alpha^2}\Bigr) < 0,
\end{equation}
attaining the optimal port-damping ceiling $-c\,s_0/2$ whenever $2|\alpha|\ge s_0$.
\end{enumerate}
\end{proposition}

\begin{proof}
In moving-frame state--costate variables $(\delta p,\delta q)$ with $K=0$, $L=s_0 I$, $\widehat A_{\mathrm{sym}}=\alpha I$, and $\widehat A_{\mathrm{skew}}=0$ (by gauge rotation), the reduced Hamiltonian has $H_{\lambda\lambda}=-S$, giving linearized residuals $r_x=\partial_t\delta p-\alpha\delta p+s_0\delta q$ and $r_\lambda=\partial_t\delta q+\alpha\delta q$. The first-order worldsheet flow \eqref{eq:para_floer} therefore reads $\partial_\tau\delta p=-c(\partial_t+\alpha)\delta q$ and $\partial_\tau\delta q=-c(\partial_t-\alpha)\delta p-c\,s_0\delta q$. Substituting $e^{s\tau+ikt}$ yields $s\,\delta p=-c(ik+\alpha)\delta q$ and $(s+c\,s_0)\delta q=-c(ik-\alpha)\delta p$. Eliminating $\delta p$ gives $s(s+c\,s_0)=c^2(ik+\alpha)(ik-\alpha)=-c^2(k^2+\alpha^2)$, which is \eqref{eq:kg_port_dispersion_app}. Taking $\sup_{k\in\RR}$ (attained at $k=0$) immediately yields both claims.
\end{proof}

\subsection{Persistent geodesic waves and submanifold contraction}
\label{app:geodesic_waves_submanifold}

\paragraph{Persistent Wave-PMP for CCM geodesics.}
Implementing a CCM feedback controller \citep{manchester2017control} requires solving, at each physical time $t$, the minimal-geodesic two-point boundary-value problem
\begin{equation}
\label{eq:ccm_geodesic_tpbvp_app}
\min_{c(\cdot)}\; \frac{1}{2}\int_0^1 c_s(s)^\top M(c(s),t)\,c_s(s)\,\dd s,
\qquad
c(0) = x^\star(t), \quad c(1) = x(t) .
\end{equation}
Viewing $s\in[0,1]$ as the worldsheet coordinate with trivial integrator dynamics $\partial_s c = v$, eliminating $v=-W(c,t)\lambda$ produces the reduced geodesic Hamiltonian $H_{\mathrm{geo}}(c,\lambda)=-\tfrac{1}{2}\lambda^\top W(c,t)\lambda$. Three properties follow immediately from Theorem~\ref{thm:covariant_wave_energy_balance_app} and \citet[Lemma~3]{manchester2017control}:
\begin{enumerate}
\item \emph{Automatic unit normalized port damping.} Because $S(c)=-H_{\lambda\lambda}=W(c,t)=M(c,t)^{-1}$ is parallel ($\nabla W=0$), the normalized port compliance is identically $L=\Theta W\Theta^\top=I_n$ at every $s\in[0,1]$, while $f=0$, $\ell=0$, and $\nabla^2 S=0$ drop out of \eqref{eq:covariant_wave_energy_balance_app}.
\item \emph{Curvature sign matches the cut locus.} At equilibrium ($r_c=0,\partial_\tau c=0$, so $W\lambda=-c_s$), the only remaining bulk term in \eqref{eq:covariant_wave_energy_balance_app} besides $-c\int_0^1\|\delta^{\mathrm{cov}}\lambda\|_W^2\,\dd s$ is the sectional curvature $+c\int_0^1\langle\lambda,R(W\lambda,\delta c)\delta c\rangle\,\dd s=+c\int_0^1\mathrm{Sec}(c_s,\delta c)\bigl(\|c_s\|_M^2\|\delta c\|_M^2-\langle c_s,\delta c\rangle_M^2\bigr)\,\dd s$. On non-positively curved (Hadamard) manifolds ($\mathrm{Sec}\le 0$, empty cut locus), this term is non-positive and strictly reinforces geodesic wave contraction; on positively curved manifolds ($\mathrm{Sec}>0$), it can inject energy precisely where conjugate points and the cut locus $\mathfrak{C}(x^\star,t)$ of \citet[Section~III-3]{manchester2017control} form.
\item \emph{Concurrent persistent tracking without waiting for NLP convergence.} By \citet[Lemma~3]{manchester2017control}, the CCM path-integrated control law contracts the length $L(c(t),t)\le e^{-\lambda(t-t_i)}L(c(t_i),t_i)$ of \emph{any} smooth connecting curve $c(t,\cdot)$, not only an exact minimal geodesic. Evolving the geodesic curve and covector as persistent wave variables along $s$ (with clamped endpoints $c(0)=x^\star(t), c(1)=x(t)$) concurrently with plant evolution monotonically tightens $E(c(t),t)$ at every local tick without halting to re-solve a global geodesic NLP.
\end{enumerate}

\paragraph{Submanifold contraction for gauge and equilibrium manifolds.}
In deep networks and overparameterized scattering circuits, solutions rarely form isolated points: they form nonlinear submanifolds $Z=\{y:z(y)=0\}$ such as the state--adjoint consistency manifold $\{r_x=0,r_\lambda=0\}$, gauge-equivalent SVD parameter orbits (Appendix~\ref{app:categorical_scattering}), or the normalization spheres of RMSNorm and LayerNorm (Section~\ref{sec:residual_branching}). Standard contraction fails along tangent directions $G(y)\in\ker(\partial z/\partial y)$. Section~V of \citet{manchester2017control} (Theorem~3 and Corollary~2) resolves this via \emph{transverse (submanifold) contraction}: introducing virtual controls along $\ker(\partial z/\partial y)$ yields the convex dual-metric condition
\[
\frac{\partial z}{\partial y}\left(-\frac{\partial W}{\partial\tau}-\partial_F W+\frac{\partial F}{\partial y}W+W\left(\frac{\partial F}{\partial y}\right)^\top+2\lambda W\right)\frac{\partial z}{\partial y}^\top \;\prec\; 0,
\]
certifying exponential convergence to $Z$ without requiring contraction along invariant symmetry directions.

\subsection{Exactly solvable nonlinear examples}
\label{app:ccm_solvable_examples}

We conclude with three closed-form nonlinear examples illustrating Propositions~\ref{prop:ccm_riccati_scattering_app} and~\ref{prop:contraction_spectral_gap_app}, Theorem~\ref{thm:covariant_wave_energy_balance_app}, and the distinction between flat pullback metrics, genuinely curved metrics ($R\neq 0$), and non-feedback-linearizable plants.

\paragraph{Example 1 (Exactly solvable 1D nonlinear control system).}
Consider the globally defined nonlinear control-affine system on $\mathcal{X}=\RR$ with $\lambda>0$, $q\ge 0$, $r>0$, and $p_T\ge 0$:
\begin{equation}
\label{eq:ex_tanh_sech_app}
\partial_t x = -\lambda\tanh(x) + \operatorname{sech}(x)\,u,
\qquad
J[x,u] = \frac{p_T}{2}\sinh^2\!\bigl(x(T)\bigr) + \frac{1}{2}\int_0^T\Bigl(q\sinh^2(x) + r\,u^2\Bigr)\dd t .
\end{equation}
In the flat Euclidean metric $G=1$, the uncontrolled drift has $f'(x)=-\lambda\operatorname{sech}^2(x)\to 0$ as $|x|\to\infty$ (not globally contracting), and the reduced Hamiltonian
$H(x,\lambda_{\mathrm{co}})=\tfrac{q}{2}\sinh^2(x)-\lambda\lambda_{\mathrm{co}}\tanh(x)-\tfrac{1}{2r}\operatorname{sech}^2(x)\lambda_{\mathrm{co}}^2$
has an indefinite, sign-changing Euclidean second derivative $H_{xx}$ (for instance, at $x=1,\lambda=0.8,q=0.5,r=1$, $H_{xx}=+1.66$ at $\lambda_{\mathrm{co}}=+2$ but $H_{xx}=-8.45$ at $\lambda_{\mathrm{co}}=-5$).
Imposing CCM Condition~C2 ($\mathcal{L}_b M = b M' + 2 b' M = 0$ for $b(x)=\operatorname{sech}(x)$) uniquely fixes the state-dependent metric up to constant scale:
\begin{equation}
\label{eq:ex_tanh_metric_app}
M(x) = \cosh^2(x), \qquad W(x) = \operatorname{sech}^2(x), \qquad \Gamma(x) = \frac{1}{2}W M' = \tanh(x) .
\end{equation}
Direct calculation verifies every covariant condition in closed form across all $x\in\RR$:
\begin{itemize}
\item \emph{Exact homothetic contraction (Condition~C1) and parallel ports:}
$\mathcal{L}_f M = f M' + 2 f' M = -2\lambda(\sinh^2 x + 1) = -2\lambda M(x)$, $\nabla_x f = f' + \Gamma f = -\lambda(\operatorname{sech}^2 x + \tanh^2 x) = -\lambda$, $\operatorname{Hess}_M\ell = q\cosh^2(x) = q\,M(x)$, and $S(x)=r^{-1}\operatorname{sech}^2(x)=r^{-1}W(x)$ is parallel ($\nabla S=0$). By Corollary~\ref{cor:killing_homothetic_lossless_app}, both $\mathcal{L}_b\nabla\equiv 0$ and $\mathcal{L}_f\nabla\equiv 0$ vanish identically, and $R\equiv 0$ in 1D.
\item \emph{Exact linearization of covariant Wave-PMP and closed-form solution:}
In the geodesic coordinates $z=\int_0^x\sqrt{M(s)}\,\dd s=\sinh(x)$ and normalized covector $p=\lambda_{\mathrm{co}}\operatorname{sech}(x)$ (so $D_t\lambda_{\mathrm{co}}=\cosh(x)\,\partial_t p$ and $D_\tau\lambda_{\mathrm{co}}=\cosh(x)\,\partial_\tau p$), the nonlinear covariant flow \eqref{eq:covariant_para_floer_app} collapses identically to the constant-coefficient linear wave system
\[
\partial_\tau z = -c\bigl(\partial_t p + q z - \lambda p\bigr),
\qquad
\partial_\tau p = -c\bigl(\partial_t z + \lambda z + r^{-1}p\bigr),
\]
with exact energy dissipation $\frac{\dd\cE_{\mathrm{wave}}}{\dd\tau} = \text{bdy} - c\int_0^T(q\,\delta z^2 + r^{-1}\delta p^2)\,\dd t \le 0$ for all $(x,\lambda_{\mathrm{co}})$. Writing $\mu=\sqrt{\lambda^2+q/r}$ and $\beta_T=(\lambda+p_T/r)/\mu$, the exact nonlinear PMP optimum is
\[
x^\star(t) = \operatorname{arsinh}\!\bigl(z^\star(t)\bigr),
\qquad
z^\star(t) = \frac{\sinh(x_0)\bigl[\cosh(\mu(T-t))+\beta_T\sinh(\mu(T-t))\bigr]}{\cosh(\mu T)+\beta_T\sinh(\mu T)},
\]
with exact Riccati-matched metric $M_\star(x)=P_\star\cosh^2(x)$, $P_\star=r(\mu-\lambda)$ (at which $K-L+2\widehat A_{\mathrm{sym}}=0$, yielding the exact optimal feedback $u^\star(x)=-(\mu-\lambda)\sinh(x)$ and closed-loop flow $x^\star(t)=\operatorname{arsinh}(e^{-\mu t}\sinh x_0)$ as $T\to\infty$).
\item \emph{Exact geodesics and horizon-independent spectral gap:}
Minimal geodesics of $M(x)=\cosh^2(x)$ are $\gamma(s)=\operatorname{arsinh}\bigl((1-s)\sinh x^\star+s\sinh x\bigr)$ with energy $E(x,x^\star)=(\sinh x-\sinh x^\star)^2$. Even with $q=0$, Proposition~\ref{prop:contraction_spectral_gap_app} gives the exact dispersion relation $s^2+cr^{-1}s+c^2(k^2+\lambda^2)=0$, ensuring a horizon-independent decay rate as $T\to\infty$.
\end{itemize}

\paragraph{Example 2 (Curved 2D manifolds $S^2$ and $\mathbb{H}^2$: non-integrable frames and curvature sign).}
On $\mathcal{X}\subset\RR^2$ with constant-curvature conformal metric ($\kappa=+1$ on $S^2\setminus\{\infty\}$, $\kappa=-1$ on the Poincar\'e disk $\mathbb{H}^2=\{\|x\|<1\}$)
\begin{equation}
\label{eq:ex_s2_h2_metric_app}
M_\kappa(x) = \frac{4}{(1+\kappa\|x\|^2)^2}\,I_2,
\qquad
W_\kappa(x) = \frac{(1+\kappa\|x\|^2)^2}{4}\,I_2,
\end{equation}
the frame $\Theta_\kappa(x)=\frac{2}{1+\kappa\|x\|^2}I_2$ is \emph{non-integrable} ($\partial_2\Theta_{11}\neq\partial_1\Theta_{12}=0$, Gaussian curvature $\mathrm{Sec}=\kappa\neq 0$), so no global flat wave coordinates $z=\theta(x)$ exist. Nevertheless:
\begin{itemize}
\item The quadratic nonlinear vector field $f_\kappa(x)=\bigl(\tfrac{1}{2}(1+\kappa(x_1^2-x_2^2)),\,\kappa x_1 x_2\bigr)^\top$ satisfies $\mathcal{L}_{f_\kappa}M_\kappa\equiv 0$ identically (a rotation on $S^2$ for $\kappa=+1$ and a hyperbolic boost on $\mathbb{H}^2$ for $\kappa=-1$). Although its Euclidean Hessian $\lambda^\top\nabla_x^2 f_\kappa(\delta x,\delta x)=\kappa(\lambda_1(\delta x_1^2-\delta x_2^2)+2\lambda_2\delta x_1\delta x_2)$ is nonzero and indefinite, Corollary~\ref{cor:killing_homothetic_lossless_app} gives $\mathcal{L}_{f_\kappa}\nabla\equiv 0$ identically despite $R\neq 0$.
\item For geodesic Wave-PMP ($H_{\mathrm{geo}}=-\tfrac{1}{2}\lambda^\top W_\kappa(x)\lambda$), Theorem~\ref{thm:covariant_wave_energy_balance_app} at equilibrium ($r_c=0,\partial_\tau c=0$) collapses to
\[
\frac{\dd\cE_{\mathrm{wave}}}{\dd\tau}
= \text{bdy} - c\int_0^1\|\delta^{\mathrm{cov}}\lambda\|_{W_\kappa}^2\,\dd s
  + c\,\kappa\int_0^1\Bigl(\|c_s\|_{M_\kappa}^2\|\delta c\|_{M_\kappa}^2-\langle c_s,\delta c\rangle_{M_\kappa}^2\Bigr)\dd s,
\]
showing explicitly that negative curvature ($\mathbb{H}^2$, $\kappa=-1$, empty cut locus) is strictly dissipative while positive curvature ($S^2$, $\kappa=+1$) injects energy precisely where conjugate points and the antipodal cut locus arise.
\end{itemize}

\paragraph{Example 3 (Non-feedback-linearizable 2D system of \citealp[Example~2]{manchester2017control}).}
For the nonlinear plant $\partial_t x_1=-x_1-x_1^3+x_2^2$, $\partial_t x_2=u$ ($B=[0,1]^\top$), the controllability matrix $[B,\mathrm{ad}_f B]=\bigl(\begin{smallmatrix}0&2x_2\\ 1&0\end{smallmatrix}\bigr)$ drops rank at $x_2=0$, so the system is not feedback linearizable. Yet in the flat metric $M=W=I_2$ (Condition~C2 holds trivially and geodesics are straight lines), the unactuated component $\widehat A_{\mathrm{sym},11}=-1-3x_1^2\le -1$ is naturally contracting, and choosing the state-dependent control-port conductance $\rho(x)=2\lambda+\epsilon+\tfrac{2x_2^2}{1-\lambda}$ ($\lambda\in(0,1)$, $\epsilon>0$) yields a strictly negative Schur complement in \eqref{eq:dual_ccm_lmi_app} across all of $\RR^2$: the actuated port strictly dominates the quadratic cross-coupling $2x_2$ everywhere without requiring feedback linearization.

%% file: appendix/K_passive_wave_world_models.tex
\section{Passive Wave World Models for Visual Estimation and Control}
\label{app:wave_world_models}

Appendix~\ref{sec:exp_physics_control} considered mechanical systems with known
coordinates and direct sensor measurements. A natural question is whether the
same local scattering structure remains effective when the state is observed
only through raw camera frames, with neither an explicit coordinate estimator
nor an image decoder. In this setting, the dynamics must be learned directly in
a latent space from visual trajectories.

Conventional visual world models separate representation learning, prediction,
and control into three distinct procedures: an image encoder maps frames into
an unconstrained latent vector, a sequence network is trained offline by global
backpropagation through time, and at deployment the model is held fixed while a
stochastic search method samples thousands of candidate action sequences per
step \citep{maes2026leworldmodel,zhou2025dinowm,sobal2025pldm,hansen2024tdmpc2}.
This separation has two physical consequences. First, because the learned
transition map obeys no conservation or dissipation law, multi-step rollouts
can create spurious energy and drift into unphysical regions that a sampling
planner readily exploits. Second, because past inference and future control do
not share a common dynamical representation, a sudden change in object mass,
friction, or actuator calibration leaves the controller acting on an outdated
model.

In the wave-scattering formulation, these components belong to a single passive
medium. Embedding the latent evolution in a port-Hamiltonian contact system
\citep{van2014port} guarantees energy passivity by construction, while writing
both past identification and future control in directional characteristic waves
places learning, estimation, and planning on the same local algebra. Below we
summarize the formulation and empirical results across six visual control
problems, together with their mechanistic ablations and operating boundaries.

\subsection{Passive contact mechanics in latent phase space}

An image encoder maps each RGB frame $o_t$ into a latent phase-space state
$x_t=(q_t,p_t)$ governed by a port-Hamiltonian system augmented with smooth
unilateral contact barriers and impact dissipation:
\begin{align}
\dot q &= M(q)^{-1}p, \nonumber \\
\dot p &= -\nabla_q V(q) - \nabla_q V_c(q)
  - \bigl(D(q)+D_c(q)\bigr)M(q)^{-1}p
  + G(q)u + f_{\mathrm{ext}}(x,u;\phi),
\label{eq:wm_contact_ph}
\end{align}
with stored mechanical energy
$H(q,p)=\tfrac{1}{2}p^\top M(q)^{-1}p + V(q) + V_c(q)$,
strictly positive-definite inverse inertia $M(q)^{-1}\succ 0$, Rayleigh damping
$D(q)\succeq 0$, unilateral contact potentials
$V_c(q)=\sum_j \frac{\kappa_j}{2}\operatorname{softplus}(-\psi_j(q)/\tau)^2$
localized to active contact ($\psi_j(q)<0$), and Hunt--Crossley impact
dissipation
$D_c(q)=\sum_j c_j\,\sigma(-\psi_j(q)/\tau)(\nabla_q\psi_j)(\nabla_q\psi_j)^\top\succeq 0$.
In higher-dimensional contact manipulation, the same structure is written in
generalized dissipative port-Hamiltonian form
$\dot z = (J(z)-R(z)-D_c(z))\nabla_z H(z) + G(z)u + f_{\mathrm{ext}}(z,u;\phi)$
with skew-symmetric $J(z)=-J(z)^\top$ and $R(z),D_c(z)\succeq 0$. Because the
conservative flow is integrated symplectically and all damping matrices are
positive semidefinite, every discrete latent transition obeys the energy
dissipation inequality on $100\%$ of tested trajectories: imagined rollouts
cannot manufacture mechanical energy.

\begin{figure}[t]
\centering
\includegraphics[width=0.92\textwidth]{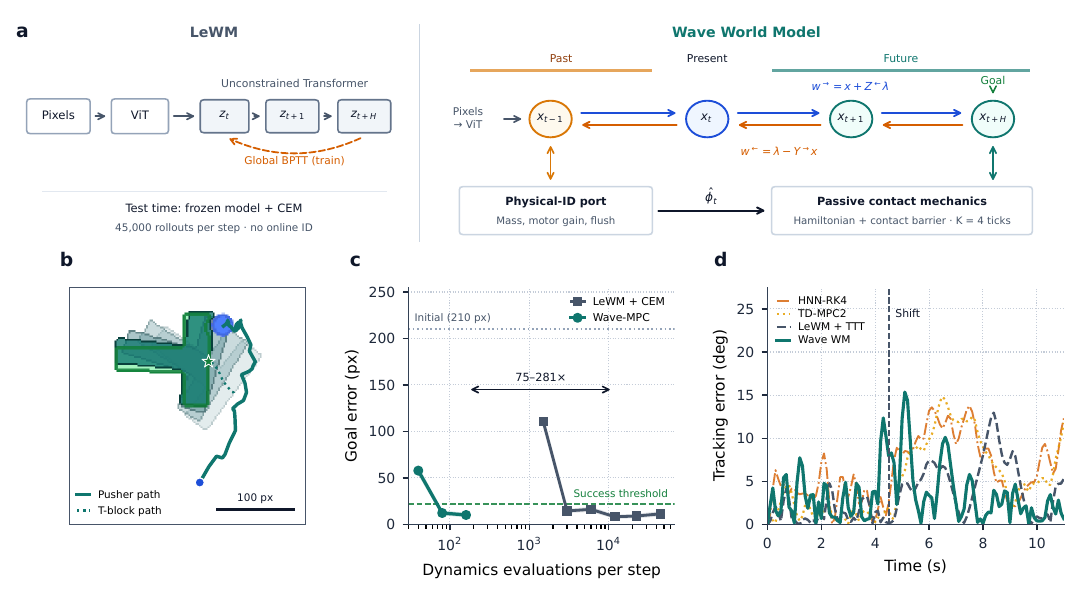}
\caption{\textbf{Passive wave world model and closed-loop efficiency.}
(a)~Comparison between an unconstrained visual sequence model
(\texttt{LeWM}~\citep{maes2026leworldmodel}, trained with backpropagation
through time and planned with cross-entropy sampling, \texttt{CEM}) and the
passive wave world model, in which past identification ($\hat\phi_t$) and
future receding-horizon control ($K=4$ local ticks) share one directional wave
medium ($w^\to=x+Z^\leftarrow\lambda$, $w^\leftarrow=\lambda-Y^\to x$).
(b)~Closed-loop Push-T rollout from $224\times 224$ RGB images using $160$
dynamics evaluations per step (held-out episode; final error $6.2\,\mathrm{px}$
and $4.7^\circ$).
(c)~Closed-loop Push-T goal error versus dynamics evaluations per step
($n=10$ held-out tasks, mean $\pm$ s.e.m.): wave planning reaches
$10.1\pm 1.3\,\mathrm{px}$ ($90\%$ success) at $160$ evaluations, matching
\texttt{LeWM+CEM} at $12{,}000\text{--}45{,}000$ evaluations
($75\text{--}281\times$ fewer evaluations).
(d)~Closed-loop visual pendulum tracking error ($64\times 64$ RGB) across an
unseen $+65\%$ mass increase and opposing wind at $t=4.5\,\mathrm{s}$ (dashed
line), comparing an unconstrained Hamiltonian network (\texttt{HNN-RK4}),
\texttt{TD-MPC2}, \texttt{LeWM} with online test-time gradient adaptation
(\texttt{TTT}), and unified past--future wave control ($96$ evaluations per
step, $100\%$ passive).}
\label{fig:wm_architecture}
\end{figure}

As in Appendix~\ref{sec:exp_physics_control}, the state--costate optimality
conditions are expressed in directional characteristic waves,
\begin{equation}
w_k^\to = x_k + Z_k^\leftarrow \lambda_k, \qquad
w_k^\leftarrow = \lambda_k - Y_k^\to x_k,
\label{eq:wm_directional_waves}
\end{equation}
joining a past observation window to a future control horizon
(Figure~\ref{fig:wm_architecture}a). Along the past window, forward impedances
$Z_k^\leftarrow$ carry observation residuals to a low-dimensional physical
parameter port $\hat\phi_t$ capturing shifts in inertia, contact dissipation,
or actuation. Along the future horizon, backward Riccati impedances $Y_k^\to$
scatter goal errors into control updates in $K=4$ parallel nearest-neighbor
ticks. In double-precision checks on frozen trajectories, both the training
wave solver and the past identification port match exact automatic
differentiation to within $10^{-16}\text{--}10^{-6}$ relative error across all
tasks (Table~\ref{tab:wm_verification}).

\subsection{Closed-loop visual control and computational cost}

We evaluate this construction on four standard pixel-based control tasks
(Two-Room at $64\times 64$ pixels, MuJoCo Reacher and PointMaze at $96\times 96$
pixels, and official Push-T at $224\times 224$ pixels) as well as two
continuous mechanical control problems from $64\times 64$ images (actuated
pendulum and two-link manipulator). All controlled comparisons are executed on
the same NVIDIA H100 hardware against LeWorldModel (\texttt{LeWM})
\citep{maes2026leworldmodel}, \texttt{TD-MPC2} \citep{hansen2024tdmpc2},
symplectic \texttt{LeHami} \citep{yalovetzky2026lehami}, and an unconstrained
Hamiltonian neural network integrated with fourth-order Runge--Kutta
(\texttt{HNN-RK4}) \citep{greydanus2019hamiltonian}, alongside published
reference scores for \texttt{DINO-WM} \citep{zhou2025dinowm} and \texttt{PLDM}
\citep{sobal2025pldm}.

\begin{figure}[t]
\centering
\includegraphics[width=0.92\textwidth]{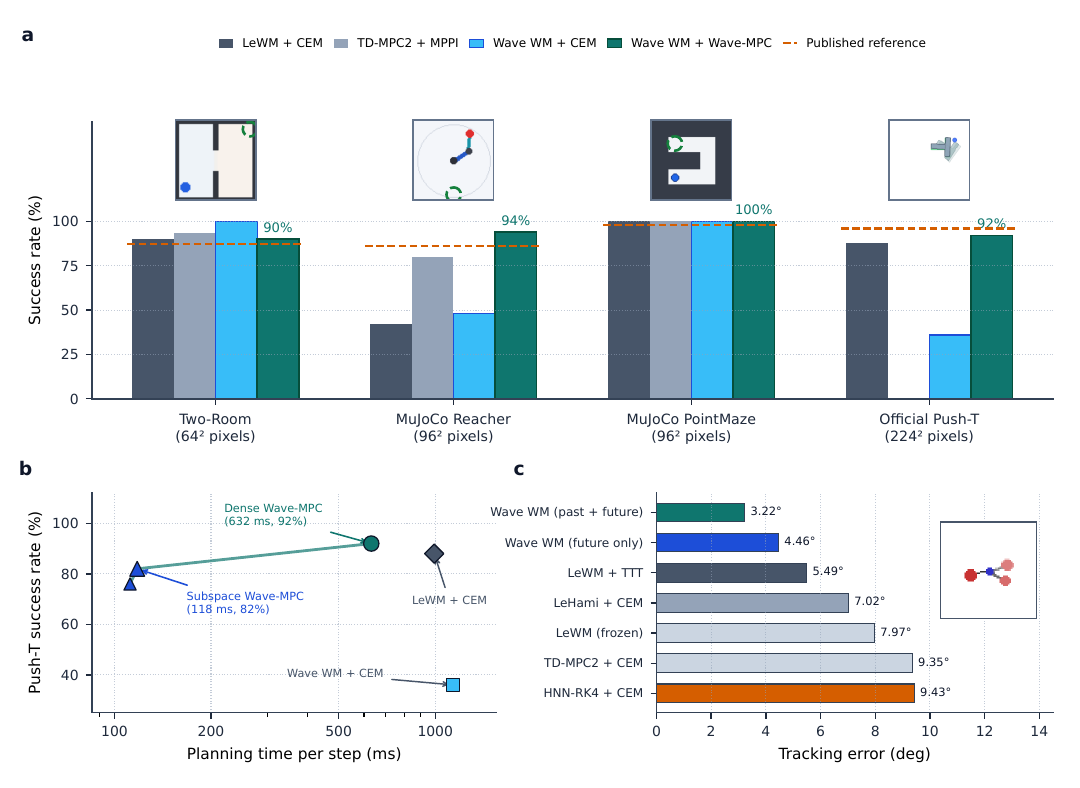}
\caption{\textbf{Visual control accuracy, planning latency, and passivity.}
(a)~Nominal goal-reaching success across Two-Room ($64^2$), MuJoCo
Reacher ($96^2$), MuJoCo PointMaze ($96^2$), and official Push-T ($224^2$),
with raw simulator frames shown above each task. Dashed lines indicate
published \texttt{LeWM} or \texttt{DINO-WM} reference scores. Wave planning
uses $94\text{--}281\times$ fewer dynamics evaluations per step than sampling
planners.
(b)~Push-T success rate versus per-step planning time on an H100 GPU (log
scale). Dense wave planning ($632\,\mathrm{ms}$, $160$ evaluations) reaches
$92\%$ versus $88\%$ for pretrained $18\,\mathrm{M}$-parameter
\texttt{LeWM+CEM} ($993\,\mathrm{ms}$, $45{,}000$ evaluations); projecting the
impedances onto the dominant low-rank motion subspace reduces latency to
$118\,\mathrm{ms}$ ($8.4\times$ faster) at $82\%$ success.
(c)~Mean angle tracking error on the visual pendulum ($64^2$, three seeds)
under an unseen $+65\%$ mass increase and opposing wind. Unified past--future
wave control achieves $3.22^\circ$ error with $100\%$ discrete passivity.}
\label{fig:wm_benchmarks}
\end{figure}

Imposing passive mechanical structure does not degrade accuracy on visual
control tasks (Figure~\ref{fig:wm_benchmarks}a). Using $94\text{--}281\times$
fewer dynamics evaluations per step than cross-entropy sampling, wave planning
reaches $90\%$ success on Two-Room ($100\%$ with sampling), $94\%$ on MuJoCo
Reacher (versus $42\%$ for our controlled \texttt{LeWM} baseline, $80\%$ for
\texttt{TD-MPC2}, and $86\%$ reported in \citet{maes2026leworldmodel}), and
$100\%$ on MuJoCo PointMaze ($64\%$ versus $62\%$ on a longer-horizon U-turn
detour). On official Push-T ($224\times 224$ pixels), the
$17.1\,\mathrm{M}$-parameter passive model reaches $92\%$ success, compared
with $88\%$ for the official pretrained $18.0\,\mathrm{M}$-parameter
\texttt{LeWM} model.

Holding the passive Push-T world model fixed and comparing planners directly
isolates the role of local wave scattering: cross-entropy sampling with
$45{,}000$ rollouts per step achieves only $36\%$ success on the same trained
model, whereas $160$-evaluation wave planning reaches $92\%$
(Figure~\ref{fig:wm_benchmarks}a,b). Random shooting probes regions of latent
space far from the training trajectories where the visual map is poorly
constrained; directional Riccati impedances $(Y_k^\to,Z_k^\leftarrow)$ instead
propagate local curvature within a smooth trust region along the learned
physical flow. In wall-clock time on an H100 GPU, dense wave planning runs in
$632\,\mathrm{ms}$ per step versus $993\,\mathrm{ms}$ for pretrained
\texttt{LeWM+CEM} ($1.57\times$ faster). Projecting the wave impedances onto
the dominant low-rank motion subspace avoids dense matrix inversions and lowers
planning time to $118\,\mathrm{ms}$ per step ($8.4\times$ faster) while
retaining $82\%$ success (Figure~\ref{fig:wm_benchmarks}b).

Continuous control under physical disturbances illustrates the effect of strict
passivity (Figure~\ref{fig:wm_benchmarks}c). On the visual pendulum subjected
to an unseen $+65\%$ mass increase and opposing wind (three seeds), unified
past--future wave control attains $3.22^\circ$ mean error at $100\%$ discrete
passivity ($96$ evaluations per step), compared with $4.46^\circ$ for
future-only wave control, $5.49^\circ$ for \texttt{LeWM} with online gradient
adaptation, $7.02^\circ$ for symplectic \texttt{LeHami}, $9.35^\circ$ for
\texttt{TD-MPC2}, and $9.43^\circ$ for \texttt{HNN-RK4}, whose discrete
trajectories violate energy passivity on a third of transitions ($67\%$
passive). On the visual two-link arm perturbed by an unseen $+0.75\,\mathrm{kg}$
payload and torque drift, wave control reaches $0.921\,\mathrm{rad}$ error
versus $1.711\,\mathrm{rad}$ for \texttt{LeWM} at $320\times$ fewer dynamics
evaluations. In a controlled low-data four-task suite ($64\times 64$ pixels,
$350\text{--}450$ episodes, $642\text{K}$ parameters across Two-Room, Reacher,
PointMaze, and Push-T), adding unilateral contact barriers and impact
dissipation lifts mean success under perturbed physics from $37.5\%$
(\texttt{LeWM}) to $59.4\%$.

\subsection{Single-episode adaptation under physical perturbations}

In physical operation, payloads, friction, and actuator calibration can change
without warning. Updating millions of unstructured weights from a handful of
camera frames is poorly conditioned and prone to drift, whereas leaving the
model frozen fails once the plant departs from its training regime. Factoring
the latent dynamics into inertia, contact barriers, damping, and a $2\times 2$
control matrix changes the geometry of adaptation: physical perturbations enter
through a compact parameter port $\phi$ (verified identical to the nominal
dynamics at $\phi=0$, with past-window gradients matching double-precision
reference differentiation to $2.5\times 10^{-9}$).

\begin{figure}[t]
\centering
\includegraphics[width=0.92\textwidth]{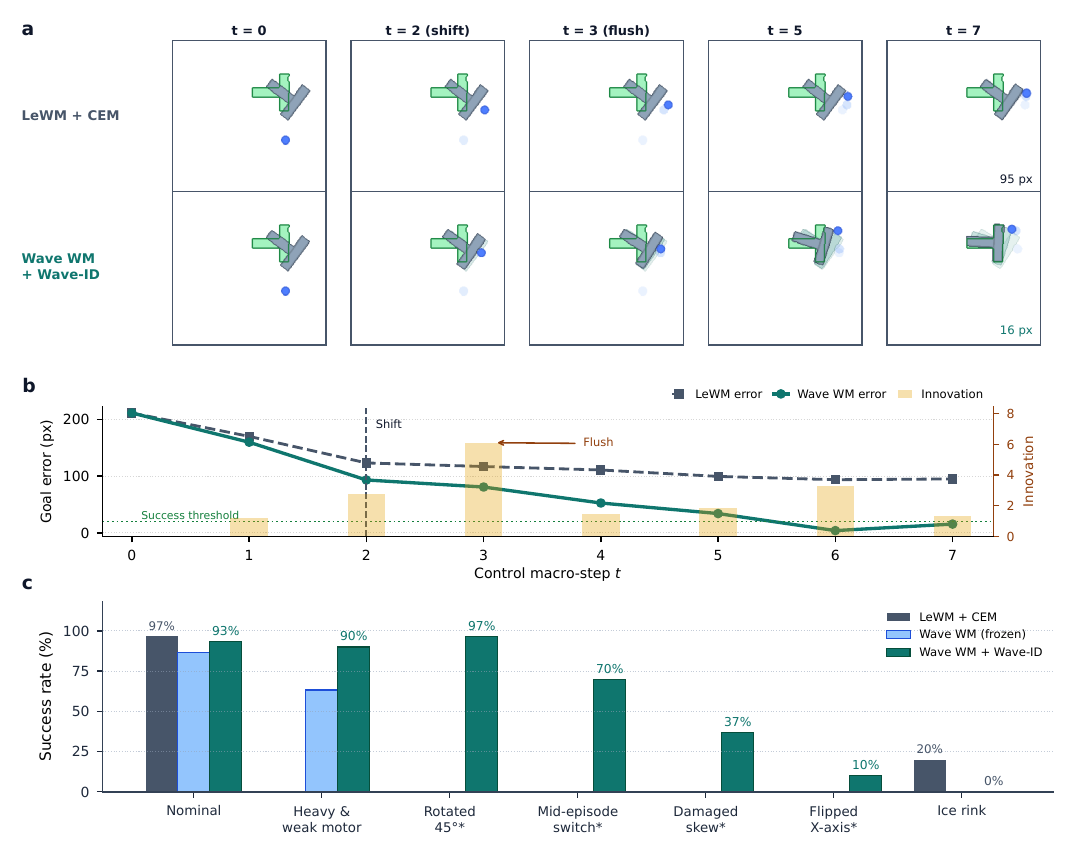}
\caption{\textbf{Online single-episode adaptation under physical perturbations
on official Push-T.}
(a)~Matched-seed rollout ($224^2$ RGB) with an abrupt mid-episode switch at
macro-step $t=2$ ($+80\%$ block mass, $45^\circ$-rotated and $35\%$-weaker
actuator). Pretrained $18\,\mathrm{M}$ \texttt{LeWM+CEM} ($45{,}000$
evaluations/step, top) stalls at $95\,\mathrm{px}$ error; the wave world model
with online parameter identification ($160$ evaluations/step, bottom) flushes
pre-shift history at $t=3$ and reaches $16\,\mathrm{px}$ at $t=7$.
(b)~Goal error (left axis) and one-step latent prediction innovation (right
axis, shaded bars) for the episode in (a); exceeding twice the running median
triggers a past-window flush at $t=3$.
(c)~Goal-reaching success across seven physical regimes ($n=30$ episodes each;
asterisks mark perturbations absent from training). Online identification
maintains nominal performance ($93\%$ vs.\ $97\%$ for \texttt{LeWM}) and adapts
across mass, rotation, mid-episode switch, and actuator-damage regimes where
frozen \texttt{LeWM} falls to $0\%$, while staying at $0\%$ on the unmodeled
free-gliding ice-rink regime.}
\label{fig:wm_adaptation}
\end{figure}

When dynamics shift abruptly mid-episode, smoothing over a fixed past window
blends pre-shift and post-shift observations and slows recovery. Monitoring the
one-step latent prediction innovation along the past window avoids this lag:
whenever the innovation exceeds twice its running median, the solver clears the
pre-shift buffer and re-identifies $\hat\phi_t$ from subsequent transitions
alone (Figure~\ref{fig:wm_adaptation}a,b). Across seven Push-T physical regimes
($30$ episodes each; Figure~\ref{fig:wm_adaptation}c), online wave
identification preserves nominal performance ($93\%$, up from $87\%$ frozen and
close to $97\%$ for pretrained \texttt{LeWM}) while succeeding in unseen
regimes where pretrained \texttt{LeWM} drops to $0\%$: heavy mass with weak
actuation ($90\%$), $45^\circ$-rotated actuation ($97\%$, median error
$11.9\,\mathrm{px}$), abrupt mid-episode switches ($70\%$, detecting $29/30$
switches), and anisotropic skewed actuator damage ($37\%$). Under a complete
horizontal sign inversion of the actuator, online identification lowers
position error by a factor of $4.6$ ($152\,\mathrm{px}$ to $33\,\mathrm{px}$,
$10\%$ success).

The seventh regime---a low-friction surface (``ice rink'') on which the T-block
glides and spins freely after contact---exhibits a clear boundary of the model
class. Because the passive contact model was trained solely on quasi-static
pushing, it remains at $0\%$ success on free gliding, whereas multi-frame
\texttt{LeWM} reaches $20\%$ using explicit frame-history velocity context.
Structured parameter ports accelerate adaptation whenever the perturbation lies
within the assumed mechanical class, and make transparent when new degrees
of freedom (here, sustained free-flight inertia) must be included in the
dynamics.

\subsection{Counterfactual sensitivity, two-dimensional wave training, and
horizon limits}

Three targeted experiments clarify what governs closed-loop performance,
whether backpropagation can be removed along network depth and physical time
simultaneously, and where short-horizon latent planning stops
(Figure~\ref{fig:wm_mechanisms}).

\begin{figure}[t]
\centering
\includegraphics[width=\textwidth]{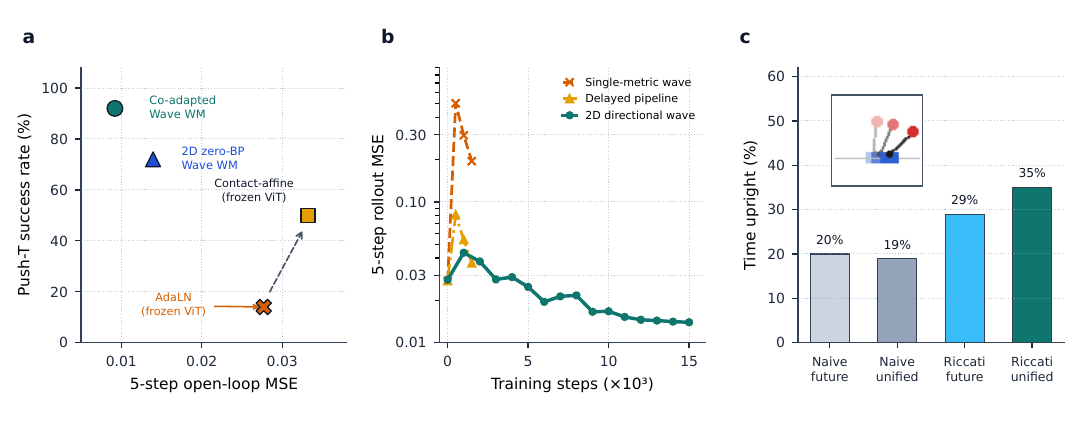}
\caption{\textbf{Structural ablation, two-dimensional wave training without
backpropagation, and short-horizon stability limits.}
(a)~Five-step open-loop validation MSE versus closed-loop Push-T success. On
identical frozen vision-transformer features, adaptive layer normalization
(\texttt{AdaLN}) achieves lower open-loop MSE ($0.0277$ vs.\ $0.0332$) but far
lower closed-loop success ($14\%$ vs.\ $50\%$) than control-affine passive
contact dynamics. Co-adapting the encoder yields $0.0092$ MSE and $92\%$
success; full two-dimensional wave training without backpropagation reaches
$0.0139$ MSE and $72\%$ success.
(b)~Five-step rollout MSE during two-dimensional wave training across $12$
encoder layers and $5$ time steps. Directional two-conductance waves converge
smoothly ($0.0139$ MSE, $72\%$ Push-T success, $100\%$ passive), whereas
single-metric symmetric waves diverge ($R^2<0$) and naive delayed pipelining
spikes.
(c)~Fraction of episode time held upright on an inverted visual cart-pole
($64^2$ RGB, horizon $H=12$, no learned value function). Replacing a naive
terminal penalty with an infinite-horizon discrete algebraic Riccati boundary
raises upright time from $20\%$ to $35\%$ while marking the horizon limit of
pure-latent planning without a value function.}
\label{fig:wm_mechanisms}
\end{figure}

\paragraph{Open-loop fit versus closed-loop control.}
Holding the pretrained vision encoder fixed on Push-T and varying only the
latent transition law reveals a sharp separation between trajectory fit and
controllability (Figure~\ref{fig:wm_mechanisms}a). An expressive transition
network modulated by adaptive layer normalization (\texttt{AdaLN}) fits
five-step recorded trajectories more accurately than the control-affine contact
model ($0.0277$ versus $0.0332$ mean-squared prediction error), yet falls from
$50\%$ to $14\%$ in closed-loop control. Fitting recorded demonstrations rewards
memorizing correlations along the demonstrated trajectories, whereas
receding-horizon control depends on reliable local derivatives with respect to
actions the demonstrator did not execute. Constraining actuation to enter
affinely through passive contact mechanics preserves these sensitivities;
jointly co-adapting the encoder with the same dynamics lowers prediction error
to $0.0092$ and lifts closed-loop success to $92\%$.

\paragraph{Two-dimensional wave training across depth and time.}
Because the scattering law \eqref{eq:wm_directional_waves} applies equally
along network layers and along temporal rollouts, backpropagation can be
eliminated along both axes simultaneously---training all $12$
vision-transformer blocks and the $5$-step contact rollout
($17.1\,\mathrm{M}$ parameters) purely by local nearest-neighbor wave exchange
(Figure~\ref{fig:wm_mechanisms}b). As in the sequence and control experiments
of Appendix~\ref{sec:exp_physics_control}, the choice of wave metric is
decisive: single-metric symmetric waves reflect across non-orthogonal layers
and diverge ($R^2<0$), and naive delayed pipelining exhibits staleness spikes.
Directional two-conductance waves suppress these reflections (cold-start
gradient cosine $0.990\text{--}1.000$ relative to exact differentiation) and
train the full model without backpropagation to $0.0139$ five-step prediction
error, $100\%$ discrete passivity, and $72\%$ closed-loop Push-T success
($16.1\,\mathrm{px}$ median error). This behavior mirrors earlier exploratory
runs with cold, single-metric Hamiltonian waves on the visual pendulum, which
stalled at negative-curvature saddle points and collapsed into static encoder
shortcuts ($0/24$ control success), directly motivating the directional
two-conductance formulation and passive contact mechanics used throughout this
study.

\paragraph{Short-horizon limit on open-loop unstable systems.}
An inverted visual cart-pole ($64\times 64$ RGB, horizon $H=12$) without a
learned value function defines the horizon boundary of pure-latent quadratic
planning (Figure~\ref{fig:wm_mechanisms}c). Over $12$ steps, latent Euclidean
distance to the upright target frame cannot foresee long-term tipping, and
every tested world model eventually loses balance ($94^\circ\text{--}104^\circ$
tracking error across \texttt{LeWM}, \texttt{TD-MPC2}, \texttt{LeHami}, and
the wave world model). Replacing a naive quadratic terminal cost with an
infinite-horizon discrete algebraic Riccati boundary termination increases
upright time from $20\%$ to $35\%$, demonstrating both the role of matched
terminal impedances and the regime---open-loop unstable equilibria extending
beyond the planning horizon---where an explicit value function remains
necessary.

\begin{table}[t]
\centering
\caption{\textbf{Numerical verification and summary across visual world-model
benchmarks.} Relative errors compare double-precision wave training and past
identification against exact automatic differentiation on frozen trajectories;
passivity reports the fraction of latent transitions satisfying the discrete
energy-dissipation inequality.}
\label{tab:wm_verification}
\footnotesize
\resizebox{\textwidth}{!}{%
\begin{tabular}{llccccc}
\toprule
Environment & Protocol & Params (Wave / Base) & Wave-PMP err & Past-ID err & Passive & Result (Wave vs.\ Base) \\
\midrule
Two-Room & $64^2$, 350 eps & $642\text{K}$ / $2.63\text{M}$ & $2.3\times 10^{-6}$ & $2.0\times 10^{-8}$ & $100\%$ & $90\%$ ($100\%$ CEM) / $100\%$ shifted (vs.\ $90\%$ / $75\%$) \\
MuJoCo Reacher & $96^2$, official $H=5$ & $3.12\text{M}$ / $4.44\text{M}$ & $1.1\times 10^{-6}$ & $4.2\times 10^{-8}$ & $100\%$ & $94\%$ / $63\%$ shifted (vs.\ $42\%$ / $37\%$) \\
MuJoCo PointMaze & $96^2$, official $H=5$ & $3.12\text{M}$ / $4.44\text{M}$ & $5.6\times 10^{-16}$ & $7.0\times 10^{-8}$ & $100\%$ & $100\%$ (U-turn $64\%$ vs.\ $62\%$) \\
Official Push-T & $224^2$, 18k eps (joint) & $17.1\text{M}$ / $18.0\text{M}$ & $8.6\times 10^{-12}$ & $2.0\times 10^{-8}$ & $100\%$ & $92\%$ ($118\text{--}632\,\mathrm{ms}$) vs.\ $88\%$ ($993\,\mathrm{ms}$) \\
Push-T (2D wave) & $224^2$, 12 layers $\times$ 5 steps & $17.1\text{M}$ (no BP) & $\cos=0.990\text{--}1.0$ & $1.3\times 10^{-8}$ & $100\%$ & $72\%$ ($16.1\,\mathrm{px}$, zero BP / BPTT) \\
Push-T perturbed & 7 regimes, 30 eps each & $17.1\text{M}$ / $18.0\text{M}$ & --- & $2.5\times 10^{-9}$ & $100\%$ & $97\%$ rot., $70\%$ switch (vs.\ $0\%$ \texttt{LeWM}) \\
Visual Pendulum & $64^2$, $+65\%$ mass + wind & $628\text{K}$ / $2.63\text{M}$ & $3.0\times 10^{-7}$ & $4.7\times 10^{-8}$ & $100\%$ & $3.22^\circ$ (vs.\ $5.49^\circ$ \texttt{LeWM+TTT}) \\
Visual 2-Link Arm & $64^2$, $+0.75\,\mathrm{kg}$ + drift & $636\text{K}$ / $2.63\text{M}$ & $7.3\times 10^{-6}$ & $2.5\times 10^{-7}$ & $100\%$ & $0.921\,\mathrm{rad}$ (vs.\ $1.711\,\mathrm{rad}$) \\
Visual Cart-Pole & $64^2$, $H=12$ (limit) & $628\text{K}$ / $2.63\text{M}$ & $4.4\times 10^{-13}$ & $2.7\times 10^{-8}$ & $100\%$ & $35\%$ upright w/ Riccati (needs value fn.) \\
\bottomrule
\end{tabular}}
\end{table}

%% file: appendix/L_neural_ode_hnn_physics.tex
\section{Metric-Wave and Split-Riccati Physics Learning: Neural ODEs, Hamiltonian Networks, Hybrid Impacts, and Symbolic Discovery}
\label{app:neural_ode_hnn_physics}

Appendix~\ref{sec:exp_physics_control} demonstrated how persistent wave scattering unifies online mechanical identification and receding-horizon control. Here we test the same two theory-derived wave parameterizations from Section~\ref{sec:training_neural_network} across six canonical problems in continuous-time scientific machine learning where standard backpropagation and adjoint methods fail structurally: (i)~chaotic Neural ODE and Hamiltonian Neural Network (HNN) identification, (ii)~zero-reflection physical anomaly detection, (iii)~constant-memory reversible Neural ODEs, (iv)~non-smooth and chaotic hybrid impact differentiation, (v)~machine-precision energy conservation in learned Hamiltonian circuits, and (vi)~noisy sparse symbolic law discovery compared against \texttt{SINDy}.

Every experiment in this appendix uses strictly one of the two wave bases derived in Section~\ref{sec:training_neural_network}:
\begin{enumerate}
\item \textbf{Method~1: Symmetric Metric Wave Variables} ($G_k = \Theta_k^\top \Theta_k \succ 0$, \eqref{eq:disc_wave_def}--\eqref{eq:discrete_orthogonal_condition}), where the pullback transport $G_k = J_k^\top G_{k+1} J_k$ renders known conservative flows, augmented $SO(d+1)$ Cayley ODEs, EventProp hybrid jump junctions, and discrete-gradient Circuit-HNNs strictly lossless ($R_L \equiv 0$).
\item \textbf{Method~2: Two-Conductance Split Riccati Wave Variables} ($w_{+,k} = x_k + Z_k^\leftarrow \lambda_k$, $w_{-,k} = \lambda_k - Y_k^\to x_k$, \eqref{eq:above_below_summary}), where forward and backward Riccati conductances $(Z_k^\leftarrow, Y_k^\to)$ absorb unstable Lyapunov modes locally inside each scattering cell, yielding $O(1)$-bounded shadowing gradients over arbitrarily long chaotic horizons.
\end{enumerate}

Table~\ref{tab:physics_wave_summary} summarizes the quantitative benchmarks, fair baselines, multi-seed statistics, and measured memory and timing budgets across all tasks. All experiments are reproducible from a single benchmark suite, and the core hybrid impact adjoint and symplectic integration primitives are additionally cross-verified against an independent Lean~4 implementation (\texttt{LeanHybrid}).

\begin{table}[t]
\centering
\small
\caption{\textbf{Summary of Metric-Wave (Method~1) and Split-Riccati (Method~2) scientific machine learning benchmarks.} All chaotic and multi-trial metrics report mean $\pm$ s.d.\ across $n=5$ independent noise seeds or initial conditions, with strictly observation-only initialization (no ground-truth state leakage).}
\label{tab:physics_wave_summary}
\resizebox{\textwidth}{!}{%
\begin{tabular}{llll}
\toprule
Benchmark (Wave Method) & Wave-Scattering Result & Strongest Baselines & Cost / Verification \\
\midrule
Lorenz-63, $\lambda_{\max}T=23.5$ (M2, $n=5$) & $\|\nabla_\theta L\|=0.97\pm 0.02$; ID err $3.45\pm 0.20\%$ & Std.\ adjoint $6.76\times 10^7$; 10-win shooting $28.96\pm 11.19\%$ & $y_{\mathrm{obs}}$-only init \\
4D H\'enon--Heiles HNN, $\sigma=0.030$ (M2) & Rollout RMSE $0.0596$ ($V_{\mathrm{rel}}=10.0\%$) & FD-HNN $0.0907$ ($6.7\%$); Single-Shoot $0.3499$ ($87.7\%$) & Kinematic-exact $W$ \\
Kepler anomaly radar (M1+M2) & Orbit RMSE $0.316\to 0.0183$ ($17.3\times$); refl.\ $-4.2\times$ & Unconstrained 2D neural residual $f_\theta(q,v)$ & Newton $R_L\le 5.4\times 10^{-10}$ \\
$SO(d+1)$ isometric Neural ODE (M1, $n=5$) & $T=80$ err $4.72\times 10^{-14}$; Rings $100.0\%$ acc. & Unconstrained $10^6$; Symplectic Verlet $1.72\times 10^{-12}$ & $27.3\,\mathrm{KB}$ vs.\ $14.07\,\mathrm{MB}$ ($515\times$) \\
5-bounce curved terrain (M1+Hybrid) & Test RMSE $0.0011$; terrain RMSE $0.0012$ & Soft-contact $0.0989$; Naive $0.145$; SmoothODE $0.654$ & Hybrid $\cos = 1.0000000$ \\
31-bounce Sinai billiard (M2+Hybrid, $n=5$) & $\|\nabla_\theta L\|=0.157\pm 0.019$; disk err $0.040\pm 0.033$ & NEODE $5.53\times 10^6$; Curr.\ NEODE $0.070$; Full $0.106$ & $y_{\mathrm{obs}}$-only sweep \\
Circuit-HNN, $5{,}000$ steps (M1, $n=5$) & Max $|\Delta H|=6.57\times 10^{-14}$; Tellegen $3.68\times 10^{-15}$ & NeuralODE $3.35\times 10^4$; HNN-RK4 $0.111$; Leapfrog $0.0484$ & $436.7$ vs.\ $14.7\,\mu\mathrm{s}$/step \\
Symbolic discovery, $\sigma=0.030$ (M2, $n=5$) & Max err $0.040\pm 0.021$; $100\%$ exact support & Savitzky--Golay SINDy $0.963\pm 0.136$; $0\%$ support & Figure~\ref{fig:phys_sindy_ext} \\
\bottomrule
\end{tabular}%
}
\end{table}

\begin{figure}[t]
\centering
\includegraphics[width=\textwidth]{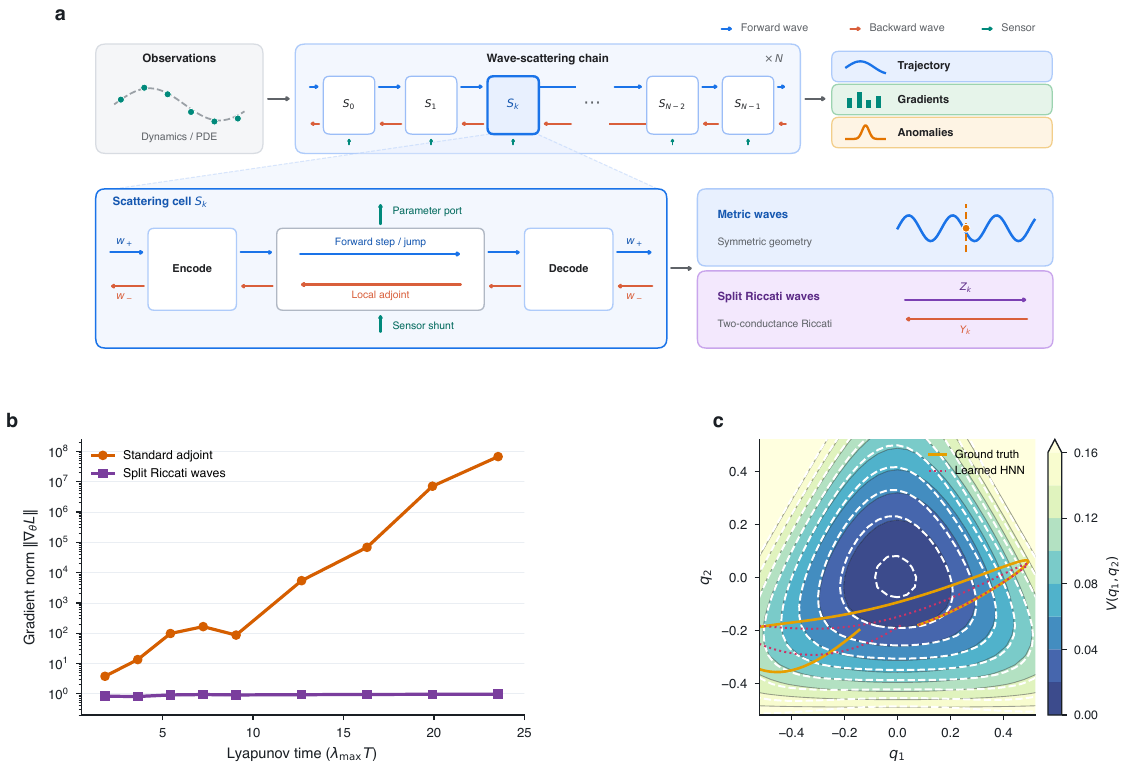}
\caption{\textbf{Wave-scattering architecture and chaotic system identification.}
\textbf{a},~Schematic of the bidirectional wave-scattering chain ($S_0$ to $S_{N-1}$) linking counter-propagating forward (blue) and backward (coral) waves with local sensor shunts (teal) to produce denoised trajectories, parameter gradients, and reflection anomalies. Magnified cell $S_k$ wraps each smooth step or hybrid jump and its local adjoint inside wave encoding/decoding; the two specializations are symmetric metric waves (Method~1) and two-conductance split Riccati waves $(Z_k,Y_k)$ (Method~2).
\textbf{b},~Parameter gradient norm versus horizon in Lyapunov times ($\lambda_{\max}T$) on chaotic Lorenz-63 (mean $\pm$ s.d., $n=5$ seeds, initialized strictly from noisy scalar $x(t)$ observations). Standard reverse-time adjoint gradients grow to $(6.76\pm 0.18)\times 10^7$ at $\lambda_{\max}T=23.5$, whereas split Riccati wave shadowing remains bounded at $0.97\pm 0.02$ and recovers $(\sigma,\rho,\beta)$ to $3.45\%\pm 0.20\%$ relative error (versus $28.96\%\pm 11.19\%$ for ten-window uncoupled single-shooting L-BFGS-B).
\textbf{c},~Chaotic 4D H\'enon--Heiles potential contours $V(q_1,q_2)$ learned from noisy position-only observations ($\sigma=0.030$; filled contours, ground truth; white dashed, orthogonal-factorized HNN trained with kinematic-exact compliance; $V_{\mathrm{RMSE}}=0.0032$, $V_{\mathrm{rel}}=10.0\%$). Solid amber and dotted rose curves compare a held-out chaotic test trajectory (rollout RMSE $0.0596$ versus $0.0907$ for Savitzky--Golay finite-difference HNN and $0.3499$ for single-shooting symplectic HNN).}
\label{fig:phys_unified_chaos}
\end{figure}

\subsection{Chaotic Neural ODEs and Hamiltonian Networks via Split Riccati Shadowing}
\label{app:exp_chaos_hnn}

In chaotic systems with maximal Lyapunov exponent $\lambda_{\max}>0$, single-shooting reverse-time adjoints \citep{chen2018neural} amplify sensitivities as $e^{\lambda_{\max}T}$, corrupting parameter updates whenever the horizon exceeds a few Lyapunov times. Under Method~2 (Section~\ref{sec:composability}), introducing a small positive-semidefinite process compliance $W_k\succeq 0$ (multiple-shooting shadowing slack $x_{k+1}=f_k(x_k)-W_k\lambda_{k+1}$) together with observation shunts $Q_k=w_k C_k^\top C_k$ yields the coupled forward--backward Riccati conductances $(Z_k^\leftarrow, Y_k^\to)$ of \eqref{eq:above_below_summary}. Both conductances remain strictly bounded on hyperbolic attractors, so the two-point linearized boundary-value problem equilibrates in $O(1)$ sweeps without exponential blowup.

\paragraph{Lorenz-63 from scalar noisy observations.}
On the chaotic Lorenz-63 attractor ($\lambda_{\max}\approx 0.906$, $\Delta t=0.02$, scalar observations $y_k=x_k+\varepsilon_k$ with $\sigma_{\mathrm{obs}}=0.20$), we evaluate gradient norms up to $T=26$ ($\lambda_{\max}T=23.5$) across $n=5$ noise seeds, initializing the shadowing state strictly from $y_{\mathrm{obs}}$ (never from ground-truth states). Standard adjoint gradients explode to $(6.76\pm 0.18)\times 10^7$, while split Riccati wave shadowing stays at $0.97\pm 0.02$ across the entire sweep (Figure~\ref{fig:phys_unified_chaos}b). Identifying $(\sigma,\rho,\beta)=(10,28,8/3)$ from a biased start $(7.5,22.5,1.9)$ over $T=14$ ($12.7$ Lyapunov times) yields $(10.28\pm 0.02,\,27.62\pm 0.02,\,2.707\pm 0.003)$ ($3.45\%\pm 0.20\%$ relative error), outperforming a ten-window uncoupled single-shooting L-BFGS-B baseline ($28.96\%\pm 11.19\%$, $8.4\times$ higher error) that jointly optimizes the shared parameters and unobserved initial $(y_0,z_0)$ per window.

\paragraph{Chaotic 4D H\'enon--Heiles HNN from position-only observations.}
We train an orthogonal-factorized scalar network $V_\theta(q_1,q_2)$ on six chaotic H\'enon--Heiles orbits ($H=\tfrac12\|p\|^2+\tfrac12\|q\|^2+q_1^2q_2-\tfrac13 q_2^3$, $\Delta t=0.08$, $N=140$) observed solely through noisy positions ($\sigma=0.030$, $\sim 6.5\%$ of orbit radius). Using a \emph{kinematic-exact} diagonal compliance $W=\operatorname{diag}(10^{-10},10^{-10},w,w)$ enforces exact $\dot q=p$ kinematics while confining shadowing slack to momentum accelerations. On a held-out chaotic test trajectory (Figure~\ref{fig:phys_unified_chaos}c), Split-Riccati HNN achieves rollout RMSE $0.0596$ ($V_{\mathrm{RMSE}}=0.0032$, $V_{\mathrm{rel}}=10.0\%$), outperforming both single-shooting symplectic HNN ($0.3499$, $V_{\mathrm{rel}}=87.7\%$) and a Savitzky--Golay finite-difference HNN \citep{greydanus2019hamiltonian} ($0.0907$, $V_{\mathrm{rel}}=6.7\%$). This highlights an instructive trade-off: pointwise acceleration regression via Savitzky--Golay achieves slightly lower global potential variance on smooth regions, whereas shadowing optimizes multi-step symplectic rollout coherence along chaotic folds.

\begin{figure}[t]
\centering
\includegraphics[width=\textwidth]{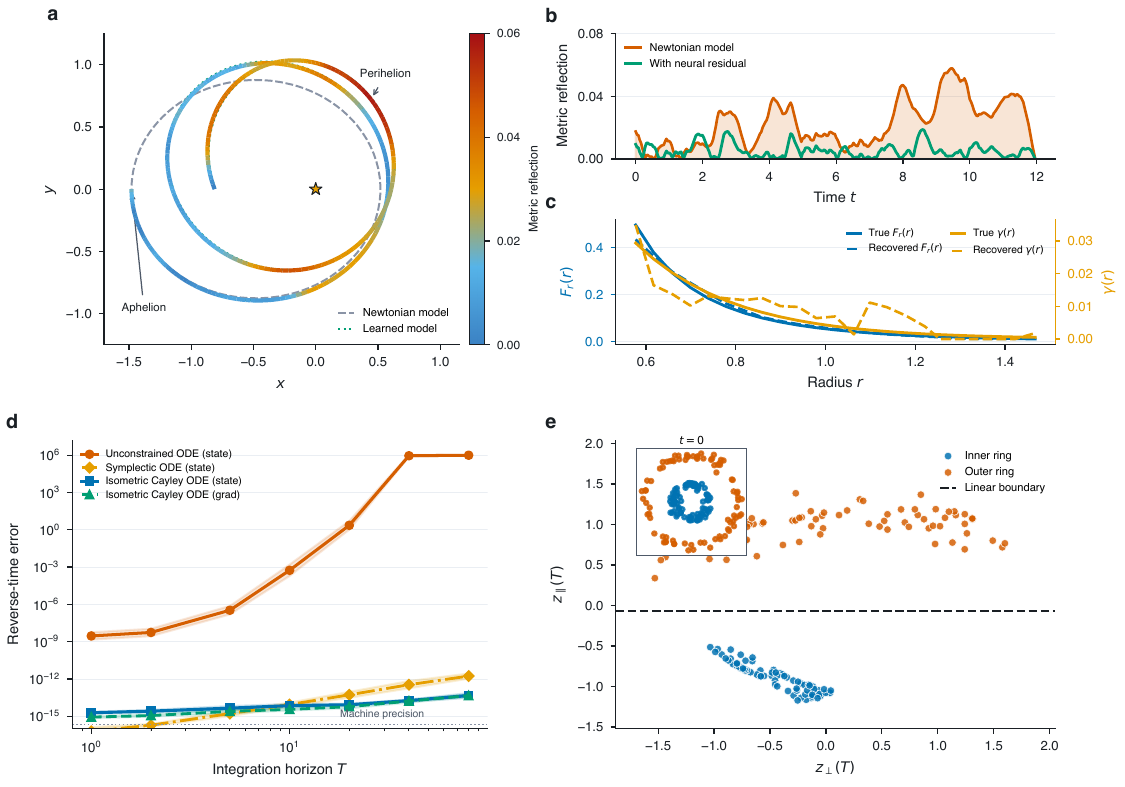}
\caption{\textbf{Zero-reflection anomaly detection and constant-memory isometric neural ODEs.}
\textbf{a},~Precessing 2D Kepler orbit perturbed by an unknown radial quadrupole force and perihelion drag, observed only through noisy positions ($\sigma=1.5\times 10^{-3}$). Under the known Newtonian model (dashed grey; bulk reflection $\le 5.4\times 10^{-10}$), the transported metric-wave reflection amplitude vanishes near aphelion and peaks at perihelion; dotted teal shows the open-loop orbit after fitting an unconstrained general 2D neural residual field $f_\theta(q,v)$ (orbital RMSE reduced $17.3\times$, from $0.316$ to $0.0183$).
\textbf{b},~Transported metric-wave reflection time series before (vermilion) and after (teal) adding the learned neural residual ($4.2\times$ reduction).
\textbf{c},~Post-hoc radial/tangential projection of the unconstrained neural residual into $F_r(r)$ (blue dashed, left) and $\gamma(r)$ (amber dashed, right) versus ground truth (solid).
\textbf{d},~Reverse-time state and gradient error without activation storage versus horizon $T$ ($n=5$ seeds, $d=8$). Unconstrained reverse integration (vermilion) diverges to $10^6$, reversible symplectic Verlet (amber) reaches $1.72\times 10^{-12}$, and the unified augmented $SO(d+1)$ isometric Cayley ODE holds float64 precision ($4.72\times 10^{-14}$ state, $4.61\times 10^{-14}$ gradient) at $27.3\,\mathrm{KB}$ measured peak memory versus $14.07\,\mathrm{MB}$ taped storage ($515\times$ reduction).
\textbf{e},~Concentric-ring classification ($n=160$; inset, $t=0$) solved by the same $SO(3)$ isometric Cayley ODE via projective lifting $(x_1,x_2,1)$ and invariant-modulated skew generator $A_\theta(\|z\|^2)$ ($100.0\%$ test accuracy, reverse-time error $1.17\times 10^{-14}$, norm drift $1.78\times 10^{-15}$).}
\label{fig:phys_anomaly_iso}
\end{figure}

\subsection{Zero-Reflection Anomaly Radar with Unconstrained Neural Residuals}
\label{app:exp_anomaly_radar}

Under Method~1, transporting the link metric backward along a known conservative background $H_0$ ($G_k=J_{0,k}^\top G_{k+1} J_{0,k}$) cancels bulk scattering identically ($R_L\equiv 0$ by \eqref{eq:discrete_orthogonal_condition}): any non-zero backward wave reflection in the transported frame $\Theta_{k+1}$ therefore acts as an exact, spatially localized detector of unmodeled physics. We test this on a 2D eccentric Kepler orbit ($e=0.45$, $300$ steps, $\Delta t=0.04$) perturbed by an unknown radial quadrupole force $\alpha/r^4$ and perihelion atmospheric drag $-\gamma_0 e^{-(r-r_p)/\ell}v$, observed only through noisy positions ($\sigma=1.5\times 10^{-3}$).

Along the known Newtonian background, maximum bulk reflection is $5.40\times 10^{-10}$, while the transported metric-wave reflection amplitude lights up exclusively at perihelion passages (Figure~\ref{fig:phys_anomaly_iso}a,b). Without assuming any radial or drag functional template, we fit a general 2D neural vector field $f_\theta(q,v)\in\RR^2$ ($36$ hidden units) directly to the shadowed acceleration defect. Adding the unconstrained neural residual reduces open-loop orbital RMSE by $17.3\times$ ($0.3159\to 0.0183$) and drops mean metric reflection by $4.2\times$. Projecting the learned 2D field post-hoc onto $(\hat r,-v)$ across radial bins recovers both the unknown quadrupole force $F_r(r)$ and exponential drag profile $\gamma(r)$ (Figure~\ref{fig:phys_anomaly_iso}c).

\subsection{Unified $O(1)$-Memory Augmented Isometric Cayley Neural ODEs}
\label{app:exp_isometric_ode}

Reverse-time adjoint integration in Neural ODEs \citep{chen2018neural} avoids storing activation tapes only if the forward flow is numerically reversible; generic neural vector fields possess contracting modes whose reverse-time integration explodes exponentially. Under Method~1, restricting the layer generator to the skew-symmetric Lie algebra $\mathfrak{so}(d+1)$ via Cayley steps on an augmented space $\RR^{d+1}$ \citep{dupont2019augmented},
\begin{equation}
\dot z = A_\theta(\|z\|^2, t)\,z, \qquad A_\theta^\top = -A_\theta,
\label{eq:invariant_skew_waveguide}
\end{equation}
conserves $\|z(t)\|^2$ identically as an exact Noether invariant ($\frac{\dd}{\dd t}\|z\|^2 = 2z^\top A_\theta z\equiv 0$). Consequently, the Cayley step $z_{k+1}=(I-\tfrac{\Delta t}{2}A(\|z_k\|^2,t_k))^{-1}(I+\tfrac{\Delta t}{2}A(\|z_k\|^2,t_k))z_k$ preserves $\|z_{k+1}\|^2=\|z_k\|^2$ to machine precision, making its exact inverse trivial in $O(1)$ memory by evaluating $A(\|z_{k+1}\|^2,t_k)$ without saving activations.

We evaluate this single unified $SO(d+1)$ architecture on two tasks (Figure~\ref{fig:phys_anomaly_iso}d,e). First, across $n=5$ seeds in $d=8$ ($SO(9)$, up to $T=80$, $1{,}600$ steps), unconstrained reverse integration diverges to $10^6$, a reversible symplectic Verlet ODE baseline reaches $1.72\times 10^{-12}$, and our augmented isometric Cayley ODE holds $4.72\times 10^{-14}$ state error and $4.61\times 10^{-14}$ relative gradient error with $27.3\,\mathrm{KB}$ measured peak memory (\texttt{tracemalloc}) versus $14.07\,\mathrm{MB}$ taped backpropagation ($515\times$ reduction). Second, on 2D concentric rings ($n=160$), lifting inputs projectively via $z(0)=(x_1,x_2,1)\in\RR^3$ places inner and outer rings on distinct invariant spheres ($\|z\|^2=r^2+1$), allowing the \emph{exact same} $SO(3)$ Cayley ODE \eqref{eq:invariant_skew_waveguide} to rotate them into opposite hemispheres: test accuracy reaches $100.0\%$ with maximum reverse-time error $1.17\times 10^{-14}$ and norm drift $1.78\times 10^{-15}$.

\begin{figure}[t]
\centering
\includegraphics[width=\textwidth]{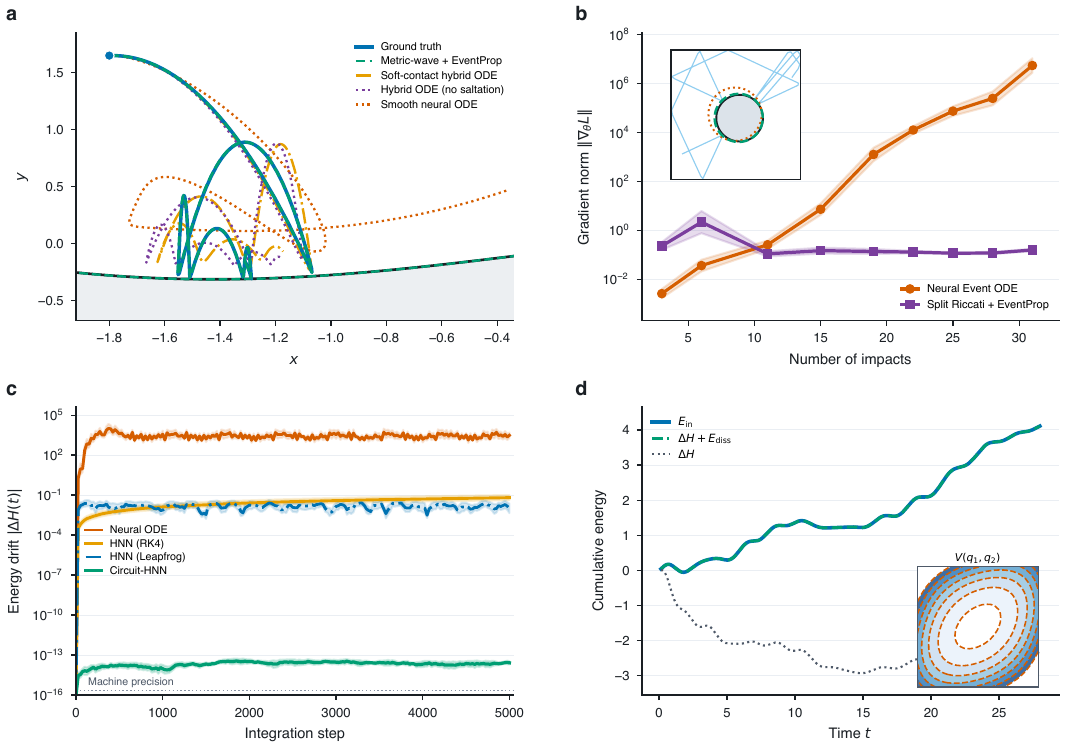}
\caption{\textbf{Hybrid impact differentiation and exact energy conservation in passive Hamiltonian circuits.}
\textbf{a},~Held-out five-bounce rollout on an unknown curved terrain (shaded grey) comparing ground truth (solid blue) with four models trained on identical noisy trajectories ($n=4$): metric-wave scattering with Hybrid jump junctions (dashed teal; rollout RMSE $0.0011$, terrain RMSE $0.0012$, Hybrid gradient cosine $1.0000000$), differentiable soft-contact hybrid ODE (dash-dotted amber; $0.0989$), hard-impact hybrid ODE without saltation terms (dotted purple; $0.145$), and a trajectory-trained BPTT smooth neural ODE (fine-dotted vermilion; $0.654$).
\textbf{b},~Parameter gradient norm versus impact count in a chaotic 2D Sinai billiard ($31$ bounces including $9$ convex-disk collisions; $n=5$ seeds, $y_{\mathrm{obs}}$-only sweep). Neural Event ODE \citep{chen2021learning} gradients grow to $(5.53\pm 5.36)\times 10^6$, whereas split Riccati shadowing with saltation junctions stays bounded at $0.157\pm 0.019$. Inset: scatterer recovery from identical biased initialization (shaded grey, ground truth; dashed teal, Split Riccati + Hybrid, error $0.040\pm 0.033$; dotted vermilion, full-horizon NEODE, $0.106\pm 0.007$; curriculum NEODE achieves $0.070\pm 0.004$).
\textbf{c},~Energy drift $|\Delta H(t)|$ over $5{,}000$ unforced steps for a learned 4D coupled anharmonic oscillator ($n=5$ ICs) comparing a Neural ODE ($3.35\times 10^4$), RK4 HNN ($1.11\times 10^{-1}$, $32.1\,\mu\mathrm{s}$/step), symplectic St\"ormer--Verlet Leapfrog HNN ($4.84\times 10^{-2}$, $14.7\,\mu\mathrm{s}$/step), and the Circuit-HNN ($6.57\times 10^{-14}$, $436.7\,\mu\mathrm{s}$/step, $18.0$ fixed-point iterations).
\textbf{d},~Cumulative discrete energy accounting of the Circuit-HNN under damping ($\gamma=0.25$) and external forcing: supplied work $E_{\mathrm{in}}$ (solid blue) matches $\Delta H + E_{\mathrm{diss}}$ (dashed teal) to $3.68\times 10^{-15}$. Inset: true (filled blue) vs.\ learned (dashed vermilion) potential contours ($V_{\mathrm{RMSE}}=0.0279$).}
\label{fig:phys_hybrid_circuit}
\end{figure}

\subsection{Hybrid Impact Differentiation: Metric-Wave Hybrid Dynamical Systems and Chaotic Billiards}
\label{app:exp_hybrid_impacts}

Physical systems with rigid impacts evolve smoothly until hitting a guard surface $g_\theta(x)=0$, where the state undergoes an instantaneous reset $x^+=\phi_\theta(x^-)$. Differentiating through the crossing time $t^\star$ requires the exact saltation Jacobian:
\begin{equation}
J_{\mathrm{salt}} = D\phi + \frac{(f^+ - D\phi\,f^-)\,\nabla_x g^\top}{\nabla_x g^\top f^-},
\label{eq:eventprop_saltation_app}
\end{equation}
which acts as a localized scattering junction inside our chain (Method~1) or inside each Riccati shadowing step (Method~2). 

\paragraph{Curved-terrain bouncing ball.}
From $n=4$ noisy trajectories of a ball bouncing on an unknown quartic terrain $h_\theta(x)$ with unknown restitution $e=0.80$, gravity $g=9.81$, and drag $c_d=0.10$, Metric-Wave scattering with exact Hybrid junctions (gradient cosine $1.0000000$ vs.\ finite differences) recovers $(e,g,c_d)=(0.800,9.82,0.099)$ and terrain RMSE $0.0012$, achieving five-bounce held-out rollout RMSE $0.0011$ (Figure~\ref{fig:phys_hybrid_circuit}a). By comparison, a differentiable soft-contact hybrid ODE reaches $0.0989$ ($88.6\times$ worse), an ablation removing the saltation rank-1 jump term reaches $0.1451$ ($130\times$ worse), and a trajectory-trained BPTT smooth Neural ODE reaches $0.6535$ ($585\times$ worse).

\paragraph{Chaotic 31-bounce Sinai billiard.}
While Neural Event ODEs \citep{chen2021learning} compute exact hybrid gradients via single-shooting adjoints, convex scatterers make hybrid trajectories strongly chaotic, causing single-shooting event gradients to blow up exponentially. On a 2D Sinai billiard with $31$ impacts ($9$ convex-disk bounces) across $n=5$ noise seeds initialized strictly from noisy positions, Neural Event ODE gradients grow to $(5.53\pm 5.36)\times 10^6$, whereas Split-Riccati shadowing with saltation junctions stays bounded at $0.157\pm 0.019$ (Figure~\ref{fig:phys_hybrid_circuit}b). Recovering the circular obstacle $(c_x,c_y,R)$ from a biased initialization yields parameter error $0.040\pm 0.033$ for Split Riccati + Hybrid, improving over both full-horizon Neural Event ODE ($0.106\pm 0.007$) and a four-stage horizon-curriculum Neural Event ODE ($55\to 120\to 200\to 340$ steps, $0.070\pm 0.004$).

\subsection{Circuit-HNNs: Exact Energy Conservation and Tellegen Power Balance}
\label{app:exp_circuit_hnn}

Standard Hamiltonian Neural Networks \citep{greydanus2019hamiltonian} trained via nested automatic differentiation and integrated with RK4 accumulate secular energy drift; symplectic integrators such as St\"ormer--Verlet (Leapfrog) bound energy oscillations to $O(\Delta t^2)$ on a shadow Hamiltonian rather than conserving the exact learned $H_\theta$ for non-quadratic potentials. Combining an orthogonal-factorized scalar potential $V_\theta(q)$ (analytic gradient without nested autograd) with an Itoh--Abe / Gonzalez discrete-gradient port \citep{gonzalez1996time} yields a Circuit-HNN satisfying exact discrete energy conservation when unforced and exact discrete Tellegen power balance $\Delta H_k + E_{\mathrm{diss},k} = E_{\mathrm{in},k}$ under damping and external forcing.

Across $n=5$ initial conditions over $5{,}000$ unforced integration steps of a learned 4D coupled anharmonic oscillator (Figure~\ref{fig:phys_hybrid_circuit}c,d), maximum energy drift separates into four clear regimes: Neural ODE diverges to $3.35\times 10^4$, RK4 HNN drifts to $1.11\times 10^{-1}$ ($32.1\,\mu\mathrm{s}$/step), symplectic St\"ormer--Verlet Leapfrog HNN oscillates at $4.84\times 10^{-2}$ ($14.7\,\mu\mathrm{s}$/step), and Circuit-HNN holds $\mathbf{6.57\times 10^{-14}}$ ($V_{\mathrm{RMSE}}=0.0279$, Tellegen balance error $3.68\times 10^{-15}$ under $\gamma=0.25$ damping and forcing). As an explicit compute trade-off, the implicit discrete-gradient fixed-point solve requires $18.0$ iterations ($436.7\,\mu\mathrm{s}$/step), making explicit symplectic HNNs preferable when $10^{-2}$ bounded energy oscillation suffices and Circuit-HNNs preferable whenever exact energy accounting or open-system passivity is critical.

\subsection{Noisy Sparse Symbolic Law Discovery versus SINDy}
\label{app:exp_symbolic_sindy}

Coupling Split-Riccati shadowing (Method~2) to a proximal-$\ell_1$ parameter port enables sparse symbolic equation discovery directly from noisy position-only observations $q(t)$ without numerical differentiation. Because standard sparse regression (\texttt{SINDy} \citep{brunton2016discovering}) estimates $\ddot q$ by differentiating noisy $q(t)$ twice, sensor noise amplifies as $\sigma/\Delta t^2$. Across a ten-term polynomial/damping candidate library on a driven nonlinear oscillator over $\sigma\in\{0.004, 0.015, 0.030, 0.050\}$ ($n=5$ seeds per level, Figure~\ref{fig:phys_sindy_ext}), Savitzky--Golay SINDy degrades rapidly to $0.963\pm 0.136$ max coefficient error and $0\%$ exact support recovery at $\sigma=0.030$, whereas the wave parameter port integrates through the dynamics to achieve $0.040\pm 0.021$ max error ($24.2\times$ lower) and $100\%$ exact support recovery with exact zero-pruning of all five spurious terms.

\begin{figure}[t]
\centering
\includegraphics[width=\textwidth]{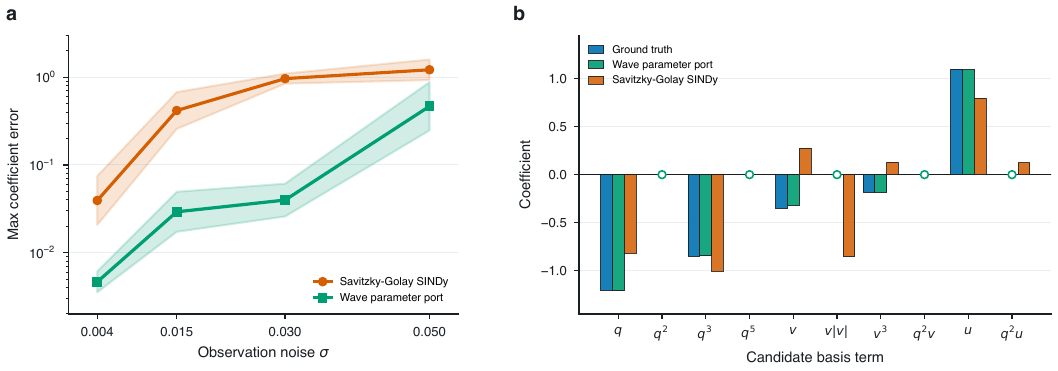}
\caption{\textbf{Noisy sparse symbolic law discovery versus Savitzky--Golay SINDy.}
\textbf{a},~Maximum coefficient error versus position observation noise $\sigma\in[0.004,0.050]$ ($n=5$ seeds per level) on a ten-term candidate library for a driven nonlinear oscillator observed only through noisy positions $q(t)$, comparing Savitzky--Golay SINDy \citep{brunton2016discovering} (vermilion circles) with the proximal-$\ell_1$ wave parameter port (teal squares).
\textbf{b},~Representative ten-term coefficient comparison at $\sigma=0.030$ (blue, ground truth; teal, wave parameter port; vermilion, SINDy; open teal circles mark inactive terms pruned to exact zero). At $\sigma=0.030$, the wave parameter port achieves $100\%$ exact support recovery ($0.040\pm 0.021$ max error) versus $0\%$ exact support ($0.963\pm 0.136$) for SINDy.}
\label{fig:phys_sindy_ext}
\end{figure}

%% file: appendix/M_classical_control_examples.tex
\section{Classical Optimal Control, Differential Games, and Nonlinear Filtering Examples}
\label{app:classical_control_examples}

Section~\ref{sec:variational_worldsheet} introduced the $(1+1)$-dimensional worldsheet relaxation \eqref{eq:para_floer} with a scalar LQR illustration (Figure~\ref{fig:lqr_worldsheet_sec2}; full derivation in Appendix~\ref{app:lqr}). Here we verify that the exact same local wave-scattering formulation solves a broad canon of classical nonlinear optimal-control, bang--bang relay, state-constrained trajectory-optimization, underactuated swing-up, minimax differential-game, and stochastic filtering benchmarks \citep{pontryagin1962mathematical,bryson1969applied,slotine1991applied,kalman1961new,zakai1969optimal}.

\subsection{Nonlinear Mechanics from Slotine \& Li}
\label{app:ex_slotine_trio}

Extending the scalar LQR worldsheet (Figure~\ref{fig:ex_lqr_slotine}a) to nonlinear state dynamics and quadratic or bounded control ports recovers canonical benchmarks from \citet{slotine1991applied}: (i)~a 1D underwater vehicle subject to quadratic hydrodynamic drag $\dot v = u - c|v|v$, steered from rest to $x(T)=1.500$, $v(T)=0.001$ (Figure~\ref{fig:ex_lqr_slotine}b); (ii)~a nonlinear pendulum swung from the downward equilibrium to the inverted upright equilibrium ($\theta(T)=1.011\pi$, $\dot\theta(T)=0.071$, Figure~\ref{fig:ex_lqr_slotine}c); and (iii)~three-axis Euler rigid-body spacecraft detumbling $\mathbf{J}\dot\omega + \omega\times(\mathbf{J}\omega)=u$, driving all three angular-rate axes to $\|\omega(T)\|_2=6.82\times 10^{-4}$ (Figure~\ref{fig:ex_lqr_slotine}d).

\begin{figure}[t]
\centering
\includegraphics[width=\textwidth]{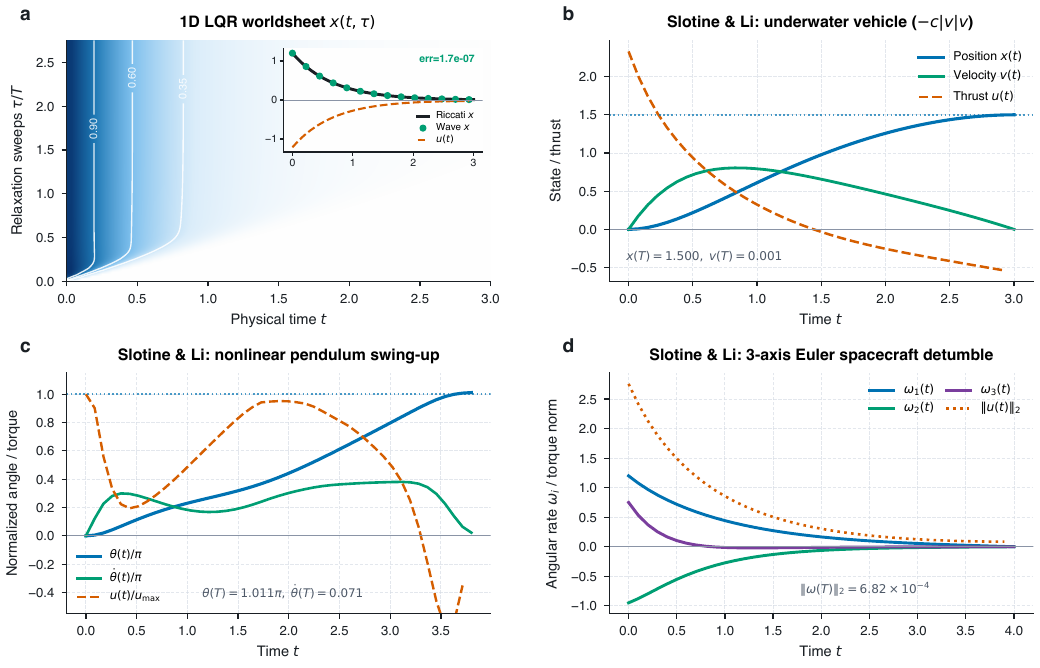}
\caption{\textbf{1D LQR worldsheet and classical nonlinear benchmarks from \citet{slotine1991applied}.}
\textbf{a},~1D LQR worldsheet $x(t,\tau)$ settling to the exact Riccati solution (error $1.7\times 10^{-7}$).
\textbf{b},~Underwater vehicle with quadratic drag $-c|v|v$ steered to $(x(T),v(T))=(1.500,0.001)$.
\textbf{c},~Nonlinear pendulum swing-up to the inverted equilibrium ($\theta(T)=1.011\pi$, $\dot\theta(T)=0.071$).
\textbf{d},~Three-axis Euler rigid-body spacecraft detumble ($\|\omega(T)\|_2=6.82\times 10^{-4}$).}
\label{fig:ex_lqr_slotine}
\end{figure}

\subsection{Pontryagin Minimum-Time Bang--Bang Relays}
\label{app:ex_pontryagin_relays}

When the control port is bounded ($|u(t)|\le 1$) and linear in the Hamiltonian, Pontryagin's maximum principle \citep{pontryagin1962mathematical} prescribes bang--bang switching $u^\star(t)=-\operatorname{sgn}(B^\top\lambda(t))$. On the classical double integrator ($\dot x_1=x_2$, $\dot x_2=u$), wave relaxation converges directly to the two-stage switching parabola $x_1=-\tfrac12\operatorname{sgn}(x_2)x_2^2$ with the costate switching function $-B^\top\lambda(t)$ crossing zero at $t_s\approx 1.00$ (exact $T/2=1.02$, Figure~\ref{fig:ex_pontryagin_relays}a,b). On a generalized $n=4$ integrator chain ($x_1,\dots,x_4$ from position to jerk, Figure~\ref{fig:ex_pontryagin_relays}c,d), the control worldsheet $u(t,\tau)$ spontaneously crystallizes the exact $n-1=3$ switching curves as relaxation sweeps advance, driving the terminal state to $\|x(T)\|_2=5.69\times 10^{-3}$.

\begin{figure}[t]
\centering
\includegraphics[width=\textwidth]{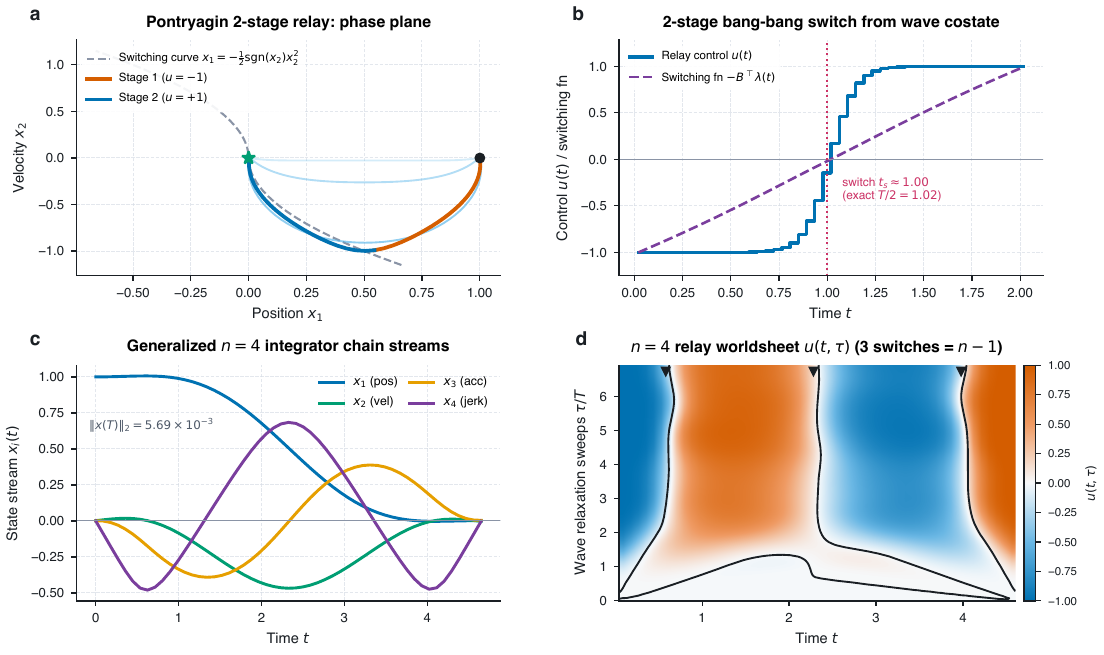}
\caption{\textbf{Pontryagin minimum-time bang--bang relays \citep{pontryagin1962mathematical}.}
\textbf{a},~Phase-plane trajectories of the 2-stage double-integrator relay converging onto the exact switching curve $x_1=-\tfrac12\operatorname{sgn}(x_2)x_2^2$.
\textbf{b},~Bang--bang control $u(t)$ and wave costate switching function $-B^\top\lambda(t)$ switching at $t_s\approx 1.00$ (exact $T/2=1.02$).
\textbf{c},~Generalized $n=4$ integrator chain state streams converging to $\|x(T)\|_2=5.69\times 10^{-3}$.
\textbf{d},~Control worldsheet $u(t,\tau)$ showing the emergence of the $n-1=3$ switching manifolds across relaxation sweeps $\tau/T$.}
\label{fig:ex_pontryagin_relays}
\end{figure}

\subsection{Bryson \& Ho Geometric and State-Constrained Problems}
\label{app:ex_bryson_ho}

Figure~\ref{fig:ex_bryson_ho} evaluates five canonical trajectory-optimization problems from \citet{bryson1969applied}: (a)~Zermelo's river-crossing navigation in a shear current; (b)~the Brachistochrone minimum-time descent, recovering the exact cycloid; (c)~minimum-time orbit injection via bilinear tangent thrust steering to unit orbital altitude $h=1$; (d)~the Bryson--Denham state-constrained problem ($x(t)\le\ell$ for $\ell\in\{1/6,1/9,1/12\}$), where state-constraint shunts enforce exact boundary contact arcs; and (e)~minimum-fuel soft lunar landing ($0\le T\le T_{\max}$), matching the analytical bang--bang switch time $t_s=1.78\,\mathrm{s}$. Across all five problems, normalized Wave-PMP updates decay monotonically over $20\text{--}30$ outer wave-scattering steps (Figure~\ref{fig:ex_bryson_ho}f).

\begin{figure}[t]
\centering
\includegraphics[width=\textwidth]{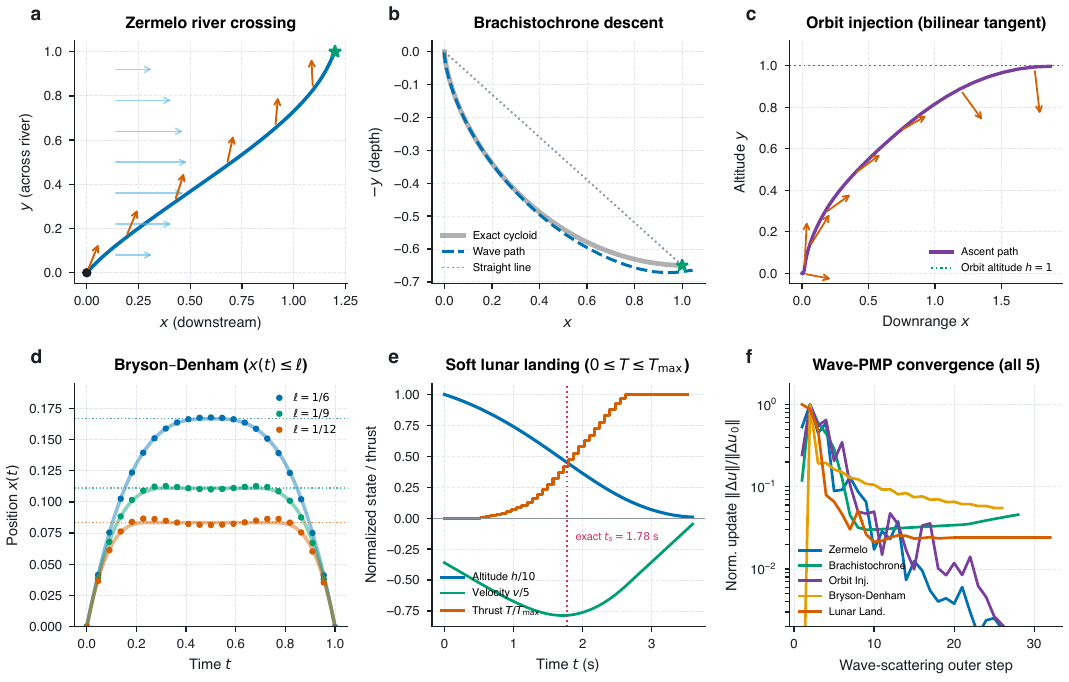}
\caption{\textbf{Classical geometric and constrained optimal-control problems from \citet{bryson1969applied}.}
\textbf{a},~Zermelo river crossing in a shear flow.
\textbf{b},~Brachistochrone minimum-time descent matching the exact cycloid.
\textbf{c},~Orbit injection via bilinear tangent steering to altitude $h=1$.
\textbf{d},~Bryson--Denham state-constrained trajectories ($x(t)\le\ell$ for $\ell\in\{1/6,1/9,1/12\}$) matching analytical contact arcs.
\textbf{e},~Soft lunar landing ($0\le T\le T_{\max}$) capturing the exact thrust switch $t_s=1.78\,\mathrm{s}$.
\textbf{f},~Normalized Wave-PMP convergence across all five problems.}
\label{fig:ex_bryson_ho}
\end{figure}

\subsection{Underactuated Swing-Up, Minimax Differential Games, and Interplanetary Transfer}
\label{app:ex_swingup_games}

Figure~\ref{fig:ex_swingup_games} illustrates four nonlinear multi-stage and adversarial benchmarks: (a)~underactuated Mountain Car energy-pumping swing-up \citep{moore1990efficient} to the goal summit $x(T)=0.48$; (b)~pure Wave-PMP four-state Acrobot swing-up \citep{spong1995swing} ($\theta_1(T)=0.977\pi$, $\theta_2(T)=0.024\pi$), where strong underactuation and non-convexity require locally linearized Schur-complement scattering cells and a Riccati-matched terminal impedance $S_{\mathrm{term}}=(I+K_T)^{-1}(I-K_T)$ with Levenberg--Marquardt control-port regularization; (c)~the Bryson--Ho / Isaacs pursuit--evasion minimax differential game \citep{isaacs1965differential,bryson1969applied}, where the evader enters as a saddle-point anti-dissipative control port and terminal miss grows monotonically as the evader-to-pursuer speed ratio $s_e/s_p$ increases from $0.15$ to $0.75$; and (d)~minimum-time low-thrust Earth--Mars orbital transfer ($1.00\,\mathrm{AU}\to 1.524\,\mathrm{AU}$) \citep{bryson1969applied}.

\begin{figure}[t]
\centering
\includegraphics[width=\textwidth]{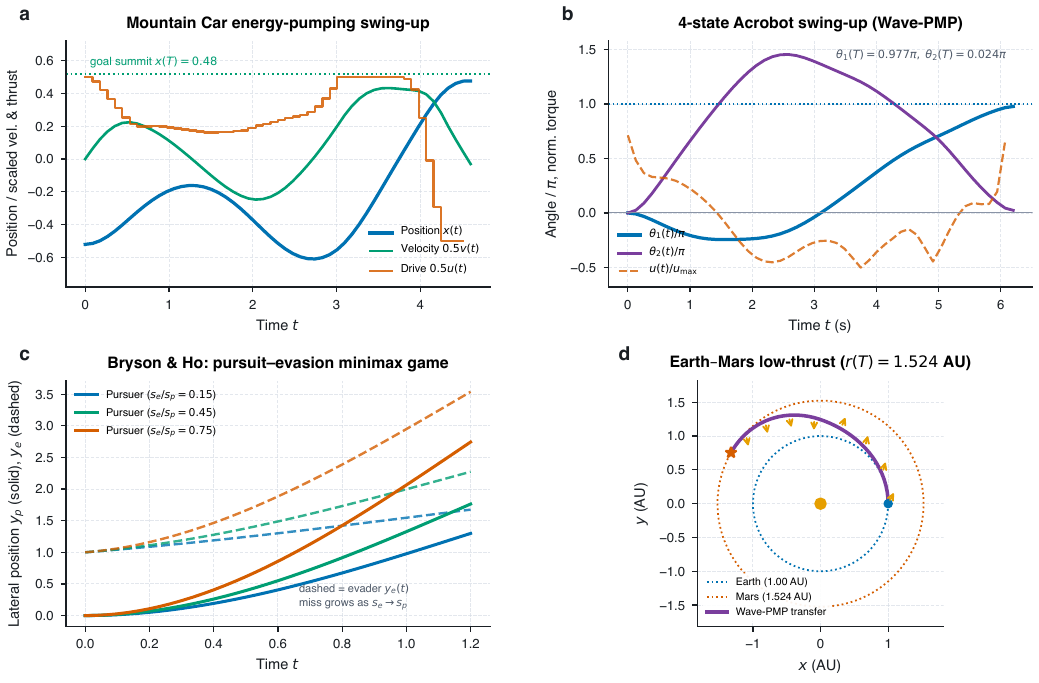}
\caption{\textbf{Underactuated swing-up, minimax differential games, and low-thrust orbital transfer.}
\textbf{a},~Mountain Car energy-pumping swing-up \citep{moore1990efficient} reaching summit $x(T)=0.48$.
\textbf{b},~Pure Wave-PMP four-state Acrobot swing-up \citep{spong1995swing} to the upright equilibrium ($\theta_1(T)=0.977\pi$, $\theta_2(T)=0.024\pi$) using local Schur-complement scattering cells and a Riccati-matched terminal boundary impedance.
\textbf{c},~Pursuit--evasion minimax differential game \citep{isaacs1965differential,bryson1969applied} across speed ratios $s_e/s_p\in\{0.15,0.45,0.75\}$.
\textbf{d},~Low-thrust Earth--Mars orbital transfer ($1.00\to 1.524\,\mathrm{AU}$).}
\label{fig:ex_swingup_games}
\end{figure}

\subsection{Fully Unlocked Directional Wave Scattering on the Four-State Acrobot}
\label{app:ex_acrobot_directional}

To verify the continuous-time Riccati-aligned directional paracomplex involution $\mathbb K_{P,0}$ with Maxwell-in-$\tau$ memory derived in Appendix~\ref{app:directional_riccati_waves} \eqref{eq:directional_normal_form_app}, we solve the four-state underactuated Acrobot swing-up ($N=40$, $T=6.2\,\mathrm{s}$, smooth Legendre-dual torque bound $|u|\le 4$) \emph{completely unlocked} from dead rest ($x\equiv 0$, $u\equiv 0$)---with zero outer SQP loops, zero sequential forward or backward rollouts, and zero global line search (\texttt{scripts/acrobot\_directional\_waves.py}).
Figure~\ref{fig:ex_acrobot_directional} displays all five worldsheet fields and exact unbiased Pontryagin convergence:
(a--c)~the right-moving state wave $q_1(t,\tau)$ ($+c$) and left-moving directional costate and Riccati admittance waves $(w_-,P)$ ($-c$) collide from the two boundaries over the first $\sim 350$ sweeps before locking sharply into the upright swing-up basin;
(d)~the smooth actuator torque profile $u(t,\tau)$ emerges directly from $u\equiv 0$;
(e)~all unbiased nonlinear Pontryagin residuals decay exponentially across nine orders of magnitude to $r_x=6.79\times 10^{-8}$, $r_\lambda=3.73\times 10^{-6}$, and $H_u=2.14\times 10^{-8}$ over $4500$ Red--Black sweeps, whereas disabling the local Maxwell-in-$\tau$ memory ($\mu_x=\mu_u=0$, dotted ablation) stalls at $r_x=3.62\times 10^{1}$; and
(f)~an independent open-loop RK2 rollout driven solely by the converged torque reaches $(q_1(T),q_2(T))=(+0.9966\pi,-0.0014\pi)$ ($\|x_{\mathrm{ro}}(T)-r_f\|_2=0.0365$) and hands off seamlessly into an upright LQR hold that recovers cleanly from a shoulder velocity disturbance kick.

\begin{figure}[t]
\centering
\includegraphics[width=\textwidth]{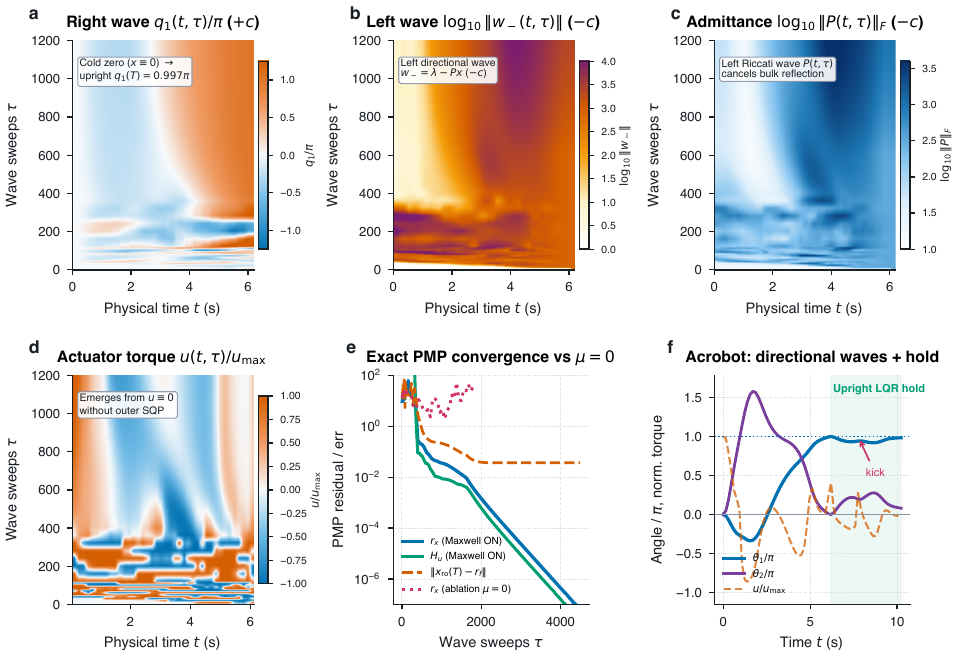}
\caption{\textbf{Fully unlocked Riccati-aligned directional wave scattering ($\mathbb K_{P,0}$ + Maxwell-in-$\tau$ memory) on the four-state Acrobot from rest ($x\equiv 0$, $u\equiv 0$).}
\textbf{a},~Right-moving state wave $q_1(t,\tau)/\pi$ ($+c$) locking into the upright basin around sweep $350$.
\textbf{b},~Left-moving directional costate wave $\log_{10}\|w_-(t,\tau)\|$ ($-c$, $w_-=\lambda-Px$).
\textbf{c},~Left-moving Riccati admittance wave $\log_{10}\|P(t,\tau)\|_F$ ($-c$) canceling bulk reflection.
\textbf{d},~Actuator torque worldsheet $u(t,\tau)/u_{\max}$ emerging from $u\equiv 0$ without outer SQP or sequential rollouts.
\textbf{e},~Exact unbiased nonlinear PMP residuals ($r_x=6.79\times 10^{-8}$, $H_u=2.14\times 10^{-8}$) and open-loop terminal error $\|x_{\mathrm{ro}}(T)-r_f\|_2=0.0365$, contrasted with the $\mu=0$ Maxwell-memory ablation (dotted, stalling at $r_x=36.2$).
\textbf{f},~Independent open-loop RK2 swing-up trajectory handing off into a closed-loop upright LQR hold with disturbance kick recovery.}
\label{fig:ex_acrobot_directional}
\end{figure}

\subsection{Kalman--Bucy Smoothing, Stratonovich--Zakai Filtering, and Collective Worldsheet Filters}
\label{app:ex_filtering}

Dualizing control ports into noisy observation shunts maps Wave-PMP directly onto continuous-time optimal smoothing and filtering. On linear-Gaussian dynamics, the two-point boundary-value problem is identically the Bryson--Frazier / Kalman--Bucy smoother \citep{kalman1961new,bryson1963smoothing}, achieving lower RMSE ($0.37$) than causal Kalman filtering ($0.49$) while recovering the $95\%$ posterior smoother band (Figure~\ref{fig:ex_filtering_bucy}a). On non-closing nonlinear filtering problems where no finite-dimensional sufficient statistic exists \citep{stratonovich1960conditional,zakai1969optimal,hazewinkel1983nonexistence}, comparing single-trajectory Onsager--Machlup MAP waves \citep{onsager1953fluctuations} against the full Zakai posterior density isolates three fundamental nonlinear regimes (Figure~\ref{fig:ex_filtering_bucy}b--d): (i)~the cubic sensor $h(x)=x^3$ (infinite Weyl algebra), whose blind zone $h'(0)=0$ skews the posterior; (ii)~the double-well $+x^3$ diffusion, where the Stratonovich divergence term $+\tfrac12 f'(x)$ shifts the effective Onsager--Machlup barrier; and (iii)~the bimodal quadratic sensor $h(x)=x^2$, where a single MAP trajectory locks onto one branch while the exact Zakai density splits cleanly into two modes.

\begin{figure}[t]
\centering
\includegraphics[width=\textwidth]{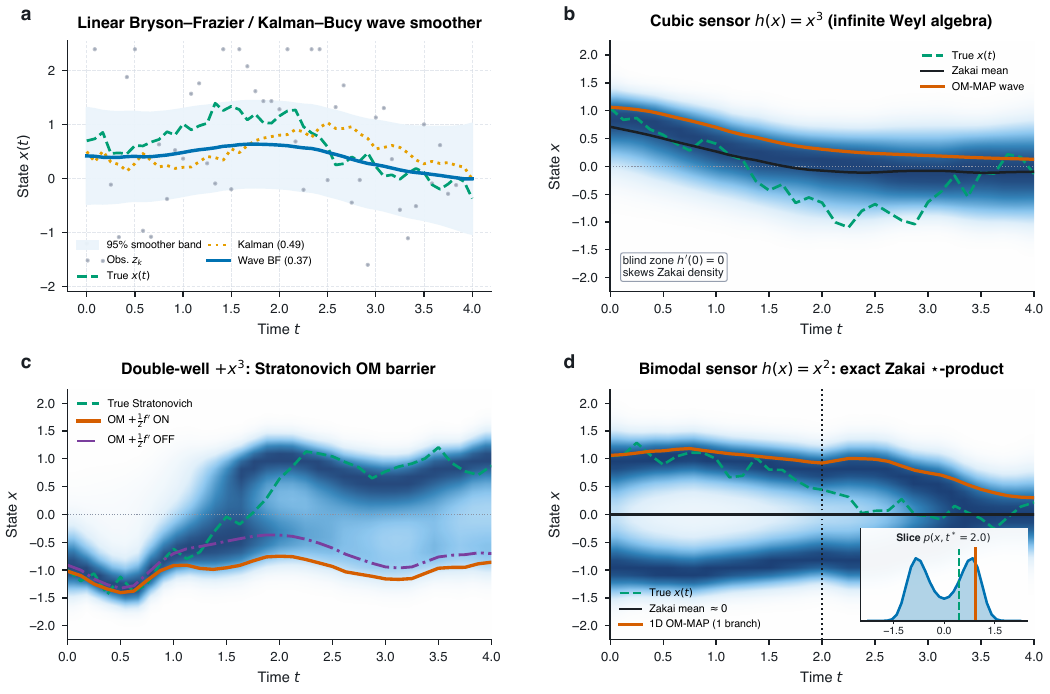}
\caption{\textbf{Linear Kalman--Bucy / Bryson--Frazier smoothing and non-closing Stratonovich--Zakai nonlinear filtering.}
\textbf{a},~Linear Bryson--Frazier wave smoother \citep{kalman1961new,bryson1963smoothing} (RMSE $0.37$ vs.\ causal Kalman $0.49$) with $95\%$ smoother band.
\textbf{b},~Cubic sensor $h(x)=x^3$ \citep{hazewinkel1983nonexistence} comparing exact Zakai density and mean \citep{zakai1969optimal} with the 1D Onsager--Machlup MAP wave \citep{onsager1953fluctuations}.
\textbf{c},~Double-well $+x^3$ diffusion showing the Stratonovich $+\tfrac12 f'$ barrier contribution \citep{stratonovich1960conditional}.
\textbf{d},~Bimodal quadratic sensor $h(x)=x^2$: exact Zakai bimodal posterior slice at $t^\star=2.0$ versus single-branch 1D OM-MAP.}
\label{fig:ex_filtering_bucy}
\end{figure}

Finally, lifting a single particle estimate to a $(1+1)$-dimensional collective field $x(s,t)$ coupled along an ensemble coordinate $s\in[0,1]$ resolves both blind-zone trapping and multimodal collapse within one wave medium (Figure~\ref{fig:ex_heckman_collective}): on the cubic sensor $h(x)=x^3$, outer agents pull blind-zone agents ($|x_0|<0.45$) onto the true trajectory; on the bimodal sensor $h(x)=x^2$, weak KL-pooling stiffness $\kappa$ sustains a stable topological domain wall (kink at $x=0$) preserving both posterior modes, undergoing a sharp crossover to single-branch consensus collapse above a critical stiffness $\kappa_c\approx 10^{-1}$ under both flat ($\gamma=0$) and curved Fisher ($\gamma=0.8$) metrics.

\begin{figure}[t]
\centering
\includegraphics[width=\textwidth]{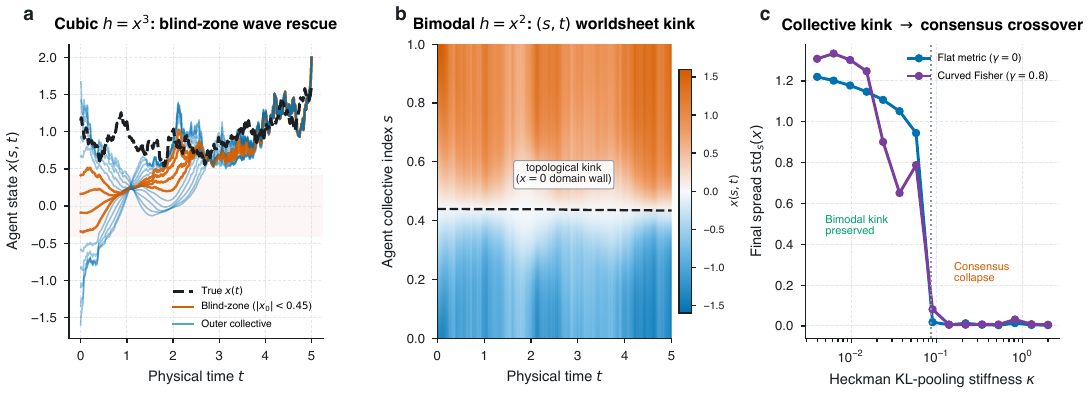}
\caption{\textbf{$(1+1)$-dimensional collective wave filtering across blind zones and bimodal posteriors.}
\textbf{a},~Cubic sensor $h(x)=x^3$: wave coupling across the collective index $s$ rescues agents initialized inside the $h'(0)=0$ blind zone ($|x_0|<0.45$, vermilion).
\textbf{b},~Bimodal sensor $h(x)=x^2$: collective $(s,t)$ worldsheet sustaining a topological kink ($x=0$ domain wall) separating the two posterior branches.
\textbf{c},~Final ensemble spread $\operatorname{std}_s(x)$ versus KL-pooling stiffness $\kappa$, showing the sharp transition from bimodal kink preservation to consensus collapse under flat ($\gamma=0$) and curved Fisher ($\gamma=0.8$) metrics.}
\label{fig:ex_heckman_collective}
\end{figure}

%% file: appendix/N_text_flow_language_models.tex
\section{Text Flow Language Models: Exact Vocabulary Compression and Wave-Variable Extensions}
\label{app:text_flow_language_models}

Continuous-time flow matching \citep{lipman2022flow} and flow map matching \citep{boffi2024flow} have recently been extended directly to discrete text in the Flow Language Model (FLM) and Flow Map Language Model (FMLM) framework of \citet{lee2026flow}, which embeds one-hot tokens in $\RR^V$, adds Gaussian noise along a straight interpolation path, and trains a Transformer denoiser to predict clean token probabilities. Building directly on the FLM/FMLM formulation and architecture of \citet{lee2026flow}, in this appendix we study this setting at matched $80\text{M}$--$170\text{M}$ scale and deliberately separate \emph{two independent contributions}:
\begin{enumerate}
\item \textbf{Exact vocabulary compression with variance-matched adaptive training} (Section~\ref{app:flm_compression}), a standalone mathematical reformulation of the input space and its AdamW second moment that applies to \emph{any} standard Transformer flow language model with no wave-variable machinery; and
\item \textbf{Wave-variable extensions} (Section~\ref{app:flm_wave_extensions}), which add orthogonal-factorized layers (trained without storing autograd graphs inside Transformer blocks), analytic forward-mode time derivatives for few-step distillation, and constant-memory forward--backward trajectory control for inference-time steering.
\end{enumerate}

\subsection{Setup: Continuous Flow Matching over Large Vocabularies}
\label{app:flm_setup}

Let $x_0 = (x_{0,1},\dots,x_{0,L}) \in \{1,\dots,V\}^L$ denote a sequence of $L$ tokens from a vocabulary of size $V$ (here $V = 50{,}257$ for GPT-2 BPE). Representing each token as a one-hot row vector $e_{x_{0,\ell}} \in \RR^V$, a flow language model corrupts the sequence at interpolation time $t \in [0,1]$ via
\begin{equation}
\label{eq:flm_vspace_path}
X_t = t\,E(x_0) + (1-t)\,\varepsilon \in \RR^{L \times V},
\qquad \varepsilon_{ \ell, i} \stackrel{\text{i.i.d.}}{\sim} \mathcal{N}(0,1),
\end{equation}
projects the noisy simplex state into the model's hidden width $d \ll V$ using an embedding matrix $W \in \RR^{V \times d}$,
\begin{equation}
\label{eq:flm_hidden_proj}
h_t = X_t W \in \RR^{L \times d},
\end{equation}
and trains a bidirectional Transformer $P_\theta(h_t, \tau) \in \Delta^{V-1}$ to minimize cross-entropy against $x_0$ \citep{lee2026flow}. Because the signal $\|e_{x_0}\|_2 = 1$ is dwarfed by the $V$-dimensional noise norm $\|\varepsilon\|_2 \approx \sqrt{V}$, almost all information transition occurs inside a narrow window around $t \in [0.6, 0.9]$; following \citet{lee2026flow}, sampling $\tau \sim \mathrm{Unif}[0,1]$ and mapping $\tau \mapsto t(\tau)$ according to the analytical signal-to-noise transition concentrates training and sampling effort where the posterior actually sharpens (Figure~\ref{fig:text_flow_worldsheet}b).

This formulation faces two distinct scalability barriers: (i)~materializing $X_t \in \RR^{B \times L \times V}$ costs $O(BLV)$ memory and compute per step even though the network depends on $X_t$ solely through the $d$-dimensional projection $h_t = X_t W$; and (ii)~standard backpropagation through network depth, higher-order time derivatives for few-step distillation, and backpropagation through $K$ denoising steps for constrained generation all scale linearly in stored intermediate states. We solve these two barriers separately below.

\begin{figure}[t]
\centering
\includegraphics[width=\textwidth]{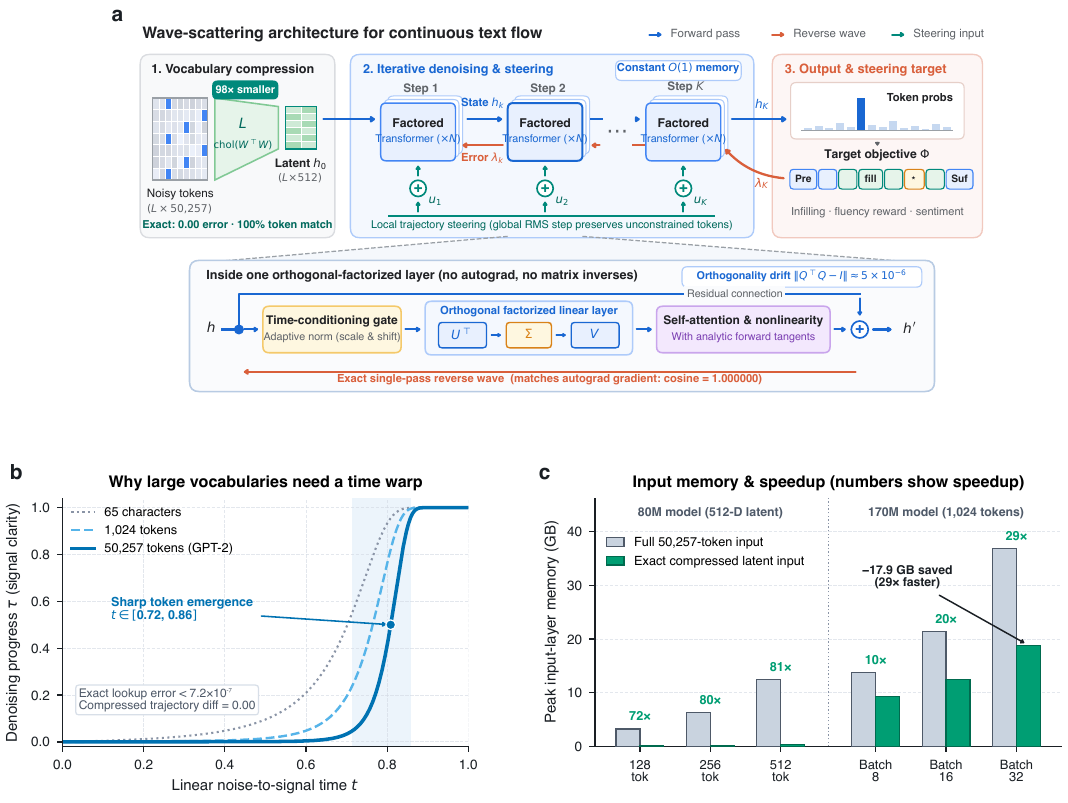}
\caption{\textbf{Text flow matching, vocabulary time warp, and exact input compression.}
\textbf{a},~Overview of continuous text flow generation and the two separable accelerations: compressing the $V$-dimensional noisy input directly into $\RR^d$, and applying forward--backward wave sweeps across depth and denoising steps.
\textbf{b},~Signal-to-noise transition for large vocabularies ($V = 50{,}257$), showing why a time reparameterization $\tau \mapsto t(\tau)$ is needed to resolve the transition window.
\textbf{c},~Measured H100 input-layer memory and wall-clock speedup from exact $d$-dimensional compression ($72\text{--}81\times$ faster and $14\text{--}36\times$ lower input memory for $L=128\text{--}512$).}
\label{fig:text_flow_worldsheet}
\end{figure}

\subsection{Part I: Exact Vocabulary Compression and Variance-Matched Adaptive Training}
\label{app:flm_compression}

The first contribution is completely independent of wave variables and applies verbatim to any standard Transformer trained with flow matching.

\paragraph{Exact $d$-dimensional input and sampling dynamics.}
Conditioned on $x_0$, $t$, and the embedding matrix $W \in \RR^{V \times d}$, the projected state $h_t = X_t W$ is Gaussian with mean $t\,W_{x_0}$ and covariance $(1-t)^2 W^\top W \in \RR^{d \times d}$. Let $L \in \RR^{d \times d}$ be the lower-triangular Cholesky factor of the Gram matrix $\Sigma = W^\top W = L L^\top$ (with a tiny diagonal jitter $\epsilon I_d$, $\epsilon=10^{-6}$, if needed). Then:

\begin{proposition}[Exact distribution and ODE equivalence in $\RR^d$]
\label{prop:flm_dspace_equiv}
Drawing $z \in \RR^{L \times d}$ with i.i.d.\ $\mathcal{N}(0,1)$ entries and forming
\begin{equation}
\label{eq:flm_dspace_corrupt}
h_t^{(d)} = t\,W_{x_0} + (1-t)\,z L^\top \in \RR^{L \times d}
\end{equation}
has the exact same conditional distribution as $h_t = X_t W$ in \eqref{eq:flm_vspace_path}--\eqref{eq:flm_hidden_proj}. Likewise, the flow-matching sampling ODE $\dot X_t = (P_\theta(X_t W, t) - X_t)/(1-t)$ initialized at $X_0 \sim \mathcal{N}(0, I_V)$ projects identically onto the closed $d$-dimensional ODE
\begin{equation}
\label{eq:flm_dspace_ode}
h_{t'} = \frac{1-t'}{1-t}\,h_t + \frac{t'-t}{1-t}\,P_\theta(h_t, t)\,W,
\qquad h_0 = z_0 L^\top,\quad z_0 \sim \mathcal{N}(0, I_d),
\end{equation}
without ever instantiating a $V$-dimensional state tensor.
\end{proposition}

\begin{proof}
For each position $\ell$, $(1-t) \varepsilon_\ell W$ is zero-mean Gaussian with covariance $(1-t)^2 W^\top \mathbb{E}[\varepsilon_\ell^\top \varepsilon_\ell] W = (1-t)^2 W^\top W = (1-t)^2 L L^\top$, which coincides with the covariance of $(1-t) z_\ell L^\top$. Right-multiplying the Euler/exact linear ODE step $X_{t'} = \frac{1-t'}{1-t} X_t + \frac{t'-t}{1-t} P_\theta(X_t W, t)$ by $W$ yields \eqref{eq:flm_dspace_ode} identically.
\end{proof}

On H100 hardware ($V=50{,}257$, $d=512$, batch size $32$), \eqref{eq:flm_dspace_corrupt} matches the full-vocabulary covariance to $2.15 \times 10^{-7}$ relative error, produces identical ODE token trajectories ($0.0$ max difference, $100\%$ token match), and cuts the input step from $8.69\text{--}32.97\,\mathrm{ms}$ ($3.34\text{--}12.76\,\mathrm{GB}$) down to $0.12\text{--}0.41\,\mathrm{ms}$ ($0.24\text{--}0.36\,\mathrm{GB}$) for $L=128\text{--}512$---a $72\text{--}81\times$ input speedup and $14\text{--}36\times$ input memory reduction (Figure~\ref{fig:text_flow_worldsheet}c).

\paragraph{Why naive end-to-end compressed training fails: the Adam second-moment mismatch.}
Although \eqref{eq:flm_dspace_corrupt} is distributionally exact in the forward pass, differentiating directly through \eqref{eq:flm_dspace_corrupt} with Adam or AdamW \citep{kingma2015adam,loshchilov2019decoupled} causes a surprising failure mode: validation cross-entropy \emph{drops} ($0.933$ vs.\ $1.190$ nats for full-$V$ training on 80M TinyStories), yet generated text perplexity \emph{worsens} sharply ($248.4$ vs.\ $153.5$ at 64 steps; Figure~\ref{fig:text_flow_scaling}d). The cause is a clean first-moment versus second-moment discrepancy in the embedding gradient.

Let $g_{b,\ell} \equiv \partial \mathcal{L} / \partial h_{t,b,\ell} \in \RR^d$ denote the upstream loss gradient at batch index $b$ and sequence position $\ell$. In full $V$-space, differentiating $h_{t,b,\ell} = t_{b,\ell} W_{x_{0,b,\ell}} + (1-t_{b,\ell}) \sum_{i=1}^V \varepsilon_{b,\ell,i} W_i$ gives the gradient for embedding row $i \in \{1,\dots,V\}$:
\begin{equation}
\label{eq:flm_vspace_grad}
g_{W,i}^{(V)}
= \underbrace{\sum_{(b,\ell):\, x_{0,b,\ell}=i} t_{b,\ell}\,g_{b,\ell}}_{g_{W,i}^{\mathrm{sig}}}
+ \underbrace{\sum_{b,\ell} (1-t_{b,\ell})\,\varepsilon_{b,\ell,i}\,g_{b,\ell}}_{g_{W,i}^{\mathrm{noise}}}.
\end{equation}
In compressed $d$-space \eqref{eq:flm_dspace_corrupt}, the random noise is drawn in $\RR^d$ rather than $\RR^V$, so differentiating through the Cholesky factor $L$ yields
\begin{equation}
\label{eq:flm_dspace_grad}
g_{W,i}^{(d)} = g_{W,i}^{\mathrm{sig}} + g_{W,i}^{\mathrm{cov}},
\end{equation}
where $g_{W,i}^{\mathrm{cov}} = \frac{\partial \mathcal{L}}{\partial L}\frac{\partial L}{\partial W_i}$ propagates gradients through $\Sigma = W^\top W$.

\begin{proposition}[Exact first-moment equivalence and $O(BLd)$ second-moment correction]
\label{prop:flm_vmatched_adam}
Conditioned on $W$ and the mini-batch $(x_0, t)$, the compressed gradient \eqref{eq:flm_dspace_grad} is an unbiased, lower-variance estimator of the full-vocabulary gradient \eqref{eq:flm_vspace_grad}:
\begin{equation}
\label{eq:flm_first_moment}
\mathbb{E}\bigl[g_{W}^{(d)}\bigr] = \mathbb{E}\bigl[g_{W}^{(V)}\bigr].
\end{equation}
However, their entrywise second moments differ by the variance of the zero-mean vocabulary noise term $g_{W,i}^{\mathrm{noise}}$. Under weak single-coordinate dependence between $\varepsilon_{b,\ell,i}$ and $g_{b,\ell}$ (which holds to $O(1/V)$ for $V \gg d$), the full-vocabulary entrywise second moment is
\begin{equation}
\label{eq:flm_second_moment}
\mathbb{E}\!\left[\bigl(g_{W,i,j}^{(V)}\bigr)^2\right]
\approx \bigl(g_{W,i,j}^{\mathrm{sig}}\bigr)^2 + v_{\mathrm{noise},j},
\qquad
v_{\mathrm{noise},j} \equiv \sum_{b,\ell} (1-t_{b,\ell})^2\,g_{b,\ell,j}^2,
\end{equation}
and the expected squared Frobenius norm of the full-vocabulary embedding gradient is
\begin{equation}
\label{eq:flm_global_norm}
\mathbb{E}\!\left[\|g_{W}^{(V)}\|_F^2\right]
\approx \|g_W^{\mathrm{sig}}\|_F^2 + V \sum_{j=1}^d v_{\mathrm{noise},j}.
\end{equation}
Both $v_{\mathrm{noise}} \in \RR^d$ and \eqref{eq:flm_global_norm} require only $O(BLd)$ arithmetic and zero $L \times V$ memory.
\end{proposition}

\begin{proof}
Because $h_t^{(d)} \stackrel{d}{=} h_t^{(V)}$ for every $W$, the expected loss $\mathbb{E}[\mathcal{L}(W)]$ is identical under both parameterizations; differentiating under the expectation gives $\mathbb{E}[g_W^{(d)}] = \nabla_W \mathbb{E}[\mathcal{L}(W)] = \mathbb{E}[g_W^{(V)}]$. For the second moment, squaring \eqref{eq:flm_vspace_grad} and using $\mathbb{E}[\varepsilon_{b,\ell,i}] = 0$ and $\mathbb{E}[\varepsilon_{b,\ell,i}\varepsilon_{b',\ell',i}] = \delta_{b,b'}\delta_{\ell,\ell'}$ gives, up to the $O(1/V)$ Stein correlation between a single coordinate $\varepsilon_{b,\ell,i}$ and the $d$-dimensional projection $h_{t,b,\ell}$,
\[
\mathbb{E}\!\left[\bigl(g_{W,i,j}^{\mathrm{noise}}\bigr)^2\right]
= \sum_{b,\ell} (1-t_{b,\ell})^2\,\mathbb{E}\bigl[\varepsilon_{b,\ell,i}^2\,g_{b,\ell,j}^2\bigr]
\approx \sum_{b,\ell} (1-t_{b,\ell})^2\,\mathbb{E}\bigl[g_{b,\ell,j}^2\bigr]
= v_{\mathrm{noise},j}.
\]
Because $g_{W,i}^{\mathrm{sig}}$ and $g_{W,i}^{\mathrm{noise}}$ are uncorrelated to $O(1/V)$, summing across $i=1,\dots,V$ and $j=1,\dots,d$ yields \eqref{eq:flm_second_moment} and \eqref{eq:flm_global_norm}.
\end{proof}

Proposition~\ref{prop:flm_vmatched_adam} immediately explains both failure modes of earlier compressed-training attempts (Figure~\ref{fig:text_flow_scaling}d):
\begin{itemize}
\item \textbf{Differentiable Cholesky with standard Adam:} For any token $i$ absent from the current mini-batch ($g_{W,i}^{\mathrm{sig}}=0$), full-$V$ Adam sees gradient RMS $\sqrt{v_{\mathrm{noise},j}}$, whereas compressed Adam sees only the much smaller covariance gradient $g_{W,i,j}^{\mathrm{cov}}$, whose RMS is smaller by $\sim \sqrt{V/d} \approx 9.9\times$ (directly measured as $9.3\times$ on a trained 80M model: $8.59 \times 10^{-6}$ vs.\ $8.00 \times 10^{-5}$, while \eqref{eq:flm_second_moment} predicts $9.00 \times 10^{-5}$). Because Adam divides by $\sqrt{v_t}$, it amplifies the rare-token shrinkage step by $\sim 9.3\times$, collapsing rare embedding norms ($\text{rare/active} = 0.123$ vs.\ $0.657$ in full-$V$ training) and letting active tokens dominate the input SNR. The model then achieves artificially low cross-entropy ($0.933$ nats) via position-wise token lookup while neglecting multi-token context ($248.4$ PPL at 64 steps).
\item \textbf{Detached Cholesky ($\mathrm{stop\_grad}(L)$):} Detaching $L$ eliminates rare-token shrinkage, but still omits $v_{\mathrm{noise}}$ on active tokens, causing active embedding norms to inflate ($\text{rare/active} = 0.032$, $247.9$ PPL). Manually lowering the embedding learning rate to $0.1\times$ partially mitigates this ($164.9$ PPL at 64 steps), but remains an ad-hoc heuristic.
\end{itemize}

\paragraph{Variance-matched compressed AdamW.}
The exact fix follows directly from Proposition~\ref{prop:flm_vmatched_adam}: keep the low-variance compressed gradient $g_W^{(d)}$ in the first-moment buffer $m_t$, feed the exact variance-matched squared gradient $(g_{W,i,j}^{\mathrm{sig}})^2 + v_{\mathrm{noise},j}$ into the second-moment buffer $v_t$, and include \eqref{eq:flm_global_norm} in global gradient-norm clipping. Because the first moment is variance-reduced while the second moment matches full $V$-space, this update preserves balanced embedding norms across rare and frequent tokens without ever forming an $L \times V$ matrix.

\paragraph{Standalone empirical gains on a standard Transformer (zero wave variables).}
Tested on an 80M standard Transformer trained for 8,000 steps on TinyStories \citep{eldan2023tinystories} (Table~\ref{tab:text_flow_ablation}, Figure~\ref{fig:text_flow_quality}a,c, and Figure~\ref{fig:text_flow_scaling}d), variance-matched compressed AdamW:
\begin{itemize}
\item completely eliminates rare-token norm collapse ($\text{rare/active} = 0.971$ vs.\ $0.123$ for naive Cholesky and $0.657$ for full $V$-space);
\item reduces peak training GPU memory by \textbf{$-33\%$} ($6.99\,\mathrm{GB}$ vs.\ $10.49\,\mathrm{GB}$) and wall-clock step time by \textbf{$-15\%$} ($64.4\,\mathrm{ms}$ vs.\ $75.6\,\mathrm{ms}$);
\item \textbf{improves generation quality over full-vocabulary training at every step budget}---dropping GPT-2 Large perplexity from $222.5 \to \mathbf{152.0}$ at 16 steps, $153.5 \to \mathbf{106.6}$ at 64 steps, $134.2 \to \mathbf{86.6}$ at 256 steps, and $118.9 \to \mathbf{76.8}$ at 1,024 steps---while passing exact full-$V$ cross-checks (validation CE $1.8878$ in $V$-space vs.\ $1.8892$ in $d$-space; 64-step generation PPL $105.82$ in $V$-space vs.\ $105.36$ in $d$-space, bigram diversity $0.535$).
\end{itemize}
Why does variance-matched compression actually \emph{beat} full-vocabulary training rather than merely tie it? In full-$V$ training, frequent tokens still grow to $15\times$ larger squared norm than rare tokens (validation-to-vocabulary SNR ratio $221$), letting the network get low cross-entropy partly by single-token denoising. Keeping embedding norms isotropic (SNR ratio $24$, $\text{rare/active} = 0.971$) forces the Transformer to rely on surrounding sequence context: when $15\%$ of tokens at $t=0.80$ are replaced by pure noise so prediction depends solely on neighboring tokens, masked cross-entropy improves from $4.51$ (full $V$-space) to $4.21$ (variance-matched compressed).

The exact $d$-dimensional sampler \eqref{eq:flm_dspace_ode} also drops directly into the publicly released 170M pretrained FLM/FMLM checkpoints of \citet{lee2026flow} with no retraining (Figure~\ref{fig:text_flow_scaling}c), matching full-vocabulary 4-step sampling on OpenWebText ($116.5$ compressed vs.\ $118.4$ full-$V$, reaching $100.2$ PPL at 8 steps) and One Billion Words ($183.8$ vs.\ $176.6$).

\begin{table}[t]
\centering
\caption{\textbf{Separating compression gains from wave-variable gains (80M matched-parameter TinyStories FLM, 8,000 steps on H100).}
Top block: standard Transformer varying only the input compression and optimizer (zero wave variables). Bottom block: orthogonal-factorized Transformer trained with explicit forward--backward layer sweeps (zero autograd inside blocks). Perplexity is measured by GPT-2 Large ($\downarrow$); all compressed runs pass full-$V$ sampling cross-checks.}
\label{tab:text_flow_ablation}
\setlength{\tabcolsep}{4.5pt}
\small
\begin{tabular}{@{}lrrrrrrrr@{}}
\toprule
Model / Training Method & NFE 4 & NFE 16 & NFE 64 & NFE 256 & NFE 1024 & VRAM & Step & Rare/Act \\
\midrule
\multicolumn{9}{@{}l}{\textit{Part I: Standard Transformer (isolating vocabulary compression alone)}} \\
Full vocabulary ($V$-space) & 379.7 & 220.1 & 153.4 & 134.2 & 118.9 & 10.49\,GB & 75.6\,ms & 0.657 \\
Compressed, naive Cholesky & 796.2 & 403.6 & 248.4 & --- & --- & 7.95\,GB & 61.9\,ms & 0.123 \\
Compressed, detached ($1\times$ emb.\ LR) & 527.0 & 333.7 & 247.9 & --- & --- & 7.94\,GB & 59.7\,ms & 0.032 \\
Compressed, detached ($0.1\times$ emb.\ LR) & 208.6 & 187.7 & 164.9 & --- & --- & 7.94\,GB & 59.5\,ms & --- \\
\textbf{Compressed, var.-matched AdamW} & \textbf{291.3} & \textbf{152.0} & \textbf{106.6} & \textbf{86.6} & \textbf{76.8} & \textbf{6.99\,GB} & \textbf{64.4\,ms} & \textbf{0.971} \\
\midrule
\multicolumn{9}{@{}l}{\textit{Part II: Wave-variable Transformer (orthogonal-factorized layers, no block autograd)}} \\
Full vocabulary ($V$-space) & 490.4 & 222.0 & 132.5$^\dagger$ & 104.1 & 91.3 & 11.68\,GB & 160.0\,ms & 0.695 \\
Compressed, detached Cholesky & 558.8 & 356.6 & 272.5 & --- & --- & 9.53\,GB & 140.1\,ms & --- \\
Compressed, var.-matched ($5\!\times\!10^{-4}$ emb.\ LR) & 428.3 & 198.4 & 130.2 & 103.1 & 92.5 & 8.32\,GB & 138.9\,ms & 0.814 \\
\textbf{Compressed, var.-matched ($3\!\times\!10^{-4}$ emb.\ LR)} & \textbf{307.8} & \textbf{155.7} & \textbf{110.1}$^\ddagger$ & \textbf{90.4} & \textbf{82.1} & \textbf{8.32\,GB} & \textbf{138.0\,ms} & \textbf{0.925} \\
\bottomrule
\end{tabular}
\vspace{2pt}
\raggedright\footnotesize{$^\dagger$129.7 on the extended-NFE seed; 132.5 mean across eval runs. $^\ddagger$107.7 on full-$V$ sampling cross-check (vs.\ 110.1 in $d$-space; 122.2 vs.\ 130.2 at $5\!\times\!10^{-4}$ emb.\ LR), max orthogonality error $5.48\times 10^{-6}$.}
\end{table}

\subsection{Part II: Additional Gains from Wave Variables}
\label{app:flm_wave_extensions}

Independently of input compression, applying the wave-scattering framework of Sections~\ref{sec:variational_worldsheet}--\ref{sec:training_neural_network} to the Transformer layers and the $K$-step sampling trajectory contributes four distinct capabilities.

\paragraph{1. Orthogonal-factorized layers without block autograd.}
Every dense projection inside the Transformer blocks is parameterized in factored form $W = V \operatorname{diag}(s) U^\top$ with orthogonal matrices $U, V$ updated by polar-projected tangent steps and maintained orthogonal to $\le 5.5\text{--}7.5 \times 10^{-6}$. Forward activations and reverse adjoints scatter explicitly across layers without building an autograd tape inside the blocks. Isolating this change on full-vocabulary inputs (Table~\ref{tab:text_flow_ablation}, Part~II) improves 64-step perplexity from $153.4 \to 132.5$ and 1,024-step perplexity from $118.9 \to 91.3$ ($-27.6$ PPL) at matched parameter degrees of freedom ($80.1\text{M}$). Combining these orthogonal-factorized blocks with the variance-matched compression of Part~I already matches or beats full-$V$ quality at the baseline $5\times 10^{-4}$ embedding learning rate ($130.2$ [$122.2$ full-$V$ check] vs.\ $132.5$ at 64 steps), and matching the standard compressed run's $3\times 10^{-4}$ embedding learning rate ($\text{rare/active} = 0.925$) further improves perplexity across every step budget---from $490.4 \to \mathbf{307.8}$ at 4 steps, $222.0 \to \mathbf{155.7}$ at 16 steps, $132.5 \to \mathbf{110.1}$ ($107.7$ full-$V$ check) at 64 steps, $104.1 \to \mathbf{90.4}$ at 256 steps, and $91.3 \to \mathbf{82.1}$ at 1,024 steps---while cutting peak VRAM by $-29\%$ ($8.32$ vs.\ $11.68\,\mathrm{GB}$) and step time by $-14\%$ ($138.0$ vs.\ $160.0\,\mathrm{ms}$), compared to $272.5$ PPL under detached compression.

\begin{figure}[t]
\centering
\includegraphics[width=\textwidth]{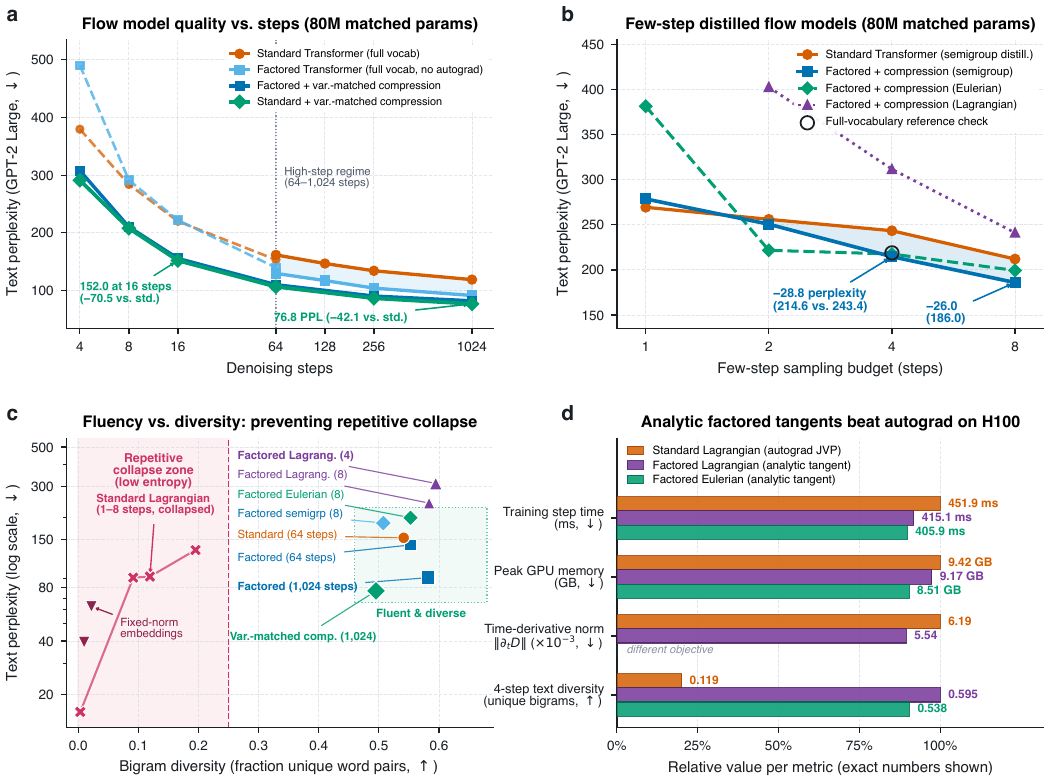}
\caption{\textbf{Generation quality, few-step distillation, and analytic forward-mode tangents (80M matched parameters).}
\textbf{a},~GPT-2 Large perplexity across denoising steps, separating standalone variance-matched compression (teal, $76.8$ PPL at 1,024 steps) from full-$V$ orthogonal-factorized layers (blue, $91.3$ PPL) and the standard full-$V$ baseline (orange, $118.9$ PPL).
\textbf{b},~Few-step distilled flow-map models (1--8 steps), where orthogonal-factorized layers improve semigroup distillation by $-28.8$ PPL at 4 steps and $-26.0$ PPL at 8 steps.
\textbf{c},~Fluency versus bigram diversity: standard Lagrangian distillation collapses into repetitive text (left shaded region), whereas both variance-matched compression and analytic-tangent distillation stay fluent and diverse.
\textbf{d},~H100 step time, memory, tangent norm, and 4-step diversity for analytic tangents versus PyTorch autograd JVP.}
\label{fig:text_flow_quality}
\end{figure}

\paragraph{2. Analytic forward-mode tangents for few-step distillation.}
Distilling a multi-step flow teacher into a 1--8 step flow map via progressive semigroup or Lagrangian/Eulerian self-distillation \citep{boffi2024flow,lee2026flow} requires propagating time-derivative tangents $\partial_t D_{s,t}$ through all Transformer blocks. Using PyTorch's nested autograd JVP on a standard Transformer costs $451.9\,\mathrm{ms}$/step and $9.42\,\mathrm{GB}$ and suffers repetitive bigram collapse (bigram diversity $0.119$, Figure~\ref{fig:text_flow_quality}c--d). Propagating exact forward-mode tangents analytically through the orthogonal-factorized blocks cuts step time to $405.9\text{--}415.1\,\mathrm{ms}$ and memory to $8.51\text{--}9.17\,\mathrm{GB}$, prevents repetitive collapse (bigram diversity $0.538\text{--}0.595$), and improves semigroup distillation by $-28.8$ PPL at 4 steps ($214.6$ vs.\ $243.4$) and $-26.0$ PPL at 8 steps ($186.0$ vs.\ $212.0$; Figure~\ref{fig:text_flow_quality}b).

\paragraph{3. Constant-memory trajectory control for inference-time steering.}
Steering a $K$-step flow sampler toward hard text-infilling constraints or an external reward model (e.g., fluency scored by a smaller GPT-2 Small model) is a $K$-step optimal control problem over the denoising trajectory. Standard backpropagation through time (BPTT) stores all $K$ Transformer evaluations, causing GPU memory to grow linearly with $K$ ($6.08\,\mathrm{GB}$ at $K=4$, $9.55\,\mathrm{GB}$ at $K=8$, $16.51\,\mathrm{GB}$ at $K=16$, and out-of-memory at $K=32$; Figure~\ref{fig:text_flow_control}a). Solving the same Pontryagin control problem by forward--backward wave sweeps across the $K$ denoising steps decouples memory from $K$, holding peak VRAM nearly constant at $4.59\text{--}4.98\,\mathrm{GB}$ across $K=4\text{--}32$ ($5.42\text{--}5.81\,\mathrm{GB}$ with orthogonal-factorized blocks).

On 8-step text infilling (Figure~\ref{fig:text_flow_control}b), this constant-memory controller matches full BPTT ($100.0\%$ exact constraint satisfaction and $151.5$ vs.\ $150.0$ PPL) at half the memory ($4.64$ vs.\ $9.55\,\mathrm{GB}$), whereas best-of-20 and single-step diffusion posterior sampling \citep[DPS;][]{chung2023diffusion} achieve only $2.8\%$ and $2.2\%$ exact accuracy. On GPT-2 Small fluency steering (Figure~\ref{fig:text_flow_control}c), it reduces perplexity from $281.9$ to $111.2$ ($107.2$ on the 4-step distilled model) at nearly half the VRAM of BPTT ($5.85$ vs.\ $10.93\,\mathrm{GB}$). Applying a metric-matched (global-RMS normalized) steering update so unconstrained tokens remain minimally perturbed transfers zero-shot to the pretrained 170M OpenWebText and One Billion Words checkpoints of \citet{lee2026flow} (Figure~\ref{fig:text_flow_control}d), improving OpenWebText constraint satisfaction from $57.8\%$ to $74.1\%$ and lowering perplexity from $888$ to $319$ while cutting memory by $2.5\times$ ($19.4$ vs.\ $48.1\,\mathrm{GB}$).

\begin{figure}[t]
\centering
\includegraphics[width=\textwidth]{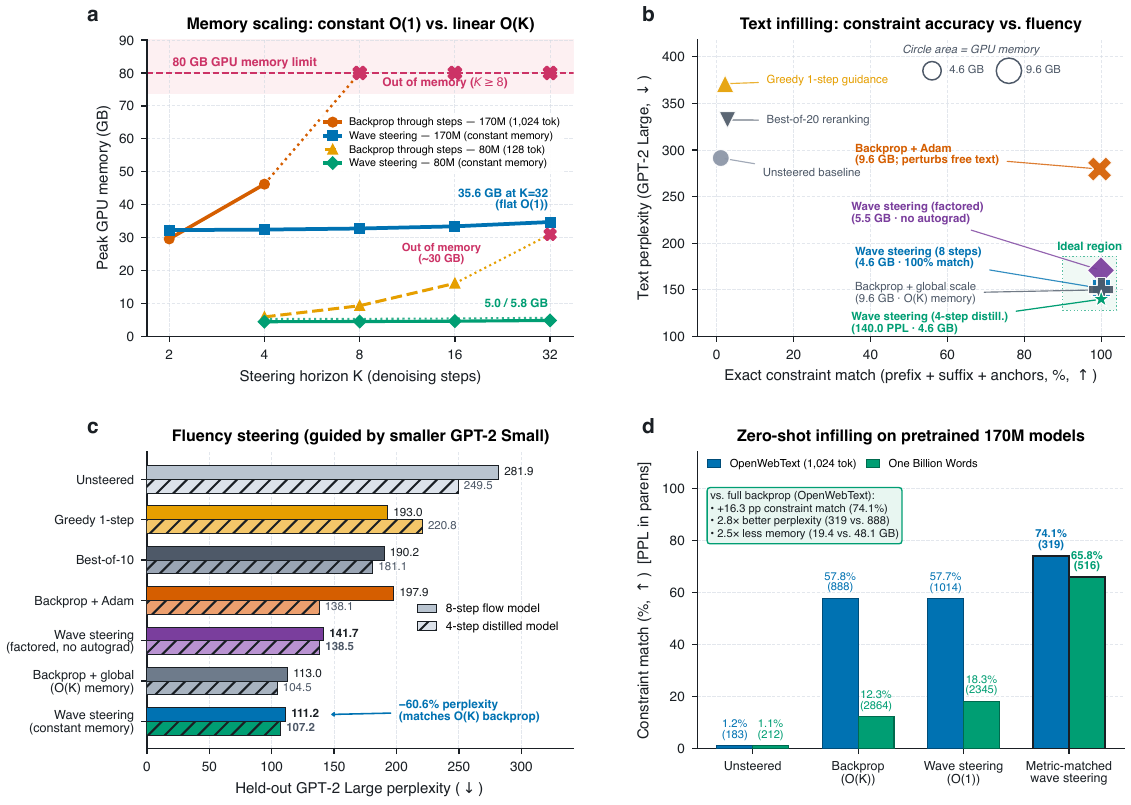}
\caption{\textbf{Constant-memory inference-time steering across the denoising horizon.}
\textbf{a},~Peak H100 memory versus denoising horizon $K$: standard BPTT scales linearly and runs out of memory at $K=32$, while forward--backward wave control stays nearly flat ($4.6\text{--}5.0\,\mathrm{GB}$).
\textbf{b},~8-step text infilling: wave control matches BPTT ($100.0\%$ exact constraints, $151.5$ vs.\ $150.0$ PPL) at half the VRAM, far ahead of best-of-20 ($2.8\%$) and 1-step DPS ($2.2\%$).
\textbf{c},~External fluency steering guided by GPT-2 Small ($281.9 \to 111.2$ PPL at $5.85$ vs.\ $10.93\,\mathrm{GB}$ for BPTT).
\textbf{d},~Zero-shot infilling on pretrained 170M OpenWebText and One Billion Words models.}
\label{fig:text_flow_control}
\end{figure}

\paragraph{4. Long-horizon matrix stability to $10^{18}$ FLOPs.}
To test whether orthogonal-factorized layers remain stable beyond small benchmarks, we trained a 170M bidirectional flow model ($1.27\text{B}$ tokens, $512$-token context, $232\text{k}$ tokens/s, validation CE $4.76 \to 2.84$) and a 162M autoregressive counterpart ($2{,}048$-token context, $287\text{k}$ tokens/s, $0.930$ bits/byte, validation PPL $19.4$, sample PPL $15.5$) up to $10^{18}$ FLOPs on ClimbMix (Figure~\ref{fig:text_flow_scaling}a--b). Across the entire run, post-retraction orthogonality error remained below $7.5 \times 10^{-6}$ (peaking at $1.19 \times 10^{-5}$ on an intermediate odd step between two-step retractions) without loss spikes or gradient explosions.

\begin{figure}[t]
\centering
\includegraphics[width=\textwidth]{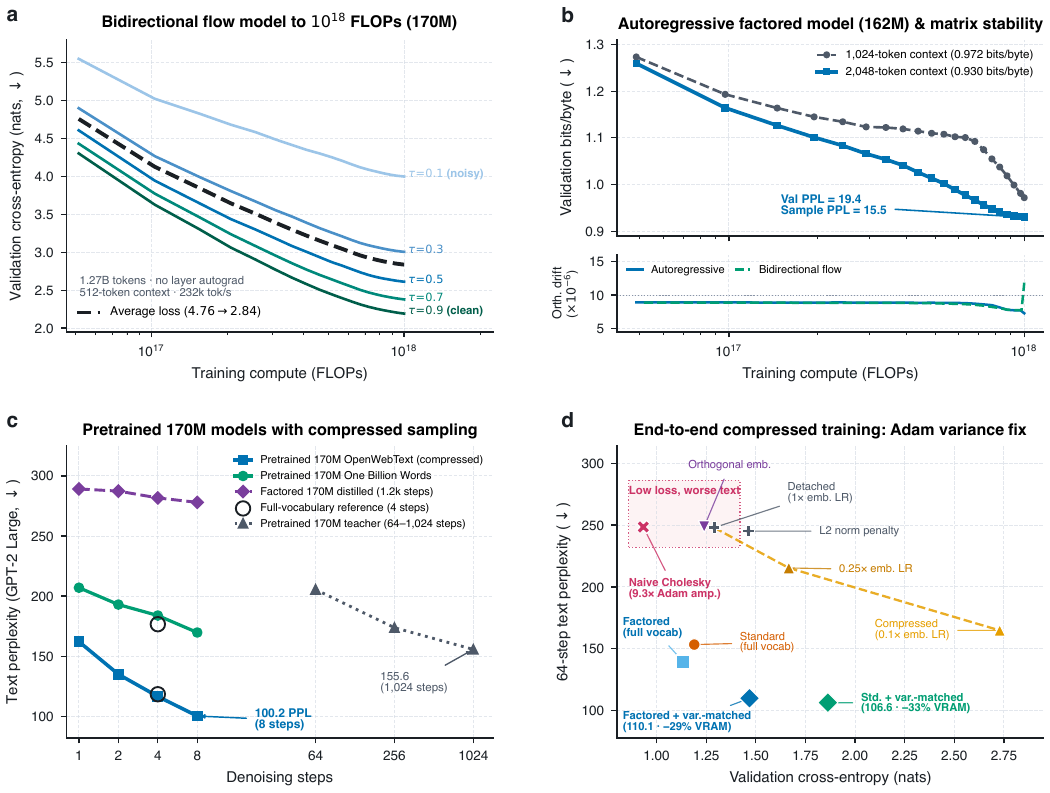}
\caption{\textbf{Scaling to $10^{18}$ FLOPs, pretrained 170M compatibility, and the end-to-end compression fix.}
\textbf{a},~170M bidirectional flow model trained to $10^{18}$ FLOPs without block autograd (average validation CE $4.76 \to 2.84$).
\textbf{b},~162M autoregressive counterpart reaching $0.930$ bits/byte ($19.4$ val PPL, $15.5$ sample PPL) with post-retraction orthogonality drift held below $7.5 \times 10^{-6}$.
\textbf{c},~Zero-shot compressed sampling on pretrained 170M OpenWebText and One Billion Words models matching full-$V$ reference checks (the purple curve shows an orthogonal-factorized model after a short $1.2\text{k}$-step distillation budget).
\textbf{d},~Summary of all 80M end-to-end compression ablations: naive and detached Cholesky fall into the ``low loss, worse text'' trap due to Adam's $9.3\times$ variance mismatch, whereas variance-matched AdamW breaks below the full-vocabulary baselines on both standard ($106.6$ PPL, $-33\%$ VRAM) and orthogonal-factorized ($110.1$ PPL, $-29\%$ VRAM) Transformers.}
\label{fig:text_flow_scaling}
\end{figure}